\documentclass[reqno,10pt,centertags]{amsart}
\usepackage{amsmath,amsthm,amscd,amssymb,latexsym,upref,stmaryrd}
\usepackage{color}
\usepackage{hyperref}

\newcommand*{\mailto}[1]{\href{mailto:#1}{\nolinkurl{#1}}}

\newcommand{\C}{{\mathbb C}}

\newcommand{\bbC}{{\mathbb{C}}}

\newcommand{\bbN}{{\mathbb{N}}}

\newcommand{\bbR}{{\mathbb{R}}}

\newcommand{\bbZ}{{\mathbb{Z}}}

\newcommand{\bfI}{{\mathbf{I}}}

\newcommand{\bsA}{{\boldsymbol{A}}}
\newcommand{\bsB}{{\boldsymbol{B}}}

\newcommand{\bsD}{{\boldsymbol{D}}}

\newcommand{\bsF}{{\boldsymbol{F}}}

\newcommand{\bsH}{{\boldsymbol{H}}}
\newcommand{\bsI}{{\boldsymbol{I}}}
\newcommand{\bsJ}{{\boldsymbol{J}}}
\newcommand{\bsK}{{\boldsymbol{K}}}
\newcommand{\bsL}{{\boldsymbol{L}}}
\newcommand{\bsM}{{\boldsymbol{M}}}
\newcommand{\bsN}{{\boldsymbol{N}}}
\newcommand{\bsP}{{\boldsymbol{P}}}
\newcommand{\bsQ}{{\boldsymbol{Q}}}
\newcommand{\bsR}{{\boldsymbol{R}}}
\newcommand{\bsS}{{\boldsymbol{S}}}
\newcommand{\bsT}{{\boldsymbol{T}}}

\newcommand{\bsV}{{\boldsymbol{V}}}

\newcommand{\cB}{{\mathcal B}}
\newcommand{\cC}{{\mathcal C}}
\newcommand{\cD}{{\mathcal D}}
\newcommand{\cE}{{\mathcal E}}

\newcommand{\cH}{{\mathcal H}}

\newcommand{\cK}{{\mathcal K}}

\newcommand{\cO}{{\mathcal O}}

\newcommand{\cS}{{\mathcal S}}
\newcommand{\cT}{{\mathcal T}}

\newcommand{\cY}{{\mathcal Y}}

\newcommand{\gE}{{\mathfrak{E}}}
\newcommand{\gF}{{\mathfrak{F}}}

\newcommand{\e}{\varepsilon}

\DeclareMathOperator{\sign}{sign}

\DeclareMathOperator{\supp}{supp}

\DeclareMathOperator{\ran}{ran}
\DeclareMathOperator{\dom}{dom}

\DeclareMathOperator{\tr}{tr}

\DeclareMathOperator*{\nlim}{n-lim}
\DeclareMathOperator*{\slim}{s-lim}
\DeclareMathOperator*{\wlim}{w-lim}

\renewcommand{\Re}{\text{\rm Re}}
\renewcommand{\Im}{\text{\rm Im}}
\renewcommand{\ln}{\text{\rm ln}}

\newcommand{\loc}{\text{\rm{loc}}}

\newcommand{\beq}{\begin{equation}}
\newcommand{\enq}{\end{equation}}

\newcommand{\ind}{\operatorname{index}}
\newcommand{\no}{\notag}
\newcommand{\lb}{\label}
\newcommand{\f}{\frac}

\newcommand{\ol}{\overline}

\newcommand{\wti}{\widetilde}
\newcommand{\Oh}{O}

\newcommand{\hatt}{\widehat} 
\newcommand{\bi}{\bibitem}

\renewcommand{\ge}{\geqslant}
\renewcommand{\le}{\leqslant}

\let\geq\geqslant
\let\leq\leqslant

\makeatletter
\def\theequation{\@arabic\c@equation}

\allowdisplaybreaks 
\numberwithin{equation}{section}

\newtheorem{theorem}{Theorem}[section]
\newtheorem{proposition}[theorem]{Proposition}
\newtheorem{lemma}[theorem]{Lemma}
\newtheorem{corollary}[theorem]{Corollary}
\newtheorem{definition}[theorem]{Definition}
\newtheorem{hypothesis}[theorem]{Hypothesis}
\newtheorem{example}[theorem]{Example}

\theoremstyle{remark}
\newtheorem{remark}[theorem]{Remark}

\begin{document}

\numberwithin{equation}{section}
\allowdisplaybreaks

\title[Relatively trace class
  perturbations]{The index formula and the spectral shift function for 
relatively trace class perturbations}

\author[F.\ Gesztesy]{Fritz Gesztesy}   
\address{Department of Mathematics,
University of Missouri, Columbia, MO 65211, USA} 
\email{\mailto{gesztesyf@missouri.edu}}
\urladdr{\url{http://www.math.missouri.edu/personnel/faculty/gesztesyf.html}}

\author[F.\ Gesztesy]{Yuri Latushkin} 
\address{Department of Mathematics,
University of Missouri, Columbia, MO 65211, USA}
\email{\mailto{latushkiny@missouri.edu}}
\urladdr{\url{http://www.math.missouri.edu/personnel/faculty/latushkiny.html}}

\author[K.\ A.\ Makarov]{Konstantin A.\ Makarov}
\address{Department of Mathematics,
University of Missouri, Columbia, MO 65211, USA}
\email{\mailto{makarovk@missouri.edu}}
\urladdr{\url{http://www.math.missouri.edu/personnel/faculty/makarovk.html}} 

\author[F.\ Sukochev]{Fedor Sukochev} 
\address{School of Mathematics and Statistics, UNSW, Kensington, NSW 2052,
Australia} 
\email{\mailto{f.sukochev@unsw.edu.au}}

\author[Y.\ Tomilov]{Yuri Tomilov} 
\address{Faculty of Mathematics and Computer Science, Nicholas 
Copernicus University, ul.\ Chopina 12/18, 87-100 Torun, Poland, and Institute of Mathematics, Polish Academy of Sciences. \'Sniadeckich str. 8, 00-956 Warsaw, Poland}
\email{\mailto{tomilov@mat.uni.torun.pl}}

\thanks{Partially supported by the US National Science
Foundation under Grant NSF DMS-0754705, by the Research Board and Research Council of the University of Missouri, by the ARC, by the  
Marie Curie ''Transfer of Knowledge'' programme, project ``TODEQ'', and by 
a MNiSzW grant Nr.\ N201384834.}

\thanks{{\it Adv. Math.} {\bf 227}, 319--420 (2011).}

\date{\today}

\hspace*{-3mm} 
\begin{abstract} 
We compute the Fredholm index, $\ind(\bsD_{\bsA})$, of the operator 
$\bsD_\bsA^{} = (d/dt) + \bsA$ on $L^2(\bbR;\cH)$ associated with the operator path 
$\{A(t)\}_{t=-\infty}^{\infty}$, where $(\bsA f)(t) = A(t) f(t)$ for a.e.\ $t\in\bbR$, and 
appropriate $f \in L^2(\bbR;\cH)$, via the spectral shift function 
$\xi(\,\cdot\,;A_+,A_-)$ associated 
with the pair $(A_+, A_-)$ of asymptotic operators $A_{\pm}=A(\pm\infty)$ on the separable complex Hilbert space $\cH$ in the case when $A(t)$ is generally an unbounded (relatively trace class) perturbation of the unbounded self-adjoint operator $A_-$. 

We derive a formula (an extension of a formula due to Pushnitski) relating the 
spectral shift function $\xi(\,\cdot\,;A_+,A_-)$ for the pair $(A_+, A_-)$, and the corresponding spectral shift function 
$\xi(\,\cdot\,;\bsH_2, \bsH_1)$ for the pair of operators 
$(\bsH_2, \bsH_1)=(\bsD_\bsA^{} \bsD_\bsA^*, \bsD_\bsA^* \bsD_\bsA^{})$ in this relative trace class context, 
\begin{equation*}
\xi(\lambda; \bsH_2, \bsH_1)=\frac{1}{\pi}\int_{-\lambda^{1/2}}^{\lambda^{1/2}}
\frac{\xi(\nu; A_+,A_-)\, d\nu}{(\lambda-\nu^2)^{1/2}} \, 
\text{ for a.e.\ $\lambda>0$.} 
\end{equation*}

This  formula is then used to identify the Fredholm index of $\bsD_\bsA^{}$ with 
$\xi(0;A_+,A_-)$. In addition, we prove that $\ind(\bsD_\bsA^{})$ coincides with the 
spectral flow $\text{\rm SpFlow} (\{A(t)\}_{t=-\infty}^\infty)$ of the family 
$\{A(t)\}_{t\in\bbR}$ and also relate it to the 
(Fredholm) perturbation determinant for the pair $(A_+, A_-)$:
\begin{align*}
\ind (\bsD_\bsA^{}) &= \text{\rm SpFlow} (\{A(t)\}_{t=-\infty}^\infty)  \\
& = \xi(0;A_+,A_-) \\
&= \pi^{-1} \lim_{\varepsilon \downarrow 0}\Im\big(\ln\big({\det}_{\cH}
\big((A_+ - i \varepsilon I)(A_- - i \varepsilon I)^{-1}\big)\big)\big)  \\
& = \xi(0_+;\bsH_2, \bsH_1)      
\end{align*}
with the choice of the branch of $\ln({\det}_{\cH}(\cdot))$ on $\bbC_+$ such that 
\begin{equation*}
\lim_{\Im(z) \to +\infty}\ln\big({\det}_{\cH} \big((A_+ - z I)(A_- - z I)^{-1}\big)\big) = 0. 
\end{equation*}
 
We also provide some applications in the context of supersymmetric quantum mechanics to zeta function and heat kernel regularized spectral asymmetries and the eta-invariant. 
\end{abstract}

\maketitle

\newpage 

{\scriptsize{\tableofcontents}}
\normalsize

\section{Introduction}  \lb{s1}

Before attempting to describe a glimpse of the extensive history of the underlying problem at hand, viz., the computation of the Fredholm index for operators of the type $\bsD_\bsA^{} = (d/dt) + \bsA$ in $L^2(\bbR;\cH)$, using a variety of different approaches, we briefly describe the principal setup and the main results in this paper.  

Let $\{A(t)\}_{t\in\bbR}$ be a family of self-adjoint operators in the complex, separable Hilbert space $\cH$, subject to a relative trace class approach described in 
Hypothesis \ref{h2.1}, and denote by $\bsA$ the operator in $L^2(\bbR;\cH)$ 
defined by
\begin{align}
&(\bsA f)(t) = A(t) f(t) \, \text{ for a.e.\ $t\in\bbR$,}   \no \\
& f \in \dom(\bsA) = \bigg\{g \in L^2(\bbR;\cH) \,\bigg|\,
g(t)\in \dom(A(t)) \text{ for a.e.\ } t\in\bbR,     \lb{1.1} \\
& \quad t \mapsto A(t)g(t) \text{ is (weakly) measurable,} \,  
\int_{\bbR} \|A(t) g(t)\|_{\cH}^2 \, dt < \infty\bigg\}.   \no 
\end{align}
Our relative trace class setup ensures that $A(t)$ has self-adjoint limiting operators
\beq
  A_+=\lim_{t\to+\infty}A(t), \quad A_-=\lim_{t\to-\infty}A(t)    \lb{1.2}
\enq
in $\cH$ in an appropriate sense (detailed in Theorem \ref{t3.7}). The principal 
novelty in our approach concerns the fact that we permit relative trace class perturbations $B(t)$ (generally, unbounded) of the asymptotic self-adjoint operator 
$A_-$ such that  
\begin{equation}
A(t) = A_- + B(t), \quad t\in\bbR.    \lb{1.2a}
\end{equation}
With the possible exception of a spectral gap at zero, no other restrictions on the 
self-adjoint unperturbed operator $A_-$ are imposed in this paper. Especially, no 
discrete spectrum hypotheses will be made in this paper. 

The first principal result to be mentioned is the extension of the following trace 
formula to our relative trace class approach,  
\begin{align} \lb{1.3}
\begin{split} 
\tr_{L^2(\bbR;\cH)}\big((\bsH_2 - z \, \bsI)^{-1}-(\bsH_1 - z \, \bsI)^{-1}\big) 
=\frac{1}{2z}\tr_\cH \big(g_z(A_+)-g_z(A_-)\big),&   \\ 
z\in\C\backslash [0,\infty),&
\end{split} 
\end{align}
where we used the abbreviations 
\begin{align}
g_z(x) & = x(x^2-z)^{-1/2}, \quad z\in\C\backslash [0,\infty), \; x\in\bbR,    \lb{1.4} \\
\bsD_\bsA^{} & = \f{d}{dt} + \bsA,
\quad \dom(\bsD_\bsA^{})= \dom(d/dt) \cap \dom(\bsA_-),     \lb{1.5}   \\
\bsH_1 &= \bsD_\bsA^* \bsD_\bsA^{}, \quad \bsH_2 = \bsD_\bsA^{} \bsD_\bsA^*,   
\lb{1.6}
\end{align}
and $\bsA_-$  in $L^2(\bbR;\cH)$ represents the self-adjoint (constant fiber) operator defined according to \eqref{1.1} (with $A(t)$ replaced throughout by $A_-$, cf.\ \eqref{2.DA-}). 

The trace formula \eqref{1.3} then implies the next main result, an extension of 
Pushnitski's formula \cite{Pu08} to our relative trace class formalism:
\beq
\xi(\lambda; \bsH_2, \bsH_1)=\frac{1}{\pi}\int_{-\lambda^{1/2}}^{\lambda^{1/2}}
\frac{\xi(\nu; A_+,A_-)\, d\nu}{(\lambda-\nu^2)^{1/2}} \, 
\text{  for a.e.\ $\lambda>0$.}  \lb{1.7}
\enq
Here $\xi(\,\cdot\,; \bsH_2, \bsH_1)$ and $\xi(\,\cdot\,; A_+,A_-)$ denote appropriately defined spectral shift functions associated with the pairs of self-adjoint operators 
$(\bsH_2, \bsH_1)$ and $(A_+, A_-)$, respectively.

Assuming that $A_-$ and $A_+$ are boundedly invertible, we prove that
$\bsD_\bsA^{}$ is a Fredholm operator in $L^2(\bbR;\cH)$. Moreover, one of the 
main results of this paper is the following pair of formulas relating the 
Fredholm index of $\bsD_\bsA^{}$ with the spectral shift function 
$\xi(\,\cdot\,; A_+, A_-)$ (for which formula \eqref{1.7} is then the major input in the proof), and with the trace of a difference of the Morse spectral projections corresponding to 
$(A_+, A_-)$, 
\beq \lb{1.8}
\ind(\bsD_\bsA^{})=\xi(0; A_+, A_-) = {\tr}_{\cH}\big(E_{A_-}((-\infty,0)) - 
E_{A_+}((-\infty,0))\big).    
\enq
Here
$\{E_{T}(\lambda)\}_{\lambda\in\bbR}$ denotes the family of spectral projections associated with the self-adjoint operator $T$. 

However, our results go considerably beyond \eqref{1.8} in the sense that we also establish the detailed connection between the spectral flow for the path 
$\{A(t)\}_{t=-\infty}^{\infty}$ of self-adjoint Fredholm operators and the Fredholm 
index of $\bsD_\bsA^{}$. More precisely, introducing the spectral flow 
$\text{\rm SpFlow} (\{A(t)\}_{t=-\infty}^\infty)$ as in Definition \ref{defSPF}, and recalling the definition of the index of a pair of Fredholm projections in Definition \ref{defFRP}, 
assuming Hypothesis \ref{h2.1} and supposing that $0 \in \rho(A_+)\cap\rho(A_-)$, we 
prove that the pair $\big(E_{A_+}((-\infty,0)),E_{A_-}((-\infty,0))\big)$ of Morse projections is Fredholm and that the following series of equalities holds:
\begin{align}
\ind (\bsD_\bsA^{}) & = \text{\rm SpFlow} (\{A(t)\}_{t=-\infty}^\infty)   \lb{1.9} \\ 
& = \xi(0_+; \bsH_2, \bsH_1)    \lb{1.9a} \\
& = \xi(0; A_+, A_-)   \lb{1.10} \\  
& = \ind(E_{A_-}((-\infty,0)),E_{A_+}((-\infty,0)))    \lb{1.11} \\
& = {\tr}_{\cH}(E_{A_-}((-\infty,0))-E_{A_+}((-\infty,0)))     \lb{1.12} \\ 
& = \pi^{-1} \lim_{\varepsilon \downarrow 0}\Im\big(\ln\big({\det}_{\cH}
\big((A_+ - i \varepsilon I)(A_- - i \varepsilon I)^{-1}\big)\big)\big),       \lb{1.13}
\end{align}
with a choice of branch of $\ln({\det}_{\cH}(\cdot))$ on $\bbC_+$ analogous to 
\eqref{1.14A} below. 

For completeness we note that $\xi(\,\cdot\,; A_+,A_-)$ can be shown to satisfy  
\beq
\xi(\lambda; A_+,A_-) = \pi^{-1}\lim_{\e\downarrow 0} 
\Im(\ln(D_{A_+/A_-}(\lambda+i\e))) \, 
\text{ for a.e.\ } \, \lambda\in\bbR,    \lb{1.14}
\enq
and we make the choice of branch of $\ln(D_{A_+/A_-}(\cdot))$ on $\bbC_+$ such that 
\begin{equation}
\lim_{\Im(z) \to +\infty}\ln(D_{A_+/A_-}(z)) = 0.       \lb{1.14A}  
\end{equation} 
Here 
\begin{equation}
D_{T/S}(z) = {\det}_{\cH} ((T-z I)(S- z I)^{-1}) = {\det}_{\cH}(I+(T-S)(S-z I)^{-1}), 
\quad z \in \rho(S),     \lb{1.15}
\end{equation}
denotes the perturbation determinant for the pair of operators $(S,T)$ in $\cH$, 
assuming $(T-S)(S-z_0)^{-1} \in \cB_1(\cH)$ for some (and hence for all) 
$z_0 \in \rho(S)$. In addition, we recall M.\ Krein's celebrated  trace formula associated 
with the pair $(A_+, A_-)$, 
\begin{align} 
\begin{split} 
\frac{d}{dz}\ln(D_{A_+/A_-}(z)) &= - {\tr}_{\cH}\big((A_+ - z I)^{-1}-(A_- - z I)^{-1}\big)   \\
& = \int_{\bbR}\frac{\xi(\lambda; A_+, A_-)\,d\lambda}{(\lambda-z)^2},    \quad 
z \in \bbC\backslash\bbR.      \lb{1.16}  
\end{split} 
\end{align} 
Analogous formulas apply of course to $\xi(\,\cdot\, ; \bsH_2, \bsH_1)$ in connection 
with the pair $(\bsH_2, \bsH_1)$.

Arguably, equations \eqref{1.3}, \eqref{1.7}, and \eqref{1.9}--\eqref{1.13} represent the central results of this paper to be developed in subsequent sections.  

The concept of spectral flow has also been developed for Breuer--Fredholm 
operators in semifinite von Neumann algebras (see, e.g., \cite{ACS07}, \cite{BCPRSW06}, \cite{Ph96}). In this context a result analogous to those that form the theme of this paper was proved. In particular, a result relating the spectral flow 
and index in the setting of Atiyah's $L^2$-index theorem was derived in 
\cite[Theorem\ 8.4]{BCPRSW06}. Using the fact that the spectral shift function can 
also be defined when working with semifinite von Neumann algebras, it is likely that
extensions of some results of this paper can be made to this wider setting.

Before describing the contents of our paper we now turn to the relevant history of this subject and a proper placement of our results in this context. Since it is impossible to do justice to a discussion of index theory for elliptic differential operators since the pioneering work of Atiyah and Singer, we only confine ourselves referring to a few 
research monographs (see, e.g., \cite{BGV92}, \cite{BB85}, \cite{BBW93}, \cite{Es98}, \cite{Gi84}, \cite{LM89}, \cite{Me93}, \cite{Mu87}, \cite{Pa65}, \cite{Ro88} and the detailed references cited therein). Operators of the form 
$\bsD_{\bsA} = (d/dt) + \bsA$ were studied by Atiyah, Patodi, and Singer 
\cite{APS75}--\cite{APS76} with $A(t)$, $t\in\bbR$, a first-order 
elliptic differential operator on a compact odd-dimensional manifold with the asymptotes $A_\pm$ boundedly invertible and $A_\pm$, $A(t)$, $t\in\bbR$, assumed to have purely 
discrete spectrum. In particular, the idea that the Fredholm index of $D_A$ equals  
the spectral flow of the family (path) of self-adjoint operators $\{A(t)\}_{t=-\infty}^{\infty}$ was put forward in this series of papers. An abstract theorem concerning the equality of the Fredholm index of 
$\bsD_{\bsA}$ and the spectral flow of the family of self-adjoint operators 
$\{A(t)\}_{t=-\infty}^{\infty}$ under the assumption of a $t$-independent domain for 
$A(t)$ which embeds densely and compactly in $\cH$, with boundedly invertible asymptotes $A_\pm$, was proved by Robbin and Salamon \cite{RS95}. This covered 
the abstract case with purely discrete 
spectra for $A_\pm$, $A(t)$, $t\in\bbR$. This paper contains a fascinating array of applications including Morse theory, Floer homology, Morse and Maslov indices, 
Cauchy--Riemann operators, all the way to oscillation theory of (matrix-valued) one-dimensional Schr\"odinger operators. In particular, both, finite and infinite-dimensional cases are treated. An extension of this approach to the Banach space setting appeared in \cite{Ra04}. Examples in which the Fredholm index and the spectral flow cease to coincide and the Fredholm index not only depends on the endpoints $A_\pm$ of the operator path, but on the path itself, are discussed in \cite{AM03}. In a related setting, necessary and sufficient conditions for $\bsD_{\bsA}$ to be Fredholm and an index formula for operators of the form $\bsD_{\bsA}$, given in terms of exponential dichotomies, can be found in \cite{LP08}, \cite{LT05} and the literature cited 
therein; the operator semigroups generated by the operators of this form were 
studied in \cite[Chapter 3]{CL99}.


These references primarily center around  the equality of the Fredholm index and 
the spectral flow as expressed in \eqref{1.9}, a fundamental part of modern index 
theory. However, the connections with the additional equalities in 
\eqref{1.9a}--\eqref{1.13} require quite different ingredients 
whose roots lie at the heart of scattering theory for the pair of self-adjoint operators 
$(\bsH_2, \bsH_1)$ and, especially, that of 
$(A_+, A_-)$, the asymptotes of the operator path $\{A(t)\}_{t=-\infty}^{\infty}$. In particular, we note that the spectral shift function $\xi(\lambda;A_+,A_-)$ (and hence 
boundary values of the perturbation determinant $D_{A_+/A_-}(\lambda+i\e)$ as 
$\varepsilon \downarrow 0$ in \eqref{1.14}) for a.e.\ 
$\lambda \in \sigma_{\rm ac}(A_\pm)$ is directly related to the determinant of 
the $\lambda$-dependent scattering matrix via the 
celebrated Birman--M.\ G.\ Krein formula \cite{BK62}. It is this additional scattering theoretic 
ingredient which represents one of the principal contributions of this paper, and, as 
evidenced in \eqref{1.9a}--\eqref{1.13}, considerably enhances the usual focus on the equality of the Fredholm index and the spectral flow.  

The first relations between Fredholm index theory and the spectral shift function 
$\xi(\,\cdot\,; \bsH_2, \bsH_1)$ were established by Boll\'e, Gesztesy, Grosse, Schweiger, and Simon \cite{BGGSS87}. In fact, inspired by index calculations of Callias \cite{Ca78} 
in connection with noncompact manifolds, the more general notion of the Witten index was studied and identified with $\xi(0_+; \bsH_2, \bsH_1)$ in \cite{BGGSS87} and 
\cite{GS88} (see also \cite{Ge86}, \cite[Ch.\ 5]{Th92}). The latter created considerable  interest, especially, in connection with certain aspects of supersymmetric quantum mechanics. Since a detailed list of references in this context is beyond the scope of this paper we only refer to \cite{An89}, \cite{An90}, \cite{An90a}, \cite{An93}, \cite{An94}, 
\cite{BS78}, \cite{Ko09}, \cite{NS86}, \cite[Ch.\ 5]{Th92} and the detailed lists of references cited therein. While \cite{BGGSS87} and \cite{Ge86} focused on index theorems for concrete one and two-dimensional  supersymmetric systems (in particular, the trace formula \eqref{1.3} and the function $g_z(\cdot)$ were discussed in 
\cite{BGGSS87} and \cite{Ge86} in the special case where $\cH = \bbC$), \cite{GS88} treated abstract Fredholm and Witten indices in terms of the spectral shift function 
$\xi(\,\cdot\,; \bsH_2, \bsH_1)$ and proved their invariance with respect to appropriate classes of perturbations. Soon after, a general abstract approach to supersymmetric scattering theory involving the spectral shift function was developed by Borisov, M\"uller, and Schrader \cite{BMS88} (see also \cite{Bu92}, \cite[Chs.\ IX, X]{Mu87}, \cite{Mu88}) and applied to relative index theorems in the context of manifolds Euclidean at infinity.

However, closest to the present paper at hand, and the prime motivation for writing 
it, is the recent work by Pushnitski \cite{Pu08} in which he went essentially beyond 
the discrete spectrum hypothesis imposed on $A_\pm$, $A(t)$, $t\in\bbR$, by Robbin 
and Salamon in \cite{RS95}. Basically, Pushnitski replaced the discrete spectrum hypothesis by the assumption 
of an arbitrary self-adjoint operator $A_-$ in $\cH$ and by imposing that $B(\cdot)$ 
in \eqref{1.2a} is trace norm differentiable and satisfies the integrability condition
\begin{equation}
\int_{\bbR} \|B'(t)\|_{\cB_1(\cH)} \, dt < \infty.     \lb{1.18}
\end{equation}
Assuming that $A_-$ and $A_+$ are boundedly invertible, Pushnitski proved that
\begin{align}
\begin{split} 
\ind(\bsD_\bsA^{}) &= \xi(0_+; \bsH_2, \bsH_1)   \\
& = \xi(0; A_+, A_-)      \lb{1.19}
\end{split} 
\end{align} 
(cf.\ \eqref{1.9a}, \eqref{1.10}) and indicated why this might imply \eqref{1.9}. Most importantly, perhaps, he proved the trace formula \eqref{1.3} and used it to derive 
his remarkable formula \eqref{1.7}. This effectively removed any discrete spectrum 
assumptions in this context. (Very recently another derivation of \eqref{1.9} without 
any discrete spectrum hypothesis was given in \cite{AW09}, but without entering a discussion of \eqref{1.9a}--\eqref{1.13}.) In the special case where $\cH$ is 
finite-dimensional, the trace formula \eqref{1.3} was first proved by Callias \cite{Ca78}.

Returning to the content of this paper, our relative trace class hypotheses detailed 
in Hypothesis \ref{h2.1} essentially replaces Pushnitski's assumption \eqref{1.18} by 
\begin{equation}
\int_{\bbR} \|B'(t) (|A_-| + I)^{-1}\|_{\cB_1(\cH)} \, dt < \infty     \lb{1.20}
\end{equation}
and certain additional technical conditions, which therefore permit the treatment 
of unbounded operators $B(\cdot)$ in $\cH$. This extension, however, comes at the price of considerably more involved proofs at every stage in this paper. In particular, we 
are using the theory of double operator integrals to justify the trace class 
property of $[g_z(A_+) - g_z(A_-)]$ (cf.\ the right-hand side of the trace formula in 
\eqref{1.3}). Moreover, the assumptions that we impose on the perturbation $B(t)$
and on $B'(t)$ are so general that some fairly delicate analysis of measurability issues 
is required (cf.\ Appendix \ref{sA} and \cite{GGST10}). 

The paper is organized as follows: In Section \ref{s2} we introduce our principal 
Hypothesis \ref{h2.1} and formulate our principal results. Our setup of relatively trace class perturbations is examined in great detail in Section \ref{s3}. Section \ref{s4} 
is of preliminary character and proves a variety of results on $\bsD_{\bsA_-}^{}$, 
$\bsD_{\bsA}$ and sets up the quadratic forms which define $\bsH_j$, $j=1,2$. In Sections \ref{s5} and \ref{s6} we deal with the left-hand side and the right-hand side of the main trace formula \eqref{1.3}, respectively. Whereas Section \ref{s5} employs various quadratic form perturbation results and associated resolvent equations, 
Section \ref{s6} employs the theory of double operator integrals (DOI) originally pioneered by Daletskii and S.\ G.\ Krein and, especially, by Birman and Solomyak. Section \ref{s7} is devoted to a careful introduction and study of the spectral shift function $\xi(\,\cdot\,;A_+,A_-)$ corresponding to the pair 
$(A_+,A_-)$. The spectral shift function $\xi(\,\cdot\,; \bsH_2, \bsH_1)$ associated 
with the pair $(\bsH_2,\bsH_1)$ is then introduced in Section \ref{s8} and the 
fundamental formula \eqref{1.7} as well as the fact that 
$\ind(\bsD_\bsA^{}) = \xi(0_+; \bsH_2, \bsH_1) =\xi(0; A_+, A_-)$ are proved. In addition, some applications to supersymmetric quantum mechanics including abstract formulas for the zeta function  and heat kernel regularized Atiyah--Patodi--Singer (APS) spectral asymmetry and the associated the eta-invariant are provided. Our final 
Section \ref{s9} details the connection between the Fredholm index and the 
spectral flow and proves the remaining equalities in \eqref{1.9}--\eqref{1.13}. 
Appendix \ref{sA} is of a technical nature and takes a close look at operators of 
the type $\bsA$ in \eqref{1.1} and establishes a precise connection with the notion 
of direct integrals over the operators $A(t)$, $t\in\bbR$, with respect to the Lebesgue 
measure $dt$. Appendix \ref{sB} is devoted to a proof of the trace norm analyticity 
of $[g_z(A_+) - g_z(A_-)]$, $z\in\bbC\backslash [0,\infty)$.  
 
Finally, we briefly summarize some of the notation used in this paper: Let $\cH$ be a separable complex Hilbert space, $(\cdot,\cdot)_{\cH}$ the scalar product in $\cH$ (linear in the second factor), and $I$ the identity operator in $\cH$.
Next, let $T$ be a linear operator mapping (a subspace of) a
Banach space into another, with $\dom(T)$, $\ran(T)$, and $\ker(T)$ denoting the
domain, range, and kernel (i.e., null space) of $T$. The closure of a closable 
operator $S$ is denoted by $\ol S$. 

The spectrum, essential spectrum, discrete spectrum, point spectrum, and resolvent set of a closed linear operator in $\cH$ will be denoted by $\sigma(\cdot)$, $\sigma_{\rm ess}(\cdot)$, $\sigma_{\rm d}(\cdot)$, $\sigma_{\rm p}(\cdot)$, and $\rho(\cdot)$, respectively. The strongly right continuous family of spectral projections of a self-adjoint operator $S$ in $\cH$ will be denoted by $E_S(\lambda)$, $\lambda \in \bbR$. (In 
particular, $E_S(\lambda) = E_S((-\infty,\lambda])$, 
$E_S((-\infty, \lambda)) = \slim_{\varepsilon\downarrow 0} E_S(\lambda - \varepsilon)$, and $E_S((\lambda_1,\lambda_2]) = E_S(\lambda_2) - E_S(\lambda_1)$, 
$\lambda_1 < \lambda_2$, $\lambda, \lambda_1, \lambda_2 \in\bbR$.)

The Banach spaces of bounded and compact linear operators on $\cH$ are
denoted by $\cB(\cH)$ and $\cB_\infty(\cH)$, respectively. Similarly,
the Schatten--von Neumann (trace) ideals will subsequently be denoted
by $\cB_p(\cH)$, $p\in (0,\infty)$. Analogous notation $\cB(\cH_1,\cH_2)$,
$\cB_\infty (\cH_1,\cH_2)$, etc., will be used for bounded, compact,
etc., operators between two Hilbert spaces $\cH_1$ and $\cH_2$. We also 
use the notation $\tr_{\cK}(\cdot)$ for the trace in the Hilbert space $\cK$. 
We use symbols $\nlim$, $\slim$ and $\wlim$ to denote the operator norm limit 
(i.e., convergence in the topology of $\cB(\cH)$), and the operator strong and weak limit. 

Throughout, we use the following functions:
\begin{align} 
& g_z(x) = x(x^2-z)^{-1/2},\quad g(x)  = g_{-1}(x) = x(x^2+1)^{-1/2},  \lb{dgvk1} \\
& \varkappa_z(x)=(x^2-z)^{1/2}, \quad 
\varkappa(x)  = \varkappa_{-1}(x) = (x^2+1)^{1/2},   \lb{dgvk2}  \\
& \hspace*{4.9cm}   z\in\C\backslash [0,\infty),\; x\in\bbR.   \no 
\end{align}

Let $A=(A)^\ast$ be a self-adjoint (and generally, unbounded) 
operator on a separable Hilbert space $\cH$, then one can introduce 
the standard scale of spaces $\cH_m(A)$,
$m\in\bbZ$ ($\cH_{0}=\cH$) associated with $A$. In particular, 
$\cH_1(A)$ is given by
$\cH_1(A)= (\dom(A), \|\cdot\|_{\cH_1(A)})$ the domain of $A$ 
equipped with the graph norm
\begin{equation}
\|f\|_{\cH_1(A)}^2 = \|A f\|_{\cH}^2 + \|f\|_{\cH}^2, \quad f \in 
\dom(A),   \lb{grA-}
\end{equation}
and the obvious scalar product $(\cdot,\cdot)_{\cH_1(A)}$ induced by 
\eqref{grA-}, rendering $\cH_1(A)$ a Hilbert space. In addition, one 
notes that
$ \varkappa (A)= (A^2 + I )^{1/2}$ is the isometric isomorphism between
$\cH_1(A)$ and $\cH$. Similarly, $\cH_2(A)=(\dom(A^2), 
\|\cdot\|_{\cH_2(A)})$ denotes
the domain of $A^2$ equipped with the corresponding graph norm. We 
recall that, of course,
\begin{equation}\lb{dfW}
\dom(A^2)=\{w\in\dom(A) \subseteq \cH \,|\, A w \in \dom(A)\}.
\end{equation}

Hilbert spaces of the type $L^2(\bbR; dt; \cH)$ will be denoted by $L^2(\bbR; \cH)$ since only the Lebesgue measure on $\bbR$ will be involved unless explicitly stated otherwise. Analogously, we will also use the shorthand notation $L^2(\bbR; \cB)$ for $L^2(\bbR; dt; \cB)$ in cases where $\cB$ is a Banach space.

Linear operators acting in the Hilbert space $L^2(\bbR; \cH)$ as defined in 
\eqref{1.1}, denoted by boldface letters, $\bsA$, $\bsB$, etc., play a special role in this paper and are discussed in some detail in Appendix \ref{sA}. 

 Given a pair $(A_-,A_+)$ of self-adjoint operators in $\cH$, we will use \eqref{dgvk2} to obtain operators $\varkappa(A_\pm)$, $\varkappa_z(A_\pm)$ in $\cH$ and $\varkappa(\bsA_-)$, $\varkappa_z(\bsA_-)$ in $L^2(\bbR; \cH)$; sometimes, in proofs, we abbreviate:
 \begin{align}
 \varkappa&=\varkappa_-=\varkappa(A_-)=(A_-^2+I)^{1/2},\quad \varkappa_+=\varkappa(A_+)=(A_+^2+I)^{1/2},\label{dvk}\\
 \widehat \varkappa&= \widehat \varkappa_-=\varkappa(\bsA_-)=(\bsA_-^2+\bfI)^{1/2},\quad \widehat \varkappa_-=\varkappa_z(\bsA_-)=(\bsA_-^2-z\bfI)^{1/2},\label{dhvk}
 \end{align}
where the operators in \eqref{dhvk} are acting in
$L^2(\bbR; \cH)$ and  $\bsA_-$ is the constant fiber operator as defined in \eqref{1.1} with $A(t)=A_-$, $t\in\bbR$. 

Finally, $\bbC_+= \{z\in\bbC\,|\, \Im(z) > 0\}$ denotes the open complex upper half-plane.

\section{Principal Results}\lb{s2}

In this section we state our main hypotheses and principal results.

Throughout, we consider a family of closed, symmetric, densely
defined (generally, unbounded) operators $B(t)$, $t\in\bbR$, that are
infinitesimally bounded with respect to $A_-$, and whose weak
derivative is given by the operators $B'(t)$,
$t\in\bbR$, that are relatively trace class with respect to $A_-$ in
the following sense:

\begin{hypothesis} \lb{h2.1}
Suppose $\cH$ is a complex, separable Hilbert space. \\
$(i)$ Assume $A_-$ is self-adjoint on $\dom(A_-) \subseteq \cH$. \\
$(ii)$ Suppose there exists a family of operators $B(t)$, $t\in\bbR$,
closed and symmetric in $\cH$, with $\dom(B(t)) \supseteq \dom(A_-)$, $t\in\bbR$. \\
$(iii)$ Assume there exists a family of operators $B'(t)$,
$t\in\bbR$, closed and symmetric in $\cH$, with 
$\dom(B'(t)) \supseteq \dom(A_-)$, such that the family
$B(t)(|A_-| + I)^{-1}$, $t\in\bbR$, is weakly locally absolutely continuous, and for a.e.\ $t\in\bbR$,  
\beq
\frac{d}{dt} (g,B(t)(|A_-| + I)^{-1}h)_{\cH}
=(g,B'(t) (|A_-| + I)^{-1}h)_{\cH}, \quad g, h\in\cH.     \lb{2.13h}
\enq
$(iv)$ Assume that $B'(t)(|A_-| + I)^{-1}\in\cB_1(\cH)$,  $t\in\bbR$, and
\begin{equation}  \lb{2.13i}
\int_\bbR \big\|B'(t) (|A_-| + I)^{-1}\big\|_{\cB_1(\cH)}\,dt<\infty.
\end{equation}
$(v)$ Suppose that the families 
\begin{equation}
\big\{\big(|B(t)|^2 + I\big)^{-1}\big\}_{t\in\bbR} \, \text{ and } \, 
\big\{\big(|B'(t)|^2 + I\big)^{-1}\big\}_{t\in\bbR}      \lb{2.13B} 
\end{equation}
are weakly measurable $($cf.\ Definition \ref{dA.1}\,$(ii)$$)$.
\end{hypothesis}

For notational simplicity later on, $B'(t)$ was defined for all $t\in\bbR$ in 
Hypothesis \ref{h2.1}\,$(iii)$; it would have been possible to introduce it for a.e.\ $t\in\bbR$ 
from the outset.

We refer to Section \ref{s3} for a thorough discussion of the implications 
of Hypothesis \ref{h2.1} and to Appendix \ref{sA} for a discussion of measurability 
questions of families of closed operators.

As discussed in detail in Section \ref{s3} (cf.\ Theorem \ref{t3.7}), Hypothesis \ref{h2.1} 
implies the existence of a family of self-adjoint operators $\{A(t)\}_{t\in\bbR}$ in $\cH$ given by
\begin{equation}
A(t) = A_- + B(t), \quad \dom(A(t)) = \dom(A_-), \; t\in\bbR,    \lb{2.A(t)}
\end{equation}
as well as a self-adjoint operator $A_+$ in $\cH$ such that
\begin{equation}
\dom(A_+) = \dom(A_-)     \lb{2.dA+}
\end{equation}
and
\begin{equation}
\nlim_{t\to \pm \infty}(A(t) - z I)^{-1} = (A_{\pm} - z I)^{-1}, \quad z\in\bbC\backslash\bbR.    \lb{2.Apm}
\end{equation}
We therefore also introduce
\begin{equation}
B_- = 0, \quad B_+ = \ol{(A_+ - A_-)}, \quad \dom(B_+) \supseteq \dom(A_-), 
\end{equation}
and note that 
\begin{equation}
A_+ = A_- + B_+, \quad \dom(A_+) = \dom(A_-).    \lb{2.A+B+}
\end{equation}

Next, let $\bsA$ in $L^2(\bbR;\cH)$ be then associated with the family $\{A(t)\}_{t\in\bbR}$ in 
$\cH$ by
\begin{align}
&(\bsA f)(t) = A(t) f(t) \, \text{ for a.e.\ $t\in\bbR$,}   \no \\
& f \in \dom(\bsA) = \bigg\{g \in L^2(\bbR;\cH) \,\bigg|\,
g(t)\in \dom(A(t)) \text{ for a.e.\ } t\in\bbR,     \lb{2.bfA} \\
& \quad t \mapsto A(t)g(t) \text{ is (weakly) measurable,} \,  
\int_{\bbR} \|A(t) g(t)\|_{\cH}^2 \, dt < \infty\bigg\}.    \no
\end{align}

To state our results, we start by introducing in $L^2(\bbR;\cH)$ the operator
\begin{equation}
\bsD_\bsA^{} = \f{d}{dt} + \bsA,
\quad \dom(\bsD_\bsA^{})= \dom(d/dt) \cap \dom(\bsA_-).   \lb{2.DA}
\end{equation}
Here the operator $d/dt$ in $L^2(\bbR;\cH)$  is defined by
\begin{align}
\begin{split}
& \bigg(\f{d}{dt}f\bigg)(t) = f'(t) \, \text{ for a.e.\ $t\in\bbR$,} 
\label{2.ddt}  \\
& \, f \in \dom(d/dt) = \big\{g \in L^2(\bbR;\cH) \, \big|\,
g \in AC_{\loc}(\bbR; \cH), \, g' \in L^2(\bbR;\cH)\big\},
\end{split}
\end{align}
especially, 
\begin{align}
& g \in AC_{\loc}(\bbR; \cH) \, \text{ if and only if $g$ is of the form }        \lb{2.ac} \\
& \quad  g(t) = g(t_0) + \int_{t_0}^t h(s) \, ds, \; t, t_0 \in \bbR, \, \text{ for some } \, 
h \in L^1_{\loc}(\bbR;\cH), \text{ and } \, g' = h \, a.e.   \no 
\end{align} 
(The integral in \eqref{2.ac} is of course a Bochner integral.) In addition, 
$\bsA$ is defined in \eqref{2.bfA} and $\bsA_-$  in $L^2(\bbR;\cH)$ represents 
the self-adjoint (constant fiber) operator defined according to
\begin{align}
&(\bsA_- f)(t) = A_- f(t) \, \text{ for a.e.\ $t\in\bbR$,}   \no \\
& f \in \dom(\bsA_-) = \bigg\{g \in L^2(\bbR;\cH) \,\bigg|\,
g(t)\in \dom(A_-) \text{ for a.e.\ } t\in\bbR,    \no \\
& \quad t \mapsto A_- g(t) \text{ is (weakly) measurable,} \,  
\int_{\bbR} \|A_- g(t)\|_{\cH}^2 \, dt < \infty\bigg\}.    \lb{2.DA-}
\end{align} 

Assuming Hypothesis \ref{h2.1}, we will prove in Lemma \ref{l2.4} 
that the operator
$\bsD_\bsA^{}$ is densely defined and closed in $L^2(\bbR; \cH)$. Similarly, 
the adjoint operator
$\bsD_\bsA^*$ of $\bsD_\bsA^{}$ in $L^2(\bbR; \cH)$ is then given by
\begin{equation}
\bsD_\bsA^*=- \f{d}{dt} + \bsA, \quad
\dom(\bsD_\bsA^*) = \dom(d/dt) \cap \dom(\bsA_-) = \dom(\bsD_\bsA^{}).
\end{equation}

Using these operators, we define in $L^2(\bbR;\cH)$ the nonnegative 
self-adjoint operators
\beq\lb{dfnHtH}
\bsH_1=\bsD_\bsA^* \bsD_\bsA^{},\quad \bsH_2=\bsD_\bsA^{} \bsD_\bsA^*.
\enq
Finally, let us define the functions
\begin{align} \lb{dfngz}
\begin{split}
g_z(x) & = x(x^2-z)^{-1/2}, \quad z\in\C\backslash [0,\infty), \; x\in\bbR,  \\
g(x) & = g_{-1}(x) = x(x^2+1)^{-1/2},  \quad x\in\bbR.
\end{split}
\end{align}

Our first principal result relates the trace of the difference of the 
resolvents of $\bsH_1$ and $\bsH_2$ in $L^2(\bbR;\cH)$, and the trace 
of the difference of $g_z(A_+)$ and $g_z(A_-)$ in $\cH$.

\begin{theorem}\lb{Ntrindf}
Assume Hypothesis \ref{h2.1} and define the operators
$\bsH_1$ and $\bsH_2$ as in \eqref{dfnHtH} and the function 
$g_z$ as in \eqref{dfngz}. Then 
\begin{align}
& \big[(\bsH_2 - z \, \bsI)^{-1}-(\bsH_1 - z \, \bsI)^{-1}\big] \in \cB_1\big(L^2(\bbR;\cH)\big),  
\quad z\in\rho(\bsH_1) \cap \rho(\bsH_2),      \lb{s.trH}  \\
& [g_z(A_+)-g_z(A_-)] \in \cB_1(\cH),  \quad 
z\in\C\backslash [0,\infty),     \lb{2.trgz}
\end{align}
and the following trace formula holds,  
\begin{align}
\begin{split}
   \lb{trfOLD}
   \tr_{L^2(\bbR;\cH)}\big((\bsH_2 - z \, \bsI)^{-1}-(\bsH_1 - z \, 
\bsI)^{-1}\big) = \frac{1}{2z}\tr_\cH \big(g_z(A_+)-g_z(A_-)\big),&  \\  
z\in\C\backslash [0,\infty).&  
\end{split}
\end{align} 
\end{theorem}
For notational convenience, cf.\ \eqref{dgvk2} and \eqref{dvk}, we also introduce the self-adjoint operator
\beq
\varkappa = \varkappa(A_-)= \big(A_-^2 + I\big)^{1/2},   \lb{dfnvk1}
\enq
in $\cH$, and for subsequent purposes also the operators
\beq \lb{dfnvk}
\varkappa_z(A_\pm) = \big(A_{\pm}^2 - z I\big)^{1/2}, \quad z \in 
\bbC\backslash [0,\infty).
\enq

We will now outline the main steps in the proof of Theorem \ref{Ntrindf}.
As in \cite{Pu08}, the essential element of our strategy is to pass 
to an appropriate approximation $A_n(t)$. The simplest way to do this 
is to consider the spectral projections $P_n=E_{A_-}((-n,n))$, 
$n\in\bbN$, associated with the operator $A_-$. The projections 
commute with $A_-$, $\varkappa(A_-)$, and their resolvents.
Using $P_n$ just introduced, \eqref{2.A(t)}, \eqref{2.A+B+}, 
\eqref{dfnvk1}, and \eqref{dfnvk}, we define the following operators:
\begin{equation}\lb{defAn}\begin{split} 
& A_n(t) =P_nA(t)P_n, \quad B_n(t)=P_nB(t)P_n, \quad 
B_n'(t) = P_nB'(t)P_n, \\
& A_{\pm,n} =P_nA_{\pm} P_n, \quad
B_n(+\infty) =P_nB(+\infty)P_n, \quad  n\in\bbN.   \end{split}   
\end{equation}
One observes that all operators introduced in \eqref{defAn} are bounded 
operators acting on the space $\cH_n =\ran (P_n)$ which is, in 
general, infinite-dimensional. (To verify this, it suffices to consider 
$A_n(t) = P_n [A(t) (A_- - i I_{\cH})^{-1}][P_n(A_- - i I_{\cH})P_n]$, etc.) For 
an operator $A_n$ acting on 
$\cH_n$, we will keep the same notation $A_n$ to denote the operator 
$A_n\oplus 0$ acting on $\cH=\cH_n \oplus 0$. The proof of the 
following formula \eqref{trfn} uses the main result in \cite{Pu08} applied to 
the bounded approximants $A_n(t)$ of $A(t)$: 

\begin{proposition}\lb{propAppr}
Assume Hypothesis \ref{h2.1}. Then the trace formula \eqref{trfOLD} 
holds for the operators $A_n(t)$, $A_{\pm,n}$ on $\cH$, defined in 
\eqref{defAn}, and the operators $\bsH_{1,n}$ and $\bsH_{2,n}$ on 
$L^2(\bbR;\cH)$, obtained by replacing $A(t)$ by $A_n(t)$ in 
\eqref{dfnHtH}, that is, one has 
\begin{align}
& \lb{trfn}
   \tr_{L^2(\bbR;\cH_n)}\big((\bsH_{2,n} - z  I)^{-1}-(\bsH_{1,n} - z I)^{-1}\big)
   =\frac{1}{2z}\tr_{\cH_n}\big(g_z(A_{+,n})-g_z(A_{-,n})\big),  \no  \\
& \hspace*{9cm} z \in \bbC\backslash [0,\infty). 
\end{align}
\end{proposition}
\begin{proof}
As $P_n=E_{A_-}((-n,n))$ are the spectral projections for $A_-$,
for each fixed $n\in\bbN$, formula \eqref{trfn} follows from 
\cite[Proposition 1.3]{Pu08}  under the assumptions in  Hypothesis 
\ref{h2.1}. Indeed, formula \eqref{trfn} has been proved in
\cite[Proposition 1.3]{Pu08}   under the assumption
\beq\lb{trclass}
\int_\bbR\|B'_n(t)\|_{\cB_1(\cH_n)}\, dt<\infty.
\enq
In the current setting, condition \eqref{2.13i} and relation 
\begin{equation}
B'_n(t)(|A_{-,n}|+I_{\cH_n})^{-1}=P_nB'(t)(|A_{-}|+I_\cH)^{-1}P_n\end{equation}  yield:
\beq\lb{reltrclass}
\int_\bbR\|B'_n(t)(|A_{-,n}|+I_{\cH_n})^{-1}\|_{\cB_1(\cH_n)}\, dt<\infty.
\enq
  Since the operator $|A_{-,n}|+I_{\cH_n}$ is bounded for each  $n\in\bbN$,
  \eqref{reltrclass} implies \eqref{trclass}, and thus \eqref{trfn} holds.
\end{proof}

In view of Proposition \ref{propAppr}, to complete the proof of 
Theorem \ref{Ntrindf} it suffices to pass  to the limit in 
$\cB_1(L^2(\bbR;\cH))$ in the left-hand side and in $\cB_1(\cH)$ in 
the right-hand side of \eqref{trfn} as $n\to\infty$. As a result, 
Theorem \ref{Ntrindf} is a consequence of the following three  
propositions proved, respectively, in Sections \ref{s5},  
\ref{s6}, and \ref{s7} (cf.\ Lemma \ref{l7.trfor}).

\begin{proposition}\lb{prTr}
Assume Hypothesis \ref{h2.1}, and consider the operators $\bsH_1$ and 
$\bsH_2$ defined in \eqref{dfnHtH}, and the operators $\bsH_{1,n}$ 
and $\bsH_{2,n}$ on $L^2(\bbR;\cH)$, obtained by replacing $A(t)$ by 
$A_n(t)$ in \eqref{dfnHtH}. Then, for each $n\in\bbN$ and 
$z\in\bbC \backslash [0,\infty)$, 
\beq\lb{2.21}
[(\bsH_2 - z \, \bsI)^{-1}-(\bsH_1 - z \, \bsI)^{-1}],\, [(\bsH_{2,n} 
- z I)^{-1}-(\bsH_{1,n} - z I)^{-1}] \in\cB_1(L^2(\bbR;\cH))
\enq
and
\begin{align} \lb{2.22}
& \lim_{n\to\infty}\big\|\big[(\bsH_2 - z \, \bsI)^{-1}-(\bsH_1 - z \, 
\bsI)^{-1}\big]     \no \\ 
& \qquad \quad \, -\big[(\bsH_{2,n} - z I)^{-1}-(\bsH_{1,n} - z 
I)^{-1}\big]\big\|_{\cB_1(L^2(\bbR;\cH))} = 0.
\end{align}
\end{proposition}

\begin{proposition}\lb{propgA}
Assume Hypothesis \ref{h2.1}. Consider the operators $A_{\pm}$ 
in \eqref{dfnAplus}, $A_{\pm,n}$ in \eqref{defAn}, and the function 
$g(x)=x(x^2+1)^{-1/2}$, $x\in\bbR$, introduced in \eqref{dfngz}. Then, 
\beq\lb{assrti}
[g(A_+)-g(A_-)],\, [g(A_{+,n})-g(A_{-,n})] \in\cB_1(\cH)
\enq
for each $n\in\bbN$ and
\beq\lb{assrtii}
\lim_{n\to\infty}\|[g(A_+) - g(A_-)] - [g(A_{+,n}) - 
g(A_{-,n})]\|_{\cB_1(\cH)} = 0.
\enq
\end{proposition}

\begin{proposition} \lb{propAnal}
Assume Hypothesis \ref{h2.1} and consider the function 
$g_z(x)=x(x^2-z)^{-1/2}$, $x\in\bbR$, $z\in\bbC\backslash [0,\infty)$ 
introduced in \eqref{dfngz}. Then, 
\begin{equation}
[g_z(A_+)-g_z(A_-)] \in \cB_1(\cH),  \quad z\in\C\backslash [0,\infty),      
\lb{2.tracegz}
\end{equation}
and 
\begin{equation}
\bbC\backslash [0,\infty) \ni z \mapsto 
\tr_{\cH} \big(g_{z}(A_+) - g_{z}(A_-)\big) \, \text{ is analytic.} 
\lb{2.tran}
\end{equation}
\end{proposition}

Assuming Propositions  \ref{prTr}--\ref{propAnal}, one finishes the proof of
Theorem \ref{Ntrindf} as follows:

\begin{proof} [Proof of Theorem \ref{Ntrindf}]
The left-hand side of formula \eqref{trfOLD} is an analytic function with respect 
to $z \in \bbC\backslash [0,\infty)$. By Proposition \ref{propAnal}, also the 
right-hand side is an analytic function with respect to 
$z \in \bbC\backslash [0,\infty)$. By analytic continuation, it suffices to show 
\eqref{trfOLD} for $z<0$ only. But for $z<0$ the
conclusions of Proposition \ref{propgA} hold if $g$ is replaced by 
$g_z$, using $g_z(x)=g(x/(-z)^{1/2})$ and rescaling $A(t)\mapsto
A(t)(-z)^{1/2}$. Thus, passing to the limit as $n\to\infty$ in \eqref{trfn},
and using Propositions \ref{prTr} and \ref{propgA}, one concludes that 
\eqref{trfOLD} holds for $z<0$.
\end{proof}

\begin{remark} \lb{r2.7}
Alternatively, one can derive relation \eqref{2.tracegz} from \eqref{assrti} as follows: One 
considers the smooth function $h_z(x)=((x^2+1)/(x^2-z))^{1/2}$ with 
equal limits $1$ as $x\to \pm\infty$. Then the inclusion \eqref{2.tracegz} follows
from the first assertion in \eqref{assrti}, from the 
formula $g_z(x)=h_z(x)g(x)$, and the representation
\beq
g_z(A_+) - g_z(A_-) = [h_z(A_+)-h_z(A_-)]g(A_+) + h_z(A_-)[g(A_+)-g(A_-)],
\enq
since the inclusion $[h_z(A_+)-h_z(A_-)]\in\cB_1(\cH)$ holds, for instance, by
\cite[Theorem 8.7.1]{Ya92}. 

However, much more is true: In Lemma \ref{lB.1} we will, in fact, prove trace norm analyticity of $[g_z(A_+)-g_z(A_-)]$, $z\in\C\backslash [0,\infty)$, which immediately 
yields analyticity of ${\tr}_{\cH}(g_z(A_+)-g_z(A_-))$, $z\in\C\backslash [0,\infty)$, and hence provides yet another alternative start for proving Theorem \ref{Ntrindf}. 
\end{remark}

The following corollary is one of our principal results saying that the difference of the Morse projections is of trace class.

\begin{corollary} \lb{c2.11}
Assume Hypothesis \ref{h2.1} and suppose that $0 \in \rho(A_+)\cap\rho(A_-)$. 
Then
\begin{equation} 
\big[E_{A_-}((-\infty,0))-E_{A_+}((-\infty,0))\big] \in \cB_1(\cH).  \lb{2.43a}
\end{equation}
\end{corollary}
\begin{proof}
By Hypothesis \ref{h2.1}, $(A_+-A_-)$ is a relatively trace class perturbation of $A_-$, that is 
$(A_+ - A_-)(A_- - z I)^{-1} \in \cB_1(\cH)$, $z\in\bbC\backslash\bbR$ 
(see \eqref{3.Apmrtrcl}), and thus the difference of the resolvents of $A_+$ 
and $A_-$ is of trace class, $\big[(A_+ - z I)^{-1} - (A_- - z I)^{-1}\big] \in \cB_1(\cH)$, 
$z\in\bbC\backslash\bbR$. In this case (cf.\ \cite[Theorem 8.7.1]{Ya92}), one has  
$[f(A_-)-f(A_+)]\in\cB_1(\cH)$ for any function $f$ having two locally bounded derivatives
and satisfying the following conditions:
\begin{align}
(x^2f'(x))'&=O(|x|^{-1-\epsilon}), \,|x|\to\infty,\, \epsilon>0,\label{infc1}\\
\lim_{x\to-\infty}f(x)&=\lim_{x\to+\infty}f(x),\, \lim_{x\to-\infty}x^2f'(x)=\lim_{x\to+\infty}x^2f'(x).\label{infc2}
\end{align}
Since $E_{A_\pm}((-\infty,0))=\frac12\big(I-\sign(A_\pm)\big)$, inclusion \eqref{2.43a} is equivalent to 
\beq \label{inclsign} 
[\sign(A_-)-\sign(A_+)] \in\cB_1(\cH). 
\enq
We choose $\varepsilon_0$ such that $[-\varepsilon_0,\varepsilon_0]\subset\rho(A_-)\cap\rho(A_+)$ and consider a smooth modification $\widetilde{g}$ of the function
$g(x)=x(1+x^2)^{-1/2}$ on $\bbR\backslash\{0\}$ such that $\widetilde{g}(x)=\sign x$ for 
$|x|<\varepsilon_0/2$ and $\widetilde{g}(x)=g(x)$ for $|x|>\varepsilon_0$.
Then $\widetilde{g}(A_\pm)=g(A_\pm)$ since $\widetilde{g}$ and $g$ coincide on the spectrum of $A_\pm$. By the first inclusion in \eqref{assrti} we have
$\big[\widetilde{g}(A_-)-\widetilde{g}(A_+)\big]\in\cB_1(\cH)$, and thus, introducing the function $f(x)=\widetilde{g}(x)-\sign(x)$, the inclusion \eqref{inclsign} is equivalent to $[f(A_-)-f(A_+)]\in\cB_1(\cH)$. But the latter inclusion holds since $f$ satisfies \eqref{infc1}, \eqref{infc2} with $\epsilon=1$.
\end{proof}

Next, we will formulate one of our principal results, relating a 
particular choice of spectral shift functions of the two pairs of 
operators,  $(\bsH_2, \bsH_1)$, and $(A_+,A_-)$. This requires some 
preparations as {\it a priori} in the present general context, the 
Krein spectral shift function for either pair is only defined up to constants.

Since by Theorem \ref{Ntrindf},
\begin{equation}
[g(A_+) - g(A_-)] \in\cB_1(\cH),    
\end{equation}
and $g(A_{\pm})$ are self-adjoint, Krein's trace formula in 
its simplest form (cf.\ \cite[Theorem 8.2.1]{Ya92} yields
\begin{equation}
\tr_{\cH}\big(g(A_+) - g(A_-)\big)
= \int_{[-1,1]} \xi(\omega; g(A_+), g(A_-)) \, d\omega.
\end{equation}
{\it Defining}
\begin{equation}
\xi(\nu; A_+, A_-) := \xi(g (\nu); g(A_+), g(A_-)),   \lb{2.46a}
\quad \nu\in\bbR,
\end{equation}
then $\xi(\,\cdot\,; A_+, A_-) $ can be shown to satisfy 
\begin{equation}
\xi(\,\cdot\,; A_+, A_-) \in L^1\big(\bbR; (|\nu| + 1)^{-2} d\nu\big).  \lb{2.46b}
\end{equation}

Next, one also needs to introduce the spectral shift function 
$\xi(\,\cdot\,; \bsH_2,\bsH_1)$ associated with the pair $(\bsH_2, 
\bsH_1)$. Since $\bsH_2\geq 0$ and
$\bsH_1\geq 0$, and 
\begin{equation}
\big[(\bsH_2 + \bsI)^{-1} - (\bsH_1 + \bsI)^{-1}\big] \in \cB_1 
\big(L^2(\bbR;\cH)\big),
\end{equation}
by Theorem \ref{Ntrindf}, one uniquely introduces $\xi(\,\cdot\,; \bsH_2,\bsH_1)$ by requiring that
\begin{equation}
\xi(\lambda; \bsH_2,\bsH_1) = 0, \quad \lambda < 0,    \lb{2.46c}
\end{equation}
and by 
\begin{align}
\begin{split} 
\tr_{L^2(\bbR;\cH)} \big((\bsH_2 - z \, \bsI)^{-1} - (\bsH_1 - z \, 
\bsI)^{-1}\big)
= - \int_{[0, \infty)}  \frac{\xi(\lambda; \bsH_2, \bsH_1) \, 
d\lambda}{(\lambda -z)^2},&   \\
z\in\bbC\backslash [0,\infty),&  \lb{new8.3} \\
\end{split} 
\end{align}
following \cite[Sect.\ 8.9]{Ya92}.

Given these preparations, we have the following result, an extension of 
Pushnitski's formula \cite{Pu08}, to be proven in Section \ref{s8}.

\begin{theorem}  \lb{t2.8}
Assume Hypothesis \ref{h2.1} and define $\xi(\,\cdot\,; A_+, A_-)$ and
$\xi(\,\cdot\,; \bsH_2,\bsH_1)$ according to \eqref{2.46a} and 
\eqref{2.46c}, \eqref{new8.3}, respectively.
Then one has for a.e.\ $\lambda>0$,
\beq
\xi(\lambda; \bsH_2, \bsH_1)=\frac{1}{\pi}\int_{-\lambda^{1/2}}^{\lambda^{1/2}}
\frac{\xi(\nu; A_+,A_-)\, d\nu}{(\lambda-\nu^2)^{1/2}},   \lb{2.53a}
\enq
with a convergent Lebesgue integral on the right-hand side of \eqref{2.53a}. 
\end{theorem}

Finally, we turn to the connection between the spectral shift function, the spectral 
flow for the path 
$\{A(t)\}_{t=-\infty}^{\infty}$ of self-adjoint Fredholm operators, and the Fredholm 
index of $\bsD_\bsA^{}$ to be studied in detail in Section \ref{s9}. Introducing the spectral flow $\text{\rm SpFlow} (\{A(t)\}_{t=-\infty}^\infty)$ as in Definition \ref{defSPF}, and recalling the definition of the index of a pair of Fredholm projections in Definition \ref{defFRP}, the following result is proved in Theorem \ref{t8.14}, Corollary \ref{c8.index},  and Theorems \ref{thSPFI}. (We note that Theorem \ref{t2.8} is the major input in the proof of the Fredholm index result \eqref{2.54}): 

\begin{theorem}\label{t3.8}
Assume Hypothesis \ref{h2.1} and suppose that $0 \in \rho(A_+)\cap\rho(A_-)$. 
Then the pair $\big(E_{A_+}((-\infty,0)),E_{A_-}((-\infty,0))\big)$ of the Morse projections is Fredholm and  the following equalities hold:
\begin{align}
\ind (\bsD_\bsA^{}) & = \text{\rm SpFlow} (\{A(t)\}_{t=-\infty}^\infty) \no \\ 
& = \xi(0;A_+,A_-)   \lb{2.54} \\  
& = \xi(0_+; \bsH_2, \bsH_1)    \lb{2.54A} \\ 
& = \ind(E_{A_-}((-\infty,0)),E_{A_+}((-\infty,0)))  \lb{2.44a} \\
& = {\tr}_{\cH}(E_{A_-}((-\infty,0))-E_{A_+}((-\infty,0)))  \lb{2.44b} \\ 
& = \pi^{-1} \lim_{\varepsilon \downarrow 0}\Im\big(\ln\big({\det}_{\cH}
\big((A_+ - i \varepsilon I)(A_- - i \varepsilon I)^{-1}\big)\big)\big),       \label{spin}
\end{align}
with a choice of branch of $\ln({\det}_{\cH}(\cdot))$ on $\bbC_+$ analogous to 
\eqref{2.52A} below. 
\end{theorem}

Here we note that $\xi$ can be shown to satisfy (cf.\ Theorem \ref{t8.14})  
\beq
\xi(\lambda; A_+,A_-)=\pi^{-1}\lim_{\e\downarrow 0} 
\Im(\ln(D_{A_+/A_-}(\lambda+i\e))) \, 
\text{ for a.e.\ } \, \lambda\in\bbR,    \lb{3.xidet}
\enq
and we make the choice of branch of $\ln(D_{A_+/A_-}(\cdot))$ on $\bbC_+$ such that 
\begin{equation}
\lim_{\Im(z) \to +\infty}\ln(D_{A_+/A_-}(z)) = 0,     \lb{2.52A}
\end{equation} 
with 
\begin{equation}
D_{T/S}(z) = {\det}_{\cH} ((T-z I)(S- z I)^{-1}) = {\det}_{\cH}(I+(T-S)(S-z I)^{-1}), 
\quad z \in \rho(S), 
\end{equation}
denoting the perturbation determinant for the pair of operators $(S,T)$ in $\cH$, 
assuming $(T-S)(S-z_0)^{-1} \in \cB_1(\cH)$ for some (and hence for all) 
$z_0 \in \rho(S)$. In addition, we recall,   
\begin{align} 
\begin{split} 
\frac{d}{dz}\ln(D_{A_+/A_-}(z)) &= - {\tr}_{\cH}\big((A_+ - z I)^{-1}-(A_- - z I)^{-1}\big)   \\
& = \int_{\bbR}\frac{\xi(\lambda; A_+, A_-)\,d\lambda}{(\lambda-z)^2},    \quad 
z \in \bbC\backslash\bbR,     \lb{2.restr}  
\end{split} 
\end{align} 
the trace formula associated with the pair $(A_+, A_-)$.

\section{The Relative Trace Class Setting} \lb{s3}

Throughout this section we assume Hypothesis \ref{h2.1} and closely examine the basic assumptions made in it. 

\subsection{A Thorough Analysis of the Main Hypothesis}
We start with the following auxiliary result: 

\begin{lemma} \lb{l3.1}
Let $\cH$ be a complex, separable Hilbert space and 
$\bbR \ni t \mapsto F(t) \in \cB_1(\cH)$. Then the following assertions $(i)$ and $(ii)$ 
are equivalent: \\
$(i)$ $\{F(t)\}_{t\in\bbR}$ is a weakly measurable family of operators in $\cB(\cH)$ 
and $\|F(\cdot)\|_{\cB_1(\cH)} \in L^1(\bbR; dt)$.  \\
$(ii)$ $F\in L^1(\bbR; \cB_1(\cH))$. 

Moreover, if either condition $(i)$ or $(ii)$ holds, then
\begin{equation}
\bigg\|\int_{\bbR} F(t)\, dt\bigg\|_{\cB_1(\cH)} \leq 
\int_{\bbR} \|F(t)\|_{\cB_1(\cH)} \, dt    \lb{3.BI}
\end{equation}  
and the $\cB_1(\cH)$-valued function
\begin{equation}
\bbR \ni t \mapsto \int_{t_0}^t F(s) \, ds, \quad t_0 \in \bbR\cup \{-\infty\},  
\lb{3.AC}
\end{equation}
is strongly absolutely continuous with respect to the norm in $\cB_1(\cH)$. \\
In addition we recall the following fact: \\
$(iii)$ Suppose that $\bbR \ni t \mapsto G(t) \in \cB_1(\cH)$ is strongly locally absolutely 
continuous in $\cB_1(\cH)$. Then $H(t) = G'(t)$ exists for a.e.\ $t\in\bbR$, $H(\cdot)$ is 
Bochner integrable over any compact interval, and hence
\begin{equation}
G(t) = G(t_0) + \int_{t_0}^t H(s) \, ds, \quad t, t_0 \in \bbR.   \lb{3.FTC}
\end{equation}
\end{lemma}
\begin{proof}
Clearly, condition $(ii)$ implies condition $(i)$. 

To prove the converse statement, that is, condition $(i)$ implies condition $(ii)$,  
one can argue as follows. Let $F: \bbR \to \cB_1(\cH)$ be a weakly measurable 
function in the sense that for every $f,g \in \cH$, the function $(f, F(\cdot) g)_{\cH}$ is  measurable on $\bbR$, and suppose that 
$\|F(\cdot)\|_{\cB_1(\cH)} \in L^1(\bbR; dt)$.

One recalls that $\cB_1(\cH)$ is a separable Banach space. Hence, if $F$ is weakly measurable in $\cB_1(\cH)$, then it is measurable in $\cB_1(\cH)$ by Pettis' theorem (cf., e.g., \cite[Theorem\ 1.1.1]{ABHN01}, \cite[Theorem\ II.1.2]{DU77}, 
\cite[3.5.3]{HP85}). Moreover, one recalls that for fixed $A \in \cB(\cH)$, 
$\tr_{\cH}(TA)$, $T\in \cB_1(\cH)$ is a continuous functional on $\cB_1(\cH)$ with 
norm $\|A\|_{\cB(\cH)}$, and every continuous functional on $\cB_1(\cH)$ is obtained in this manner. In particular, one can identify $\cB_1(\cH)^*$ and $\cB(\cH)$ as Banach spaces. 

Next, one notes that $F$ is weakly measurable in $\cB_1(\cH)$ if and only if 
$ \tr_{\cH} (F(\cdot) A)$ is measurable on $\bbR$ for every $A$ taken from a 
separating set $\cS \subseteq (\cB_1(\cH))^*=\cB(\cH)$ (cf.\ \cite[Corollary\ 1.1.3]{ABHN01}). As $\cS$ one may take, for example, the set $\cO$ of rank-one operators on $\cH$. It will clearly be separating since for any $0 \neq T \in \cB_1(\cH)$ one can find an operator $A = (f_0, \cdot)_{\cH} \, g_0 \in \cO$ such that 
$\tr_{\cH} (T A)= (f_0,T g_0)_{\cH} \neq 0$. However, by hypothesis, 
$\tr_{\cH} (F(\cdot) A)=(f_0, F(\cdot) g_0)_{\cH}$ is measurable
on $\bbR$ for every $A = (f_0, \cdot)_{\cH} \, g_0$. Thus $F$ is weakly measurable and hence measurable in $\cB_1(\cH)$. Moreover, since also 
$\|F(\cdot)\|_{\cB_1(\cH)} \in L^1(\bbR; dt)$ by assumption, $F$ is Bochner integrable in $\cB_1(\cH)$ by Bochner's theorem (cf., e,g., \cite[Theorem\ 1.14]{ABHN01}, 
\cite[Theorem\ II.2.2]{DU77}, \cite[Theorem\ 3.7.4]{HP85}). 

The estimate \eqref{3.BI} and the strong absolute continuity of the function in 
\eqref{3.AC} is well-known in the context of Bochner integrals (cf., e.g., 
\cite[p.\ 6--21]{ABHN01}, \cite[p.\ 44--50]{DU77}, \cite[p.\ 71--88]{HP85}).

Finally, also \eqref{3.FTC} is standard (cf., e.g., \cite[Proposition\ 1.2.3]{ABHN01}) since 
$\cB_1(\cH)$ has the Radon--Nikodym property by the Dunford--Pettis Theorem 
(cf., e.g., \cite[Theorem\ 1.2.6]{ABHN01}) as 
$(\cB_\infty(\cH))^* = \cB_1(\cH)$ is a separable dual space (cf., e.g., 
\cite[Theorem\ III.7.1]{GK69}, \cite[Sect.\ IV.1]{Sc60}).
\end{proof}

An application of Lemma \ref{l3.1} yields the following observations:
 
\begin{remark}  \lb{r2.2}
Hypothesis \ref{h2.1}\,(iii) implies that 
 \beq\label{2.10a}
\{B'(t)(|A_-| + I)^{-1}\}_{t\in\bbR} \, \text{ is weakly measurable}
\enq 
since for all $g, h\in\cH$, $(g,B'(\cdot) (|A_-| + I)^{-1}h)_{\cH}$ arises as a pointwise 
a.e.\ limit of measurable functions. Thus, applying Lemma \ref{l3.1}, one concludes 
that assumption \eqref{2.13i},
\beq\lb{2.11a}
\int_\bbR \big\|B'(t) (|A_-| + I)^{-1}\big\|_{\cB_1(\cH)}\,dt<\infty, 
\enq
together with condition \eqref{2.10a}, are equivalent to the (seemingly stronger) condition
\begin{equation}
B'(\cdot)(|A_-| + I)^{-1} \in L^1(\bbR;\cB_1(\cH)).   \lb{2.12a}
\end{equation}
In particular, it would have been possible to just assume
$B'(\cdot)(|A_-| + I)^{-1} \in L^1(\bbR;\cB_1(\cH))$ in Hypothesis \ref{h2.1}\,$(iv)$. 
\end{remark}

\begin{remark}  \lb{r2.3}
We temporarily introduce the Bochner integral in $\cB_1(\cH)$, 
\begin{equation}
C(t) = \int_{-\infty}^t B'(s) (|A_-| + I)^{-1} \,ds \in \cB_1(\cH), \quad t \in \bbR. 
\end{equation}
Applying Lemma \ref{l3.1}\,$(iii)$, one concludes that 
\begin{equation}
C'(t) = B'(t) (|A_-| + I)^{-1} \, \text{ for a.e.\  $t \in \bbR$,} 
\end{equation}
and hence, in particular, for all $f, g \in \cH$, 
\begin{equation}
(f,C'(t) g)_{\cH} = (f,B'(t) (|A_-| + I)^{-1} g)_{\cH} \, \text{ for a.e.\  $t \in \bbR$.} 
\end{equation}
Thus, by Hypothesis \ref{h2.1}\,$(iii)$,  
\begin{align}
\f{d}{dt} (f, C(t) g)_{\cH} &= (f,C'(t) g)_{\cH} = (f,B'(t) (|A_-| + I)^{-1} g)_{\cH}   \no \\
& = \f{d}{dt} (f, B(t)  (|A_-| + I)^{-1} g)_{\cH}  \, \text{ for a.e.\  $t \in \bbR$.} 
\end{align}
Consequently, one arrives at 
\begin{equation}
C(t) = B(t)  (|A_-| + I)^{-1} + C_0 \, \text{ for some $C_0 \in\cB_1(\cH)$.}
\end{equation}
In particular, one infers that 
\begin{equation}
\lim_{t\to -\infty} B(t) (|A_-| + I)^{-1} = D_- \, \text{ exists in the $\cB_1(\cH)$-norm.}  \lb{3.D-} 
\end{equation}

We now choose the convenient normalization 
\begin{equation}
D_- = 0    \lb{3.D-=0}
\end{equation}
and hence obtain 
\beq
B(t) (|A_-| + I)^{-1} = \int_{-\infty}^t B'(s) (|A_-| + I)^{-1} \,ds \in \cB_1(\cH), \quad t \in \bbR,  
\lb{2.13jk}
\enq
(a fact that will be used later in the proof of Lemma \ref{lHYP}), and hence one also has the estimate
\beq
\big\|B(t) (|A_-| + I)^{-1}\big\|_{\cB_1(\cH)}
\leq \int_{-\infty}^t \big\|B'(s) (|A_-| + I)^{-1}\big\|_{\cB_1(\cH)} 
\,ds, \quad t \in \bbR.   \lb{2.13jl}
\enq
\end{remark}

In the following we draw some conclusions from Hypothesis \ref{h2.1}:

We start by recalling the following standard convergence property for 
trace ideals:

\begin{lemma}\lb{lSTP}
Let $p\in[1,\infty)$ and assume that $R,R_n,T,T_n\in\cB(\cH)$, 
$n\in\bbN$, satisfy
$\slim_{n\to\infty}R_n = R$  and $\slim_{n\to \infty}T_n = T$ and that
$S,S_n\in\cB_p(\cH)$, $n\in\bbN$, satisfy 
$\lim_{n\to\infty}\|S_n-S\|_{\cB_p(\cH)}=0$.
Then $\lim_{n\to\infty}\|R_n S_n T_n^\ast - R S T^\ast\|_{\cB_p(\cH)}=0$.
\end{lemma}
This follows, for instance, from \cite[Theorem 1]{Gr73}, \cite[p.\ 
28--29]{Si05}, or
\cite[Lemma 6.1.3]{Ya92} with a minor additional effort (taking 
adjoints, etc.). We note that by the uniform boundedness principle, weak (and hence strong) convergence of $R_n\in\cB(\cH)$ to an operator $R\in\cB(\cH)$ implies the uniform boundedness of the sequence $\{R_n\}_{n\in\bbN}$, that is, the existence of a constant $C \in (0,\infty)$ such that $\sup_{n\in\bbN} \|R_n\|_{\cB(\cH)} \leq C$ and 
$\|R\|_{\cB(\cH)} \leq \liminf_{n\to\infty} \|R_n\|_{\cB(\cH)}$ 
(cf., e.g., \cite[Theorem\ 4.26]{We80}). (In particular, the uniform boundedness hypothesis $\sup_{n\in\bbN} \|R_n\|_{\cB(\cH)} \leq C$ (and similarly for $T_n$) 
used in \cite[p.\ 28]{Si05} need not be assumed in Lemma \ref{lSTP}.) 

\begin{lemma}\lb{lHYP}
Assume Hypotheses \ref{h2.1} and introduce the open cone  
$C_{\varepsilon} = \{z\in\bbC \,|\, |\arg(z)| < \varepsilon\}$ for some 
$\varepsilon \in(0, \pi/2)$. Then
\begin{equation}
\sup_{t\in\bbR} \big\|B(t) (|A_-| - z I)^{-1}\big\|_{\cB_1(\cH)}
\underset{\substack{z\to \infty \\ z \notin C_{\varepsilon}}}{=} o(1).     \lb{2.13g}   
\end{equation}
\end{lemma}
\begin{proof}
One estimates, assuming for simplicity that $|z| \geq 1$, $z\notin C_{\varepsilon}$, 
\begin{align}
& \sup_{t\in\bbR} \big\|B(t) (|A_-| - z I)^{-1}\big\|_{\cB_1(\cH)}=
\sup_{t\in\bbR}\bigg\|\int_{-\infty}^tB'(s) (|A_-| -
z)^{-1}\,ds \bigg\|_{\cB_1(\cH)}  \no  \\
& \quad \le\int_{\bbR} \big\|B'(s) (|A_-| - z 
I)^{-1}\big\|_{\cB_1(\cH)}\,ds  \lb{Bprime}  \\
& \quad =\int_{\bbR} \big\|B'(s)(|A_-| + I)^{-1} (|A_-| + I)
(|A_-| - z I)^{-1}\big\|_{\cB_1(\cH)}\,ds    \no  \\
& \quad \le \big\|(|A_-| + I) (|A_-| - z I)^{-1}\big\|_{\cB(\cH)}
\int_{\bbR} \big\|B'(s)(|A_-| + I)^{-1}\big\|_{\cB_1(\cH)}\,ds <\infty   
\end{align}
due to condition \eqref{2.13i}, since
\beq\lb{bddAm}
\big\|(|A_-| + I) (|A_-| - z I)^{-1}\big\|_{\cB(\cH)} \leq c(\varepsilon), \quad 
|z|\geq 1, \; z\notin C_{\varepsilon}, 
\enq
for some constant $c(\varepsilon) > 0$. By the dominated convergence theorem and 
\eqref{Bprime}, it remains to show that
\beq \lb{convz}
\big\|B'(s) (|A_-| - z I)^{-1}\big\|_{\cB_1(\cH)}
\underset{\substack{|z|\to \infty \\ z \notin C_{\varepsilon}}}{\longrightarrow}0 \,
\text{ for each $s\in\bbR$.}
\enq
Introducing the normal operators
$W_z=(|A_-| + I) (|A_-| - z I)^{-1}$, $W_z^*=(|A_-| + I) (|A_-| - \ol{z} I)^{-1}$, 
$|z|\geq 1$, $z\notin C_{\varepsilon}$, the norms of $W_z$
are uniformly bounded due to \eqref{bddAm}. In addition, one has
$\|(|A_-| - z I)^{-1}\|_{\cB(\cH)}\to0$
as $|z|\to \infty$, $z\notin C_{\varepsilon}$, and thus for all $f\in\dom(|A_-|)$, 
$(W_z)^* f\to0$ in $\cH$ as $|z|\to \infty$, $z\notin C_{\varepsilon}$.
Since $\dom(|A_-|)$ is dense in $\cH$, it follows that
$(W_z)^* \to0$ strongly in $\cH$ as $|z|\to \infty$, $z\notin C_{\varepsilon}$. 
Due to the fact that
\beq
B'(s) (|A_-| - z I)^{-1}=B'(s) (|A_-| + I)^{-1}W_z,
\enq
and the operator $B'(s) (|A_-| + I)^{-1}$ is in $\cB_1(\cH)$,
Lemma \ref{lSTP} implies \eqref{convz}. 
\end{proof}

\begin{remark}  \lb{r2.6}
Since $B(t)$ and $B'(t)$ are symmetric with 
$\dom(B(t))\cap\dom(B'(t))\supseteq\dom(A_-)$, $t\in\bbR$, one 
concludes that
\begin{align}
\begin{split} 
& B(t)^* (|A_-|+I)^{-1} = B(t) (|A_-|+I)^{-1}, \\ 
& (B'(t))^* (|A_-|+I)^{-1} = B'(t) (|A_-|+I)^{-1}, \quad   t\in\bbR.   \lb{2.19}
\end{split} 
\end{align}
Consequently, \eqref{2.13i}, \eqref{2.13jk}, \eqref{2.13jl}, \eqref{2.12a}, and 
\eqref{2.13g} hold with $B(t)$, $B'(t)$ replaced by $B(t)^*$, $(B'(t))^*$, respectively. 
\end{remark}

Next, assuming Hypothesis \ref{h2.1}, we recall the definition of the family of operators 
$\{A(t)\}_{t\in\bbR}$ in $\cH$ with constant domain $\dom(A_-)$ (cf.\ \eqref{2.A(t)}) by 
\begin{equation}\lb{defAT}
A(t)=A_-+B(t),\quad \dom(A(t))=\dom(A_-), \quad  t\in\bbR,
\end{equation}
and note that (cf.\ \eqref{2.13g})
\begin{align} 
\begin{split}
\sup_{t\in\bbR}\|A(t)\|_{\cB(\cH_1(A_-),\cH)}
& = \sup_{t\in\bbR} \|A(t)(|A_-| + I)^{-1}\|_{\cB(\cH)}    \lb {eq2.22} \\ 
& = \sup_{t\in\bbR} \|[A_-+B(t)](|A_-| + I)^{-1}\|_{\cB(\cH)}<\infty.
\end{split} 
\end{align}

We now turn to a closer examination of the family $\{A(t)\}_{t\in\bbR}$:

\begin{theorem} \lb{t3.7}
Assume Hypothesis \ref{h2.1} and define $A(t)$, $t\in\bbR$, as in \eqref{defAT}. Then the following assertions hold: \\
$(i)$ For all $t\in\bbR$, $A(t)$ with domain $\dom(A(t)) = \dom(A_-)$ is self-adjoint in $\cH$. \\
$(ii)$ For all $t\in\bbR$, $B(t)$ is relatively trace class with respect to $A_-$, that is, 
\begin{equation}
B(t) (A_- - z I)^{-1} \in \cB_1(\cH), \quad z\in\bbC\backslash\bbR, \; t\in\bbR.   \lb{3.B(t)rtrcl}
\end{equation}
$(iii)$ There exists a self-adjoint operator $A_+$ in $\cH$ such that 
\begin{equation}
\dom(A_+) = \dom(A_-),
\end{equation} 
and
\begin{equation}
\nlim_{t\to \pm \infty} (A(t) - z I)^{-1} = (A_{\pm} - z I)^{-1}, 
\quad z\in\bbC\backslash\bbR.      \lb{3.Anlim}
\end{equation}
$(iv)$ $(A_+ - A_-)$ is relatively trace class with respect to $A_-$, that is, 
\begin{equation}
(A_+ - A_-) (A_- - z I)^{-1} \in \cB_1(\cH), \quad z\in\bbC\backslash\bbR.   
\lb{3.Apmrtrcl} 
\end{equation}
$(v)$ One has 
\begin{align}
& \big[(A(t) - z I)^{-1} - (A_- - z I)^{-1}\big] \in \cB_1(\cH), \quad  t\in\bbR, 
\; z \in \bbC\backslash\bbR,    \lb{3.trA(t)} \\
& \big[(A_+ - z I)^{-1} - (A_- - z I)^{-1}\big] \in \cB_1(\cH), \quad  z \in \bbC\backslash\bbR,  
 \lb{3.trApm}
\end{align}
and hence, 
\beq
\sigma_{\text{ess}}(A(t))=\sigma_{\text{ess}}(A_-)
=\sigma_{\text{ess}}(A_+), \quad  t\in\bbR.     \lb{3.esssp}
\enq
\end{theorem}
\begin{proof} 
$(i)$ Self-adjointness of $A(t)$ on $\dom(A(t)) = \dom(A_-)$ for all $t\in\bbR$ immediately 
follows from \eqref{2.13g}, which implies 
\begin{equation}
\big\|B(t) (A_- - z I)^{-1}\big\|_{\cB(\cH)} < 1 \, \text{ for $|\Im(z)|>0$ sufficiently large,}   \lb{3.KR}
\end{equation}
and the Kato--Rellich Theorem (cf.\ \cite[Theorem\ V.4.3]{Ka80}). \\
$(ii)$ This instantly follows from \eqref{2.13g}. \\
$(iii)$ Since by \eqref{2.12a}, $B'(t)(|A_-| + I)^{-1} \in L^1(\bbR;\cB_1(\cH))$, one infers in addition 
to \eqref{3.D-}  and \eqref{3.D-=0} that  
\begin{equation}
\lim_{t\to \pm \infty} B(t) (A_- - z I)^{-1} = \begin{cases} D_+(z) \\
0 \end{cases}  \text{exist in the $\cB_1(\cH)$-norm}      \lb{3.Dpm}
\end{equation}
for $|\Im(z)|>0$ sufficiently large. 
Moreover, by \eqref{3.KR} and \eqref{2.13g} (in fact, in this context it would be sufficient to replace 
$\cB_1(\cH)$ by $\cB(\cH)$ in \eqref{2.13g}) one has that 
\begin{equation}
\big[I + B(t) (A_- - z I)^{-1}\big]^{-1}, \, [I + D_+(z)]^{-1}  \in \cB(\cH) \, 
\text{ for $|\Im(z)|>0$ sufficiently large.}    \lb{3.BD}
\end{equation}
Employing the second resolvent equation for $A(t)$ one obtains, using \eqref{3.BD}, 
\begin{equation}
(A(t) - z I)^{-1} = (A_- - z I)^{-1} - (A(t) - z I)^{-1} \big[B(t) (A_- - z I)^{-1}\big],  \quad t\in\bbR,  
\lb{3.resA(t)}  
\end{equation}
for $|\Im(z)|>0$ sufficiently large. Thus, applying \eqref{3.Dpm}, one obtains 
\begin{equation}
\nlim_{t\to\pm\infty} (A(t) - z I)^{-1} = \begin{cases} (A_- - z I)^{-1} [I + D_+(z)]^{-1} \\
(A_- - z I)^{-1} \end{cases}
\end{equation}
for $|\Im(z)|>0$ sufficiently large, and hence also 
\begin{equation}
(A(t) -z I)^{-1} = (A_- - z I)^{-1} \big[I + B(t) (A_- - z I)^{-1}\big]^{-1},  \quad t\in\bbR,  \lb{3.ResA(t)}
\end{equation}
for $|\Im(z)|>0$ sufficiently large.

Next, one notes that the strong (and hence in particular the norm) limit of resolvents of self-adjoint operators is necessarily a pseudoresolvent. The latter is the resolvent of a closed, linear operator if 
and only if the $z$-independent nullspace of the pseudoresolvent equals $\{0\}$ 
(cf.\ \cite[Sect.\ VIII.1.1]{Ka80}). Since 
\begin{equation}
\ker\big((A_- - z I)^{-1} [I + D_+(z)]^{-1}\big) = \{0\} \, \text{ for $|\Im(z)|>0$ sufficiently large,} 
\end{equation}
one thus concludes that 
\begin{equation}
\nlim_{t\to\pm\infty} (A(t) - z I)^{-1} = (A_{\pm} - z I)^{-1} \text{ for $|\Im(z)|>0$ sufficiently large,}  
\lb{3.nlim}
\end{equation}
for some closed, linear operator $A_+$ in $\cH$. Thus, \eqref{3.resA(t)} yields
\begin{equation}
(A_+ - z I)^{-1} = (A_- - z I)^{-1} - (A_+ - z I)^{-1} D_+(z)    \lb{3.resAp}
\end{equation}
for $|\Im(z)|>0$ sufficiently large, and hence (cf.\ also \eqref{3.ResA(t)})
\begin{equation}
(A_+ - z I)^{-1} = (A_- - z I)^{-1} [I + D_+(z)]^{-1} \text{ for $|\Im(z)|>0$ sufficiently large.} 
\lb{3.ApmD}
\end{equation}
Equation \eqref{3.ApmD} then yields 
\begin{equation}
(A_+ - z I )= [I + D_+(z)] (A_- - z I) \text{ for $|\Im(z)|>0$ sufficiently large,}   \lb{3.AD+}
\end{equation}
and hence confirms that $\dom(A_+) = \dom(A_-)$.  
Self-adjointness of $A_+$ then follows from
\begin{align}
\nlim_{t\to\infty}\big[(A(t) - z I)^{-1} \big]^* = \nlim_{t\to\infty}(A(t) - {\ol z} I)^{-1} 
& = (A_+ - {\ol z} I)^{-1} = \big[(A_+ - z I)^{-1} \big]^*   \no \\
& = (A_+^* - {\ol z} I)^{-1}  
\end{align}
for $|\Im(z)|>0$ sufficiently large. Having established self-adjointness of 
$A_{\pm}$, an analytic continuation with respect to $z$ in \eqref{3.nlim}
 then yields \eqref{3.Anlim}. \\
$(iv)$ This immediately follows from \eqref{3.Dpm} and \eqref{3.AD+}, which imply  
\begin{equation}
(A_+ - A_-) (A_- - z I)^{-1} = D_+(z) \in \cB_1(\cH) \, \text{ for $|\Im(z)|>0$ sufficiently large.}     
\lb{3.Artrcl}
\end{equation}
An analytic continuation with respect to $z$ in \eqref{3.Artrcl} then yields \eqref{3.Apmrtrcl}. \\
$(v)$ Relation \eqref{3.trA(t)} follows from \eqref{3.B(t)rtrcl} and \eqref{3.resA(t)}, relation 
\eqref{3.trApm} follows from \eqref{3.Dpm} and \eqref{3.resAp}. Finally, \eqref{3.esssp} follows from 
\eqref{3.trA(t)} and \eqref{3.trApm} (in fact, replacing $\cB_1(\cH)$ by 
$\cB_{\infty}(\cH)$ would be sufficient for this purpose in both equations) and Weyl's Theorem (cf., e.g., \cite[Corollary\ XIII.4.2]{RS78}).  
\end{proof}

Given Theorem \ref{t3.7} one can introduce the densely defined, symmetric (and hence closable)
operator $\dot B(+\infty)$ in $\cH$ by
\begin{equation}
\dot B(+\infty) = A_+ - A_-, \quad \dom(\dot B(+\infty)) = \dom(A_-),   \lb{3.dotB+}
\end{equation}
and its closure $B(+\infty)$ in $\cH$,
\begin{equation}
B(+\infty) = \ol{\dot B(+\infty)}, \quad \dom(B(+\infty)) \supseteq \dom(A_-).    \lb{3.B+}
\end{equation}
In addition, and in accordance with our normalization $D_- = 0$ in \eqref{3.D-=0}, we also introduce
\begin{equation}
B(-\infty) = 0, \quad \dom(B(-\infty)) = \cH.    \lb{3.B-}
\end{equation} 

By \eqref{2.13jk}, \eqref{3.Dpm}, and $D_+(z) = B(+\infty) (A_- - z I)^{-1}$, 
and recalling notation \eqref{dfnvk1},
 one may thus summarize some of the properties of $B(t)$, $B(+\infty)$ by  
\begin{align}\lb{limB}
& \lim_{t\to\infty}\|[B(t)-B(+\infty)](A_-^2+I)^{-1/2}\|_{\cB_1(\cH)}=0, \\
& B(+\infty)(A_-^2+I)^{-1/2}= \int_\bbR B'(s)(A_-^2+I)^{-1/2}\,ds \in\cB_1(\cH),    \lb{intB} \\
& B(t)(A_-^2+I)^{-1/2} =  \int_{-\infty}^t B'(s) 
(A_-^2+I)^{-1/2}\,ds \in\cB_1(\cH), \quad t\in\bbR.   \lb{intBp}
\end{align}
Finally, one also has
\beq\lb{dfnAplus}
A_+=A_-+B(+\infty),\quad\dom(A_+)=\dom(A_-).
\enq 

Next, we denote by $\cH_{1/2}(|A|)$ the domain of the operator $|A|^{1/2}$ equipped with its graph norm. The following lemma shows that the graph norms associated with $A(t)$ and $|A(t)|^{1/2}$, respectively, are equivalent for different $t$ with constants {\em uniform} with respect to $t\in\bbR$:

\begin{lemma}\label{UNIFnorms} 
Assume Hypothesis \ref{h2.1}.
Then there are positive constants $c_1$ and $c_2$ such that for all $t\in\bbR$ one has, 
\begin{align} \label{UNn1} 
\|f\|_{\cH_1(A_-)} & \le c_1 \|f\|_{\cH_1(A(t))} \le c_2\|f\|_{\cH_1(A_-)},\\  
&\hspace*{.74cm} f\in\dom(A_-)=\dom(A(t)),   \no \\
\|f\|_{\cH_{1/2}(|A_-|)} & \le c_1\|f\|_{\cH_{1/2}(|A(t)|)} \le 
c_2\|f\|_{\cH_{1/2}(|A_-|)},\label{UNn2} \\ 
&\hspace*{.1cm}  f\in\dom\big(|A_-|^{1/2}\big)=\dom\big(|A(t)|^{1/2}\big).   \no 
\end{align}
\end{lemma}
\begin{proof} Since $B(t)$ is relatively compact with respect to $A_-$,
one concludes that (cf., \cite[Theorems\ 9.4(b), 9.7, 9.9]{We80}) 
\begin{align}
\begin{split} 
& \dom(A_-)= \dom(|A_-|) = \dom(|A(t)|) = \dom(A(t)), \\ 
& \dom\big(|A_-|^{1/2}\big)=\dom\big(|A(t)|^{1/2}\big), \quad t\in\bbR. 
\end{split}
\end{align}
For each $t\in\bbR$, the set $\dom(|A(t)|)$ is a core for $|A(t)|^{1/2}$ (see, e.g., 
\cite[Theorem\ V.3.24]{Ka80}). This implies that \eqref{UNn2} follows from \eqref{UNn1}. The second inequality in \eqref{UNn1} is just a reformulation of \eqref{eq2.22}. 
To prove the first inequality in \eqref{UNn1}, we will use Lemma \ref{lHYP}: Fix 
$z \in \bbC\backslash [0,\infty)$ such that $\|B(t)(|A_-|-zI)^{-1}\|_{\cB(\cH)}^2<1/6$, uniformly with respect to $t\in\bbR$. Then, for each $f\in\dom(A_-)$,
\begin{align}
\|f\|_{\cH_1(A_-)}^2 &=\|f\|_\cH^2+\|A(t)f-B(t)f\|_\cH^2\le\|f\|_\cH^2 
+\big(\|A(t)f\|_\cH+\|B(t)f\|_\cH\big)^2  \no \\
&\le \|f\|_\cH^2+2\big(\|A(t)f\|_\cH^2+\|B(t)(|A_-|-zI)^{-1}(|A_-|-zI)f\|_\cH^2\big)  \no \\
&\le \|f\|_{\cH}^2+2\|A(t)f\|_\cH^2+(2/3)\big(\|\,|A_-|f\|_\cH^2+z^2\|f\|_\cH^2\big) \no \\
&\le c(z) \|f\|_{\cH_1(A(t))}^2 + (2/3)\|f\|_{\cH_1(A_-)}^2,
\end{align}
where $c(z)$ is independent of $t$.
\end{proof}

\begin{remark} \label{r2.8} 
Given the operators $\varkappa(A_\pm)=\big((A_\pm)^2+I\big)^{1/2}$ with 
$\dom (\varkappa(A_+)) = \dom (\varkappa(A_-)) = \dom (A_-)$, one concludes that 
$\varkappa(A_-)^{1/2}\varkappa(A_+)^{-1/2}\in\cB(\cH)$ by
the closed graph theorem (cf.\ \cite[Remark IV.1.5]{Ka80}). 
Passing to the adjoint (cf.\ \cite[Theorem\ 4.19\,(b)]{We80}),
one infers that $\varkappa(A_+)^{-1/2}\varkappa(A_-)^{1/2}
\subseteq \big[\varkappa(A_-)^{1/2}\varkappa(A_+)^{-1/2}\big]^*\in\cB(\cH)$ and hence 
 \begin{equation}
 \overline{\varkappa(A_+)^{-1/2}\varkappa(A_-)^{1/2}}
 =\big[\varkappa(A_-)^{1/2}\varkappa(A_+)^{-1/2}\big]^*\in\cB(\cH).   \lb{3.kappapm}
 \end{equation}
\end{remark}

\subsection{The Role of $N$-Measurability}
We continue this section with some remarks concerning the relevance of 
Hypothesis \ref{h2.1}\,$(v)$. Let $\bsT$ in $L^2(\bbR; \cH)$ be defined in terms of the weakly measurable family of densely defined, closed, linear operators 
$T(t)$, $t\in\bbR$, in $\cH$ in analogy to \eqref{A.17}, that is, 
\begin{align}
&(\bsT f)(t) = T(t) f(t) \, \text{ for a.e.\ $t\in\bbR$,}   \no \\
& f \in \dom(\bsT) = \bigg\{g \in L^2(\bbR;\cH) \,\bigg|\,
g(t)\in \dom(T(t)) \text{ for a.e.\ } t\in\bbR,     \lb{2.28}  \\
& \quad t \mapsto T(t)g(t) \text{ is (weakly) measurable,} \,  
\int_{\bbR} \|T(t) g(t)\|_{\cH}^2 \, dt < \infty\bigg\}.   \no 
\end{align}
Then $\bsT$ is closed in $L^2(\bbR; \cH)$, but may not be densely defined. Also,
it is of interest to know if $\bsT$ can be written as the direct integral of the operators $T(t)$. Adding the hypothesis that the family $\{T(t)\}_{t\in\bbR}$ is $N$-measurable (cf.\ the discussion of 
$N$-measurability in Appendix \ref{sA}) guarantees that 
$\bsT$ is densely defined by Theorem \ref{tA.6}. In 
particular, one then has 
\begin{equation}
\bsT = \int_{\bbR}^{\oplus} T(t) \, dt, \quad \bsT^* = \int_{\bbR}^{\oplus} T(t)^* \, dt, 
\quad |\bsT| = \int_{\bbR}^{\oplus} |T(t)| \, dt,    \lb{2.29}
\end{equation}
moreover, the remaining analogs of the direct integral formulas in Theorem \ref{tA.6} 
(such as \eqref{A.21}, \eqref{A.22}) apply to $\bsT$ as well. 

\begin{remark}\label{rem.condv}
We will show in Lemma \ref{l2.9} that Hypotheses \ref{h2.1}\,$(i)$--$(iv)$, in addition to 
Hypothesis \ref{h2.1}\,$(v)$, imply that $\{B(t)\}_{t\in\bbR}$ and 
$\{B'(t)\}_{t\in\bbR}$ are $N$-measurable as introduced in Definition \ref{dA.1}\,$(iii)$ and further discussed in Remark \ref{rA.2}\,$(iv)$. Consequently, $\bsB$ and $\bsB'$, defined according to \eqref{2.28},  
are densely defined in $L^2(\bbR; \cH)$, and the analogs of \eqref{2.29} hold in either case by Theorem \ref{tA.6}. 
\end{remark}

\begin{remark} \lb{r2.10}
$(i)$ Assuming Hypothesis \ref{h2.1}, the weak measurability of $\{B(t)\}_{t\in\bbR}$ 
and $\{B'(t)\}_{t\in\bbR}$, proven in Lemma \ref{l2.9}, yield an alternative and 
direct proof (without relying on Theorem \ref{tA.6}) that $\bsB$ and $\bsB'$ are densely defined in $L^2(\bbR; \cH)$ as follows: Since the function $B'(\,\cdot\,)(|A|+I)^{-1}$ is weakly measurable, for each $f \in L^2(\bbR; \cH)$ with compact support, the function  
$B'(\,\cdot\,)(|A|+I)^{-1} f$ taking values in $L^2(\bbR; \cH)$ is weakly measurable as well. The fact that 
\begin{equation}
S:=\{(|A_{-}|+I)^{-1}f \,|\, {\rm ess \, supp} (f) \, \text{compact}\} \, 
\text{ is dense in $L^2(\bbR; \cH)$,} 
\end{equation}
then implies that the maximal domain of
$\bsB'$ is dense in $L^2(\bbR; \cH)$. Analogous ideas yield that the maximal
domain of $\bsB$ is dense in $L^2(\bbR; \cH)$. 
To see that $S$ is indeed dense in $L^2(\bbR; \cH)$, one argues as follows: Assume that 
there exists $f \in L^2(\bbR; \cH)$ such that $(f,g)_{L^2(\bbR; \cH)} =0$  for 
every $g \in S$. Then $((|A_{-}|+I)^{-1}f, \wti g)_{L^2(\bbR; \cH)} =0$ 
for every $\wti g \in L^2(\bbR; \cH)$ with compact support. Since the latter set is 
dense in $L^2(\bbR; \cH)$ one gets $(|A_{-}|+I)^{-1}f = 0$ a.e.\ Since 
$(|A_{-}|+I)^{-1}$ is injective in $\cH$, $f =0$ a.e., that is, the set $S$ is 
dense in $L^2(\bbR; \cH)$. \\
$(ii)$ It is of course possible to interchange $B(t)$ by $B(t)^*$ in \eqref{2.13B}, and analogously, one may replace $B'(t)$ by $(B'(t))^*$ in \eqref{2.13B}.
\end{remark}

\begin{remark}\label{essHv} We will show by means of Example \ref{e2.11}
that Hypothesis \ref{h2.1}$(v)$ is essential, and cannot be derived from 
assertions $(i)$--$(iv)$ in Hypothesis \ref{h2.1}; in particular, we will show 
that weak measurability of the family 
$\big\{\big(|B'(t)|^2 + I\big)^{-1}\big\}_{t\in\bbR}$ does not follow from weak 
measurability of $\{B'(t)\}_{t\in\bbR}$ and weak measurability of 
$\big\{B'(t)(|A_-| + I)^{-1}\big\}_{t\in\bbR}$. 
\end{remark}

\begin{remark}
In the special case where $\dom(A_-)$ is a core for $B(t)$ for all $t\in\bbR$, that is,
\begin{equation}
\ol{B(t)\big|_{\dom(A_-)}} = B(t), \quad t \in\bbR, 
\end{equation}
an application of Lennon's \cite{Le74} result \eqref{A.43} then yields 
$N$-measurability of the family $\{B(t)\}_{t\in\bbR}$ and 
\begin{align}
\begin{split} 
\bsB &= \int_{\bbR}^{\oplus} B(t) \, dt 
= \int_{\bbR}^{\oplus} \ol{\big[B(t) (|A_-| + I)^{-1}\big] (|A_-| + I)^{-1}} \, dt  \\
&= \ol{\big[\bsB (|\bsA_-| + I)^{-1}\big] (|\bsA_-| + I)^{-1}} 
= \ol{\bsB\big|_{\dom(\bsA_-)}}. 
\end{split} 
\end{align}
Using \eqref{2.13i} and \eqref{2.13jl}, one concludes that (cf.\ \eqref{normT} below)
\begin{equation}
\big\|\bsB (|\bsA_-| - z \, \bsI)^{-1}\big\|_{\cB(L^2(\bbR;\cH))}   
= \sup_{t\in\bbR} \big\|B(t) (|A_-| - z I)^{-1}\big\|_{\cB(\cH)} < \infty. 
\end{equation}
\end{remark}

\begin{remark} \lb{TnormT} 
In the particular case where $T(t)\in\cB(\cH)$, $t\in\bbR$, and 
$T(\cdot)\in L^\infty(\bbR;\cB(\cH))$, the operator $\bsT$ defined in 
\eqref{2.28} is bounded in $L^2(\bbR;\cH)$ and
\begin{equation} \lb{normT}
\|\bsT\|_{\cB(L^2(\bbR;\cH))}=\sup_{t\in\bbR}\|T(t)\|_{\cB(\cH)}.
\end{equation} 
\end{remark}

\subsection{Some Multi-Dimensional PDE Examples} 
We conclude this section with two elementary examples illustrating the feasibility of  Hypothesis \ref{h2.1}.

\begin{example} \lb{e3.13} 
Let $n\in\bbN$, $p > n$, $q \in ((n/2), p -(n/2))$, and $\varepsilon > 0$.  
Consider
\begin{align}
& 0 \leq V_- \in L^2(\bbR^n; (1 + |x|^2)^q d^n x) \cap L^\infty(\bbR^n; d^n x), \label{n3.78}\\
& 0 \leq V(t,\cdot) \in L^2(\bbR^n; (1 + |x|^2)^q d^n x) \cap L^\infty(\bbR^n; d^n x), 
\quad t\in\bbR, \label{n3.79}
\end{align}
and suppose in addition that
\begin{align}
& \partial_tV(t,\cdot) \in L^2(\bbR^n; (1 + |x|^2)^q d^n x) \cap L^\infty(\bbR^n; d^n x), 
\quad t\in\bbR,      \\
& \bbR\ni t\mapsto V(t,\cdot) \in C^1(\bbR; L^\infty(\bbR^n; d^n x)). 
\end{align}
Denoting the operator of multiplication by $V_-$, $V$, and $\partial_tV$ in 
$L^2(\bbR^n; d^n x)$ by the same symbol, respectively, we introduce the linear operators 
\begin{align}
& A_- = (-\Delta)^{p/2} + V_- + \varepsilon I, \quad 
\dom(A_-) = \dom\big((-\Delta)^{p/2}\big),    \lb{3.A-} \\
& B(t) = V(t,\cdot) - V_-, \quad \dom(B(t)) = L^2(\bbR^n; d^n x), \quad t\in\bbR, \\
& A(t) = A_- + B(t), \quad \dom(A(t)) = \dom(A_-), \quad t\in\bbR,   \lb{3.At}
\end{align}
in $L^2(\bbR^n; d^n x)$, with $-\Delta$ abbreviating the self-adjoint Laplacian in 
$L^2(\bbR^n; d^n x)$ whose graph domain equals the usual Sobolev space 
$W^{2,2}(\bbR^n)$.  

Repeatedly applying \cite[Corollary\ 4.8]{Si05}, one verifies that all assumptions 
in Hypothesis \ref{h2.1} are satisfied. Specifically, since 
\begin{align}
\begin{split} 
& (|k|^2 + 1)^{- p/2} \in L^2\big(\bbR^n; (1 + |k|^2)^q d^n k\big),     \\
& V_-, V(t,\cdot) \in L^2\big(\bbR^n; (1 + |x|^2)^q d^n x\big), \quad t\in\bbR, 
\end{split} 
\end{align}
\cite[Corollary\ 4.8]{Si05} implies that 
\begin{equation}
V_- \big((-\Delta)^{p/2} + I\big)^{-1}, \, 
V(t,\cdot) \big((-\Delta)^{p/2} + I\big)^{-1} \in \cB_1\big(L^2(\bbR^n; d^n x)\big),   
\quad t \in \bbR.      \lb{3.trc}
\end{equation}
In addition, one has 
\begin{equation}
\sigma (A(t)) = \sigma (A_-) = [\varepsilon, \infty), \quad t\in\bbR.     \lb{3.85}
\end{equation}
Indeed, to show \eqref{3.85}, one recalls that $V\ge 0$ and $V_-\ge 0$, and 
since both operators are relatively compact (in fact, relatively trace class) with 
respect to $(-\Delta)^{p/2}$ by \eqref{3.trc}, and hence also with respect to 
$A_-$ and $A(t)$ (cf.\ \eqref{3.A-}, \eqref{3.At}), one obtains
\begin{align}
& \sigma(A_-)\subseteq[\varepsilon,\infty), \quad 
\sigma(A(t))\subseteq[\varepsilon,\infty), \quad  t\in\bbR,       \\
& \sigma_{\rm ess}(A_-)=\sigma_{\rm ess}(A(t))=[\varepsilon,\infty), \quad  t\in\bbR, 
\end{align}
implying \eqref{3.85}.
\end{example}

We note that $L^2(\bbR^n; (1 + |x|^2)^q d^n x)$, $q>(n/2)$, in Example \ref{e3.13} 
can be replaced by the Birman--Solomyak space $\ell^1(L^2(\bbR^n))$ (cf., e.g.,  
\cite[Chapter 4]{Si05}). In addition, the $L^\infty$-assumptions in Example \ref{e3.13} 
can be replaced by appropriate relatively boundedness assumptions with respect 
to $A_-$, but we omit further details in the interest of simplicity. 

A similar example, removing the positivity property of $A(t)$ in Example \ref{e3.13}, 
can be constructed as follows:

\begin{example} \lb{e3.14} 
Let $n\in\bbN$, $p > n$, $q \in ((n/2), p -(n/2))$, and $\varepsilon > 0$.  
Consider the self-adjoint $2 \times 2$ matrices $V_- = (V_{-,j,k})_{1\leq j, k \leq 2}$, 
$V(t,\cdot) = (V(t,\cdot)_{j,k})_{1\leq j, k \leq 2}$, with  
\begin{align}
& V_{-,j,k} \in L^2(\bbR^n; (1 + |x|^2)^q d^n x) \cap L^\infty(\bbR^n; d^n x), 
\quad 1\leq j, k \leq 2,  \\
& V(t,\cdot)_{j,k} \in L^2(\bbR^n; (1 + |x|^2)^q d^n x) \cap L^\infty(\bbR^n; d^n x), 
\quad t\in\bbR, \; 1\leq j, k \leq 2,
\end{align}
and suppose in addition that
\begin{align}
& \partial_tV(t,\cdot)_{j,k} \in L^2(\bbR^n; (1 + |x|^2)^q d^n x) \cap L^\infty(\bbR^n; d^n x), 
\quad t\in\bbR, \; 1\leq j, k \leq 2,     \\
& \bbR\ni t\mapsto V(t,\cdot)_{j,k} \in C^1(\bbR; L^\infty(\bbR^n; d^n x)),  \quad  1\leq j, k \leq 2. 
\end{align}
Next, we introduce the linear operators 
\begin{align}
& A_- = \begin{pmatrix} (-\Delta)^{p/2} + \varepsilon I + V_{-,1,1} & V_{-,1,2} 
\\ V_{-,2,1} & - (-\Delta)^{p/2} - \varepsilon I + V_{-,2,2} \end{pmatrix},    \no \\
& \quad \dom(A_-) = \dom\big((-\Delta)^{p/2}\big) \oplus \dom\big((-\Delta)^{p/2}\big),   \\
& B(t) = V(t,\cdot) - V_-, \quad \dom(B(t)) = L^2(\bbR^n; d^n x) \oplus 
L^2(\bbR^n; d^n x), \quad t\in\bbR,     \\
& A(t) = A_- + B(t), \quad \dom(A(t)) = \dom(A_-), \quad t\in\bbR, 
\end{align}
in $L^2(\bbR^n; d^n x) \oplus L^2(\bbR^n; d^n x)$.  Then 
\begin{equation}
\sigma_{\rm ess}(A(t)) = \sigma_{\rm ess}(A_-) 
= (-\infty, - \varepsilon] \cup [\varepsilon, \infty), \quad t\in\bbR,    \lb{3.93}
\end{equation}
and repeatedly applying \cite[Corollary\ 4.8]{Si05} one again verifies that all assumptions 
in Hypothesis \ref{h2.1} are satisfied. In the particular case where 
\begin{align} 
& V_{-,1,2} = V_{-,2,1} = 0, \quad V_{-,1,1} \geq 0, \quad V_{-,2,2} \leq 0,   \\   
& V_{1,2} (t,\cdot) = V_{2,1} (t,\cdot) = 0, \quad  
V_{1,1} (t,\cdot) \geq 0, \quad V_{2,2} (t,\cdot) \leq 0, \quad t\in\bbR, 
\end{align} 
then also 
\begin{equation}
\sigma (A(t)) = \sigma (A_-) = (-\infty, - \varepsilon] \cup [\varepsilon, \infty), 
\quad t\in\bbR,     \lb{3.96}
\end{equation}
holds as in the proof of \eqref{3.85}. 
\end{example}

Employing the norm resolvent convergence as $t \to +\infty$ in \eqref{3.Anlim} 
then shows that $A_+$, constructed according to Theorem \ref{t3.7}, also satisfies 
\eqref{3.85} and \eqref{3.96} (cf., e.g., \cite[Sect.\ VIII.7]{RS80}).

\section{Preliminaries in Connection with the Trace Formula} \lb{s4}

In this section we collected some preliminary results used in the proof
of Propositions \ref{prTr}  and \ref{propgA}.

The following interpolation result (and others) have been proved in \cite{GLST15}. They extend results originally discussed by Lesch \cite{Le05}:

\begin{theorem} [\cite{GLST15}] \lb{hadamard1} Let $\cH$ be a separable Hilbert space and 
$T \geq 0$ a self-adjoint operator with $T^{-1}\in\cB(\cH)$. Assume that
$S$ is closed and densely defined in $\cH$, with
$\big(\dom(S)\cap\dom(S^*)\big) \supseteq \dom(T)$, implying
$ST^{-1}\in\cB(\cH)$ and $S^*T^{-1}\in\cB(\cH)$. If, in addition, 
$ST^{-1}\in\cB_1(\cH)$ and $S^*T^{-1}\in\cB_1(\cH)$, then 
\begin{equation}
T^{-1/2}ST^{-1/2} \in\cB_1(\cH), \quad 
(T^{-1/2}ST^{-1/2})^* = T^{-1/2}S^*T^{-1/2} \in\cB_1(\cH). 
\end{equation}
Moreover,  
\begin{equation} \lb{normin1}
\big\|T^{-1/2}ST^{-1/2}\big\|_{\cB_1(\cH)} 
= \big\|T^{-1/2}S^*T^{-1/2}\big\|_{\cB_1(\cH)}
\le \big\|ST^{-1}\big\|_{\cB_1(\cH)}^{1/2} \, \big\|S^*T^{-1}\big\|_{\cB_1(\cH)}^{1/2}.
\end{equation}
\end{theorem}

Next, we study properties of the operator $\bsD_{\bsA}$ defined in \eqref{2.DA}
starting with the constant coefficient case $A(t)=A_-$, $t\in\bbR$. We recall that the operator of differentiation $d/dt$ in $L^2(\bbR;\cH)$, defined in 
\eqref{2.ddt}, is closed, and the graph norm on $\dom(d/dt)$ is 
equivalent to the norm in $W^{1,2}(\bbR;\cH)$, where $W^{1,2}(\cdot)$ denotes the 
usual Sobolev space of $L^2(\bbR; \cH)$-functions with the first distributional 
derivative in $L^2(\bbR; \cH)$. We note that $(d/dt)^* = - (d/dt)$ which will be used in \eqref{2.DA-*}. For a self-adjoint operator $A_-$ in $\cH$ on $\dom(A_-)\subseteq\cH$, 
the operator $\bsA_-$, defined by \eqref{2.DA-}, is closed in $L^2(\bbR;\cH)$ since 
$A_-$ is closed in $\cH$. In addition, the graph norm
$\|\cdot\|_{\cH_1(\bsA_-)}$ on $\dom(\bsA_-)$ is equivalent to the norm in 
$L^2(\bbR;\cH_1(A_-))$ since
\beq\label{eq37}
\begin{split}
\|f\|_{\cH_1(\bsA_-)}^2& = \|\bsA_-f\|_{L^2(\bbR;\cH)}^2 + 
\|f\|_{L^2(\bbR;\cH)}^2
=\int_\bbR \big[\|A_-f(t)\|_\cH^2 +  \|f(t)\|_\cH^2 \big]\,dt  \\
& =\int_\bbR\|f(t)\|_{\cH_1(A_-)}^2\,dt = \|f\|_{L^2(\bbR;\cH_1(A_-))}^2, \quad
f \in \dom(\bsA_-).
\end{split}
\enq
We recall the definition of the constant 
coefficient operator in $L^2(\bbR;\cH)$, 
\begin{equation}
\bsD_{\bsA_-}^{} = \f{d}{dt} + \bsA_-, \quad
\dom(\bsD_{\bsA_-}^{})=\dom(d/dt) \cap \dom(\bsA_-).
\label{3.DA-1}
\end{equation}
\begin{lemma} \lb{l2.3}
Suppose $A_-$ is self-adjoint  in $\cH$ on $\dom(A_-)\subseteq\cH$, 
and define the operator $\bsD_{\bsA_-}^{}$ as in \eqref{3.DA-1}. Then the following
assertions hold: \\
$(i)$ The graph norm
$\|\cdot\|_{\cH_1(\bsD_{\bsA_-}^{})}$ on $\dom(\bsD_{\bsA_-}^{})$ is 
equivalent to the norm on
$W^{1,2}(\bbR;\cH)\cap L^2(\bbR;\cH_1(A_-))$ defined as the maximum 
of the norms in $W^{1,2}(\bbR;\cH)$ and $L^2(\bbR;\cH_1(A_-))$; 
consequently, the operator $\bsD_{\bsA_-}^{}$ is closed. \\
$(ii)$ The adjoint $\bsD_{\bsA_-}^*$ of the operator 
$\bsD_{\bsA_-}^{}$ in $L^2(\bbR;\cH)$ is given by
\begin{equation}
\bsD_{\bsA_-}^* = - \f{d}{dt} + \bsA_-, \quad \dom(\bsD_{\bsA_-}^*)
= \dom(d/dt) \cap \dom(\bsA_-) = \dom(\bsD_{\bsA_-}^{}).  \lb{2.DA-*}
\end{equation}
$(iii)$  The operator $\bsD_{\bsA_-}^{}$ is a normal operator in 
$L^2(\bbR; \cH)$. \\
$(iv)$ The spectra of the operators $\bsD_{\bsA_-}^{}$ in $L^2(\bbR; \cH)$ and $A_-$ in $\cH$ satisfy: 
\begin{equation}\label{sprel}
\sigma(\bsD_{\bsA_-}^{})=\sigma(A_-) + i \, \bbR.
\end{equation}
\end{lemma}
\begin{proof} As we will see, the lemma follows by letting $A=\bsA_-$ and $B=(-id/dt)$ in the next assertion (cf.\  \cite[Ex.\
XII.9.11, p.1259]{DS88}, \cite[Ex.\ 7.48]{We80}).\\

\noindent{\bf Assertion.} {\em Suppose that $A$ and $B$ are two resolvent commuting self-adjoint operators in a complex, separable Hilbert space $\cK$, and define the operators $C$ and $C'$ by 
\begin{equation}\label{C}
C=A+iB, \quad C'=A-iB, \quad \dom(C)=\dom( C')=\dom(A)\cap \dom(B).
\end{equation}
Then
\begin{align}\label{gen1}
\|Ch\|_{\cK}^2&=\|Ah\|_{\cK}^2+\|Bh\|_{\cK}^2, \quad h\in \dom (C),\\
\label{gen}
\| C'h\|_{\cK}^2&=\|Ah\|_{\cK}^2+\|Bh\|_{\cK}^2, \quad h\in \dom ( C'),
\end{align}
the operator $C$ is normal,
$C^{*}= C'$, and  
\begin{equation}\label{roAC} 
\rho(A) + i \, \bbR\subseteq \rho(C).
\end{equation}}

To prove this assertion, we introduce the strongly right continuous families of spectral projections $E_A(\lambda)=E_A((-\infty,\lambda])$ and 
$E_B(\lambda)=E_B((-\infty,\lambda])$, $\lambda\in\bbR$, of the operators $A$ and $B$, respectively. Since by hypothesis the resolvents of $A$ and $B$ commute, 
the spectral projections also commute, that is,
$E_A(\lambda) E_B(\mu) = E_B(\mu) E_A(\lambda)$, $\lambda,\mu\in \bbR$, and
\begin{align} \label{s}
\begin{split}  
& E_A(\lambda) A \subseteq A E_A(\lambda), \quad 
E_A(\lambda) B \subseteq B E_A(\lambda), \\
& E_B(\lambda) B \subseteq B E_B(\lambda), \quad 
E_B(\lambda) A \subseteq A E_B(\lambda), \quad \lambda\in \bbR.
\end{split}
\end{align}
It follows from \eqref{s} that $C$ and $ C'$ are densely defined, 
closable operators in $\cK$ and
 \begin{equation}\label{sop}
 C'\subseteq C^{*},\quad C\subseteq  C'^{*}.
 \end{equation}
Next, we define
 \begin{equation}\label{defQ}
 Q_n=E_A([-n,n])E_B([-n,n])=E_B([-n,n])E_A([-n,n]), \quad n\in\bbN,
 \end{equation}
so that  $\lim_{n\to\infty}\|Q_nh-h\|\to 0$ for each $h\in\cK$ and,
in addition,
\begin{align}\label{QQ}
& Q_nC\subseteq CQ_n, \quad Q_nC^{*} \subseteq C^{*}Q_n, \quad n\in\bbN,  \\
 \label{comm}
& ABQ_n=BAQ_n, \quad n \in \bbN.
\end{align}
Let $h\in \dom (C')$ and denote $h_n=Q_n h$, $n\in\bbN$.
Then \eqref{comm} yields
\begin{align}
\|C' h\|_{\cK}^2&=\lim_{n\to\infty} (A h_n-iB h_n,A h-iB h)_{\cK}    \no \\
&=\lim_{n\to\infty}\big[(A h_n,A h)_{\cK} + (B h_n,B h)_{\cK} + i(A h_n,B h)_{\cK} -
i(B h_n,A h)_{\cK}\big]     \no \\
&=\lim_{n\to\infty}\big[(A h_n,A h)_{\cK} + (B h_n,B h)_{\cK}\big] 
= \|A h\|_{\cK}^2 + \|B h\|_{\cK}^2,
\end{align}
proving \eqref{gen}; the proof of \eqref{gen1} is similar. By \eqref{gen1}, the 
graph norm of $C$ is equivalent to the norm 
$\max\big\{\|h\|_{\cH_1(A)}, \|h\|_{\cH_1(B)}\big\}$ on $\dom(A)\cap\dom(B)$. 
Since the latter space is complete, $C$ is closed; similarly, $C'$ is closed. Next, 
let $ h\in \dom(C^{*})$. By \eqref{QQ}, we have
\begin{equation}
\lim_{n\to\infty} C' h_n=\lim_{n\to\infty} C^{*} h_n=C^{*} h.
\end{equation}
Since  $C'$ is closed, one concludes that $ h\in\dom (C')$ and
$C' h=C^{*} h$, which, together with \eqref{sop}, implies that 
$C'=C^{*}$. Since 
\begin{equation} 
\|C h\|_{\cK} = \|C^{*} h\|_{\cK}, \quad h\in \dom (C)=\dom (C^{*}), 
\end{equation}
due to \eqref{gen} and \eqref{gen1}, the normality of $C$ follows by 
\cite[Section\ 5.6]{We80}. Finally, to prove \eqref{roAC}, let us fix a 
$\mu+i\nu\in \rho(A)+i \, \bbR$ and apply \eqref{gen1} with $A$ and $B$ replaced 
by $A-\mu I_{\cK}$ and $B-\nu I_{\cK}$, respectively. Since the operator 
$A-\mu I_{\cK}$ is uniformly 
bounded from below, for some $c>0$,  
\begin{align}\label{unbddC}
\|(C-(\mu+i\nu) I_{\cK})h\|_{\cK}^2 &= \|(A-\mu I_{\cK})h\|_{\cK}^2 
+ \|(B-\nu I_{\cK})h\|_{\cK}^2 \ge \|(A-\mu I_{\cK})h\|_{\cK}^2   \no \\ 
& \ge c\|h\|_{\cK}^2, 
\quad h\in\dom(C), 
\end{align}
proving that the operator $C-(\mu+i\nu) I_{\cK}$ is uniformly bounded from 
below. Using \eqref{gen}, a similar argument for $C^\ast$ completes the proof 
of the inclusion $(\mu+i\nu) \in \rho(C)$, thus finishing the proof of the assertion.

Returning to the proof of Lemma \ref{l2.3}, we remark that items $(i)$, $(ii)$, $(iii)$ 
follow directly from the assertion just proved (with $A=\bsA_-$, $B=(-id/dt)$, and $C=\bsD_{\bsA_-}^{}$). In particular, the equivalence of the norms in item 
$(i)$ follows from \eqref{gen1}, 
\begin{equation}
\|\bsD_{\bsA_-}^{} f\|_{L^2(\bbR; \cH)}^2=\|f'\|_{L^2(\bbR; \cH)}^2+\|\bsA_-f\|_{L^2(\bbR; \cH)}^2, \quad f\in\dom(\bsD_{\bsA_-}^{}),
\end{equation}
which, in turn, for each $ f\in\dom(\bsD_{\bsA_-}^{})=\dom(d/dt)\cap\dom(\bsA_-)$ yields
\begin{align}
& \|f\|_{W^{1,2}(\bbR;\cH)\cap  L^2(\bbR;\cH_1(A_-))}^2=\max\big[
\|f\|_{\cH_1(d/dt)}^2,\|f\|_{\cH_1(\bsA_-)}^2\big]  \no \\
& \quad =\max\big[
\|f\|_{L^2(\bbR;\cH)}^2+\|f'\|_{L^2(\bbR;\cH)}^2,\|f\|_{L^2(\bbR;\cH)}^2
+\|\bsA_-f\|_{L^2(\bbR;\cH)}^2\big] \no \\
& \quad \le \|f\|_{L^2(\bbR;\cH)}^2+\|f'\|_{L^2(\bbR;\cH)}^2 
+ \|\bsA_-f\|_{L^2(\bbR;\cH)}^2      \lb{nnor} \\ 
& \quad \le 2\max\big[\|f\|_{L^2(\bbR;\cH)}^2 
+ \|f'\|_{L^2(\bbR;\cH)}^2,\|f\|_{L^2(\bbR;\cH)}^2 
+ \|\bsA_-f\|_{L^2(\bbR;\cH)}^2\big]   \no \\
& \quad =2\|f\|_{W^{1,2}(\bbR;\cH)\cap  L^2(\bbR;\cH_1(A_-))}^2,     \no 
\end{align}
since the term in \eqref{nnor} is equal to 
$\|f\|_{L^2(\bbR;\cH)}^2+\|\bsD_{\bsA_-}^{} f\|_{L^2(\bbR; \cH)}^2
=\|f\|_{\cH_1(\bsD_{\bsA_-}^{})}^2$. Therefore, $\dom(\bsD_{\bsA_-}^{})$ with the graph norm is a complete space, and thus $\bsD_{\bsA_-}^{}$ is closed.

To finish the proof of item $(iv)$, it remains to show that 
$(\mu+i\nu)\in\rho(\bsD_{\bsA_-}^{})$ implies $\mu\in\rho(A_-)$. As in the proof of \cite[Theorem 3.13]{CL99}, one considers 
the unitary operator of multiplication $\bsM$ in $L^2(\bbR; \cH)$ by the scalar 
function $m(t)=e^{-i\nu t}$, that is, 
\begin{equation} 
(\bsM f)(t)=e^{-i\nu t}f(t),  \quad f\in L^2(\bbR; \cH).
\end{equation}  
In addition, 
\begin{equation}
\text{If $f\in\dom(\bsD_{\bsA_-}^{})$ then $\bsM f\in\dom(\bsD_{\bsA_-}^{})$ and 
$\bsD_{\bsA_-}^{} \bsM f = \bsM (-i\nu \bsI + \bsD_{\bsA_-}^{})f$.} 
\end{equation} 
Thus, 
$\rho(\bsD_{\bsA_-}^{})=\rho(-i\nu \bsI + \bsD_{\bsA_-}^{})$, and therefore 
$(\mu+i\nu) \in\rho(\bsD_{\bsA_-}^{})$ implies $\mu\in\rho(\bsD_{\bsA_-}^{})$. Similarly to \eqref{unbddC}, for some $c>0$, 
\begin{align}\label{bdbel} 
& c\|f\|_{L^2(\bbR;\cH)}^2\le\|(\bsD_{\bsA_-}^{} - \mu \bsI)f\|_{L^2(\bbR;\cH)}^2=
\|(\bsA_- - \mu \bsI)f\|_{L^2(\bbR;\cH)}^2+\|f'\|_{L^2(\bbR;\cH)}^2,   \no   \\ 
& \hspace*{8.5cm} f\in\dom(\bsD_{\bsA_-}^{}). 
\end{align}
For each $k\in\bbN$, we choose a smooth function $\chi_k:\bbR\to[0,1]$
such that 
\begin{equation}
\chi_k(t)= \begin{cases} 1, & |t|\le k, \\ 0, & |t|\ge k+1, \end{cases}  
\quad  |\chi'_k(t)|\le 2, \; t\in\bbR. 
\end{equation} 
We fix any $h\in\dom(A_-)$ and denote $f_k(t)=\chi_k(t)h$, $t\in\bbR$. Then
$f_k\in\dom(\bsD_{\bsA_-}^{})$ with $f'_k(t)=\chi'_k(t)h$ and $(\bsA_-f_k)(t)=\chi_k(t)A_-h$. In addition,  
\begin{equation}\label{limfh}
\frac{\|f'_k\|_{L^2(\bbR;\cH)}}{\|f_k\|_{L^2(\bbR;\cH)}}
=\frac{\|\chi'_k\|_{L^2(\bbR; dt)}}{\|\chi_k\|_{L^2(\bbR; dt)}} 
\underset{k\to\infty}{\longrightarrow} 0.
\end{equation}
Applying \eqref{bdbel} with $f$ replaced by $f_k$ yields:
\begin{equation}
c\|h\|_\cH^2\|\chi_k\|_{L^2(\bbR; dt)}^2 
\le \|(A_--\mu I)h\|_\cH^2\|\chi_k\|_{L^2(\bbR; dt)}^2 
+ \|h\|_\cH^2\|\chi'_k\|_{L^2(\bbR; dt)}^2.
\end{equation}
Using \eqref{limfh}, we arrive at the inequality 
$c\|h\|_\cH^2\le\|(A_- - \mu I)h\|_\cH^2$, thus proving
$\mu\in\rho(A_-)$. 
\end{proof}

\begin{remark} \lb{r3.3}
$(i)$ Lemma \ref{l2.3}\,$(iv)$ shows that if $A_-$ (and hence 
$\bsA_-$) has a spectral gap at $0$, then $(\bsD_{\bsA_-}^{})^{-1} \in 
\cB(L^2(\bbR;\cH))$ and thus $\bsD_{\bsA_-}^{}$ is a Fredholm operator of 
index zero. \\
$(ii)$ One notes the peculiar fact that if $\sigma(A_-)=\bbR$, then 
$\bsD_{\bsA_-}^{}$ has empty resolvent set, or equivalently, 
$\sigma(\bsD_{\bsA_-}^{})=\bbC$.
\end{remark}

For an alternative proof of Lemma \ref{l2.3}\,$(iii)$ using the notion 
of $N$-measurability we refer to Lemma \ref{lA.10}. 

Throughout the remaining part of this section, we continue to assume 
Hypothesis \ref{h2.1}. 

We recall that $\bsA$ denotes the maximally defined operator in 
$L^2(\bbR;\cH)$ associated with the family of operators 
$A(t)$, $t\in\bbR$, in $\cH$, defined by 
\begin{align}\label{dombbA}
& (\bsA f)(t)=A(t)f(t) \, \text{ for a.e.\ $t\in\bbR$,}     \no \\
& f \in \dom(\bsA)=\bigg\{g \in L^2(\bbR;\cH) \, \bigg| \,
g(t)\in\dom(A(t)) \text{ for a.e.\ } t\in\bbR,    \no \\
& \quad t \mapsto A(t)g(t) \text{ is (weakly) measurable,} \, 
  \int_{\bbR} \|A(t)f(t)\|_{\cH}^2\, dt < \infty \bigg\}.
\end{align}

Next, we define in $L^2(\bbR;\cH)$ the operator
\begin{equation}
\bsD_\bsA^{} = \f{d}{dt} + \bsA,   \quad
\dom(\bsD_\bsA^{})=\dom(d/dt)\cap\dom(\bsA), \label{3.DA1}
\end{equation}
as the operator sum of $d/dt$ and $\bsA$.

Assuming Hypothesis \ref{h2.1}, we next prove that $\bsD_\bsA^{}$ is a 
densely defined and closed operator in $L^2(\bbR; \cH)$, and that the 
domain of $\bsD_\bsA^{}$ actually coincides with that of $\bsD_{\bsA_-}^{}$ in 
\eqref{3.DA-1}.

\begin{lemma} \lb{l2.4}
Assume Hypothesis \ref{h2.1}.
Then $\bsD_\bsA^{}$ as defined in \eqref{3.DA1} is a densely defined and closed 
operator in
$L^2(\bbR; \cH)$ and
\beq
\dom(\bsD_\bsA^{}) = \dom(\bsD_\bsA^*) = \dom(\bsD_{\bsA_-}^{}) 
= \dom(d/dt) \cap \dom(\bsA_-).
\enq
Moreover, the adjoint operator $\bsD_\bsA^*$ of 
$\bsD_\bsA^{}$ in $L^2(\bbR; \cH)$ is given by
\begin{align}
\begin{split} 
& \, \bsD_\bsA^*=-\f{d}{dt} + \bsA,   \\ 
& \dom(\bsD_\bsA^*) = \dom(d/dt) \cap \dom(\bsA) = \dom(d/dt) \cap \dom(\bsA_-).
\end{split}
\end{align}
In addition, the graph norm $\|\cdot\|_{\cH_1(\bsD_\bsA^{})}$ on 
$\dom(\bsD_\bsA^{})$ is equivalent to the norm on
$W^{1,2}(\bbR;\cH) \cap L^2(\bbR;\cH_1(A_-))$ defined as the maximum 
of the norms in $W^{1,2}(\bbR;\cH)$ and $L^2(\bbR;\cH_1(A_-))$. 
\end{lemma}
\begin{proof}
Since $\bsD_{\bsA_-}^{}$ may have an  empty resolvent set, we will focus on 
the self-adjoint operator $|\bsD_{\bsA_-}^{}|$ at first. Consider the unitary 
vector-valued Fourier transform 
\begin{equation}\label{ft}
\gF_{\cH}:L^2(\bbR; \cH)\to L^2(\bbR; 
\cH),
\end{equation} first defined by 
\begin{equation}\label{ft1}
F\mapsto \widehat F, \quad 
 \widehat F(\lambda)= (2 \pi)^{-1/2} 
\int_{\bbR} e^{-i\lambda s} F(s)\, ds, \quad  
\lambda \in \bbR,  
\end{equation}
for all $F\in \cS(\bbR; \cH)$, the $\cH$-valued Schwartz class, and then 
extended to a unitary operator in $L^2(\bbR; \cH)$ by taking the
 closure (see, e.g., \cite[Lemma 2]{GN83}, \cite[p.\ 16]{LM72}). 

Via the Fourier 
transform $\gF_{\cH}$, the operator
$|\bsD_{\bsA_-}^{}|$ is unitarily equivalent to the operator 
$|it \bsI + \bsA_-|$ in the space $L^2(\bbR;\cH)$ with domain 
\beq
\dom(|it \, \bsI + \bsA_-|) = \dom(it \, \bsI + \bsA_-) = \dom(it \, 
\bsI) \cap \dom(\bsA_-).   \lb{4.34a}
\enq
Using \eqref{4.34a}, Remark \ref{TnormT},  and the spectral theorem for 
$A_-$, one obtains
\begin{align}
& \big\|(|\bsA_-| - z \, \bsI)(|\bsD_{\bsA_-}^{}| - z \, 
\bsI)^{-1}\big\|_{\cB(L^2(\bbR;\cH))}   \no  \\
& \quad =\sup_{t\in\bbR} \big\|(|A_-| - z I))(|-it I + A_-|-z 
I)^{-1}\big\|_{\cB(\cH)} \no  \\
& \quad =\sup_{t\in\bbR}\sup_{\lambda\in\sigma(A_-)}
\bigg|\frac{|\lambda|-z}{(t^2+\lambda^2)^{1/2}-z}\bigg| = 1, \quad z<0.
\end{align}
This in turn implies (still assuming $z<0$),
\begin{align}
& \big\|\bsB (|\bsD_{\bsA_-}^{}| - z \, 
\bsI)^{-1}\big\|_{\cB(L^2(\bbR;\cH))}   \no  \\
& \quad = \big\|\bsB (|\bsA_-| - z \, \bsI)^{-1}
(|\bsA_-| - z \, \bsI)(|\bsD_{\bsA_-}^{}| - z \, 
\bsI)^{-1}\big\|_{\cB(L^2(\bbR;\cH))}   \no  \\
& \quad \leq  \big\|\bsB (|\bsA_-| - z \, \bsI)^{-1}\big\|_{\cB(L^2(\bbR;\cH))}
  \big\|(|\bsA_-| - z \, \bsI))(|\bsD_{\bsA_-}^{}| - z \, 
\bsI)^{-1}\big\|_{\cB(L^2(\bbR;\cH))}     \no  \\
& \quad = \big\|\bsB (|\bsA_-| - z \, 
\bsI)^{-1}\big\|_{\cB(L^2(\bbR;\cH))}       \no  \\
& \quad = \sup_{t\in\bbR} \big\|B(t) (|A_-| - z 
I)^{-1}\big\|_{\cB(\cH)} \underset{z\downarrow -\infty}{=} o(1),
\lb{3.13}
\end{align}
by \eqref{2.13g}.  Put differently, \eqref{3.13} implies the 
existence of $\varepsilon (z)>0$ with
$\varepsilon (z) \underset{z\downarrow -\infty}{=} o(1)$ and $\eta(z)>0$, such that 
the Kato--Rellich-type bound 
\begin{align}
\begin{split}
\|\bsB f\|_{L^2(\bbR;\cH)} \leq \varepsilon(z) \| |\bsD_{\bsA_-}^{}| 
f\|_{L^2(\bbR;\cH)} + \eta(z) \| f\|_{L^2(\bbR;\cH)},&
\\
f\in \dom(|\bsD_{\bsA_-}^{}|) = \dom(\bsD_{\bsA_-}^{}),&    \lb{3.14}
\end{split}
\end{align}
holds. 
Next, one recalls that the polar decomposition of a densely defined, 
closed, linear operator $T$ in a complex Hilbert space $\cK$ is of 
the form  $T = U_T |T|$, with $U_T$ (and hence $U_T^*$) a partial 
isometry in $\cK$, implying $|T|=U_T^* T$. Applying the latter fact 
to $\bsD_{\bsA_-}^{}$ in \eqref{3.14}, one finally obtains 
\begin{align}
\begin{split}
\|\bsB f\|_{L^2(\bbR;\cH)} \leq \varepsilon(z) \| \bsD_{\bsA_-}^{} 
f\|_{L^2(\bbR;\cH)} + \eta(z) \| f\|_{L^2(\bbR;\cH)}&,
\\
f\in \dom(\bsD_{\bsA_-}^{})&.    \lb{3.15}
\end{split}
\end{align}
Thus, $\bsB$ is relatively bounded with respect to $\bsD_{\bsA_-}^{}$ 
in $L^2(\bbR;\cH)$ with relative bound zero (cf.\ \cite[Sect.\ 4.1.1]{Ka80}). 
Since $\bsD_{\bsA_-}^{}$ is a closed operator in $L^2(\bbR;\cH)$ 
by Lemma \ref{l2.3} ($i$), also $\bsD_\bsA^{} = \bsD_{\bsA_-}^{} + \bsB$ defined on 
$\dom(\bsD_{\bsA_-}^{})$ is closed in $L^2(\bbR;\cH)$.

To prove that $\dom(\bsD_\bsA^*)=\dom(\bsD_\bsA^{})$ one can argue as follows: 
Since $\bsB$ is symmetric on $\dom (\bsD_{\bsA_-}^{})$ and the operator
$\bsD_{\bsA_-}^{}$ is normal, and hence
$\dom({\bsD_{\bsA}^*}_-)=\dom({\bsD_{\bsA}}_-)$, one obtains that 
\begin{equation}
\|\bsB^* f\|_{L^2(\bbR;\cH)}= \|\bsB f\|_{L^2(\bbR;\cH)}, \quad f\in \dom(\bsD_{\bsA_-}^{}), 
\end{equation}
and that 
\begin{equation}
\|\bsD_{\bsA_-}^{} f\|_{L^2(\bbR;\cH)}=\|\bsD_{\bsA_-}^*f\|_{L^2(\bbR;\cH)}, \quad 
f\in \dom(\bsD_{\bsA_-}^{}). 
\end{equation}
Therefore, \eqref{3.15} can be rewritten as 
\begin{align}
\begin{split}
\|\bsB^* f\|_{L^2(\bbR;\cH)} \leq \varepsilon(z) \|\bsD_{\bsA_-}^* 
f\|_{L^2(\bbR;\cH)} + \eta(z) \| f\|_{L^2(\bbR;\cH)}&,
\\
f\in \dom(\bsD_{\bsA_-}^*)& ,   \lb{3.152}
\end{split}
\end{align}
implying that also $\bsB^*$ is relatively bounded with respect to $\bsD_{\bsA_-}^*$ 
with relative bound zero. By the Hess--Kato result \cite{HK70} (see also 
\cite[p.\ 111]{We80}),  
\begin{equation}
\dom(\bsD_\bsA^*) = \dom(\bsD_{\bsA_-}^*) \cap \dom(\bsB^*) = \dom(\bsD_{\bsA_-}^*) 
= \dom(\bsD_{\bsA_-}^{}) = \dom(\bsD_{\bsA}^{})
\end{equation}
and  
\begin{align}
\begin{split}
& \, \bsD_\bsA^* = \bsD_{\bsA_-}^*+ \bsB^* = \bsD_{\bsA_-}^*+\bsB  
= \bigg(- \f{d}{dt} + \bsA_-\bigg) + \bsB = - \f{d}{dt} + \bsA,  \\
& \dom(\bsD_\bsA^*) = \dom(\bsD_{\bsA}^{}) = \dom(\bsD_{\bsA_-}^{}). 
\end{split}
\end{align} 
Here we used again that ${\bsB^*}f={\bsB}f$ for all $f\in\dom(\bsD_{\bsA_-}^{})$.

The statements about graph norms have been proved in Lemma \ref{l2.3}.
\end{proof}

Next, we will discuss some operators needed in the proof of 
Proposition \ref{prTr}. We start with the operator $\bsH_0$ in 
$L^2(\bbR;\cH)$ defined by
\begin{equation}
\bsH_0 =\bsD_{\bsA_-}^* \bsD_{\bsA_-}^{} = \bsD_{\bsA_-}^{} \bsD_{\bsA_-}^*     
\lb{3.19}
\end{equation}
(cf.\ Lemma \ref{l2.3}\,$(iii)$). 
In particular, $\bsH_0$ is self-adjoint since $\bsD_{\bsA_-}^{}$ is closed, and
$\bsH_0 \geq 0$. In addition, one obtains that
\beq\lb{H0}
\dom\big(\bsH_0^{1/2}\big) = \dom(\bsD_{\bsA_-}^{}) = \dom(\bsD_{\bsA_-}^*) 
= \dom(d/dt) \cap\dom(\bsA_-) .
\enq

We will use the following representation for the resolvent of $\bsH_0$,
\begin{align}\lb{R0}
\bsR_0(z)=(\bsH_0 - z \, \bsI)^{-1} 
=\frac12 \big(\bsA_-^2 - z \, \bsI \big)^{-1/2}\widehat{\bsK}_0 (z),
\quad z\in\bbC\backslash [0,\infty), 
\end{align}
where $\widehat{\bsK}_0 (z)$ denotes the operator of convolution with
$e^{- (\bsA_-^2 - z \, \bsI )^{1/2} |t|}$ on $L^2(\bbR;\cH)$, that is, $\bsR_0(z)$ is an integral operator with the operator-valued integral kernel
\beq
R_0 (z,s,t) = \f{1}{2} \varkappa_z(A_-)^{-1} e^{-\varkappa_z(A_-) |t-s|} \in \cB(\cH), 
\quad s, t \in \bbR.    \lb{4.R0}
\enq
Here we used the notation $\varkappa_z(A_-)=(A_-^2-z I)^{1/2}$ in \eqref{dfnvk}. In the scalar-valued context formula \eqref{4.R0}  can be found, for instance, in 
\cite[Theorem\ 9.5.2]{Sc81}.

For subsequent purpose we also recall the integral kernel $\bsR_0^{1/2}(z,s,t)$ 
of $\bsR_0(z)^{1/2}$,
\beq
R_0^{1/2}(z,s,t) = \pi^{-1}K_0(\varkappa_z(A_-)|t-s|), \quad s, t\in\bbR, \;
s\neq t,    \lb{4.R01/2}
\enq
where $K_0(\cdot)$ denotes the modified (irregular) Bessel function of order zero
(cf.\ \cite[Sect.\ 9.6]{AS72}.) Formulas such as \eqref{4.R0} and \eqref{4.R01/2} follow 
from elementary Fourier transform arguments as detailed in \cite[p.\ 57--59]{RS75}.   
Relation \eqref{4.R01/2} requires in addition the integral representation 
\cite[No.\ 3.7542]{GR80} for $K_0(\cdot)$.

Next, we study some properties of $\bsB'$. For this purpose the following 
known result will turn out to be useful:

\begin{lemma} \lb{lHS}
Suppose $T(s,t) \in \cB_2(\cH)$ for a.e.\ $(s,t)\in\bbR^2$ and assume that 
\begin{equation}
\int_{\bbR^2} \|T(s,t)\|^2_{\cB_2(\cH)} \, ds\, dt < \infty.   \lb{3.30}
\end{equation}
Define the operator $\bsT$ in $L^2(\bbR; \cH)$ by
\begin{equation} 
(\bsT f)(s) = \int_{\bbR} T(s,t) f(t) \, dt \, \text{ for a.e.\ $s\in\bbR$, } 
\quad f \in L^2(\bbR; \cH).    \lb{3.31}
\end{equation}
Then $\bsT \in \cB_2(L^2(\bbR; \cH))$ and 
\begin{equation}
\|\bsT\|_{\cB_2(L^2(\bbR; \cH))}^2 = \int_{\bbR^2} \|T(s,t)\|^2_{\cB_2(\cH)} \, ds\, dt. 
\end{equation}
Conversely, any operator $\bsT \in \cB_2(L^2(\bbR; \cH))$ arises in the manner 
\eqref{3.30}, \eqref{3.31}. 
\end{lemma}

For the proof of an extension of Lemma \ref{lHS} we refer to 
\cite[Theorem 11.3.6]{BS87}. 

At this point it is worth noting that by Theorem \ref{tA.6}, $\bsB$ and $\bsB'$ are 
densely defined, symmetric, and closed operators in $L^2(\bbR;\cH)$ (cf.\ 
Lemma \ref{l2.9}).

\begin{lemma} \lb{12bprime}
Assume Hypothesis \ref{h2.1}. Then
\begin{align}
\begin{split} 
& | \bsB'|^{1/2} \, (\bsH_0 - z \, \bsI)^{-1/2} \in\cB_2(L^2(\bbR;\cH)),   \\ 
& | (\bsB')^*|^{1/2} \, (\bsH_0 - z \, \bsI)^{-1/2} \in\cB_2(L^2(\bbR;\cH)), 
\quad  z\in \bbC\backslash [0,\infty).    \lb{3.32}
\end{split}
\end{align}
Moreover,  
\begin{align} 
& \big\||\bsB'|^{1/2} \, (\bsH_0 - z \, \bsI)^{-1/2}\big\|_{\cB_2(L^2(\bbR;\cH))}^2  
 \le |z|^{-1/2}\int_\bbR \big\|B'(t) (A_-^2 + I)^{-1/2}\big\|_{\cB_1(\cH)}\,dt,  \no \\
 & \big\||(\bsB')^*|^{1/2} \, (\bsH_0 - z \, \bsI)^{-1/2}\big\|_{\cB_2(L^2(\bbR;\cH))}^2  
 \le |z|^{-1/2}\int_\bbR \big\|B'(t) (A_-^2 + I)^{-1/2}\big\|_{\cB_1(\cH)}\,dt,  \no \\
& \hspace*{9.6cm} z < -1.  \lb{zdecay} 
\end{align}
\end{lemma}
\begin{proof}
Abbreviating $\bsR_0^{1/2}=\bsR_0^{1/2}(z)$,  
$\widehat\varkappa_- =\big(\bsA_-^2 - z \, \bsI\big)^{1/2}$ (cf.\ \eqref{dhvk}), and 
$\varkappa_-= \varkappa_z(A_-)= (A_-^2-z I)^{1/2}$ (cf.\ \eqref{dfnvk}), with $z<0$, one estimates 
\begin{align}
& \big\| |\bsB'|^{1/2} \, \bsR_0^{1/2}\big\|_{\cB_2(L^2(\bbR;\cH))}^2 = 
\big\| |\bsB'|^{1/2} \, \widehat \varkappa_-^{-1/2} \,
\widehat \varkappa_-^{1/2} \, \bsR_0^{1/2}\big\|^2_{\cB_2(L^2(\bbR;\cH))}  \no \\
& \quad = \int_\bbR\int_\bbR \big\| |B'(t)|^{1/2} \varkappa_-^{-1/2} 
\varkappa_-^{1/2} R_0^{1/2}(t,s)\big\|^2_{\cB_2(\cH)}\,ds \,dt   \no \\
& \quad \le\int_\bbR \bigg( \big\||B'(t)|^{1/2} \varkappa_-^{-1/2}\big\|^2_{\cB_2(\cH)}
\int_\bbR \big\|\varkappa_-^{1/2} R_0^{1/2}(t,s)\big\|^2_{\cB(\cH)}\,ds\bigg)\,dt  \no \\ 
& \quad = \int_\bbR \bigg( \big\|\varkappa_-^{-1/2} |B'(t)| 
\varkappa_-^{-1/2}\big\|_{\cB_1(\cH)}
\int_\bbR \big\|\big[\varkappa_-^{1/2} R_0^{1/2}(t,s)\big]^2\big\|_{\cB(\cH)}
\,ds\bigg)\,dt  \no \\
& \quad = \int_\bbR \bigg( \big\|\varkappa_-^{-1/2} |B'(t)| 
\varkappa_-^{-1/2}\big\|_{\cB_1(\cH)}
\int_\bbR \big\|\varkappa_- R_0(t,s)\big\|_{\cB(\cH)}\,ds\bigg)\,dt  \no \\ 
& \quad \leq \int_\bbR \bigg( \big\||B'(t)| \varkappa_-^{-1}\big\|_{\cB_1(\cH)}
\int_\bbR \big\|\varkappa_- R_0(t,s)\big\|_{\cB(\cH)}\,ds\bigg)\,dt  \no \\
& \quad \leq \int_\bbR\bigg(\big\| B'(t) \varkappa_-^{-1}\big\|_{\cB_1(\cH)} \, 
\frac{1}{2} \int_\bbR \big\|e^{-\varkappa_- |s-t|}\big\|_{\cB(\cH)}\,ds \bigg)\,dt    \no \\
& \quad = \int_\bbR \big\| B'(t) \varkappa_-^{-1}\big\|_{\cB_1(\cH)}\, dt \, 
\frac{1}{2} \int_\bbR \big\|e^{-\varkappa_- |s|}\big\|_{\cB(\cH)}\,ds    \no \\ 
& \quad \leq |z|^{-1/2} \int_\bbR \big\|B'(t)\varkappa_z(A_-)^{-1}\big\|_{\cB_1(\cH)}\,dt 
< \infty.    \lb{e4.26}
\end{align}
Here we used Lemma \ref{lHS}, employed the fact that 
$\varkappa_-^{1/2}$ and $R_0^{1/2}(t,s)$ commute and that 
$\varkappa_-^{1/2}R_0^{1/2}(t,s)$ is self-adjoint, applied 
Theorem \ref{hadamard1} to obtain 
\begin{equation}
\big\|\varkappa_-^{-1/2} |B'(t)| \varkappa_-^{-1/2}\big\|_{\cB_1(\cH)} 
\leq \big\||B'(t)| \varkappa_-^{-1}\big\|_{\cB_1(\cH)}, 
\end{equation}
used the polar decomposition $B'(t) = U_{B'(t)} |B'(t)|$ of $B'(t)$, employed the explicit form of $R_0(s,t)$ in terms of the convolution operator $\hatt \bsK_0$ in \eqref{R0}, and finally, used the estimate 
\begin{equation}
 \big\|e^{-\varkappa_z(A_-) |s|}\big\|_{\cB(\cH)}
= \sup_{\lambda \in \sigma(A_-)} \Big[e^{-(\lambda^2+|z|)^{1/2}|s|}\Big]  
\leq e^{-|z|^{1/2} |s|}, \quad s\in\bbR, \; z<0, 
\end{equation}
and hence, 
\begin{equation}
\f{1}{2} \int_{\bbR} \big\|e^{-\varkappa_z(A_-)|s|}\big\|_{\cB(\cH)}\,ds 
\leq |z|^{-1/2}, \quad z<0.  
\end{equation}
Next, one notes that 
\begin{align}
\int_\bbR \big\|B'(t)\varkappa_z(A_-)^{-1}\big\|_{\cB_1(\cH)}  \,dt
& = \int_\bbR \big\|B'(t) \varkappa(A_-)^{-1} 
[\varkappa(A_-) \, \varkappa_z(A_-)^{-1}]\big\|_{\cB_1(\cH)} \,dt     \no \\
& \leq \int_\bbR \big\|B'(t) \varkappa(A_-)^{-1} \big\|_{\cB_1(\cH)} \,dt ,   \quad z<-1, 
\end{align}
since $\big\|\varkappa(A_-) \, \varkappa_z(A_-)^{-1}\big\|_{\cB(\cH)} =1$, 
$z \leq -1$, for $\varkappa(A_-)=(A_-^2+I)^{1/2}$, finishing the proof of the first 
relation in \eqref{zdecay} and, using Lemma \ref{lHS}, the first inclusion in \eqref{3.32} 
(for $z<-1$).

An application of Remark \ref{r2.6} (using \eqref{2.19} repeatedly) then yields the second relation in \eqref{3.32} (for $z<-1$) and \eqref{zdecay}.  

The extension of \eqref{3.32} to $z\in\bbC\backslash [0,\infty)$ then follows from
\begin{equation}
\big\|\bsR_0(\zeta)^{-1/2} \bsR_0(z)^{1/2}\big\|_{\cB(L^2(\bbR; \cH))} \leq C(\zeta,z) < \infty, 
\quad \zeta < 0, \; z \in \bbC\backslash [0,\infty).
\end{equation} 
\end{proof}

\begin{lemma}  \lb{l3.6}
Assume Hypothesis \ref{h2.1}. Then,
\begin{align}
& \big\|(\bsA_-^2-z \, \bsI)^{1/2} 
(\bsH_0 - z \, \bsI)^{-1/2}\big\|_{\cB(L^2(\bbR;\cH))} = 1, 
\quad   z<0,  \label{n4.60}\\
& \big\|\bsA_- (\bsH_0 - z \, \bsI)^{-1/2}\big\|_{\cB(L^2(\bbR;\cH))} \leq 1,   
\quad   z<0,    \lb{rhbdd}  \\
& \big\|\bsB (\bsH_0 - z \, \bsI)^{-1/2}\big\|_{\cB(L^2(\bbR;\cH))} 
\underset{z\downarrow -\infty}{=} o(1).   \lb{3.48}
\end{align}
\end{lemma}
\begin{proof}
Passing to the Fourier transform (cf.\ \eqref{ft}, \eqref{ft1}), and using 
Remark \ref{TnormT}, and the spectral theorem, one obtains,
\begin{align}
& \big\| (\bsA_-^2 - z\, \bsI)^{1/2} (\bsH_0 - z \, \bsI)^{-1/2}\big\|_{\cB(L^2(\bbR;\cH)}
 =\sup_{t\in\bbR} \big\|\varkappa_z(A_-)  
\big(t^2+\varkappa_z(A_-)^2\big)^{-1/2}\big\|_{\cB(\cH)} \no  \\
& \quad =\sup_{t\in\bbR}\sup_{\lambda\in\sigma(A_-)}
\bigg|\frac{\lambda^2-z}{t^2+\lambda^2-z}\bigg|^{1/2} = 1, \quad z<0, 
\end{align}
proving \eqref{n4.60}. The inequality \eqref{rhbdd}
is proved analogously.
Next, one estimates,
\begin{align}
& \big\|\bsB (\bsH_0 - z \, \bsI)^{-1/2}\big\|_{\cB(L^2(\bbR;\cH))}   \no \\
& \quad = \big\|\bsB (\bsA_-^2 -z \, \bsI)^{-1/2} 
(\bsA_-^2 -z \, \bsI)^{1/2}
(\bsH_0 - z \, \bsI)^{-1/2}\big\|_{\cB(L^2(\bbR;\cH))}   \no  \\
& \quad \leq \big\|\bsB (\bsA_-^2 -z \, \bsI)^{-1/2}\big\|_{\cB(L^2(\bbR;\cH))}
\big\| (\bsA_-^2 -z \, \bsI)^{1/2} 
(\bsH_0 - z \,\bsI)^{-1/2}\big\|_{\cB(L^2(\bbR;\cH))}   \no  \\
& \quad = \sup_{t\in\bbR} \big\|B(t) (A_-^2 - z I)^{-1/2}\big\|_{\cB(\cH)} 
\underset{z\downarrow -\infty}{=} o(1),   
\end{align}
by \eqref{2.13g} and \eqref{rhbdd}. 
\end{proof}

For subsequent purposes, we recall the generalized polar decomposition of a densely defined and 
closed operator $T$ in a complex separable Hilbert space $\cK$
\beq\
T = |T^*|^{1/2} U_T |T|^{1/2},    \lb{PDT} 
\enq
derived in \cite{GMMN09}, where $U_T$ is the partial isometry in $\cK$ 
in the standard polar decomposition $T = U_T |T|$ of $T$, with 
$|T| = (T^* T)^{1/2}$. 

Next, we introduce the following sesquilinear forms in $L^2(\bbR;\cH)$,
\begin{align}
Q_{\bsH_0} (f,g) & = \big(\bsH_0^{1/2} f, \bsH_0^{1/2} g\big)_{L^2(\bbR;\cH)}, 
\quad f, g \in \dom(Q_{\bsH_0}) = \dom\big(\bsH_0^{1/2}\big),   \lb{h0h} \\
Q_{\bsV_j} (f,g) & =  (\bsA_- f, \bsB g)_{L^2(\bbR;\cH)} 
+ (\bsB f, \bsA_- g)_{L^2(\bbR;\cH)} + (\bsB f, \bsB g)_{L^2(\bbR;\cH)}    \no \\
& \quad + (-1)^j \big(|(\bsB')^*|^{1/2} f, U_{\bsB'} |\bsB'|^{1/2} g\big)_{L^2(\bbR;\cH)},   
\lb{hsh} \\
& \hspace*{4mm}
f,g \in \dom(Q_{\bsV_j}) = \dom\big(\bsH_0^{1/2}\big), \; j=1,2,    \no \\
Q_{\bsV} (f,g) & = (\bsA_- f, \bsB g)_{L^2(\bbR;\cH)} + (\bsB f, \bsA_- g)_{L^2(\bbR;\cH)} 
+ (\bsB f, \bsB g)_{L^2(\bbR;\cH)},    \lb{hsht2} \\
& \hspace*{4.45cm} f, g \in \dom(Q_{\bsV}) = \dom\big(\bsH_0^{1/2}\big),  \no
\end{align} 
where we employed the generalized polar decomposition 
\begin{equation}
\bsB' = |(\bsB')^*|^{1/2} U_{\bsB'} |\bsB'|^{1/2}     \lb{4.B'}
\end{equation}
of $\bsB'$. 

By Lemmas \ref{12bprime} and \ref{l3.6}, the sesquilinear forms $Q_{\bsV_j}$, 
$j=1,2$, and $Q_{\bsV}$ are well-defined. In addition, $Q_{\bsH_0}$, 
$Q_{\bsV_j}$, $j=1,2$, and $Q_{\bsV}$ are symmetric forms.  

\begin{lemma}\lb{l4.relbdd}
Assume Hypothesis \ref{h2.1}. Then the symmetric forms $Q_{\bsV_j}$, $j=1,2$, and 
$Q_{\bsV}$, defined in  \eqref{hsh}, \eqref{hsht2}, are infinitesimally 
bounded with respect to the form $Q_{\bsH_0}$ of the self-adjoint operator $\bsH_0$
in $L^2(\bbR; \cH)$. Thus, the form sums
\begin{align}
Q_{\hatt \bsH_j} (f,g) &= Q_{\bsH_0} (f,g) + Q_{\bsV_j} (f,g), 
\quad f, g \in \dom(Q_{\hatt \bsH_j}) = \dom(Q_{\bsH_0}),  \quad j=1,2, \\
Q_{\bsH} (f,g) &= Q_{\bsH_0} (f,g) + Q_{\bsV} (f,g), 
\quad f, g \in \dom(Q_{\hatt \bsH_j}) = \dom(Q_{\bsH_0}), 
\end{align}
are densely defined, closed, and bounded from below. Consequently, the forms
$Q_{\hatt \bsH_j}$, $j=1,2$, and $Q_{\bsH}$ uniquely define self-adjoint operators 
$\hatt \bsH_j$, $j=1,2$, and $\bsH$ in $L^2(\bbR; \cH)$, respectively, with 
$\hatt \bsH_j$, $j=1,2$, and $\bsH$ bounded from below, satisfying
\begin{align}
\dom\big(\hatt \bsH_j\big) &= \big\{f \in \dom(Q_{\bsH_0}) \,\big|\,  
\text{the map:} \, \dom(Q_{\bsH_0}) \ni g \mapsto Q_{\hatt \bsH_j} (f,g)   \lb{3.57} \\
& \hspace*{2.8cm}  \text{is continuous in the norm of $L^2(\bbR; \cH)$}\big\}, 
\quad j=1,2,  \no \\
Q_{\hatt \bsH_j} (f,g) &= (f, \hatt \bsH_j g)_{L^2(\bbR;\cH)}, 
\quad f \in \dom(Q_{\bsH_0}), \, g \in \dom\big(\hatt \bsH_j\big),  \quad j=1,2,  \lb{Hsh} \\
\dom(\bsH) &= \big\{f \in \dom(Q_{\bsH_0}) \,\big|\,  
\text{the map:} \, \dom(Q_{\bsH_0}) \ni g \mapsto Q_{\bsH} (f,g) \\
& \hspace*{2.7cm}  \text{is continuous in the norm of $L^2(\bbR; \cH)$}\big\},  \no \\
Q_{\bsH} (f,g) &= (f, \bsH g)_{L^2(\bbR;\cH)}, 
\quad f \in \dom(Q_{\bsH_0}), \, g \in \dom(\bsH),   \lb{Hsht2}
\end{align}
and 
\begin{equation}
\dom\big(|\widehat{\bsH}_j|^{1/2}\big) = \dom\big(|\bsH|^{1/2}\big) 
= \dom\big(\bsH_0^{1/2}\big).   \lb{3.61}
\end{equation} 
\end{lemma}
\begin{proof}
Applying \eqref{rhbdd} and \eqref{3.48} one obtains
\begin{align}
& |(\bsA_- f, \bsB f)_{L^2(\bbR;\cH)}|   \no \\
& \quad = \big|\big(\bsA_- \big(\bsH_0 - z \bsI)^{-1/2}\big(\bsH_0 -  
z \bsI)^{1/2}f,
\bsB \big(\bsH_0 - z \bsI)^{-1/2}\big(\bsH_0 - z \bsI)^{1/2}f\big)_ 
{L^2(\bbR;\cH)}\big|
\no \\
& \quad \leq \big\|\bsA_- \big(\bsH_0 - z \bsI)^{-1/2}\big\|_{\cB(L^2 
(\bbR; \cH))} \big\|\bsB \big(\bsH_0 - z \bsI)^{-1/2}\big\|_{\cB(L^2(\bbR; \cH))}   
\no \\
& \qquad\quad \times \big\|\big(\bsH_0 - z \bsI)^{1/2}f\big\|_{L^2(\bbR; 
\cH)}^2    \no \\
& \quad = \big\|\bsB \big(\bsH_0 - z \bsI)^{-1/2}\big\|_{\cB(L^2 
(\bbR; \cH))}
\big\|\big(\bsH_0 - z \bsI)^{1/2}f\big\|_{L^2(\bbR;\cH)}^2  \no \\
& \quad = a(z) \big\|\big(\bsH_0 - z \bsI)^{1/2}f\big\|_{L^2(\bbR; 
\cH)}^2\,, \quad
  f \in \dom\big(\bsH_0^{1/2}\big),
\end{align}
with
\begin{equation}
a(z) \geq 0 \, \text{ and } \, a(z) \underset{z\downarrow -\infty} 
{\longrightarrow} 0.
\end{equation}
The same estimate now applies to the sesquilinear forms 
\begin{equation} 
(\bsB f, \bsA_- f)_{L^2(\bbR;\cH)}, \quad (\bsB f, \bsB f)_{L^2(\bbR;\cH)}, 
\quad f \in \dom\big(\bsH_0^{1/2}\big).
\end{equation} 
Moreover, Lemma \ref{12bprime} yields the same estimate also for the sesquilinear form 
\begin{equation} 
\big(|(\bsB')^*|^{1/2}f, U_{\bsB'}  |\bsB'|^{1/2}f\big)_{L^2(\bbR;\cH)}, \quad f \in \dom\big(\bsH_0^{1/2}\big).
\end{equation}   
Thus, by \eqref{hsh} and \eqref{hsht2}, the sesquilinear forms 
$Q_{\bsV_j}$, $j=1,2$, and $Q_{\bsV}$ are infinitesimally bounded with respect to 
$Q_{\bsH_0}$. The first and second representation theorem for sesquilinear forms (cf., e.g., \cite[Sect.\ IV.2]{EE89}, \cite[Sect.\ 6.2]{Ka80}) then yields 
\eqref{3.57}--\eqref{3.61} and completes the proof.
\end{proof}

Being defined as a self-adjoint form sum, we note that $\bsH$ is an extension of the operator sum 
$-(d^2/dt^2) + \bsA^2$ defined on $\dom(d^2/dt^2) \cap \dom(\bsA^2)$.

Next we will prove that $\widehat{\bsH}_j$ coincides with $\bsH_j$, $j=1,2$:

\begin{lemma}\lb{l4.hatHH}
Assume Hypothesis \ref{h2.1}. Then, 
\begin{equation}
\widehat{\bsH}_j = \bsH_j, \; j=1,2,   \lb{Hshj}
\end{equation}
where
\begin{equation}
\bsH_1 = \bsD_\bsA^* \bsD_\bsA^{}, \quad \bsH_2 = \bsD_\bsA^{} \bsD_\bsA^*.    \lb{Hj}
\end{equation}
In particular,
\begin{align}
\begin{split} 
& \dom\big(\bsH_1^{1/2}\big) = \dom\big(\bsH_2^{1/2}\big) 
= \dom\big(\bsH^{1/2}\big) = \dom\big(\bsH_0^{1/2}\big)   \\
& \quad = \dom(\bsD_\bsA^{}) = \dom(\bsD_\bsA^*) = \dom(d/dt) \cap \dom(\bsA_-).
\lb{doms}
\end{split} 
\end{align}
\end{lemma}
\begin{proof}
It suffices to prove $\widehat{\bsH}_1 = \bsH_1$. The sesquilinear 
form $Q_{\bsH_1}$ uniquely associated with $\bsH_1$ is given by
\begin{equation}
Q_{\bsH_1} (f,g) = (\bsD_\bsA^{} f,\bsD_\bsA^{} g)_{L^2(\bbR;\cH)},
\quad f, g \in \dom(Q_{\bsH_1}) = \dom(\bsD_\bsA^{})
= \dom\big(\bsH_1^{1/2}\big),
\end{equation}
with
\begin{equation}
Q_{\bsH_1} (f,g) = (f,\bsH_1 g)_{L^2(\bbR;\cH)},
\quad f \in \dom(Q_{\bsH_1}) = \dom(\bsD_\bsA^{}), \,
g \in \dom(\bsH_1).      \lb{3.36}
\end{equation}
Thus, one computes
\begin{align}
Q_{\bsH_1} (f,g) & = (\bsD_\bsA^{} f,\bsD_\bsA^{} g)_{L^2(\bbR;\cH)}   \no  \\
& = ((\bsD_{\bsA_-}^{} + \bsB) f, (\bsD_{\bsA_-}^{} + \bsB) g)_{L^2(\bbR;\cH)}   \no  \\
& = (\bsD_{\bsA_-}^{} f, \bsD_{\bsA_-}^{} g)_{L^2(\bbR;\cH)} + (\bsD_{\bsA_-}^{} f, 
\bsB g)_{L^2(\bbR;\cH)}
+ (\bsB f, \bsD_{\bsA_-}^{} g)_{L^2(\bbR;\cH)}  \no  \\
& \quad + (\bsB f, \bsB g)_{L^2(\bbR;\cH)}   \no  \\
& = (\bsD_{\bsA_-}^{} f, \bsD_{\bsA_-}^{} g)_{L^2(\bbR;\cH)} + (((d/dt)+ \bsA_-) 
f, \bsB g)_{L^2(\bbR;\cH)}   \no  \\
& \quad + (\bsB f, ((d/dt) + \bsA_-) g)_{L^2(\bbR;\cH)}
+ (\bsB f, \bsB g)_{L^2(\bbR;\cH)}   \no  \\
& = (\bsD_{\bsA_-}^{} f, \bsD_{\bsA_-}^{} g)_{L^2(\bbR;\cH)} + (\bsA_- f, \bsB 
g)_{L^2(\bbR;\cH)}
+ (\bsB f, \bsA_- g)_{L^2(\bbR;\cH)}   \no  \\
& \quad + (\bsB f, \bsB g)_{L^2(\bbR;\cH)} + (f', \bsB g)_{L^2(\bbR;\cH)}
+ (\bsB f, g')_{L^2(\bbR;\cH)}   \no  \\
& = (\bsD_{\bsA_-}^{} f, \bsD_{\bsA_-}^{} g)_{L^2(\bbR;\cH)} + (\bsA_- f, \bsB 
g)_{L^2(\bbR;\cH)}
+ (\bsB f, \bsA_- g)_{L^2(\bbR;\cH)}    \no  \\
& \quad + (\bsB f, \bsB g)_{L^2(\bbR;\cH)} 
- \big(|(\bsB')^*|^{1/2} f, U_{\bsB'} |\bsB'|^{1/2} g\big)_{L^2(\bbR;\cH)},   \lb{3.37} \\
& \hspace*{4cm} f, g \in \dom(\bsD_\bsA^{}) = \dom(\bsD_{\bsA_-}^{}).   \no
\end{align} 
The last step is a consequence of the following observations:
\begin{align}
& (f', \bsB g)_{L^2(\bbR;\cH)} + (\bsB f, g')_{L^2(\bbR;\cH)} \no  \\
& \quad = (f', \bsB g)_{L^2(\bbR;\cH)} + (\bsB f, g')_{L^2(\bbR;\cH)}
+ \big(|(\bsB')^*|^{1/2} f, U_{\bsB'} |\bsB'|^{1/2} g\big)_{L^2(\bbR;\cH)}  \no \\
& \qquad - \big(|(\bsB')^*|^{1/2} f, U_{\bsB'} |\bsB'|^{1/2} g\big)_{L^2(\bbR;\cH)}   
\no  \\
& \quad = \lim_{R\to\infty} \int_{-R}^R \big[(f'(t),B(t) g(t))_{\cH} 
+ (f(t),B(t) g'(t))_{\cH} + (f(t),B'(t) g(t))_{\cH}\big] \, dt   \no \\
& \qquad - \big(|(\bsB')^*|^{1/2} f, U_{\bsB'} |\bsB'|^{1/2} g\big)_{L^2(\bbR;\cH)}   
\no  \\
& \quad = \lim_{R\to\infty} \int_{-R}^R \frac{d}{dt} (f(t),B(t) g(t))_{\cH} \, dt 
- \big(|(\bsB')^*|^{1/2} f, U_{\bsB'} |\bsB'|^{1/2} g\big)_{L^2(\bbR;\cH)}   \no  \\
& \quad = \lim_{R\to\infty}(f(R),B(R) g(R))_{\cH} 
- \lim_{R\to\infty}(f(-R),B(-R) g(-R))_{\cH}    \no \\
& \qquad - \big(|(\bsB')^*|^{1/2} f, U_{\bsB'} |\bsB'|^{1/2} g\big)_{L^2(\bbR;\cH)}   
\no  \\
& \quad = - \big(|(\bsB')^*|^{1/2} f, U_{\bsB'} |\bsB'|^{1/2} g\big)_{L^2(\bbR;\cH)}, 
\quad f, g \in \dom(\bsD_\bsA^{}) = \dom(\bsD_{\bsA_-}^{}).     \lb{3.38}
\end{align}
Here we used the fact that the limits $\lim_{R\to\pm\infty}(f(R),B(R) g(R))_{\cH}$, 
exist since 
\beq
(f'(\cdot),B(\cdot) g(\cdot))_{\cH}, \, (f(\cdot),B(\cdot) g'(\cdot))_{\cH}, \, 
(f(\cdot),B'(\cdot) g(\cdot))_{\cH} \in L^1(\bbR; dt). 
\enq
Moreover, since also 
\beq
(f(\cdot),B(\cdot) g(\cdot))_{\cH} \in L^1(\bbR; dt),  
\enq
one concludes that 
\begin{equation}
\lim_{R\to\pm\infty}(f(R),B(R) g(R))_{\cH} = 0,  
\end{equation}
completing the derivation of \eqref{3.38} and hence of \eqref{3.37}. Equations 
\eqref{3.19} and \eqref{3.37} then imply
\begin{align}
Q_{\bsH_1} (f,g) & = (f, \bsH_0 g)_{L^2(\bbR;\cH)} 
+ (f, \bsA_- \bsB g)_{L^2(\bbR;\cH)}
+ (f, \bsB \bsA_- g)_{L^2(\bbR;\cH)}   \no  \\
& \quad + (f, \bsB^2 g)_{L^2(\bbR;\cH)} 
- \big(|(\bsB')^*|^{1/2} f, U_{\bsB'} |\bsB'|^{1/2} g\big)_{L^2(\bbR;\cH)},     \no  \\
& =  Q_{\hatt \bsH_1} (f,g) , \quad  f, g \in 
\dom(Q_{\bsH_1}) = \dom(Q_{\hatt \bsH_1}) = \dom\big(\bsH_0^{1/2}\big),    \lb{3.42}
\end{align}
and hence $\bsH_1 = \widehat{\bsH}_1$.
\end{proof}

We will use the following notations for the resolvents of the 
operators $\bsH_j$, $j=1,2$, and $\bsH$:
\begin{equation} \lb{RRR}
\bsR_j(z)=(\bsH_j - z \, \bsI)^{-1}, \quad z\in\rho(\bsH_j), \; j=1,2, \quad
\bsR(z)=(\bsH - z \, \bsI)^{-1}, \quad z\in \rho(\bsH).
\end{equation}

Next, we will discuss in detail the properties of the approximative operators
introduced in \eqref{defAn}. Since $P_n=E_{A_-}((-n,n))$, $n\in\bbN$, 
is the spectral projection for $A_-$, we recall the following 
commutation formulas (cf.\ \eqref{defAn}):
\begin{align}
& A_{-,n}= P_n A_{-,n}=A_{-,n} P_n=P_n A_{-,n} P_n = A_-P_n
=P_nA_-=P_nA_- P_n,   \no \\ 
& A_{+,n} = P_n A_+ P_n,    \no \\ 
& (A_- - z I)^{-1}P_n=P_n(A_- - z I)^{-1},  \quad z \in \rho(A_-),     \lb{comAm} \\ 
& B_n(t)=P_nB(t)P_n, \quad B_n'(t) = P_nB'(t)P_n, \quad 
B_n(+\infty) =P_nB(+\infty)P_n, \quad n\in\bbN.   \no 
\end{align}

Next, one recalls the following properties of the spectral projections $P_n$ in $\cH$:
\begin{align}\lb{Pn1}
&\slim_{n\to\infty}P_n = I,     \\
\lb{Pn2}
& \ran (P_n) \subseteq \dom(A_-), \quad n\in\bbN,  \\
\lb{Pn3}
& \lim_{n\to\infty}\|P_nA_-P_nw-A_-w\|_{\cH} = 0, \quad w\in\dom(A_-).
\end{align}
We collect some basic properties of the operators introduced in 
\eqref{defAn} in the next lemma:

\begin{lemma}\lb{propappr} 
Assume Hypothesis \ref{h2.1}. Then
\begin{align}
&\int_\bbR\|[B'(t)-B'_n(t)](A_-^2+I)^{-1/2}\|_{\cB_1(\cH)}\,dt\to0
\, \text{ as } \, n\to\infty;\lb{Pn4} \\
& \lim_{n\to\infty}\|[B(+\infty)-B_n(+\infty)](A_-^2+I)^{-1/2}\|_{\cB_1(\cH)}  \no \\
& \quad 
= \lim_{n\to\infty}\|[A_+-A_--A_{+,n}+A_{-,n}](A_-^2+I)^{-1/2}\|_{\cB_1(\cH)} 
= 0,   \lb{Pn5}  \\
&A_{\pm,n}\to A_{\pm} \, \text{ in the strong resolvent sense in 
$\cH$ as $n\to\infty$},
\lb{Pn6}\\
& \lim_{n\to\infty}
\big\|(\bsB - \bsB_n)\big(\bsA_-^2 + \bsI\big)^{-1/2}\big\|_{\cB(L^2(\bbR;\cH))}= 0.
\lb{Pn7}
\end{align}
\end{lemma}
\begin{proof}
As usual, we abbreviate $\varkappa=(A_-^2+I)^{1/2}$. To prove \eqref{Pn4}, we will employ the dominated convergence theorem. 
By \eqref{comAm} one infers that
\begin{align}
\begin{split}
&\|[B'(t)-B'_n(t)]\varkappa^{-1}\|_{\cB_1(\cH)}=\|B'(t)\varkappa^{-1}
-P_nB'(t)\varkappa^{-1} P_n\|_{\cB_1(\cH)}      \\
& \quad \le 2\|B'(t)\varkappa^{-1}\|_{\cB_1(\cH)},\lb{3.24}
\end{split} 
\end{align}
and the function in the right-hand side of \eqref{3.24} is summable thanks to
\eqref{2.13i}. For each $t\in\bbR$, due to \eqref{comAm}, one may write 
\begin{align}
\begin{split}
& 
[B'(t)-B'_n(t)]\varkappa^{-1}=B'(t)\varkappa^{-1}-P_nB'(t)\varkappa^{-1}P_n 
\\
& \quad =P_nB'(t)\varkappa^{-1}(I-P_n)
+(I-P_n)B'(t)\varkappa^{-1}.
\end{split} 
\end{align}
Since $B'(t)\varkappa^{-1}=B'(t)(|A_-|+I)^{-1}\cdot (|A_-|+I)\varkappa^{-1}\in\cB_1(\cH)$ by Hypothesis 
\ref{h2.1}\,$(iv)$, and since $P_n\to I$ in $\cH$ strongly as 
$n\to\infty$, one can apply Lemma \ref{lSTP},  thus finishing the 
proof of \eqref{Pn4}. By definition, the operators under the $\cB_1(\cH)$-norm 
on either side in equation \eqref{Pn5} are equal (cf.\ \eqref{defAn}, \eqref{dfnAplus}).  Because of 
\begin{equation}
[B(+\infty)-B_n(+\infty)]\varkappa^{-1}=\int_\bbR [B'(t)-B'_n(t)]\varkappa^{-1}\,dt,
\end{equation}
assertion \eqref{Pn5} follows from \eqref{Pn4}. That $A_{-,n}\to A_-$ in
strong resolvent sense follows from \eqref{Pn3} and \cite[Theorem 
VIII.25(a)]{RS80}. To see that  $A_{+,n}\to A_+$ in
strong resolvent sense as $n\to\infty$, one writes 
\begin{align}
& (A_+ + i I)^{-1} - (A_{+,n} + i I)^{-1}=-(A_{+,n} + i 
I)^{-1}(A_+-A_{+,n})(A_+ + i I)^{-1}   \no \\
& \quad =-(A_{+,n} + i I)^{-1}[A_+-A_{+,n}-A_-+A_{-,n}]\varkappa^{-1}
\varkappa(A_+ + i I)^{-1}
\lb{strre1}   \\
& \qquad  -(A_{+,n} + i I)^{-1}(A_--A_{-,n})(A_+ + i I)^{-1}.\lb{strre2}
\end{align}
Since $\|(A_{+,n} + i I)^{-1}\|_{\cB(\cH)}\le1$ for all $n$ for the 
self-adjoint operator $A_{+,n}$, and $\varkappa(A_+ + i 
I)^{-1} = (A_-^2 + I)^{1/2}(A_+ + i I)^{-1} \in\cB(\cH)$ due to \eqref{dfnAplus}, the 
sequence of operators
in \eqref{strre1} converges to zero as $n\to\infty$ (even in $\cB_1(\cH)$
due to \eqref{Pn5}) while the sequence of the operators
in \eqref{strre2} converges to zero as $n\to\infty$ strongly in $\cH$
due to \eqref{Pn3}. Finally, relation \eqref{Pn7} follows from Remark \ref{TnormT}, the estimate
\begin{align}
& \big\|(\bsB - \bsB_n)(\bsA_-^2 + \bsI)^{-1/2}\big\|_{\cB(L^2(\bbR;\cH))}
=\sup_{t\in\bbR}\|[B(t)-B_n(t)]\varkappa^{-1}\|_{\cB(\cH)}    \no \\
&\quad =\sup_{t\in\bbR}\Big\|
\int_{-\infty}^t [B'(\tau)-B'_n(\tau)]\varkappa^{-1}\,d\tau\Big\|_{\cB(\cH)} 
\no \\
& \quad 
\le\int_{-\infty}^\infty \|[B'(\tau)-B'_n(\tau)] \varkappa^{-1}\|_{\cB_1(\cH)}\,d\tau\,,
\end{align}
and \eqref{Pn4}.
\end{proof}

\section{The Left-Hand Side of the Trace Formula and Approximations} \lb{s5}

In this section we deal with the left-hand sides of formulas \eqref{trfOLD} 
and \eqref{trfn}, assuming Hypothesis \ref{h2.1}. We also recall the 
notations introduced in \eqref{dfnvk1}, \eqref{H0}, \eqref{Hsht2}, 
\eqref{Hj}, and \eqref{RRR}, and Lemma \ref{l4.relbdd}.

We start by proving the first inclusion in \eqref{2.21} (the second inclusion is 
proved similarly) and repeatedly use the generalized polar decomposition 
described in \eqref{PDT}. In addition, we will frequently rely on resolvent formulas 
familiar from the perturbation theory of quadratic forms (and more generally, 
for perturbations permitting appropriate factorizations) as pioneered by 
Kato \cite{Ka66} and applied to Schr\"odinger operators by Simon \cite{Si71}  
(see also \cite[Sections\ 2, 3]{GLMZ05}).   

\begin{lemma} \lb{trclLHS}
Assume Hypothesis \ref{h2.1}. Then
\beq
\big[(\bsH_2 - z \, \bsI)^{-1}-(\bsH_1 - z \, \bsI)^{-1}\big] \in\cB_1(L^2(\bbR;\cH)),
\quad z\in\rho(\bsH_2) \cap \rho(\bsH_1),
\enq
for the resolvents of the operators defined in \eqref{dfnHtH}.
\end{lemma}
\begin{proof}
By Lemma \ref{12bprime}, one infers that 
\beq 
\big[|(\bsB')^*|^{1/2} \bsR_0(z)^{1/2}\big]^* U_{\bsB'} |\bsB'|^{1/2} \, 
\bsR_0(z)^{1/2} \in \cB_1(L^2(\bbR;\cH)), 
\quad z \in\bbC\backslash [0,\infty).     \lb{clB1}
\enq
Combining \eqref{h0h}, \eqref{hsh}, \eqref{4.B'}, Lemmas \ref{l4.relbdd} and 
\ref{l4.hatHH}, and equation \eqref{clB1}, one computes (for simplicity) for $z<0$, 
\begin{align}
&(\bsH_2 - z \, \bsI)^{-1}-(\bsH_1 - z \, \bsI)^{-1}  \no \\
& \quad = - 2 \ol{(\bsH_1 - z \, \bsI)^{-1} |(\bsB')^*|^{1/2} U_{\bsB'} |\bsB'|^{1/2}
\, (\bsH_2 - z \, \bsI)^{-1}}  \no \\ 
& \quad = - 2 \big[|(\bsB')^*|^{1/2} (\bsH_1 - z \, \bsI)^{-1}\big]^*  
U_{\bsB'} |\bsB'|^{1/2} \, (\bsH_2 - z \, \bsI)^{-1}   \no \\
& \quad = 2 (\bsH_1 - z \, \bsI)^{-1/2} \, 
\big[(\bsH_0 - z \, \bsI)^{1/2}(\bsH_1 - z \, \bsI)^{-1/2}\big]^*  \no \\ 
& \qquad \times \big[|(\bsB')^*|^{1/2} (\bsH_0 - z \, \bsI)^{-1/2}\big]^*  
U_{\bsB'} |\bsB'|^{1/2} \, (\bsH_0 - z \, \bsI)^{-1/2}  \no \\
& \qquad \times \big[(\bsH_2 - z \, \bsI)^{-1/2}(\bsH_0 - z \, \bsI)^{1/2}\big]^* 
(\bsH_2 - z \, \bsI)^{-1/2} \in \cB_1(L^2(\bbR;\cH)). 
\end{align}
By analytic continuation with respect to $z$ based on resolvent equations in a standard manner, this extends to $z\in\rho(\bsH_2) \cap \rho(\bsH_1)$. We note that the 
resolvent equations used repeatedly at the beginning of this computation follow from 
the results in \cite[Sect.\ 1]{Ka66} (see also \cite[Sects.\ 2, 3]{GLMZ05}, \cite[Ch.\ II]{Si71}).
\end{proof}

To prove \eqref{2.22} in Proposition \ref{prTr}, we will need one more technical lemma. We recall the notation introduced in \eqref{H0}, \eqref{hsht2}, 
\eqref{Hsht2}, \eqref{Hj}, \eqref{RRR}, and introduce the following bounded operators 
in $L^2(\bbR; \cH)$:
\begin{align} 
\bsL(z) &= \bsI+ \big[\bsA_- \bsR_0^{1/2}(z)\big]^* \, \bsB \, \bsR_0^{1/2}(z) 
+ \big[\bsB \bsR_0^{1/2}(z)\big]^* \bsA_- \, \bsR_0^{1/2}(z)   \no \\
& \quad + \big[\bsB \bsR_0^{1/2}(z)\big]^* \, \bsB \, \bsR_0^{1/2}(z),    
\quad z<0,      \lb{dfnLa}  \\
\bsL_n(z) &= \bsI+ \big[\bsA_{-,n} \bsR_{0,n}^{1/2}(z)\big]^* \, 
\bsB_n \, \bsR_{0,n}^{1/2}(z) 
+ \big[\bsB_n \bsR_{0,n}^{1/2}(z)\big]^* \bsA_{-,n} \, \bsR_{0,n}^{1/2}(z)   
\no \\
& \quad + \big[\bsB_n \bsR_{0,n}^{1/2}(z)\big]^* \, 
\bsB_n \, \bsR_{0,n}^{1/2}(z), \quad z<0.    \lb{dfnLb} 
\end{align}
In what follows, we use the subscript $n\in\bbN$ for the operators defined in \eqref{H0}, 
\eqref{hsht2}, \eqref{Hsht2}, \eqref{Hj}, and \eqref{RRR}, with $A(t)$, $B(t)$, 
$A_-$ replaced by the operators $A_n(t)$, $B_n(t)$, $A_{-,n}$ introduced in 
\eqref{defAn}.
In addition, one observes that
  \beq\lb{commPn}
\bsR_{0,n}(z)=\bsP_n \, \bsR_0(z) \, \bsP_n=\bsR_0(z) \, 
\bsP_n=\bsP_n \, \bsR_0(z), \quad z\in\bbC\backslash\bbR, 
\enq
with $\bsP_n= E_{\bsA_-}((-n,n))$ the spectral projection for $\bsA_-$.

\begin{lemma}\lb{invLLn}
Assume Hypothesis \ref{h2.1}. Then the following assertions hold for 
the operators defined in \eqref{dfnLa}, \eqref{dfnLb}: \\
$(i)$\,  $ \lim_{n\to\infty} \|\bsL(z) - \bsL_n(z)\|_{\cB(L^2(\bbR;\cH))} = 0$ uniformly for $z\le-1$. \\
$(ii)$ The operators $\bsL(z)$, $\bsL_n(z)$, $n\in\bbN$, are boundedly invertible
on $L^2(\bbR;\cH)$ for $z<0$ and 
\beq\label{n5.7}
\sup_{z\le-1}\|\bsL(z)^{-1}\|_{\cB(L^2(\bbR;\cH))}<\infty,\quad \sup_{z\le-1}\sup_{n\in\bbN} \|\bsL_n(z)^{-1}\|_{\cB(L^2(\bbR;\cH))}<\infty.\enq
\end{lemma}
\begin{proof}
Using \eqref{defAn}, \eqref{commPn}, the fact that
\begin{align}
& \bsA_{-,n} = \bsA_{-,n} \bsP_n = \bsP_n \bsA_{-,n} = \bsP_n \bsA_{-,n} \bsP_n 
= \bsA_- \, \bsP_n = \bsP_n \, \bsA_- = \bsP_n \, \bsA_- \, \bsP_n,  \no \\
& \hspace*{10cm} n\in\bbN,  
\end{align}
and abbreviating $\widehat\varkappa=(\bsA_-^2 + \bsI)^{1/2}$, one obtains 
the following representation:
\begin{align}
\bsL(z)&-\bsL_n (z)
= \big[\widehat \varkappa \bsR_0(z)^{1/2}\big]^* \no\\
&\times \Big[[\bsA_- \widehat \varkappa^{-1}]^* \, \bsB \, \widehat \varkappa^{-1} 
+ [ \bsB \widehat \varkappa^{-1}]^* \bsA_- \widehat \varkappa^{-1}
+ [\bsB \widehat \varkappa^{-1}]^* \, \bsB \, \widehat \varkappa^{-1}   \no \\
&\quad - [\bsA_{-,n} \widehat \varkappa^{-1}]^* \, \bsB_n \widehat \varkappa^{-1}
- [ \bsB_n \widehat \varkappa^{-1}]^* \bsA_{-,n} \widehat \varkappa^{-1}
- [\bsB_n \widehat \varkappa^{-1}]^* \, \bsB_n \widehat \varkappa^{-1}\Big]\no\\
&\quad\quad\times \big[\widehat \varkappa \, \bsR_0(z)^{1/2} \big]  \no  \\
&= \bsJ_1+ \bsJ_2+ \bsJ_3+ \bsJ_4,
\end{align}
where we denoted 
\begin{align}
\bsJ_1&= \big[\widehat \varkappa \bsR_0(z)^{1/2}\big]^* 
[\bsA_- \widehat \varkappa^{-1}]^* [(\bsB - \bsB_n) 
\widehat \varkappa^{-1}] \widehat \varkappa \, \bsR_0(z)^{1/2}, \\
\bsJ_2&= \big[\widehat \varkappa \bsR_0(z)^{1/2}\big]^* 
[(\bsB - \bsB_n) \widehat \varkappa^{-1}]^* 
[\bsA_- \widehat \varkappa^{-1}] \widehat \varkappa \, \bsR_0(z)^{1/2},\\
\bsJ_3&= \big[\widehat \varkappa \bsR_0(z)^{1/2}\big]^* 
[(\bsB - \bsB_n) \widehat \varkappa^{-1}]^* [\bsB \widehat \varkappa^{-1}]
\widehat \varkappa \, \bsR_0(z)^{1/2},\\
\bsJ_4&= \big[\widehat \varkappa \bsR_0(z)^{1/2}\big]^* 
[\bsB_n \widehat \varkappa^{-1}]^*
[(\bsB - \bsB_n)\widehat \varkappa^{-1}] \widehat \varkappa \, \bsR_0(z)^{1/2}.
\end{align}
One observes that 
\beq
\lim_{n\to\infty} \|(\bsB - \bsB_n) \widehat 
\varkappa^{-1}\|_{\cB(L^2(\bbR;\cH))} = 0 \, \text{ and } \, 
\sup_{n\in\bbN} \|\bsB_n \widehat \varkappa^{-1}\|_{\cB(L^2(\bbR;\cH))}<\infty
\enq
by \eqref{Pn7}, and
\beq 
\|\widehat \varkappa \, \bsR_0(z)^{1/2}\|_{\cB(L^2(\bbR;\cH))}\leq 1
\, \text{ uniformly for $z\le-1$}
\enq
by \eqref{n4.60} and 
$\|(\bsA_-^2+\bsI)^{1/2}(\bsA_-^2-z\,\bsI)^{-1/2}\|_{\cB(L^2(\bbR; \cH))}=1$, $z\le -1$. Thus, assertion $(i)$ in Lemma \ref{invLLn} holds. 

 That the operator $\bsL(z)$ is boundedly invertible for 
$z\in\rho(\bsH)\cap(-\infty,0)$ is well-known. In addition, one has the identity 
\begin{align}
& (\bsH - z \, \bsI)^{-1} = \bsR_0(z)^{1/2} [\bsL(z)]^{-1} \bsR_0(z)^{1/2}   \no \\ 
& \quad = \bsR_0(z)^{1/2}
\Big[\bsI+ \big[\bsA_- \bsR_0^{1/2}(z)\big]^* \, \bsB \bsR_0^{1/2}(z) 
+ \big[\bsB \bsR_0^{1/2}(z)\big]^* \bsA_- \bsR_0^{1/2}(z)   \no \\
& \qquad + \big[\bsB \bsR_0^{1/2}(z)\big]^* \, \bsB \bsR_0^{1/2}(z)\Big]^{-1}
\bsR_0(z)^{1/2}, \quad z\in\rho(\bsH)\cap(-\infty,0). 
\end{align}
This is proved as Tiktopoulos' formula in \cite[Section\ II.3]{Si71} by first choosing 
$z<0$ with $|z|$ sufficiently large followed by an analytic continuation with respect 
to $z$. In particular, 
\begin{align}
\bsL^{-1} (z) &= (\bsH_0 - z \, \bsI)^{1/2} \, \bsR(z)^{1/2} \, \big[(\bsH_0 
- z \, \bsI)^{1/2} \bsR(z)^{1/2}\big]^*,  \quad z\in\rho(\bsH)\cap(-\infty,0),  \\
\bsL (z) &= \big[(\bsH - z \, \bsI)^{1/2}\bsR_0(z)^{1/2}\big]^* 
(\bsH - z \, \bsI)^{1/2}\, \bsR_0(z)^{1/2}, \quad z<0, 
\end{align}
illustrating again that both operators $\bsL(z)$ and $\bsL^{-1}(z)$ are 
bounded in $L^2(\bbR; \cH)$ by \eqref{doms}. Analogous considerations 
apply to $\bsL_n(z)$, $n\in\bbN$. 

The rest  of assertion $(ii)$ follows from item $(i)$. Indeed, 
we conclude from \eqref{dfnLa}, \eqref{dfnLb} that
\beq\lb{n5.18} \lim_{z\downarrow -\infty}\|\bsL(z)- \bsI\|_{\cB(L^2(\bbR;\cH))}=0,
\quad  \lim_{z\downarrow -\infty}\|\bsL_n(z)- \bsI\|_{\cB(L^2(\bbR;\cH))}=0
\enq
for each $n\in\bbN$ by Lemma \ref{l3.6}. This implies 
\beq \sup_{z\le-1}\|L(z)^{-1}\|_{\cB(L^2(\bbR;\cH))}<\infty,\quad
\sup_{z\le-1}\|L_n(z)^{-1}\|_{\cB(L^2(\bbR;\cH))}<\infty\enq for each $n\in\bbN$,
and now item $(i)$ implies the second assertion in \eqref{n5.7}.
\end{proof}

At this point we are ready to prove Proposition \ref{prTr}:

\begin{proof} We will abbreviate $\bsR_0^{1/2}=\bsR_0^{1/2}(z)$,
$\bsR_{0,n}^{1/2}=\bsR_{0,n}^{1/2}(z)$ 
and $L=L(z)$, $L_n=L_n(z)$. In view of Lemma \ref{trclLHS}, it remains to show 
\eqref{2.22}. Using Lemma \ref{invLLn}\,$(ii)$ and Lemma \ref{12bprime} 
we choose $z<-1$ with $|z|$ so large that
\begin{align} 
\begin{split}
& \big\|\bsL^{-1/2} \,\big[|(\bsB')^*|^{1/2}\bsR_0^{1/2}\big]^* 
U_{\bsB'} |\bsB'|^{1/2} \, \bsR_0^{1/2} 
\, \bsL^{-1/2}\big\|_{\cB(L^2(\bbR;\cH))}\le1/2,    \lb{boundLL}   \\
& \sup_{n\in\bbN} \big\|\bsL^{-1/2}_n \, 
\big[|(\bsB'_n)^*|^{1/2}\bsR_{0,n}^{1/2}\big]^* U_{\bsB'_n} |\bsB'_n|^{1/2} \,
\bsR_{0,n}^{1/2} \, \bsL^{-1/2}_n \big\|_{\cB(L^2(\bbR;\cH))}\le1/2.    
\end{split} 
\end{align}
Using \eqref{hsh}, \eqref{hsht2} one infers 
\begin{align}
& (\bsH_1-z \, \bsI)^{-1} =\bsR_0^{1/2}\Big[\bsI +
\big[\bsA_- \bsR_0^{1/2}\big]^* \, \bsB \, \bsR_0^{1/2}   \no \\
& \qquad + \big[\bsB \bsR_0^{1/2}\big]^* \bsA_- \, \bsR_0^{1/2}    
 + \big[\bsB \bsR_0^{1/2}\big]^* \, \bsB \, \bsR_0^{1/2}    \no \\
& \qquad - \big[|(\bsB')^*|^{1/2} \bsR_0^{1/2}\big]^*  
U_{\bsB'} |\bsB'|^{1/2} \, \bsR_0^{1/2} 
\Big]^{-1}\bsR_0^{1/2}   \no \\
& \quad = \bsR_0^{1/2} 
\Big[\bsL- \big[|(\bsB')^*|^{1/2} \bsR_0^{1/2}\big]^*  
U_{\bsB'} |\bsB'|^{1/2} \, \bsR_0^{1/2} \Big]^{-1}\bsR_0^{1/2}   \no \\
& \quad = \bsR_0^{1/2} \, \bsL^{-1/2} \Big[\bsI-\bsL^{-1/2} \, 
\big[|(\bsB')^*|^{1/2} \bsR_0^{1/2}\big]^* U_{\bsB'} |\bsB'|^{1/2} \,
\bsR_0^{1/2} \, \bsL^{-1/2}\Big]^{-1}    \no \\
& \qquad \times \bsL^{-1/2} \, \bsR_0^{1/2}.  
\end{align} 
A similar calculation for $\bsH_2$ and \eqref{boundLL} show that
the resolvents $\bsR_1 = \bsR_1 (z)$ and $\bsR_2=\bsR_2(z)$ can be 
computed as follows (and similarly for $\bsR_{1,n}$, $\bsR_{2,n}$):
\begin{align}
\bsR_1&=\bsR_0^{1/2}\,\bsL^{-1/2} 
\Big[\bsI-\bsL^{-1/2} \, \big[|(\bsB')^*|^{1/2} \bsR_0^{1/2}\big]^*  
U_{\bsB'} |\bsB'|^{1/2} \, \bsR_0^{1/2} \, \bsL^{-1/2}\Big]^{-1} \no \\ 
& \quad \times \bsL^{-1/2}\, \bsR_0^{1/2},    \\
\bsR_2&=\bsR_0^{1/2}\,\bsL^{-1/2} \Big[\bsI+\bsL^{-1/2}\, 
\big[|(\bsB')^*|^{1/2} \bsR_0^{1/2}\big]^* 
U_{\bsB'} |\bsB'|^{1/2} \, \bsR_0^{1/2} \,\bsL^{-1/2}\Big]^{-1} \no \\  
& \quad \times \bsL^{-1/2}\, \bsR_0^{1/2}.
\end{align}
Introducing the bounded operators
\begin{align}
\bsM&=\bsR_0^{1/2}\,\bsL^{-1/2} \Big[\bsI+\bsL^{-1/2}\, 
\big[|(\bsB')^*|^{1/2} \bsR_0^{1/2}\big]^*  
U_{\bsB'} |\bsB'|^{1/2} \, \bsR_0^{1/2} \, \bsL^{-1/2}\Big]^{-1} \, \bsL^{-1/2},  
\lb{MN}\\
\bsN&=\bsL^{-1/2} \Big[\bsI-\bsL^{-1/2}\, 
\big[|(\bsB')^*|^{1/2} \bsR_0^{1/2}\big]^* 
U_{\bsB'} |\bsB'|^{1/2} \, \bsR_0^{1/2}
\, \bsL^{-1/2}\Big]^{-1} \, \bsL^{-1/2} \, \bsR_0^{1/2},     \lb{MN1}\\
\bsM_n&=\bsR_{0,n}^{1/2} \,\bsL^{-1/2}_n 
\Big[\bsI+\bsL^{-1/2}_n\, \big[|(\bsB'_n)^*|^{1/2 }\bsR_{0,n}^{1/2}\big]^* 
U_{\bsB'_n} |\bsB'_n|^{1/2} 
\,\bsR_{0,n}^{1/2} \, \bsL^{-1/2}_n\Big]^{-1} \, \bsL^{-1/2}_n,    \lb{MN2}\\
\bsN_n&=\bsL^{-1/2}_n \Big[\bsI-\bsL^{-1/2}_n \, 
\big[|(\bsB'_n)^*|^{1/2 }\bsR_{0,n}^{1/2}\big]^* 
U_{\bsB'_n} |\bsB'_n|^{1/2} \, \bsR_{0,n}^{1/2}
 \, \bsL^{-1/2}_n\Big]^{-1} \, \bsL^{-1/2}_n \, 
\bsR_{0,n}^{1/2} ,  \lb{MN3}
\end{align}
one obtains the following identities:
\begin{align} 
\begin{split}
\bsR_1-\bsR_2 &= 2\, \bsM \, 
\big[|(\bsB')^*|^{1/2} \bsR_0^{1/2}\big]^* 
U_{\bsB'} |\bsB'|^{1/2} \, \bsR_0^{1/2} \,\bsN,    \lb{resdif} \\
\quad \bsR_{1,n}-\bsR_{2,n} &= 2 \, \bsM_n\,  
\big[|(\bsB'_n)^*|^{1/2 }\bsR_{0,n}^{1/2}\big]^* 
U_{\bsB'_n} |\bsB'_n|^{1/2} \, \bsR_{0,n}^{1/2} \, \bsN_n.
\end{split}
\end{align}

We need two more preparatory facts to finish the proof of Proposition 
\ref{prTr}: First, we claim that
\begin{align} 
\begin{split} 
& \lim_{n\to\infty} \Big\|\big[|(\bsB')^*|^{1/2} \bsR_0^{1/2}\big]^*  
U_{\bsB'} |\bsB'|^{1/2} \, \bsR_0^{1/2}     \lb{12resb1} \\
& \qquad \quad - \big[|(\bsB'_n)^*|^{1/2 }\bsR_{0,n}^{1/2}\big]^* 
U_{\bsB'_n} |\bsB'_n|^{1/2} \, \bsR_{0,n}^{1/2} 
\Big\|_{\cB_1(L^2(\bbR;\cH))} = 0.  
\end{split} 
\end{align}
Indeed, since the spectral projection $\bsP_n$ and the operator 
$-\frac{d^2}{dt^2}$ commute (cf.\ \eqref{commPn}),
\beq\lb{PnR}
\bsP_n \, \bsR_{0,n}^{1/2} =\bsP_n \, \bsR_0^{1/2},
\quad \bsR_{0,n}^{1/2} \, \bsP_n=\bsR_0^{1/2} \, \bsP_n.
\enq
Since $\bsB'_n= \bsP_n \, \bsB' \, \bsP_n$, one can write
\begin{equation} 
\big[|(\bsB'_n)^*|^{1/2 }\bsR_{0,n}^{1/2}\big]^* 
U_{\bsB'_n} |\bsB'_n|^{1/2} \, \bsR_{0,n}^{1/2} 
= \big[|(\bsB'_n)^*|^{1/2 }\bsR_0^{1/2}\big]^* 
U_{\bsB'_n} |\bsB'_n|^{1/2} \, \bsR_0^{1/2}, 
\end{equation}
and, after a short calculation with scalar products using \eqref{4.B'} for $\bsB'$, $\bsB'_n$, and $\bsB'-\bsB'_n$,  obtain the estimate
\begin{align}
& \Big\|\big[|(\bsB')^*|^{1/2} \bsR_0^{1/2}\big]^* 
U_{\bsB'} |\bsB'|^{1/2} \, \bsR_0^{1/2}   \no \\
& \qquad - \big[|(\bsB'_n)^*|^{1/2 }\bsR_{0,n}^{1/2}\big]^* 
U_{\bsB'_n} |\bsB'_n|^{1/2} \, \bsR_{0,n}^{1/2} 
\Big\|_{\cB_1(L^2(\bbR;\cH))}  \\
& \quad = \Big\|\big[|(\bsB' - \bsB'_n)^*|^{1/2} \bsR_0^{1/2}\big]^*
U_{[\bsB' - \bsB'_n]} |\bsB' - \bsB'_n|^{1/2} 
\bsR_0^{1/2}\Big\|_{\cB_1(L^2(\bbR;\cH))}   \no   \\
& \quad \le c\int_\bbR \|[B'(t) - B'_n(t)] (A_-^2+I)^{-1/2}\|_{\cB_1(\cH)}\,dt,
\end{align}
using Lemma \ref{12bprime} with $\bsB'$ replaced by 
$[\bsB' - \bsB'_n]$. Now claim
\eqref{12resb1} follows from \eqref{Pn4}.

Second, we claim that
\beq\lb{strMN}
\slim_{n\to\infty} \bsM_n = \bsM \, \text{ and } \, \slim_{n\to\infty} 
\bsN_n = \bsN
\, \text{ in } \, L^2(\bbR;\cH).
\enq
Indeed, referring to equations \eqref{MN}--\eqref{MN3}, one notes that
$\slim_{n\to\infty} \bsR_{0,n}^{1/2} = \bsR_0^{1/2}$ in $L^2(\bbR;\cH)$, while
$\lim_{n\to\infty} \big\|\bsL^{-1/2}_n - 
\bsL^{-1/2}\big\|_{\cB(L^2(\bbR;\cH))}=0$,
because of the strong resolvent convergence of the self-adjoint operators 
$\bsL_n$ to $\bsL$ as $n\to\infty$ by Lemma \ref{invLLn}. 
Also, due to \eqref{boundLL}, the norms of the operators satisfy
\beq\label{n5.34}
\sup_{n\in\bbN} \Big\|\Big[\bsI\pm \bsL^{-1/2}_n \, 
\big[|(\bsB'_n)^*|^{1/2 }\bsR_{0,n}^{1/2}\big]^* 
U_{\bsB'_n} |\bsB'_n|^{1/2} \, \bsR_{0,n}^{1/2} \, 
\bsL^{-1/2}_n\Big]^{-1}\Big\|_{\cB(L^2(\bbR;\cH))} < \infty.
\enq
Combining this with \eqref{12resb1} proves the claim \eqref{strMN}.

Finally, using \eqref{resdif} and \eqref{PnR}, one infers
\begin{align}
& \bsR_1 - \bsR_2-(\bsR_{1,n}-\bsR_{2,n})    \no \\
& \quad =2\, \bsM\, \big[|(\bsB')^*|^{1/2} \bsR_0^{1/2}\big]^* 
U_{\bsB'} |\bsB'|^{1/2} \, \bsR_0^{1/2}  \, \bsN    \no \\
& \qquad - 2 \, \bsM_n \, 
\big[|(\bsB'_n)^*|^{1/2 }\bsR_{0,n}^{1/2}\big]^* 
U_{\bsB'_n} |\bsB'_n|^{1/2} \, \bsR_{0,n}^{1/2} \,\bsN_n    \no \\
& \quad =\bsJ^{(n)}_1+ \bsJ^{(n)}_2,
\end{align}
where we denoted
\begin{align}
\bsJ^{(n)}_1&=2(\bsM-\bsM_n)\, 
\big[|(\bsB')^*|^{1/2} \bsR_0^{1/2}\big]^*  
U_{\bsB'} |\bsB'|^{1/2} \, \bsR_0^{1/2}  \,\bsN,  \\
\bsJ^{(n)}_2&= 2 \Big[\bsN \, 
\big[|(\bsB')^*|^{1/2} \bsR_0^{1/2}\big]^* 
U_{\bsB'} |\bsB'|^{1/2} \, \bsR_0^{1/2}   \no \\
& \qquad  - \bsN_n \,
\big[|(\bsB'_n)^*|^{1/2 }\bsR_{0,n}^{1/2}\big]^*  
U_{\bsB'_n} |\bsB'_n|^{1/2} \, \bsR_{0,n}^{1/2} \Big]^*   \no \\
&= 2 \Big[(\bsN-\bsN_n) \, 
\big[|(\bsB')^*|^{1/2} \bsR_0^{1/2}\big]^* 
U_{\bsB'} |\bsB'|^{1/2} \, \bsR_0^{1/2}  \no \\
& \qquad + \bsN_n \Big(
\big[|(\bsB')^*|^{1/2} \bsR_0^{1/2}\big]^* 
U_{\bsB'} |\bsB'|^{1/2} \, \bsR_0^{1/2}   \no \\
& \qquad \qquad \quad 
- \big[|(\bsB'_n)^*|^{1/2 }\bsR_{0,n}^{1/2}\big]^*  
U_{\bsB'_n} |\bsB'_n|^{1/2} \, \bsR_{0,n}^{1/2} \Big)\Big]^*.
\end{align}
Since $\big[|(\bsB')^*|^{1/2} \bsR_0^{1/2}\big]^* 
U_{\bsB'} |\bsB'|^{1/2} \, \bsR_0^{1/2} \in\cB_1(L^2(\bbR;\cH))$ by
Lemma \ref{12bprime}, one concludes that
\beq
\lim_{n\to\infty} \big\|\bsJ^{(n)}_{j}\big\|_{\cB_1(L^2(\bbR;\cH))} = 0, \quad j=1,2,
\enq
by \eqref{12resb1}, \eqref{strMN}, and Lemma \ref{lSTP}.
\end{proof}

\section{The Right-Hand Side of the Trace Formula  
and Double Operators Integrals} \lb{s6}

In this section we deal with the right-hand side of the trace
formula \eqref{trfOLD}, and prove Proposition \ref{propgA}. Our
approach based on the theory of double operator integrals. This
theory originated in \cite{BS65}, \cite{BS67},\cite{BS68},
\cite{BS73}, \cite{BS75}, \cite{BS93} (see also the  reviews in 
\cite{BS03}, \cite{Pe09}, \cite{PS10} 
and more recent further developments in  \cite{CPS09}, \cite{dPS04},
\cite{dPSW02}, \cite{PS08}, \cite{PS09}, \cite{PS10}).

To show the first inclusion in assertion
\eqref{assrti} of Proposition \ref{propgA}, we will follow the strategy
in \cite{CPS09}, \cite{PS09}; in particular, see equation (23) in 
\cite[Section\ 6]{CPS09}, where the inclusion  
\beq\lb{subtle}
[g(S_+)-g(S_-)] \in\cB(\cH)
\enq
is proved, assuming $(S_+ - S_-)(S_-^2+I)^{-1/2}\in\cB(\cH)$.
Lemma \ref{ext_subtle} below yields this inclusion with $\cB(\cH)$ replaced by
$\cB_1(\cH)$, but assuming
$(S_+ - S_-)(S_-^2 + I)^{-1/2} \in\cB_1(\cH)$. The argument in 
Lemma \ref{ext_subtle}
involves the concept of double operator integrals. 

We begin by recalling some relevant background
material regarding double operator integrals (cf.\ \cite{CPS09},
\cite{dPS04}, \cite{dPSW02}, \cite{PS08}, \cite{PS09}) and fix two
unbounded self-adjoint operators $S_+$ and $S_-$ in $\cH$.

Let $\mathfrak{A}_0$ denote the set of all bounded Borel functions
$\phi$ admitting the representation 
\beq
\phi(\lambda,\mu)=\int_\bbR\alpha_s(\lambda)\beta_s(\mu)\,d\nu(s), \quad
(\lambda,\mu)\in\bbR^2,    \lb{phialbe} \enq
where 
$\alpha_s(\cdot)$, $\beta_s(\cdot):\bbR\to\bbC$, for
each $s\in\bbR$,  are bounded Borel functions satisfying
\beq
\int_\bbR\|\alpha_s\|_\infty\|\beta_s\|_\infty\,d\nu(s)<\infty,\lb{intalbe}
\enq
and $d\nu$ is a positive Borel measure on $\bbR$ (cf.\ \cite[Proposition 4.7]{dPS04} or
\cite[Corollary 2]{PS09}). We introduce the norm  on $\mathfrak{A}_0$ as the
infimum of the integrals in \eqref{intalbe} taken over all
possible representations in \eqref{phialbe}. It is easy to see that
$\mathfrak{A}_0$ is a Banach algebra.

Given two self-adjoint operators $S_+$ and
$S_-$ in $\cH$, one defines for each $\phi\in\mathfrak{A}_0$ the
operator $T_{\phi,1}=T_\phi^{(S_+,S_-)}\in\cB(\cB_1(\cH))$, that is, a 
bounded operator from $\cB_1(\cH)$ to itself, as the following integral,
absolutely convergent in $\cB_1(\cH)$-norm 
\beq \lb{intreprOI}
T_{\phi,1}(K)=\int_\bbR\alpha_s(S_+)K\beta_s(S_-)d\nu(s),  \quad
K\in\cB_1(\cH).
\enq 
We will call $T_{\phi,1}=T_\phi^{(S_+,S_-)}$
the {\it  operator integral}; the proof of the fact that
$T_{\phi,1}$ is well-defined follows along the same lines as in
\cite[Lemma 4.3]{ACDS09}. The definition above (see also
\cite{ACDS09}) is a particular case of the definition of the
double operator integrals considered in \cite{CPS09}, \cite{dPS04},
\cite{dPSW02}, \cite{PS08}, \cite{PS09}. Replacing
$\cB_1(\cH)$ above with $\cB(\cH)$ one obtains a bounded operator
$T_{\phi,\infty}$ from $\cB(\cH)$ to $\cB(\cH)$. If 
$\phi\in \mathfrak{A}_0$ satisfies the condition 
$\phi(\lambda,\mu)=\phi(\mu,\lambda)$, $(\lambda,\mu)\in\bbR^2$ (and we 
will consider only such $\phi$'s) then $T_{\phi,1}^*=T_{\phi,\infty}$ and 
$T_{\phi,\infty}|_{\cB_1(\cH)}=T_{\phi,1}$ (cf.\ \cite[Lemma 2.4]{PS08}). 
We note that
\beq\lb{normTphi}
\|T_{\phi,1}\|_{\cB(\cB_1(\cH))}=\|T_{\phi,\infty}\|_{\cB(\cB(\cH))}\leq
  \|\phi\|_{ \mathfrak{A}_0} 
  \enq
(cf.\ \cite{dPS04}, \cite{dPSW02}). In what follows we use the notation $T_\phi$ 
for either $T_{\phi,1}$ or $T_{\phi,\infty}$, which should not lead to a confusion.
We remark the following two
properties of the mapping $\phi \to T_\phi$ for which we
again refer to \cite[Lemma 2.4]{PS08}:

$(i)$ $\phi \to T_\phi$ is a homomorphism of $\mathfrak{A}_0$
into $\cB(\cB_1(\cH))$ (or $\cB(\cB(\cH))$), that is,
$T_{\phi\psi}=T_\phi T_\psi$ for $\phi, \psi\in \mathfrak{A}_0$.

$(ii)$ $T_\phi$ is wo-continuous (i.e., continuous in the weak operator topology, 
or ultraweakly continuous) on $\cB(\cH)$. Indeed, $T_{\phi,\infty}$ on $\cB(\cH)$ 
is dual to $T_{\phi,1}$ and therefore is ultra-weakly continuous as a dual operator.

In addition, given bounded Borel functions $\alpha,\beta:\bbR\to\bbC$, one
notes that if $\phi(\lambda,\mu)=\alpha(\lambda)$, then
$T_\phi(K)=\alpha(S_+)K$, and if $\phi(\lambda,\mu)=\beta(\mu)$, 
then $T_\phi(K)=K\beta(S_-)$, $K\in\cB_1(\cH)$ (or $K\in\cB(\cH)$,
cf.\ \cite{dPS04}, \cite{dPSW02}).

\begin{hypothesis}\label{spm}
Assume  that  $S_+$ and $S_-$ are  self-adjoint operators in $\cH$. 
Given two bounded $($real-valued\,$)$ Borel functions $\alpha $ and $\beta$ on $\bbR$,
suppose that $\cD\subseteq \dom(S_-)$ is a core for the operator $S_-$ such that
\beq
\beta(S_-) \cD \subseteq \dom (S_+\alpha(S_+)).
\enq
Assume that the operator $K=K(S_+,S_-)$ in $\cH$ defined by
\beq\label{Kinit}
K=S_+\alpha(S_+)\beta(S_-) -\alpha(S_+)  S_-\beta(S_-), \quad \dom(K)=\cD, 
\enq
is closable and $\overline{K} \in \cB(\cH)$.
\end{hypothesis}

\begin{lemma}\label{zam}
Assume Hypothesis \ref{spm}. Then
\beq\lb{n6.8}
\beta(S_-) \dom (S_-) \subseteq \dom (S_+\alpha(S_+)),
\enq
and hence the operator $K$ admits a natural extension from the initial domain 
$\cD$ to $\dom (S_-)$ provided by the same formula \eqref{Kinit}.
\end{lemma}
\begin{proof}
Since $\alpha$ and $\beta$ are bounded, the corresponding
   operators $\alpha(S_+)$ and $\beta(S_-)$ leave the 
   domains $\dom (S_+)$ and $\dom (S_-)$ invariant, 
\beq
   \alpha(S_+) \dom (S_+) \subseteq \dom (S_+) \, 
   \text{ and } \, \beta(S_-) \dom (S_-) \subseteq \dom
   (S_-).
\enq
Next, one considers the following sesquilinear form
   \begin{equation}
     \label{CommDefAssumptionsExplainedSLF}
  (\beta (S_-) f, S_+ \alpha (S_+) g)_{\cH} - (S_- \beta (S_-) f, \alpha (S_+) g)_{\cH} 
     = (\overline{K} f, g)_{\cH},
   \end{equation}
   where $f \in \cD$ and $g \in \dom (S_+)$.  Since $\overline{ K}$ is
   bounded, the form in the left-hand side of 
   \eqref{CommDefAssumptionsExplainedSLF} is also bounded and thus for every
   fixed $g \in \dom (S_+)$ the linear mapping
\beq
   \cD \ni f
   \mapsto (S_- \beta(S_-) f, \alpha (S_+) g)_{\cH}
\enq
  is continuous.  Since $\cD$ is a core for $S_-$ and hence 
   it is also a core for $S_- \beta(S_-)$, this implies
   that $\alpha (S_+) g \in \dom ( S_- \beta(S_-))$, or
\beq
  \alpha (S_+) \dom (S_+) \subseteq \dom (S_- \beta (S_-)).
\enq
  This observation allows one to rewrite
\eqref{CommDefAssumptionsExplainedSLF}
   as
   \begin{equation}
     \label{exten}
    (\beta (S_-) f, S_+ \alpha (S_+) g)_{\cH} - (f, S_- \beta (S_-) \alpha (S_+) g)_{\cH} 
    = (\overline{ K} f, g)_{\cH},
   \end{equation}
   for $f \in \cD$ and $g \in \dom (S_+)$ and then to conclude that
\eqref{exten} holds for all $f \in \cH$ and $g \in \dom (S_+)$, since
the right-hand side  of \eqref{exten} is a bounded sesquilinear form.
  In particular, it follows from \eqref{exten} that
  for every fixed $f \in \dom (S_-)$, the mapping
\beq
\dom (S_+)\ni g  \mapsto (\beta(S_-) f, S_+ \alpha (S_+) g)_{\cH}
\enq
is continuous and thus $\beta(S_-) f \in \dom ( S_+ \alpha(S_+))$, proving \eqref{n6.8}.
\end{proof}

We will use operator integrals via the following result which is 
a variation of \cite[Theorem 15]{CPS09}:

\begin{lemma}\lb{ffTOI}
Assume Hypothesis \ref{spm}. Suppose that $h$ is a  bounded Borel 
function on $\bbR$ such that the function $\phi$ defined by
\beq\lb{def16IO}
\phi(\lambda, 
\mu)=\frac{h(\lambda) - h(\mu)}{\alpha(\lambda)(\lambda- 
\mu)\beta(\mu)},
\quad (\lambda, \mu) \in \bbR^2,
\enq
belongs to the class $\mathfrak{A}_0$. Then the closure  $\overline{K} \in \cB(\cH)$ of the operator $K=K(S_+,S_-)$ satisfies:
\beq\lb{n6.17a}
h(S_+) - h(S_-)=T_\phi(\overline{K}) \in \cB(\cH),
\enq
where $T_\phi$ represents the operator integral 
$T_{\phi,\infty}=T_{\phi,\infty}^{(S_+,S_-)}$.
In addition, assume that $\ol K \in \cB_1(\cH)$. Then
\beq\lb{ggAAIO}
h(S_+) - h(S_-)=T_\phi(\overline{K}) \in \cB_1(\cH),
\enq
where $T_\phi$ represents the operator integral 
$T_{\phi,1}=T_{\phi,1}^{(S_+,S_-)}$.
\end{lemma}
\begin{proof}
Due to the observation $T_{\phi,\infty}|_{\cB_1(\cH)}=T_{\phi,1}$ 
made above, \eqref{ggAAIO} follows from \eqref{n6.17a}.
To begin the proof of \eqref{n6.17a}, we let
$E^\pm_n = E_{S_\pm} ([-n, n])$
denote the spectral projections associated with the self-adjoint operators
$S_\pm$, and  introduce the sequence of bounded operators
\beq
  K_n = E^+_n \, \overline{K} E^-_n, \quad n\in \bbN.
\enq
Clearly, $ \wlim_{n \rightarrow \infty} K_n = \overline{ K} $, where the limit is taken 
with respect to the weak operator topology.
Lemma \ref{zam} implies that
\beq
\overline{K}f=S_+\alpha(S_+)\beta(S_-)f -\alpha(S_+)  S_-\beta(S_-)f,
  \quad f \in \dom (S_-),
\enq
and therefore, the operator $K_n$ may be alternatively represented by 
\beq\lb{n6.20}
   K_n = E^+_n\alpha(S_+)  S_+ \beta(S_-)E^-_n - E^+_n\alpha(S_+) S_-
   \beta(S_-)E^-_n.
\enq 
We claim that 
\begin{equation} \lb{n6.31}  
E_n^+ \, \left(h(S_+) - h(S_-) \right)\, E^-_n = T^{(S_+, S_-)}_\phi (K_n).  
\end{equation}     
Assuming the claim, one finishes the proof of the lemma as follows:
Since $\phi \in \mathfrak{A}_0$, the operator $T_\phi^{(S_+, S_-)}:\cB(\cH)\to\cB(\cH)$ is
   continuous with respect to the weak operator topology.  Observing
   also that 
\begin{equation}   
\wlim_{n \rightarrow \infty} E^+_n \, \left(h(S_+) - h(S_-) \right)\, E^-_n = h(S_+) - h(S_-), 
\end{equation}     
and passing to the limit as $n\to\infty$ in \eqref{n6.31}, one obtains \eqref{n6.17a}, 
completing the proof of Lemma \ref{ffTOI} (subject to \eqref{n6.31}).

It remains to prove the claim \eqref{n6.31} (which is a slight generalization of 
\cite[Lemma\ 7.1]{dPSW02}), that is, we need to show the identity (cf.\ \eqref{n6.20})
 \begin{align}
  \begin{split}
   \label{GeneralPerturbationApp}
& E^+_n \, \left( h(S_+) - h(S_-)  \right)\, E^-_n   \\
& \quad = T^{(S_+, S_-)}_\phi \left( E_n^+ \alpha(S_+) S_+ \beta(S_-) \,
       E^-_n - E_n^+ \alpha(S_+) S_- \beta(S_-) \, E^-_n  \right). 
\end{split}
\end{align}
For this purpose we let $\chi_n$ denote the characteristic function corresponding to spectral projections $E^\pm_n$ and introduce the functions $\phi_\pm$ by
\begin{equation}\begin{split}
\phi_+(\lambda,\mu)&=\chi_n(\lambda)\alpha(\lambda)\,\lambda\,\phi(\lambda,\mu)\beta(\mu)\chi_n(\mu),\\
\phi_-(\lambda,\mu)&=\chi_n(\lambda)\alpha(\lambda)\phi(\lambda,\mu)\,\mu\,\beta(\mu)\chi_n(\mu).\end{split}
\end{equation} 
Since the mapping $\phi \to T_\phi$ is a homomorphism of 
$\mathfrak{A}_0$ into $\cB(\cB(\cH))$ one has 
\begin{equation}\label{identity1}
T_{\phi_+}^{(S_+, S_-)}(I)=T_{\phi_+}^{(S_+, S_-)} (E_n^+\alpha(S_+)S_+\beta(S_-)E_n^-),
\end{equation}
and
\begin{equation}\label{identity2}
T_{\phi_-}^{(S_+, S_-)}(I)=T_{\phi_+}^{(S_+, S_-)} (E_n^+\alpha(S_+)S_-\beta(S_-)E_n^-),
\end{equation} 
implying 
\begin{equation}\label{GeneralPerturbationPhi1}
T_{\phi_+-\phi_-}^{(S_+, S_-)}(I)=T_{\phi}^{(S_+, 
S_-)}\left(E_n^+\alpha(S_+)S_+\beta(S_-)E_n^- - 
E_n^+\alpha(S_+)S_-\beta(S_-)E_n^-\right).
\end{equation}
Indeed, the operators 
$E_n^+\alpha(S_+)S_\pm\beta(S_-)E_n^-$ in the identities 
\eqref{identity1} and \eqref{identity2} are bounded and hence the 
application of the double operator integral $T_{\phi}^{(S_+, S_-)}$ to 
these operators is justified.
A direct computation shows that
\begin{equation}
\phi_+(\lambda, \mu)-\phi_-(\lambda, 
\mu)=\chi_n(\lambda) \left(h(\lambda) - h(\mu) \right) 
\chi_n(\mu),
\end{equation} 
and therefore (again appealing to the fact that the mapping $\varphi 
\to T_\varphi$ is a homomorphism of $\mathfrak{A}_0$ into $\cB(\cB(\cH))$), one has 
\begin{equation} \label{GeneralPerturbationPhi2}
T_{\phi_+-\phi_-}^{(S_+, S_-)}(I)=E_n^+\Big(T_{h(\lambda)}^{(S_+, 
S_-)}(I) -T_{h(\mu)}^{(S_+, S_-)}(I) \Big) 
E_n^-=E_n^+\big(h(S_+) - h(S_-)\big)E_n^-.
\end{equation}
Combining \eqref{GeneralPerturbationPhi1} and  
\eqref{GeneralPerturbationPhi2} yields \eqref{GeneralPerturbationApp}.
\end{proof}


Next, we turn to the discussion of the analogue of \eqref{subtle} for 
$S_\pm$ in the trace class setting. We recall our usual notation
$g(\lambda) =  \lambda (\lambda^2 + 1)^{-1/2}$,  $\lambda\in\bbR$.
Our intention is to use Lemma \ref{ffTOI} with $h(\lambda)=g(\lambda)$, and $\alpha(\lambda)=(\lambda^2+1)^{-1/4}$ and $\beta(\mu)=(\mu^2+1)^{-1/4}$.
First, we verify the condition $\phi\in\mathfrak{A}_0$ in Lemma \ref{ffTOI}.
\begin{lemma}\lb{newL6.7}
The function $\phi$ defined by
\beq\lb{def_phi}
\phi(\lambda,\mu):=\frac{\lambda(\lambda^2 + 1)^{-1/2}-\mu(\mu^2 + 1)^{-1/2}}{(\lambda^2 + 1)^{-1/4}\,(\lambda-\mu)\,(\mu^2 + 1)^{-1/4}}, 
\quad (\lambda,\mu)\in\bbR^2,
\enq
belongs to the class $\mathfrak{A}_0$.
\end{lemma}
\begin{proof}
Let $(\lambda,\mu)\in\bbR^2$.
A direct calculation (carried out in \cite[(4.3)]{PS09}) reveals:
\begin{align}
\phi(\lambda,\mu)&=(\lambda^2 + 1)^{1/4}\frac{\lambda(\lambda^2 + 1)^{-1/2}-\mu(\mu^2 + 1)^{-1/2}}{\lambda-\mu}(\mu^2 + 1)^{1/4} 
\no \\
&=\frac{(\lambda^2 + 1)^{1/2}((\lambda^2 + 1)^{1/2}
-(\mu^2 + 1)^{1/2})(\mu^2 + 1)^{1/2}}{(\lambda^2 + 1)^{1/4}((\lambda^2 + 1)-(\mu^2 + 1))(\mu^2 + 1)^{1/4}} 
\no \\
&\quad
+\frac{(1-\lambda\mu)((\lambda^2 + 1)^{1/2}-(\mu^2 + 1)^{1/2})}{(\lambda^2 + 1)^{1/4}((\lambda^2 + 1)-(\mu^2 + 1))(\mu^2 + 1)^{1/4}}. 
\lb{dfnphi}
\end{align}
As a result, one can write
\beq\lb{splitphi}
\phi(\lambda,\mu)=\psi(\lambda,\mu)+\frac{\psi(\lambda,\mu)}{(\lambda^2 + 1)^{1/2}(\mu^2 + 1)^{1/2}}-
\frac{\lambda\psi(\lambda,\mu)\mu}{(\lambda^2 + 1)^{1/2}(\mu^2 + 1)^{1/2}},
\enq
where we introduced the function 
\beq \lb{defpsi}
\psi(\lambda,\mu)=\frac{(\lambda^2 + 1)^{1/4}(\mu^2 + 1)^{1/4}}{(\lambda^2 + 1)^{1/2}+(\mu^2 + 1)^{1/2}}. 
\enq
As soon as one knows that $\psi\in\mathfrak{A}_0$, it is  
straightforward that $\phi\in\mathfrak{A}_0$ and
$\|\phi\|_{\mathfrak{A}_0}\leq 3\|\psi\|_{\mathfrak{A}_0}$.
To begin the proof of the assertion $\psi\in\mathfrak{A}_0$, one introduces 
the function
\beq \lb{deffg}
\zeta(x)=\frac{1}{e^{x/2}+e^{-x/2}}, \quad x\in\bbR,
\enq
and observes that $\psi(\lambda,\mu)$ in \eqref{defpsi} can be written as 
\beq \lb{chvarf}
\psi(\lambda,\mu)=\zeta\big(\log((\lambda^2 + 1)^{1/2})-\log((\mu^2 + 1)^{1/2})\big).
\enq 
Since $\zeta\in W^{1,2}(\bbR)$, the Sobolev space of functions satisfying 
$\zeta,\zeta'\in L^2(\bbR; dx)$, one concludes that $\widehat{\zeta}\in
L^1(\bbR; ds)$ for the Fourier transform $\widehat{\zeta}=\widehat{\zeta}(s)$. Since also $\zeta\in
L^1(\bbR; dx)$, the inverse Fourier transform formula yields  
\beq \lb{infrakA}
\zeta(\lambda-\mu)=\frac{1}{2\pi}\int_\bbR
e^{is\lambda}e^{-is\mu} \widehat{\zeta}(s)\,ds, \quad
\lambda,\mu\in\bbR. 
\enq 
Combining \eqref{chvarf} and \eqref{infrakA} yields 
\beq \lb{final_psi}
  \psi(\lambda,\mu)=\frac{1}{2\pi}\int_\bbR
(\lambda^2 + 1)^{is/2}(\mu^2 + 1)^{-is/2}\widehat{\zeta}(s)\,ds, \quad
\lambda,\mu\in\bbR, \enq which immediately implies  $\psi
\in\mathfrak{A}_0$ due to $\widehat{\zeta}\in
L^1(\bbR; ds)$, completing the proof.
\end{proof}

\begin{remark}\lb{rSplit}
In the course of the proof of Lemma \ref{newL6.7} we established  
formula \eqref{splitphi}, yielding the following decomposition of $T_\phi$,  
\beq \lb{splitT}
\begin{split}
T_\phi&=T_\psi+(S_+^2 + I)^{-1/2}T_\psi(S_-^2 + I)^{-1/2}      \\
&\quad -S_+(S_+^2 + I)^{-1/2}T_\psi S_-(S_-^2 + I)^{-1/2},
\end{split}
\enq
where $T_\psi=T_\psi^{(S_+,S_-)}$ is the  operator integral
for the function $\psi$ defined in \eqref{defpsi} for which we proved the integral representation \eqref{final_psi}. Since $\psi\in\mathfrak{A}_0$, which in turn 
implies $T_\psi\in\cB(\cB_1(\cH))$, and since both
operators $(S_{\pm}^2 + I)^{-1/2}$ and $S_{\pm}(S_{\pm}^2 + I)^{-1/2}$
belong to $\cB(\cH)$, it follows from \eqref{splitT} that $T_\phi\in\cB(\cB_1(\cH))$.
We will use the decomposition \eqref{splitT} and the representation \eqref{final_psi} in the proof of Proposition \ref{propgA}.
\end{remark}

\begin{lemma}\lb{ext_subtle}
 Assume that $S_\pm$ are self-adjoint operators in $\cH$ such that
 \beq\lb{domeqdom} \dom(S_+)=\dom(S_-)
 \enq
 and
  \beq\label{trcl}
(S_+ - S_-)\big(S_-^2 + I)^{-1/2} \in\cB_1(\cH).
\enq
Then the closure $\overline{K}$ of the operator $K=K(S_+,S_-)$ in $\cH$ defined by
\beq\label{preclK}
K=(S_+^2+I)^{-1/4}(S_+ - S_-)(S_-^2+I)^{-1/4}, \quad  
\dom (K)=\dom(S_-),
\enq
satisfies $\ol K \in \cB_1(\cH)$. Moreover,
\beq\lb{subtle2}
g(S_+)-g(S_-)=T_\phi \big[\overline{K}\big]\in\cB_1(\cH),
\enq
where $T_\phi \in\cB(\cB_1(\cH))$. 
\end{lemma}
\begin{proof} 
Assumption \eqref{domeqdom} yields 
\beq
\overline{(S_+^2 + I)^{-1/4} (S_-^2 + I)^{1/4}} \in \cB(\cH)      \lb{6.Spm}
\enq 
by Remark \ref{r2.8} with $A_\pm$ replaced by $S_\pm$. In addition, the operator 
$K$ on $\dom (K)=\dom(S_-)$ can be represented as follows,  
\beq\label{prelim}
\begin{split}
K&=(S_+^2+I)^{-1/4}(S_+ - S_-)(S_-^2+I)^{-1/4}  \\
&= \big[(S_+^2+I)^{-1/4}(S_-^2+I)^{-1/4}\big]
(S_-^2+I)^{-1/4}(S_+ - S_-)(S_-^2+I)^{-1/4}.
\end{split}
\enq
Due to \eqref{trcl} and Theorem \ref{hadamard1}, the closure of the operator
\beq
\widetilde K=
(S_-^2+I)^{-1/4}(S_+ - S_-)(S_-^2+I)^{-1/4}, \quad \dom(\widetilde K)=\dom (S_-), 
\enq
is a trace class operator, and hence from \eqref{6.Spm} and \eqref{prelim} one concludes that $\overline{K}\in \cB_1(\cH)$. Next, we choose 
$h(\lambda)=g(\lambda)$ and 
$\alpha(\lambda)=(\lambda^2+1)^{-1/4}$, $\beta(\mu)=(\mu^2+1)^{-1/4}$ in 
Lemma \ref{ffTOI}.  By \eqref{domeqdom} and $\overline{K}\in \cB(\cH)$, 
Hypothesis \ref{spm} with $\cD=\dom(S_-)$ holds. 
By Lemma \ref{newL6.7}, the function $\phi$ in \eqref{def16IO} 
belongs to the class $\mathfrak{A}_0$, and thus all assumptions of 
Lemma \ref{ffTOI}  are verified.
As a result,  \eqref{subtle2} follows from \eqref{ggAAIO}.
\end{proof}

We note in passing, that we could have used the weaker hypotheses
\begin{equation} 
\dom(|S_+|^{1/2})=\dom(|S_-|^{1/2}) \, \text{ and } \, 
\dom(S_+) \supseteq \dom(S_-)
\end{equation}  
in place of \eqref{domeqdom} in Lemma \ref{ext_subtle}, but we do not 
pursue this here.

At this point we are ready to prove Proposition \ref{propgA}. We switch back to our 
original notation $A_{\pm}$, that is, we will now identify $S_{\pm}$ with the 
self-adjoint operators $A_{\pm}$ studied in the previous sections. In 
particular, we emphasize that Lemma  \ref{ext_subtle} is applicable as assumption \eqref{domeqdom} holds by Theorem \ref{t3.7}\,$(iv)$ and assumption \eqref{trcl} is satisfied by \eqref{3.Apmrtrcl} 
(cf.\ also \eqref{intB}).

\begin{proof}[Proof of Proposition \ref{propgA}]
The first inclusion in assertion \eqref{assrti} of
Proposition \ref{propgA} is proved in Lemma  \ref{ext_subtle}. 
The second inclusion in \eqref{assrti} is proved similarly.

To begin the proof of assertion \eqref{assrtii}, one considers the
operator integral
\beq
T_\phi^{(n)}=T_\phi^{(A_{+,n},A_{-,n})},
\enq
with $\phi$ given in \eqref{def_phi} and the operators $K(A_{+},A_{-})$
and $K(A_{+,n},A_{-,n})$ defined by \eqref{preclK} with $S_{\pm}$
replaced by $A_{\pm}$ and $A_{\pm,n}$, respectively.
  Using formula \eqref{subtle2}, one obtains 
\begin{align}
& g(A_+) - g(A_-)-(g(A_{+,n})-g(A_{-,n}))  \\
& \quad =T_\phi^{(A_+,A_-)}(\overline{K(A_+,A_-)})-
T_\phi^{(A_{+,n},A_{-,n})}(K(A_{+,n},A_{-,n}))    \\
& \quad = \big(T_\phi^{(A_+,A_-)}-T_\phi^{(A_{+,n},A_{-,n})}\big)(
\overline{K(A_+,A_-)})\lb{term1}  \\
&\qquad + 
T_\phi^{(A_{+,n},A_{-,n})}\big(\overline{K(A_+,A_-)}-K(A_{+,n},A_{-,n})\big).\lb{term2}
\end{align}

Since $\phi\in\mathfrak{A}_0$ by Lemma \ref{newL6.7}, the
  sequence   of the operators
$T_\phi^{(A_{+,n},A_{-,n})}$  is uniformly bounded in the Banach space
$\cB(\cB_1(\cH))$ (see
\eqref{normTphi}). Thus, to complete the proof of assertion \eqref{assrtii}
it suffices to establish that
\begin{equation}\lb{ass1}
\lim_{n\to\infty}\big\|\overline{K(A_+,A_-)}-K(A_{+,n},A_{-,n})\big\|_{\cB_1(\cH)}=0     
\end{equation}
and 
\begin{align}\lb{ass2}
 \lim_{n\to\infty}\big\|T_\phi^{(A_+,A_-)}(K) 
 - T_\phi^{(A_{+,n},A_{-,n})}(K)\big\|_{\cB_1(\cH)}=0 \,
\text{ for each } K\in\cB_1(\cH).    
\end{align}

Starting the proof of \eqref{ass1}, we recall that $P_n=E_{A_-}((-n,n))$ is the 
spectral projection associated with $A_-$, that $A_{\pm,n}=P_nA_{\pm} P_n$, 
and that we abbreviate $\varkappa_\pm=((A_\pm)^2+I)^{1/2}$, 
$\varkappa_{\pm,n}=((A_{\pm,n})^2+I)^{1/2}$. It is clear that 
\begin{equation} 
P_n\varkappa_{-,n}^{-1/2}=P_n\varkappa_-^{-1/2},     \lb{KAP-} 
\end{equation} 
and hence one obtains
\begin{align} \lb{pnleft}
& K(A_{+,n},A_{-,n})=((A_{+,n})^2+I)^{-1/4}(A_{+,n}-A_{-,n})((A_{-,n})^2+I)^{-1/4}   \no  \\
& \quad =\big[(A_{+,n}^2+I)^{-1/4}(A_{-,n}^2+I)^{1/4}\big](A_{-,n}^2+I)^{-1/4}P_n(A_{+}-A_{-})(A_-^2+I)^{-1/4}P_n   \no \\
& \quad =
\big[(A_{+,n}^2+I)^{-1/4}(A_{-,n}^2+I)^{1/4}\big]P_n(A_{-}^2+I)^{-1/4}(A_{+}-A_{-})(A_-^2+I)^{-1/4}P_n.
\end{align} 
In addition, $\ol{(A_{+}^2+I)^{-1/4}(A_{-}^2+I)^{1/4}} 
= \big[A_{-}^2+I)^{1/4}(A_{+}^2+I)^{-1/4}\big]^* \in \cB(\cH)$ by Remark \ref{r2.8} , 
and hence one can write
\begin{align}
& K(A_+,A_-) = (A_{+}^2+I)^{-1/4}(A_{+}-A_{-})(A_{-}^2+I)^{-1/4}   \no \\
& \quad =\big[\ol{(A_{+}^2+I)^{-1/4}(A_{-}^2+I)^{1/4}}\big](A_{-}^2+I)^{-1/4}(A_{+}-A_{-})(A_-^2+I)^{-1/4}.
\end{align}
Thus one can represent the difference under the norm in \eqref{ass1} as follows, 
\begin{align}
& K(A_+,A_-)-K(A_{+,n},A_{-,n})  \no  \\
& \quad = \big[\ol{(A_{+}^2+I)^{-1/4}(A_{-}^2+I)^{1/4}}\big]\,(A_{-}^2+I)^{-1/4}(A_{+}-A_{-})(A_-^2+I)^{-1/4}    \\
&\qquad -\big[(A_{+,n}^2+I)^{-1/4}(A_{-,n}^2+I)^{1/4}\big] \,P_n(A_{-}^2+I)^{-1/4}(A_{+}-A_{-})(A_-^2+I)^{-1/4}P_n.   \no 
\end{align}
Since $(A_+-A_-)\varkappa_-^{-1}\in\cB_1(\cH)$ by \eqref{intB}, one has 
\beq\lb{n6.62}
(A_{-}^2+I)^{-1/4}(A_{+}-A_{-})(A_-^2+I)^{-1/4}\in\cB_1(\cH)
\enq
by Theorem \ref{hadamard1}. Hence, Lemma \ref{lSTP} implies  
\begin{align}
\begin{split} 
& \lim_{n \to \infty}\big\|P_n(A_{-}^2+I)^{-1/4}(A_{+}-A_{-})(A_-^2+I)^{-1/4}P_n     \\
& \hspace*{1cm} - (A_{-}^2+I)^{-1/4}(A_{+}-A_{-})(A_-^2+I)^{-1/4}\big\|_{\cB_1(\cH)} =0.
\end{split} 
\end{align}
Again appealing to Lemma \ref{lSTP}, one concludes that for the convergence \eqref{ass1} it  suffices to show that
\begin{equation}
\slim_{n\to\infty} (A_{+,n}^2+I)^{-1/4}(A_{-,n}^2+I)^{1/4} 
= \ol{(A_{+}^2+I)^{-1/4}(A_{-}^2+I)^{1/4}}.      \lb{n6.59}
\end{equation}

By \cite[Theorem\ VIII.25]{RS80} (or \cite[Theorem\ 9.16]{We80}), the sequence $A_{+,n}$ converges to $A_+$ in the strong resolvent sense, and so  \cite[Theorem VIII.20]{RS80} 
(or \cite[Theorem\ 9.17]{We80})) implies that $(A_{+,n}^2+I_{\cH})^{-1/4}$ converges to 
$(A_{+}^2+I)^{-1/4 }$ in the strong operator topology. Moreover, the sequence 
$\{(A_{+,n}^2+I)^{-1/4}\}_{n\in\bbN}$ is uniformly bounded in $\cB(\cH)$. In addition, for fixed 
$g \in \dom (A_-)$ one has  
\begin{equation} 
\slim_{n\to\infty} (A_{-,n}^2+I)^{1/4}g   
= \slim_{n\to\infty} P_n(A_{-}^2+I)^{1/4}f = (A_{-}^2+I)^{1/4}g,
\end{equation}
and hence 
\begin{align}
\begin{split}
\slim_{n\to\infty} (A_{+,n}^2+I)^{-1/4}P_n(A_{-}^2+I)^{1/4}g  
= (A_{+}^2+I)^{-1/4}(A_{-}^2+I)^{1/4} g,&  \\  
g \in \dom(A_-).&
\end{split} 
\end{align}
Since $\dom (A_-)$ is dense in $\cH$, \eqref{n6.59} indeed holds if one can prove that 
\begin{align}
\begin{split} 
& \big\|(A_{+,n}^2+I)^{-1/4}(A_{-,n}^2+I)^{1/4}\big\|_{\cB(\cH)}     \\ 
& \quad = \big\|(A_{-,n}^2+I)^{1/4}(A_{+,n}^2+I)^{-1/4}\big\|_{\cB(\cH)} \leq C  \lb{6.C}
\end{split} 
\end{align}
for some $C \in (0, \infty)$, independent of $n \in \bbN$.
This concludes the proof of \eqref{ass1}, subject to \eqref{6.C}. 

In the remainder of this argument we now establish \eqref{6.C}: Employing \eqref{dfnAplus}, 
one estimates 
\begin{align}
& \|P_n A_- P_n f\|_{\cH} \leq \|P_n A_+ P_n f\|_{\cH} + \|P_n B(+\infty) P_n f\|_{\cH}  \no \\
& \quad = \|P_n A_+ P_n f\|_{\cH} 
+ \big\|P_n B(+\infty) (A_- - i y I)^{-1} P_n (A_- - i y I) P_n f \big\|_{\cH}  \no \\
& \quad \leq \|P_n A_+ P_n f\|_{\cH} + \big\|P_n B(+\infty) (A_- - i y I)^{-1} \big\|_{\cB(\cH)}
\|P_n (A_- - i y I)P_n f\|_{\cH}    \no \\
& \quad  \leq \|P_n A_+ P_n f\|_{\cH} + (1/2) [\|P_n A_- P_nf\|_{\cH} + |y| \|f\|_{\cH}], 
\quad f \in \cH,
\end{align} 
choosing $y>0$ sufficiently large such that 
$\big\|P_n B(+\infty) (A_- - i y I)^{-1} \big\|_{\cB(\cH)} \leq (1/2)$, which is possible since 
by \eqref{intB} $B(+\infty)$ is relatively compact (in fact, relatively trace class) with respect 
to $A_-$. Thus, employing 
\begin{equation}
\|Tg\|_{\cH} = \| |T| g\|_{\cH},  \quad g \in \dom(T) = \dom(|T|)
\end{equation} 
for any closed, densely defined operator $T$ in $\cH$ (using the polar decomposition 
for $T$), one concludes  
\begin{equation}
 \||A_{-,n}|f\|_{\cH} \leq 2 \||A_{+,n}| f\|_{\cH} + |y| \|f\|_{\cH}, \quad f \in \cH,
\end{equation}
implying
\begin{equation}
 \|[|A_{-,n}| + I] f\|_{\cH}^2 \leq 16 \|[|A_{+,n}| + I] f\|_{\cH}^2 
 + 2 [2 |y|^2 + 1] \|f\|_{\cH}^2, \quad f \in \cH.    \lb{n6.66} 
\end{equation}

At this point it suffices to apply the L\"owner--Heinz inequality in the following form: 
Assume that $T$ is a self-adjoint operator in $\cH$ with $T^{-1} \in \cB(\cH)$ and suppose 
that $S$ is a closed symmetric operator in $\cH$ satisfying
\begin{equation}
\dom(S) \supseteq \dom(T).   
\end{equation}
Then $S$ is relatively bounded with respect to $T$ and hence
there exist $a>0$ and $b>0$ such that
\begin{align}
\|Sf\|_{\cH}^2 & \leq a^2 \|Tf\|_{\cH}^2 + b^2 \|f\|_{\cH}^2, \quad f \in \dom(T).   
\end{align}
The L\"owner--Heinz inequality (cf.\ \cite[Sect.~3.2.1]{Fu02}, 
\cite{GLST15}, \cite[Theorem\ 3]{He51},
\cite{Ka52}, \cite[Theorem\ IV.1.11]{KPS82}, \cite{Lo34}, \cite{Mc80}), then entails that 
\begin{equation}
\dom\big(|S|^\alpha\big) \supseteq \dom\big(\big(a^2|T|^2 + b^2 I\big)^{\alpha/2}\big)
= \dom\big(|T|^{\alpha}\big), \quad \alpha \in (0,1],    
\end{equation}
and
\begin{align}
\begin{split} 
\big\||S|^\alpha f\big\|_{\cH}^2 & \leq \big\|[a^2 |T|^2 + b^2 I]^{\alpha/2} f\big\|_{\cH}^2     \\
& \leq a^{2 \alpha} \big\||T|^{\alpha} f \big\|_{\cH}^2 + b^{2 \alpha} \|f\|_{\cH}^2, 
\quad \alpha \in (0,1].    \lb{n6.70}
\end{split} 
\end{align}
Here we used the spectral theorem for $|T|$ and the elementary inequality, 
\begin{equation} 
(x^2 + y^2)^{\alpha} \leq x^{2 \alpha} + y^{2 \alpha}, \quad x, y \in [0, \infty), \; \alpha \in (0,1],  
\end{equation} 
to arrive at the second inequality in \eqref{n6.70}. Thus, by the estimate \eqref{n6.66}, 
identifying $S = |A_{-,n}| + I$, 
$T = |A_{+,n} | + I$, $\alpha = 1/2$, $a = 4$, and $b = 2^{1/2} [2 |y|^2 + 1]^{1/2}$ in 
\eqref{n6.70} yields
\begin{equation}
\big\|[|A_{-,n}| + I]^{1/2} f \big\|_{\cH}^2 \leq 4 \big\|[|A_{+,n}| + I]^{1/2} f \big\|_{\cH}^2 
 + 2^{1/2} [2 |y|^2 + 1]^{1/2} \|f\|_{\cH}^2, \quad f \in \cH.     \lb{n6.72} 
\end{equation}

Given a self-adjoint operator $R$ in $\cH$, and using once again the spectral theorem, 
there exist constants $c_j \in (0, \infty)$, $j=1,2$, such that 
\begin{align}
\begin{split} 
c_1 \big\|(|R|^2 + I)^{1/4}g\big\|_{\cH} 
\leq \big\|(|R| + I)^{1/2}g\big\|_{\cH}  
\leq  c_2 \big\|(|R|^2 + I)^{1/4}g\big\|_{\cH},&   \\
g \in \dom\big(|R|^{1/2}\big).&      \lb{n6.73}
\end{split} 
\end{align} 
Hence, applying the inequalities \eqref{n6.73} to $R = A_{\pm,n}$ in \eqref{n6.72}, 
finally yields 
\begin{equation} 
d_1 \big\|[|A_{-,n}|^2 + I]^{1/4} f \big\|_{\cH}^2 \leq d_2 
\big\|[|A_{+,n}|^2 + I]^{1/4} f \big\|_{\cH}^2 + d_3 \|f\|_{\cH}^2, \quad f \in \cH,     \lb{n6.74} 
\end{equation} 
for appropriate constants $d_j \in (0,\infty)$, $j=1,2,3$, implying \eqref{6.C}. 

Starting the proof of assertion \eqref{ass2}, one uses \eqref{splitT} with $S_\pm$ 
replaced by $A_\pm$ and $A_{\pm,n}$ and writes 
\begin{equation}
T_\phi^{(A_+,A_-)}(K)-T_\phi^{(A_{+,n},A_{-,n})}(K)=T_\psi(K)-T_\psi^{(n)}(K)+
\Delta_n^{(1)}+\Delta_n^{(2)}. 
\end{equation}
Here, we introduced the notation
\begin{align}
&\Delta_n^{(1)}=\varkappa_+^{-1} T_\psi(K)\varkappa_-^{-1}-
\varkappa_{+,n}^{-1}T_\psi^{(n)}(K)\varkappa_{-,n}^{-1},\\
&\Delta_n^{(2)}=A_+\varkappa_+^{-1} T_\psi(K)A_- \varkappa_-^{-1}-
A_{+,n} \varkappa_{+,n}^{-1}T_\psi^{(n)}(K)A_{-,n} \varkappa_{-,n}^{-1},
\end{align}
the abbreviations $T_\psi=T_\psi^{(A_+,A_-)}$ and
$T_\psi^{(n)}=T_\psi^{(A_{+,n},A_{-,n})}$, and used the function $\psi$ 
defined in \eqref{defpsi}. One observes that
\beq\slim_{n\to\infty}\varkappa_{\pm,n}^{-1}=\varkappa_{\pm}^{-1},\quad
\slim_{n\to\infty}\,(A_{\pm,n}
\varkappa_{\pm,n}^{-1})=A_{\pm} \varkappa_{\pm}^{-1}\enq
 by the strong resolvent convergence in
\eqref{Pn6} and \cite[Theorem VIII.20(b)]{RS80}. Thus, by Lemma
\ref{lSTP}, to finish the proof of assertion \eqref{ass2}, it suffices 
to show that 
\beq \lb{ass3}
\lim_{n\to\infty}\big\|T_\psi(K)-T_\psi^{(n)}(K)\big\|_{\cB_1(\cH)}=0 \text{ for each } K\in\cB_1(\cH). 
\enq 
We will employ the integral representation \eqref{final_psi}, 
\beq \lb{IntRep}
T_\psi(K)-T_\psi^{(n)}(K)=\frac{1}{2\pi}\int_\bbR \left (\varkappa_+^{is}K
\varkappa_-^{-is} - \varkappa_{+,n}^{is}K
\varkappa_{-,n}^{-is}\right) \widehat{\zeta}(s)\,ds.
\enq 
Again, $\slim_{n\to\infty}\varkappa_{\pm,n}^{\pm is}=\varkappa_{\pm}^{\pm
is}$  by the strong resolvent
convergence in \eqref{Pn6} and \cite[Theorem VIII.20(b)]{RS80}. By
Lemma \ref{lSTP}, the integrand in \eqref{IntRep} converges to zero
in $\cB_1(\cH)$ as $n\to\infty$ for each $s\in\bbR$. Since
$\widehat{\zeta}\in L^1(\bbR; ds)$, the dominated convergence theorem
yields \eqref{ass3}, completing the proof of Proposition \ref{propgA}.
\end{proof}

\section{The Spectral Shift Function for the Pair $(A_+, A_-)$ and Perturbation Determinants}   
\lb{s7}

In this section we provide a detailed study of the spectral shift function associated 
with the pair $(A_+, A_-)$. 

Introducing the spectral shift function associated with the pair $(A_+, A_-)$  
via the invariance principle one can proceed as follows: One recalls that by 
Theorem \ref{Ntrindf}, the difference of the 
self-adjoint operators
$g(A_+) $ and $ g(A_-)$, with 
\begin{equation}
g(x) = g_{-1} (x) = x (x^2+1)^{-1/2}, \quad x\in \bbR,   \lb{7.0} 
\end{equation} is of trace class, that is,
\begin{equation}
[g(A_+) - g(A_-)] \in\cB_1(\cH).    \lb{7.1} 
\end{equation}
Bearing in mind the membership \eqref{7.1}, we {\it define}
(cf.\ also \cite[eq.\ 8.11.4]{Ya92})
\begin{equation}
\xi(\nu; A_+, A_-) := \xi(g (\nu); g(A_+), 
g(A_-)),   \lb{2.31}
\quad \nu\in\bbR,
\end{equation}
 where
$\xi(\,\cdot\,;g(A_+), g(A_-))$ is the spectral shift function associated with the pair
$(g(A_+), g(A_-))$ uniquely determined by the requirement 
(cf.\ \cite[Sects.\ 9.1, 9.2]{Ya92})
\begin{equation}
\xi(\,\cdot\,;g(A_+), g(A_-)) \in L^1(\bbR; d\omega).  \lb{l1}
\end{equation} 
One recalls that since $\|g(A_\pm)\|\le 1$, $\xi(\,\cdot\,; g(A_+), g(A_-))$ is a
real-valued function supported on the interval $ [-1,1]$,
\begin{equation}
\supp(\xi(\,\cdot\,;g(A_+),g(A_-))) \subseteq [-1,1], 
\end{equation}
and 
\begin{align} 
& \xi(\omega; g(A_+), g(A_-))   \no \\
& \quad = \pi^{-1} \lim_{\varepsilon \downarrow 0} 
\Im \big(\ln\big({\det}_{\cH} \big(I + (g(A_+) - g(A_-))
(g(A_-) - (\omega + i \varepsilon) I)^{-1}\big)\big)\big)      \lb{7.5}  \\
& \hspace*{8.1cm} \text{ for a.e.\ } \, \omega \in [-1,1].   \no 
\end{align}
Here the choice of branch of $\ln({\det}_{\cH}(\cdot))$ on $\bbC_+$ is again 
chosen such that
\begin{equation}
\lim_{\Im(z) \to +\infty} \ln\big({\det}_{\cH}\big(I + (g(A_+) - g(A_-))
(g(A_-) - z I)^{-1}\big)\big) = 0. 
\end{equation}

Moreover, since \eqref{7.1} holds, Krein's trace formula in 
its simplest form yields (cf.\ \cite[Theorem\ 8.2.1]{Ya92})
\begin{equation}\label{7.2}
\tr_{\cH}\big(g(A_+) - g(A_-)\big)
= \int_{[-1,1]} \xi(\omega; g(A_+), g(A_-)) \, d\omega.
\end{equation}

Alternatively, one can also introduce  the spectral shift function associated with the pair $(A_+, A_-)$ taking into account that  the difference of the resolvents of the operators $A_+$ and $A_-$ is of trace class (cf. \eqref{3.trApm}), that is,
\begin{equation}\label{7.1s}
\big[(A_+-zI)^{-1}-(A_--zI)^{-1}\big] 
\in \cB_1(\cH),\quad  z \in \rho(A_+)\cap \rho(A_-).
\end{equation}
Since in this case the difference of the Cayley transforms of the 
operators $A_+$ and $A_-$ is of trace class, 
one can introduce the spectral shift function 
$\widehat \xi(\,\cdot\,; A_+, A_-)$ associated with the pair $(A_+, A_-)$  
upon relating 
$\widehat \xi(\,\cdot\,; A_+, A_-)$ to the spectral shift function associated with the 
Cayley transforms of $A_+$ and $A_-$ as in \cite[eq.\ (8.7.4)]{Ya92}.  
The spectral shift function introduced in this way 
is not unique, in fact, any two of them differ by an integer-valued homotopy invariant
(see a comprehensive discussion of this phenomenon in 
\cite[Sect.\ 8.6]{Ya92}). Moreover,
\begin{equation}
\widehat \xi(\,\cdot\,; A_+, A_-) \in L^1\big(\bbR; (|\nu| + 1)^{-2} d\nu\big)  \lb{2.33}
\end{equation}
for any concrete choice of the  integer-valued constant (cf. \cite[Sect.\ 8.7]{Ya92}). Given the pair $(A_+, A_-)$, we now arbitrarily fix the undetermined  integer-valued constant, and for simplicity, keep denoting the corresponding spectral shift function 
by $\widehat \xi(\,\cdot\,; A_+, A_-)$. 

Our next result states that the  functions 
$\xi(\,\cdot\,\; A_+, A_-)$ and $\widehat \xi(\,\cdot\,\; A_+, A_-)$ differ at most by a constant: 

\begin{lemma} \label{l7.xiconst}
Assume Hypothesis \ref{h2.1}. Let the spectral shift function   
$\xi(\,\cdot\,; A_+, A_-)$ be defined according to \eqref{2.31} and 
$\widehat \xi(\,\cdot\,; A_+, A_-)$  as in \cite[eq.\ (8.7.4)]{Ya92} $($with some determination of the associated integer-valued constant\,$)$.  Then 
there exists a $C\in\bbR$ such that
\begin{equation}
\widehat \xi(\nu; A_+, A_-)=\xi(\nu; A_+, A_-)+C \, \text{ for a.e.\ } \, \nu\in \bbR. 
\lb{7.xiC}
\end{equation}
\end{lemma}
\begin{proof}
First we note that by \cite[Theorem\ 8.7.1]{Ya92} the trace formula 
\begin{equation}
{\tr}_{\cH} (f(A_+) - f(A_-)) = \int_{\bbR} f'(\nu) \widehat\xi(\nu;A_+,A_-) \, d\nu   \lb{B.152} 
\end{equation}
holds for the   class of functions $f$ 
having two locally bounded derivatives and satisfying the conditions
\begin{equation}\label{yaf1}
\text{for some $\varepsilon > 0$, } \, 
(\nu^2f'(\nu))' \underset{|\lambda|\to \infty}{=} O(|\nu|^{-1-\varepsilon})  
\end{equation}
and 
\begin{equation}\label{yaf2}
\lim_{\nu\to -\infty} f(\nu)
=\lim_{\nu\to +\infty} f(\nu),
\quad
\lim_{\nu\to -\infty} \nu^2f'(\nu)=\lim_{\nu\to +\infty} \nu^2f'(\nu).
\end{equation}
This class includes, in particular, the functions of the type 
\begin{equation} \label{wide}
f\in C_0^{\infty}(\bbR) \, \text{ and } \, (\cdot - z)^{-n}, \; 
z\in\bbC\backslash\bbR, \; n\in\bbN, \, n\geq 1.    
\end{equation}
Since \eqref{7.1} holds, \cite[Lemma\ 8.11.3]{Ya92}
applies to the $\xi$-function given by the invariance principle
 \eqref{2.31} and hence  the trace formula 
\begin{equation}
{\tr}_{\cH} (f(A_+) - f(A_-)) = \int_{\bbR} f'(\nu) 
\xi(\nu;A_+,A_-) \, d\nu        \lb{B.15} 
\end{equation}
holds for all $f \in C_0^{\infty}(\bbR)$. 
Comparing \eqref{B.15} and \eqref{B.152} one obtains that 
\begin{equation}
\int_{\bbR} f'(\nu) 
\xi(\nu;A_+,A_-) \, d\nu =\int_{\bbR} f'(\nu) 
\widehat\xi(\nu;A_+,A_-) \, d\nu, \quad f \in C_0^{\infty}(\bbR), 
\end{equation}
and therefore, by the Du Bois--Raymond Lemma 
(see, e.g., \cite[Theorem\ 6.11]{LL01}), 
the functions $\xi(\,\cdot\, ;A_+,A_-)$ and  $\widehat \xi(\nu;A_+,A_-)$
 differ a.e. at most by a constant.
\end{proof}

\begin{remark} \label{r7.trff} The fact that 
$\xi(\,\cdot\,;g(A_+), g(A_-)) \in L^1(\bbR; d\omega)$ according to \eqref{l1},  
implies the membership
\begin{equation}
\xi(\,\cdot\,; A_+, A_-) \in L^1\big(\bbR; (|\nu| + 1)^{-3} d\nu\big)  \lb{2.337}
\end{equation}
which can easily be verified taking into account the definition \eqref{2.31} 
of $\xi(\ \cdot \,; A_+, A_-)$ and using the change of variables \eqref{7.cv} below. 
While \eqref{2.337} is correct, it is not  optimal, since, in fact, 
\begin{equation}
\xi(\,\cdot\,; A_+, A_-) \in L^1\big(\bbR; (|\nu| + 1)^{-2} d\nu\big)  \lb{2.33xi}
\end{equation}
as a consequence of \eqref{2.33} and \eqref{7.xiC}. 
Moreover, the following trace formulas hold,
\begin{align}
\begin{split}
- {\tr}_{\cH}\big((A_+ - z I)^{-1}-(A_- - z I)^{-1}\big)   
 &= \int_{\bbR}\frac{\widehat \xi(\nu;A_+,A_-) \, d\nu}{(\nu - z)^2}    \\
&  = \int_{\bbR}\frac{\xi(\nu;A_+,A_-) \, d\nu}{(\nu - z)^2}\label{trtr}, 
\quad z \in \bbC\backslash\bbR,  
\end{split} 
\end{align} 
with two convergent Lebesgue integrals in \eqref{trtr}. 
Indeed, the first equality in \eqref{trtr} follows from \eqref{B.152} and \eqref{wide}, 
and the second from the observation that 
\begin{equation}
\int_\bbR\frac{d\nu}{(\nu - z)^2}=0, \quad z\in \bbC \backslash \bbR, 
\end{equation}
and the fact that by \eqref{7.xiC}, $\xi(\,\cdot\, ;A_+,A_-)$ and 
$\widehat \xi(\,\cdot\,;A_+,A_-)$ differ at most by a constant. 
\end{remark}

Our next result provides a refinement of the trace formula \eqref{7.2}. For 
this purpose we recall the function $g_z$ defined by 
 \begin{equation}
g_z(x)= x (x^2-z)^{-1/2}, \quad x\in \bbR, \;  z\in\bbC\backslash [0,\infty). 
\end{equation}

\begin{lemma}\label{l7.trfor}
Assume Hypothesis \ref{h2.1} and define $\xi(\,\cdot\,; A_+, A_-)$ 
according to \eqref{2.31}.  Then 
\begin{equation}
[g_{z}(A_+) - g_{z}(A_-)] \in\cB_1(\cH), \quad z\in \bbC\backslash [0, \infty), 
\lb{7.trclgz}
\end{equation}
and the following trace formula holds
\begin{equation}
\tr_{\cH}\big(g_{z}(A_+) - g_{z}(A_-)\big)
  = - z \int_{\bbR} \frac{\xi(\nu; A_+, A_-) \, d\nu}{(\nu^2 - z)^{3/2}},
\quad  z\in\bbC\backslash [0,\infty).         \lb{2.381}
\end{equation}
In particular, 
\begin{equation}
\bbC\backslash [0,\infty) \ni z \mapsto 
\tr_{\cH} \big(g_{z}(A_+) - g_{z}(A_-)\big) \, \text{ is analytic.} 
\lb{7.tran}
\end{equation}
\end{lemma}
\begin{proof}
We start with the representation \eqref{7.2}
\begin{equation}
\tr_{\cH}\big(g(A_+) - g(A_-)\big)
= \int_{[-1,1]} \xi(\omega; g(A_+), g(A_-)) \, d\omega.
\end{equation}
 Since
\begin{equation}
g^\prime (\nu) = (\nu^2 + 1)^{-3/2} > 0, 
\quad \nu\in\bbR,
\end{equation}
one can introduce the change of variables
\begin{equation}
\omega = g (\nu) = \nu (\nu^2 + 1)^{-1/2}, \quad \nu\in\bbR,  \lb{7.cv}
\end{equation}
implying
\begin{equation}
\tr_{\cH}\big(g(A_+) - g(A_-)\big)
= \int_{\bbR} \frac{\xi(g (\nu); g(A_+), g(A_-)) \, d\nu}{(\nu^2 + 1)^{3/2}}
\end{equation}
and hence, in accordance with the definition \eqref{2.31}
 of the spectral shift function 
$\xi(\,\cdot \,; A_+, A_-)$, one also obtains that 
\begin{equation}\lb{2.30}
\tr_{\cH}\big(g(A_+) - g(A_-)\big) =  
\int_{\bbR} \frac{\xi(\nu; A_+, A_-) \, d\nu}{(\nu^2 + 
1)^{3/2}}   
\end{equation}
which proves the trace formula \eqref{2.381} for $z=-1$.

To handle the case of arbitrary $z\in \bbC\backslash  [0,\infty)$, 
we remark that 
the function $G_z$ given by 
\begin{equation}
G_z(\nu)=g_z(\nu)-g_{-1}(\nu), \quad \nu \in \bbR,
\; z\in \bbC\backslash  [0, \infty),    \lb{7.Gz}
\end{equation}
satisfies the conditions \eqref{yaf1} and \eqref{yaf2}, and hence  
by \cite[Theorem\ 8.7.1]{Ya92},
one obtains 
\begin{equation}
[G_z(A_+)-G_z(A_-)] \in \cB_1(\cH), \quad z\in \bbC\backslash  [0, \infty), 
\lb{7.trclGz}
\end{equation}
and the trace formula
\begin{equation}\label{pri}
\tr_{\cH}\left (G_z(A_+)-G_z(A_-)\right )=\int_\bbR G_z'(\nu)\widehat \xi 
(\nu; A_+, A_-)\,d\nu,\quad z\in \bbC\backslash  [0, \infty),
\end{equation}
where the spectral shift function 
$\widehat \xi(\,\cdot\,; A_+, A_-)$ associated with the pair $(A_+, A_-)$  
is introduced as in \cite[eq.\ (8.7.4)]{Ya92}.  By Lemma \ref{l7.xiconst}, 
the spectral shift functions $\widehat \xi(\,\cdot\,; A_+, A_-)$
 and $ \xi(\,\cdot\,; A_+, A_-)$ differ at most by a constant and hence, since
\begin{equation}
\lim_{\nu\to +\infty} G_z(\nu)=\lim_{\nu\to +\infty}G_z(\nu)=0,
\end{equation}
eq. \eqref{pri} can be rewritten as
\begin{align}
\begin{split} 
& \tr_{\cH}\left (G_z(A_+)-G_z(A_-)\right )=\int_\bbR G_z'(\nu) \xi 
(\nu; A_+, A_-)\,d\nu,     \\
& \quad =\int_\bbR \left [\frac{-z}{(\nu^2-z)^{3/2}}-\frac{1}{(\nu^2+1)^{3/2}}\right]
\xi (\nu; A_+, A_-)\,d\nu,\lb{pri2}, \quad z\in \bbC\backslash [0, \infty).
\end{split}
\end{align} 
Combining \eqref{7.1}, \eqref{7.Gz}, and \eqref{7.trclGz}, one concludes that 
\eqref{7.trclgz} and the trace formula \eqref{2.381} hold.
\end{proof}

The following result, an improvement of \eqref{7.trclgz} and \eqref{7.tran}, 
will be proved in Appendix \ref{sB}:

\begin{lemma} \lb{l7.4}
Assume Hypothesis \ref{h2.1} and let $z\in \bbC\backslash [0,\infty)$. 
Then $[g_z(A_+) - g_z(A_-)]$ is differentiable with respect to the 
$\cB_1(\cH)$-norm and 
\begin{align}
& \f{d}{dz} \tr_{\cH} \big(g_z(A_+) - g_z(A_-)\big) 
= \tr_{\cH}\bigg(\frac{d}{dz} g_z(A_+) - \frac{d}{dz} g_z(A_-)\bigg)     \\
& \quad = \f{1}{2} {\tr}_{\cH} \big(A_+ (A_+^2 - z I)^{-3/2} 
- A_- (A_-^2 - z I)^{-3/2}\big), 
\quad z\in \bbC\backslash [0,\infty).    \no 
\end{align} 
\end{lemma}

We note that Lemmas \ref{l7.trfor} and \ref{l7.4} extend to 
$z\in \rho(A_+^2) \cap \rho(A_-^2)$.

Next, we prove the following result which justifies equalities  of \eqref{2.54} and \eqref{2.44b} in Theorem \ref{t3.8}:

\begin{lemma}\label{l7.trxi}
Assume Hypothesis \ref{h2.1} and $0\in\rho(A_-)\cap\rho(A_+)$. Then
\beq \lb{treqxi} 
\tr_{\cH}\big(E_{A_-}((-\infty,0))-E_{A_+}((-\infty,0))\big)=\xi(0; A_+,A_-).
\enq
\end{lemma}
\begin{proof}Since $0\in\rho(A_\pm)$, the spectral mapping property implies 
$0\in\rho(g(A_\pm))$ for $g(x)=x(x^2+1)^{-1/2}$. 
Fixing $\nu_0>0$ such that $[-\nu_0,\nu_0]\subset\rho(g(A_-))\cap\rho(g(A_+))$, 
one notes that $\xi(\,\cdot\,; g(A_+), g(A_-))=\xi(0; g(A_+), g(A_-))$ a.e.\ on the interval 
$(-\nu_0,\nu_0)$. In addition, we introduce a smooth cut-off function 
$\varphi \in C^{\infty}(\bbR)$ satisfying 
\beq\label{defphi}
\varphi(\nu) = \begin{cases} 1, & \nu \le -\nu_0, \\
0, & \nu \ge\nu_0, \end{cases}  \, \text{ and }  \, 
\int_{-\nu_0}^{\nu_0}\varphi'(\nu) \, d \nu=-1.
\enq
Next, using a change of variables in the spectral theorem 
\cite[Theorem\ XII.2.9(c)]{DS88}, and noting that $\varphi$ coincides with the characteristic function of $(-\infty,0)$ on the spectrum of $g(A_\pm)$, one infers, 
\beq \lb{eqsppr} 
E_{A_\pm}((-\infty,0))=E_{A_\pm}(g^{-1}(-1,0))=E_{g(A_\pm)}(-1,0)
=\varphi(g(A_\pm)). 
\enq 
We recall that $[g(A_+)-g(A_-)]\in\cB_1(\cH)$ by Proposition \ref{propgA}. Thus, Krein's trace formula holds for the pair of bounded operators $g(A_+)$ and $g(A_-)$ and the spectral shift function $\xi(\,\cdot\,\; g(A_+),g(A_-))$ (cf.\ \cite[Theorem 8.2.1]{Ya92}). 
Using  \eqref{defphi}, \eqref{eqsppr}, and the trace formula, one then completes the proof as follows: 
\begin{align}
\tr_{\cH}&\big(E_{A_+}((-\infty,0))-E_{A_-}((-\infty,0))\big)   \no \\ 
& = \tr_{\cH}\big(\varphi(g(A_+))-\varphi(g(A_-))\big)=\int_{-\infty}^\infty\xi(\nu; g(A_+),g(A_-))\varphi'(\nu)\,d\nu  \no \\
&=\int_{-\nu_0}^{\nu_0}\xi(\nu; g(A_+),g(A_-))\varphi'(\nu)\,d\nu
= \xi(0; g(A_+),g(A_-))\int_{-\nu_0}^{\nu_0}\varphi'(\nu)\,d\nu  \no \\
&=-\xi(0; g(A_+),g(A_-))=-\xi(0; A_+,A_-),
\end{align}
utilizing \eqref{2.31} in the last equality. 
\end{proof}

In the final part of this section we detail the precise connection between $\xi$ and Fredholm perturbation  
determinants associated with the pair $(A_-, A_+)$. In particular, this will justify the perturbation determinants formula \eqref{spin} in the index computation in Theorem \ref{t3.8}.
 In practice, these determinants are often simpler to handle than the projection operators used in \eqref{2.44a} and \eqref{2.44b}.

Let 
\begin{equation}
D_{T/S}(z) = {\det}_{\cH} ((T-z I)(S- z I)^{-1}) = {\det}_{\cH}(I+(T-S)(S-z I)^{-1}), 
\quad z \in \rho(S), 
\end{equation}
denote the perturbation determinant for the pair of operators $(S,T)$ in $\cH$, 
assuming $(T-S)(S-z_0)^{-1} \in \cB_1(\cH)$ for some (and hence for all) 
$z_0 \in \rho(S)$. 

\begin{theorem} \lb{t8.14} 
Assume Hypothesis \ref{h2.1} and $0\in\rho(A_-)\cap\rho(A_+)$. Then
\beq
\xi(\lambda; A_+,A_-) = \pi^{-1}\lim_{\e\downarrow 0} 
\Im(\ln(D_{A_+/A_-}(\lambda+i\e))) \, 
\text{ for a.e.\ } \, \lambda\in\bbR,    \lb{8.xidet}
\enq
where $\xi(\,\cdot\,; A_+,A_-)$ is introduced by \eqref{2.31} and we make the choice 
of branch of $\ln(D_{A_+/A_-}(\cdot))$ on $\bbC_+$ such that 
$\lim_{\Im(z) \to +\infty}\ln(D_{A_+/A_-}(z))=0$.
In particular, for a continuous representative of  $\xi(\,\cdot\,; A_+,A_-)$
 in a neighborhood of $\lambda = 0$ the equality 
\beq
\xi(0; A_+,A_-)= \pi^{-1} \lim_{\varepsilon \downarrow 0}
\Im(\ln(D_{A_+/A_-}(i\e)))   \lb{8.xiind}
\enq
holds.
\end{theorem}
\begin{proof} By \eqref{trtr} in Remark \ref{r7.trff},
\begin{align}
- {\tr}_{\cH}\big((A_+ - z I)^{-1}-(A_- - z I)^{-1}\big)   
 &= \int_{\bbR}\frac{\xi(\nu;A_+,A_-) \, d\nu}{(\nu - z)^2}\label{trtr1},  
\quad  z \in \bbC\backslash\bbR,   
\end{align} 
with a convergent Lebesgue integral in \eqref{trtr1}.

The general formula for the logarithmic derivative of 
the perturbation determinant  (see, e.g., \cite[Sect.\ IV.3]{GK69})
 yields  
\begin{equation} \label{restr}
\frac{d}{dz}\ln(D_{A_+/A_-}(z))
 = - {\tr}_{\cH}\big((A_+ - z I)^{-1}-(A_- - z I)^{-1}\big),  \quad 
 z \in \rho(A_+)\cap \rho(A_-).  
\end{equation}

A comparison of \eqref{trtr1} and \eqref{restr} yields
 \begin{align} 
\frac{d}{dz}\ln(D_{A_+/A_-}(z)) =\int_{\bbR}\frac{\xi(\nu;A_+,A_-) \, 
d\nu}{(\nu - z)^2},  \quad z \in \bbC\backslash\bbR.  \label{N9.41}
\end{align} 

Integrating \eqref{N9.41} with respect to $z$ (cf.\ also \cite[eq.\ (1.10)]{KY81}), one obtains 
\begin{align}
\ln(D_{A_+/A_-}(z)) & = \gamma+\int_{\bbR}
\bigg(\frac{1}{\nu - z}-\frac{\nu}{\nu^2 + 1}\bigg)\xi(\nu;A_+,A_-) \, d\nu, 
\quad z\in \bbC_+, \lb{KY1}  
\end{align}
for some constant $\gamma \in \bbC$. 

Next, we claim that actually, 
\begin{equation}
\gamma \in \bbR.   \lb{8.gam}
\end{equation} 
Indeed, taking $z \in \bbC$ and letting $|\Im(z)| \to\infty$, one infers that 
\beq
\lim_{|\Im(z)| \to +\infty} D_{A_+/A_-}(z) = 1,     \lb{8.limdet}
\enq
similarly to the proof of Lemma \ref{lHYP}. More precisely, one uses the fact that
\beq
(A_+ - A_-)(A_- - z I)^{-1} = [(A_+ - A_-) A_-^{-1}] [A_- (A_- - z I)^{-1}], 
\quad z \in \bbC\backslash\bbR,    \lb{3.A+A-}
\enq
implying
\beq
\lim_{|\Im(z)| \to \infty} \|(A_+ - A_-)(A_- - z I)^{-1}\|_{\cB_1(\cH)} =0   \lb{8.limtrn}
\enq
since 
\beq
(A_+ - A_-) A_-^{-1} \in \cB_1(\cH) \, \text{ and } \, 
\slim_{|\Im(z)| \to \infty} A_- (A_- z I)^{-1} = 0,
\enq
employing Lemma \ref{lSTP}. Clearly, \eqref{3.A+A-} and \eqref{8.limtrn} yield 
\eqref{8.limdet}. Hence we now fix the branch of $\ln(D_{A_+/A_-}(\cdot))$ 
on $\bbC_+$ by requiring 
\beq
\lim_{\Im(z) \to +\infty} \ln(D_{A_+/A_-}(z)) = 0.    \lb{8.limlndet}
\enq
Rewriting 
\eqref{KY1} in the form 
\begin{align}
\begin{split} 
\ln(D_{A_+/A_-}(iy)) & = \Re(\gamma) + \int_{\bbR}
\bigg(\frac{\nu}{\nu^2 + y^2}-\frac{\nu}{\nu^2 + 1}\bigg) \xi(\nu;A_+,A_-) \, 
d\lambda \\
& \quad + i \bigg[\Im(\gamma) + y \int_{\bbR}
\f{\xi(\nu;A_+,A_-) \, d\nu}{\nu^2 + y^2}\bigg], \quad y>0, \lb{8.Img}  
\end{split} 
\end{align}
and applying the dominated convergence theorem to conclude that 
\begin{equation}
\lim_{y\to\infty} y \int_{\bbR}
\f{\xi(\nu;A_+,A_-) \, d\nu}{\nu^2 + y^2} = 0,
\end{equation} 
combining \eqref{8.limdet} with taking the limit $y\to\infty$ in \eqref{8.Img} 
yields $\Im(\gamma) = 0$ and hence \eqref{8.gam}. 

Decomposing $\xi$ into its positive and negative parts $\xi_{\pm}$, respectively, 
\begin{align} 
\begin{split} 
\xi(\,\cdot\,; A_+,A_-) &=  \xi_+ (\,\cdot\,;A_+,A_-) - \xi_- (\,\cdot\,;A_+,A_-),   \\ 
\xi_{\pm}(\,\cdot\,; A_+,A_-) 
&= \big[|\xi(\,\cdot\,; A_+,A_-)| \pm \xi(\,\cdot\,; A_+,A_-)\big]\big/2,
\end{split} 
\end{align}  
and applying the Stieltjes inversion formula to the absolutely continuous measures 
$\xi_{\pm} (\nu;A_+,A_-) \, d\nu$  (cf., e.g., \cite[p.\ 328]{AD56}, \cite[App.\ B]{We80}),  
then yields \eqref{8.xidet}. Since by hypothesis, $0\in\rho(A_-)\cap\rho(A_+)$, one concludes \eqref{8.xiind}  (cf.\ also the discussion in connection with \eqref{2.47} which 
defines $\xi(0;A_+,A_-)$) as follows: Given the fact \eqref{8.gam}, one obtains that 
\begin{align}
& \ln(D_{A_+/A_-}(z)) = \gamma + \int_{\bbR}
\bigg(\frac{1}{\nu - z}-\frac{\nu}{\nu^2 + 1}\bigg)\xi(\nu;A_+,A_-) \, d\nu  \no \\
& \quad = \gamma + \xi(0;A_+,A_-) \int_{\bbR} 
\bigg(\frac{1}{\nu - z}-\frac{\nu}{\nu^2 + 1}\bigg) d\nu   \no \\
& \qquad + \int_{\bbR} \bigg(\frac{1}{\nu - z}-\frac{\nu}{\nu^2 + 1}\bigg)
[\xi(\nu;A_+,A_-) - \xi(0;A_+,a_-)] \, d\nu    \no \\
& \quad = \gamma + i \pi \, \xi(0;A_+,A_-)    \no \\
& \qquad + \int_{\bbR} \bigg(\frac{1}{\nu - z}-\frac{\nu}{\nu^2 + 1}\bigg)
[\xi(\nu;A_+,A_-) - \xi(0;A_+,a_-)] \, d\nu, \quad z\in \bbC_+,      \lb{7.57}
\end{align}
using 
\begin{equation}
\int_{\bbR} \bigg(\frac{1}{\nu - z}-\frac{\nu}{\nu^2 + 1}\bigg) d\nu = i \pi. 
\end{equation}
Since the last integral in \eqref{7.57} is supported in 
$(-\infty, - \varepsilon) \cup (\varepsilon,\infty)$ for some $\varepsilon > 0$ and hence 
real-valued for $z=0$ (as $\xi(\,\cdot\,;A_+,A_-)$ is constant a.e.\ in a sufficiently small neighborhood of the origin), \eqref{7.57} proves \eqref{8.xiind} taking $z=i\varepsilon$ 
and $\varepsilon\downarrow 0$. 
\end{proof}

\begin{remark} \lb{r8.15}
Given the fact \eqref{8.gam}, one explicitly obtains
\begin{equation}
\gamma = \Re(\ln(D_{A_+/A_-}(i))).       \lb{8.gamma}
\end{equation}
Moreover, from 
\begin{align}
& \ln(D_{A_+/A_-}(z)) - \ln(D_{A_+/A_-}(i)) 
= \int_{\bbR} \bigg(\frac{1}{\nu - z}-\frac{1}{\nu - i}\bigg)\xi(\nu;A_+,A_-) \, d\nu  \no \\ 
& \quad = -i \int_{\bbR} \f{\xi(\nu;A_+,A_-) \, d\nu}{\nu^2 + 1} 
+ \int_{\bbR} \bigg(\frac{1}{\nu - z}-\frac{\nu}{\nu^2 + 1}\bigg)\xi(\nu;A_+,A_-) \, d\nu, 
\quad z\in \bbC_+,
\end{align}
one concludes that 
\begin{equation}
\Im(\ln(D_{A_+/A_-}(i))) = \int_{\bbR} \f{\xi(\nu;A_+,A_-) \, d\nu}{\nu^2 + 1}. 
\end{equation}
\end{remark}

\begin{remark} \lb{r8.15a}
To illustrate the relevance of the choice of branch of $\ln(D_{A_+/A_-}(\cdot))$ we briefly look at the following elementary situation where $\cH = \bbC^2$, $A_{\pm} = \pm I_2$. Then obviously, 
\begin{equation}
D_{I_2/-I_2}(z) = \bigg(\f{z-1}{z+1}\bigg)^2, \quad z\in\bbC\backslash \{-1\}. 
\end{equation}
The function $\ln(D_{I_2/-I_2}(\cdot))$ has the branch points $\pm 1$ (we note, however, that 
the point $z=\infty$ is not a branch point of this function). Applying our convention of choosing the principal branch of $\ln(D_{I_2/-I_2}(\cdot))$ near infinity then yields that
\begin{equation}
\ln(D_{I_2/-I_2}(z)) = 2 \, \ln\big(1 - 2(z+1)^{-1}\big) \underset{|z|\to\infty}{=} \f{- 4}{z+1} + \Oh(|z|^{-2}). 
\end{equation}
Taking into account the branch cut $[-1,1]$ for $\ln(D_{I_2/-I_2}(\cdot))$ then implies 
\begin{equation}
\lim_{\varepsilon\downarrow 0} \ln(D_{I_2/-I_2}(\lambda \pm i \varepsilon)) = 
\begin{cases} 2 \, \ln(|(\lambda - 1)/(\lambda + 1)|), 
& \lambda \in \bbR\backslash [-1,1], \\[1mm]
2 \, \ln(|(\lambda - 1)/(\lambda + 1)|) \pm 2\pi i, & \lambda \in (-1,1), 
\end{cases}
\end{equation}
and hence, 
\begin{equation}
\xi(\lambda; I_2,-I_2) =\begin{cases} 0, &  \lambda \in \bbR\backslash [-1,1], \\
2, & \lambda \in (-1,1), \end{cases} 
\end{equation}
consistent with the spectral flow $\text{\rm SpFlow} (\{A(t)\}_{t=-\infty}^\infty) =2$ in an example 
where $A(t)$, $t\in\bbR$, has asymptotes $A_\pm = \pm I_2$ as $t\to \pm \infty$ (cf.\ Section \ref{s9} 
for the notion of the spectral flow). 
\end{remark}

We conclude this section with the following known fact under the additional hypothesis of $A_-$ being 
bounded from below:

\begin{remark} [\cite{BY93}, Proposition\ 6.5, \cite{KY81}, \cite{Ya64}] \lb{r8.16} 
Assume Hypothesis \ref{h2.1} and, also,  $0\in\rho(A_-)\cap\rho(A_+)$. In addition, 
assume that $A_-$ $($and hence $A_+$ and $A(t)$, $t\in\bbR$$)$ is bounded from below. 
Then one obtains the following refinements of \eqref{2.33xi}, \eqref{restr}, 
\eqref{8.gam}, and \eqref{8.gamma}, 
\begin{align}
& \xi(\,\cdot\,; A_+,A_-) \in L^1\big(\bbR; (|\lambda| + 1)^{-1} d\lambda\big),    \lb{xiint} \\
& \ln(D_{A_+/A_-}(z)) = \int_{\bbR} \f{\xi(\lambda; A_+,A_-) \, d\lambda}{\lambda-z}, 
\quad z \in \bbC_+,     \label{xiDA} \\
& \gamma = \int_{\bbR} \lambda \f{\xi(\lambda; A_+, A_-) \, d\lambda}{\lambda^2 + 1}. 
\end{align} 
\end{remark}

\section{The Spectral Shift Function for the Pair
 $(\bsH_2,\bsH_1)$  and an Index Computation}   
\lb{s8}

In this section we will prove one of our principal results, an extension 
of Pushnitski's formula \cite{Pu08}, 
relating a particular choice of spectral shift functions of the two 
pairs of operators,  $(\bsH_2, \bsH_1)$, and
$(A_+,A_-)$. 

\subsection{Pushnitski's Formula}
We start by introducing the spectral shift function
$\xi(\,\cdot\,; \bsH_2,\bsH_1)$ associated with the pair $(\bsH_2, 
\bsH_1)$. Since $\bsH_2\geq 0$ and $\bsH_1\geq 0$, and 
\begin{equation}
\big[(\bsH_2 + \bsI)^{-1} - (\bsH_1 + \bsI)^{-1}\big] \in \cB_1 
\big(L^2(\bbR;\cH)\big),     \lb{2.34a}
\end{equation}
by Lemma \ref{trclLHS}, one uniquely introduces $\xi(\,\cdot\,; \bsH_2,\bsH_1)$ by requiring that
\begin{equation}
\xi(\lambda; \bsH_2,\bsH_1) = 0, \quad \lambda < 0,    \lb{2.35}
\end{equation}
and
\begin{align}
\begin{split}
\tr_{L^2(\bbR;\cH)} \big((\bsH_2 - z \, \bsI)^{-1} - (\bsH_1 - z \, 
\bsI)^{-1}\big)
= - \int_{[0, \infty)}  \frac{\xi(\lambda; \bsH_2, \bsH_1) \, 
d\lambda}{(\lambda -z)^2},&  \\
z\in\bbC\backslash [0,\infty),&      \lb{2.36}
\end{split}
\end{align} 
following \cite[Sect.\ 8.9]{Ya92}. In addition, one has
\begin{equation}
\xi(\,\cdot\,; \bsH_2, \bsH_1) \in L^1\big(\bbR; (|\lambda| + 1)^{-2} 
d\lambda\big).   \lb{2.36a}
\end{equation}
However, \eqref{2.36a} can be improved as follows:

\begin{lemma} \lb{l8.1}
Assume Hypothesis \ref{h2.1}. Then 
\begin{equation}
\xi(\,\cdot\,; \bsH_2, \bsH_1) \in L^1\big(\bbR; (|\lambda| + 1)^{-1} 
d\lambda\big)   \lb{2.36aa}
\end{equation}
and 
\begin{equation}
\xi(\lambda; \bsH_2, \bsH_1) = \pi^{-1} \lim_{\varepsilon \downarrow 0} 
\Im\big(\ln\big(\wti D_{\bsH_2/\bsH_1}(\lambda + i \varepsilon)\big)\big) 
\, \text{ for a.e.\ } \, \lambda \in \bbR,   
\end{equation}
where we used the abbreviation 
\begin{align}
& \wti D_{\bsH_2/\bsH_1} (z) = {\det}_{L^2(\bbR;\cH)} 
\big((\bsH_1 - z \bsI)^{-1/2} (\bsH_2 - z \bsI) (\bsH_1 - z \bsI)^{-1/2}\big)  \no \\
& \quad = {\det}_{L^2(\bbR;\cH)} 
\big(\bsI + 2 (\bsH_1 - z \bsI)^{-1/2} \bsB' (\bsH_1 - z \bsI)^{-1/2}\big), 
\quad z \in \rho(\bsH_1).   
\end{align}
\end{lemma}
\begin{proof}
This follows from results of Krein and Yavryan \cite{KY81} (see also 
\cite[Proposition\ 6.5]{BY93}) and the fact that
\begin{align}
& (\bsH_1 - z \bsI)^{-1/2} \bsB' (\bsH_1 - z \bsI)^{-1/2} =  
 \big[\big((\bsH_0 - z \bsI)^{1/2}\big)^* \big((\bsH_1 - z \bsI)^{-1/2}\big)^*\big]^*    \\
& \quad \times \big[|(\bsB')^*|^{1/2} \big((\bsH_0 - z \bsI)^{-1/2}\big)^*\big]^* 
U_{\bsB'} \big[|\bsB'|^{1/2}(\bsH_0 - z \bsI)^{-1/2}\big]   \\
& \quad \times 
\big[(\bsH_0 - z \bsI)^{1/2} (\bsH_1 - z \bsI)^{-1/2}\big] \in \cB_1\big(L^2(\bbR; \cH)\big), 
\quad z \in \rho(\bsH_1),
\end{align}
applying \eqref{3.32}, \eqref{4.B'}, and \eqref{3.61}.
\end{proof}

Given these preparations, one can prove the following result, an extension 
of Pushnitski's formula \cite{Pu08}: 

\begin{theorem}  \lb{t8.xi}
Assume Hypothesis \ref{h2.1} and define $\xi(\,\cdot\,; A_+, A_-)$ and
$\xi(\,\cdot\,; \bsH_2, \bsH_1)$ according to \eqref{2.31} and 
\eqref{2.35}, \eqref{2.36}, respectively. Then, 
\begin{equation}
\xi(\lambda; \bsH_2, \bsH_1) = \frac{1}{\pi}\int_{-\lambda^{1/2}}^{\lambda^{1/2}}
\frac{\xi(\nu; A_+,A_-)\, d\nu}{(\lambda-\nu^2)^{1/2}} 
\, \text{ for a.e.\ } \, \lambda>0,     \lb{2.37} 
 \end{equation}
with a convergent Lebesgue integral on the right-hand side of \eqref{2.37}. 
\end{theorem}
\begin{proof} By Lemma \ref{l7.trfor} one has 
\begin{equation}
\tr_{\cH}\big(g_{z}(A_+) - g_{z}(A_-)\big)
  = - z \int_{\bbR} \frac{\xi(\nu; A_+, A_-) \, d\nu}{(\nu^2 - z)^{3/2}},
\quad z\in\bbC\backslash [0,\infty).       \lb{2.38}
\end{equation}
The trace identity \eqref{trfOLD} then yields
\begin{equation}
\int_{[0, \infty)}  \frac{\xi(\lambda; \bsH_2, \bsH_1) \, 
d\lambda}{(\lambda -z)^{2}}
= \frac{1}{2} \int_{\bbR} \frac{\xi(\nu; A_+, A_-) \, d\nu}{(\nu^2 - z)^{3/2}},
\quad z\in\bbC\backslash [0,\infty),
\end{equation}
and hence,
\begin{align}
& \int_{[0, \infty)}  \xi(\lambda; \bsH_2, \bsH_1) 
\bigg(\frac{d}{dz}(\lambda -z)^{-1}\bigg)
  d\lambda
= \int_{\bbR} \xi(\nu; A_+, A_-)\bigg(\frac{d}{dz} (\nu^2 - z)^{-1/2}\bigg) d\nu,  \no  \\
& \hspace*{9cm} z\in\bbC\backslash [0,\infty).     \lb{2.40}
\end{align}
Integrating \eqref{2.40} with respect to $z$ from a fixed point $z_0 
\in (-\infty,0)$ to
$z\in\bbC\backslash\bbR$ along a straight line connecting $z_0$ and 
$z$ then results in
\begin{align}
\begin{split}
& \int_{[0, \infty)}  \xi(\lambda; \bsH_2, \bsH_1)
\bigg(\frac{1}{\lambda - z} - \frac{1}{\lambda - z_0}\bigg) d\lambda      \\
& \quad = \int_{\bbR} \xi(\nu; A_+, A_-)\big[(\nu^2 - z)^{-1/2} - (\nu^2 - 
z_0)^{-1/2}\big] \, d\nu,
\quad z\in\bbC\backslash [0,\infty).    \lb{2.42}
\end{split} 
\end{align}
One notes that $\big[(\nu^2 - z)^{-1/2} - (\nu^2 - z_0)^{-1/2}\big] = O 
\big(|\nu|^{-3}\big)$ as
$|\nu|\to\infty$, compatible with the fact \eqref{2.33} and similarly,
$\big[(\lambda - z)^{-1} - (\lambda - z_0)^{-1}\big] = O 
\big(|\lambda|^{-2}\big)$, compatible with the fact \eqref{2.36a}.

Applying the Stieltjes inversion formula (cf., e.g., \cite{AD56}, 
\cite[Theorem\ B.3]{We80}) to \eqref{2.42} then yields
\begin{align}
\xi (\lambda; \bsH_2, \bsH_1)
&= \lim_{\varepsilon\downarrow 0} \frac{1}{\pi} \int_{[0,\infty)}
\xi (\lambda'; \bsH_2, \bsH_1) \Im\big((\lambda'  - 
\lambda) - i \varepsilon)^{-1}\big) d\lambda'
\no  \\
& =  \lim_{\varepsilon\downarrow 0} \frac{1}{\pi}  \int_{\bbR} \xi(\nu; A_+, A_-)
\Im\big((\nu^2 - \lambda - i \varepsilon)^{-1/2}\big) d\nu     \no  \\
& = \frac{1}{\pi}  \int_{- \lambda^{1/2}}^{\lambda^{1/2}}
\frac{\xi(\nu; A_+, A_-) \, d\nu}{(\lambda - \nu^2)^{1/2}} \, \text{ for 
a.e.\ $\lambda > 0$.}   \lb{2.43}
\end{align}
The last step in \eqref{2.43} still warrants some comments: One 
splits $\bbR$ into the two regions $0 \leq \nu^2 \leq \lambda + 1$ and 
$\nu^2 \geq \lambda + 1$. In the compact region $0 \leq \nu^2 \leq 
\lambda + 1$ one can immediately apply Lebesgue's dominated 
convergence theorem since $\xi(\,\cdot\,; A_+, A_-)$ is locally 
integrable. One also uses that $\Im\big((\nu^2 -\lambda)^{1/2}\big) = 
0$ for $\nu^2 \in [\lambda, \lambda + 1]$ and that
$\xi(\nu; A_+, A_-) (\lambda - \nu^2)^{-1/2}$ is locally integrable for 
a.e.\ $\lambda > 0$.

The latter fact can be seen as follows: Decomposing $(\lambda - 
\nu^2)^{-1/2}$ into
$(\lambda^{1/2} - \nu)^{-1/2}(\lambda^{1/2} + \nu)^{-1/2}$, and focusing 
on the case
$\nu\geq 0$ at first, one sees that only the factor $(\lambda^{1/2} - 
\nu)^{-1/2}$ is relevant in this case and one can reduce matters to a 
convolution estimate. Thus, one introduces
\begin{equation}
  f_R(\nu) = \begin{cases} \nu^{-1/2}, & 0 < \nu < R, \\
0, & \nu>R, \, \nu<0,  \end{cases} \quad R>0,   \quad
  g(\nu) = \begin{cases} |\xi(\nu; A_+, A_-)|, & \nu>0,  \\
0, & \nu<0.  \end{cases}
\end{equation}
Then $f_R, \, g \in L^1(\bbR; d\nu)$ and hence a special case of 
Minkowski's inequality (which in turn is a special case of Young's 
inequality, $\|h*k\|_r \leq \|h\|_p \, \|k\|_q$,
$1\leq p, q, r \leq \infty$, $p^{-1}+q^{-1}=1+r^{-1}$, with 
$\|\cdot\|_p$ the norm in
$L^p(\bbR; d\lambda)$, cf., e.g., \cite[p.\ 20--22]{Gr04}), shows that $f_R 
* g \in L^1(\bbR; d\lambda)$, in particular, $(f_R * g)(\lambda)$ exists 
for a.e.\ $\lambda>0$. Since $R>0$ is arbitrary, $\xi(\nu; A_+, A_-) 
(\lambda - \nu^2)^{-1/2}$ is locally integrable with respect to $\nu$ on
$[0,\infty)$ for a.e.\ $\lambda > 0$. The case $\nu\leq 0$ is handled 
analogously.

Finally, in the region $\nu^2 \geq \lambda + 1$ one estimates that
\begin{equation}
\big|\Im\big((\nu^2 - \lambda - i \varepsilon)^{-1/2}\big)\big| \leq 
\frac{\varepsilon}{(\nu^2 - \lambda)^{3/2}},
\quad \nu^2 \geq \lambda + 1,
\end{equation}
completing the proof of \eqref{2.43}. 
\end{proof}

One notes that while the outline of this proof still closely follows 
the corresponding proof of Theorem\ 1.1 by Pushnitski in \cite{Pu08}, 
the finer details of our 
approach now necessarily deviate from his proof due to  our more 
general Hypothesis \ref{h2.1}.

The next result also follows Pushnitski in \cite{Pu08} closely (but 
again necessarily deviates in some details):

\begin{lemma}  \lb{l8.2}
Assume Hypothesis \ref{h2.1} and suppose that $0 \in \rho(A_+)\cap\rho(A_-)$. 
Then $\bsH_1$ $($and hence $\bsH_2$$)$ has an essential spectral gap near 
$0$, that is, there exists an $a>0$ such that
\beq \lb{ess}
\sigma_{\rm ess}(\bsH_1) = \sigma_{\rm ess}(\bsH_2) \subseteq [a,\infty).
\enq
\end{lemma}
\begin{proof}
By Lemma \ref{trclLHS} and a variant of Weyl's theorem one concludes that
\begin{equation}
\sigma_{\rm ess}(\bsH_1)  = \sigma_{\rm ess}(\bsH_2).      \lb{6.22}
\end{equation}
Next, one recalls the definition of the operators $\bsH$ and $\bsH_1$ from Lemma \ref{l4.relbdd}, and  
\begin{equation}
\dom\big(\bsH_1^{1/2}\big) =  \dom\big((\bsH)^{1/2}\big)
= \dom\big(\bsH_0^{1/2}\big) = \dom(d/dt) \cap\dom(\bsA_-).
\end{equation}
By Lemma \ref{12bprime}, one obtains
\begin{align}
& \big\|(\bsH - z \, \bsI)^{-1/2} \bsB' (\bsH - z \, \bsI)^{-1/2}\big\|_{\cB_1(L^2(\bbR;\cH))} 
\no \\
& \quad  \leq \big\|(\bsH_0 - z \, \bsI)^{1/2}(\bsH - z \, 
\bsI)^{-1/2}\big\|_{\cB(L^2(\bbR;\cH))}^2   \no  \\
& \qquad \times \big\|(\bsH_0 - z \, \bsI)^{-1/2} \bsB' (\bsH_0 - z \, 
\bsI)^{-1/2}\big\|_{\cB_1(L^2(\bbR;\cH))} < \infty, \quad z<0,
\end{align}
and hence $\bsB'$ is relatively form compact with respect to $\bsH_0$ and $\bsH$.
Hence,
\begin{equation}
\sigma_{\rm ess}(\bsH_j)  = \sigma_{\rm ess}(\bsH), \quad j=1,2.
\end{equation}
Since by hypothesis $0 \in \rho(A_+)\cap\rho(A_-)$, one obtains the 
existence of $a>0$ and $T_0>0$ such that
\begin{equation}
A(t)^2 \geq a I \, \text{ for all $|t|\geq T_0$.}
\end{equation}
Next, one writes 
\begin{align}
& (A(t) g,A(t) g)_{\cH} 
= \big[(A(t) g, A(t) g)_{\cH} - a \|g\|_{\cH}^2 \big]  
+ a \|g\|_{\cH}^2    \no  \\
& \quad = \big(g, \big[A(t)^2 - a I\big] E_{A(t)^2}([0,a]) g\big)_{\cH} \no \\
& \quad + 
(A(t) g, E_{A(t)^2}((a,\infty)) A(t) g)_{\cH} - a (g, E_{A(t)^2}((a,\infty))g)_{\cH} 
+ a \|g\|_{\cH}^2   \no  \\
& \quad \geq \big(g, \big[A(t)^2 - a I\big] E_{A(t)^2}([0,a]) g\big)_{\cH}  + a \|g\|_{\cH}^2   \no  \\
& \quad = (g, F(t) g)_{\cH} + a \|g\|_{\cH}^2,  \quad g \in \dom(A_-), \; 
t \in \bbR,    \lb{6.28}
\end{align}
where
\begin{equation}
F(t) = \big[A(t)^2 - a I\big] E_{A(t)^2}([0,a]) = \begin{cases} 0,\, 
|t| \geq T_0, \\
\text{of finite rank for all $t\in\bbR$,} \end{cases}
\end{equation}
choosing $a>0$ sufficiently small.
Indeed, the strongly right continuous family of spectral projections 
of $A(t)^2$ is given in terms of that of $A(t)$ by
\begin{equation}
E_{A(t)^2} (\lambda) = \begin{cases} 0, & \lambda < 0, \\
E_{A(t)}(\{0\}), & \lambda =0, \\
E_{A(t)}\big([-\lambda^{1/2}, \lambda^{1/2}]\big), & \lambda > 0.\end{cases}
\end{equation}
Since $A(t) = A_- + B(t)$ on $\dom(A(t)) = \dom(A_-)$, $t\in\bbR$, and
\begin{equation}
B(t) (A_- - z I)^{-1} \in \cB_1(\cH), \quad z \in\bbC\backslash\bbR, \; t\in\bbR, 
\end{equation}
by Theorem \ref{t3.7}, one infers
\begin{equation}
\sigma_{\rm ess}(A(t))  = \sigma_{\rm ess}(A_-), \quad t\in\bbR.
\end{equation}
Since $0\in\rho(A_-)$, choosing $a>0$ sufficiently small, $A(t)$ 
has at most finitely many eigenvalues of finite multiplicity in 
the interval $[-a^{1/2}, a^{1/2}]$, and thus $A(t)^2$ has at 
most finitely many eigenvalues of finite multiplicity in the interval 
$[0, a]$, implying the finite rank property of $F(t)$ for all 
$t\in\bbR$. Thus, one obtains
\begin{equation}
\int_{\bbR} \|F(t)\|_{\cB_1(\cH)} < \infty,
\end{equation}
and applying \cite[Lemma 2.2]{Pu08} to the operator $F(t)$, 
$t\in\bbR$,  then proves that $\bsF$ in $L^2(\bbR;\cH)$, defined by
\begin{equation}
(\bsF f)(t) = F(t) f(t), \quad t\in\bbR, \; f \in L^2(\bbR;\cH),
\end{equation}
is form compact relative to the operator $\bsH_{0,0}$ in $L^2(\bbR;\cH)$ defined 
by $\bsH_{0,0} = - \frac{d^2}{dt^2}$ with maximal domain. Thus,
\begin{equation}
\sigma_{\rm ess}(\bsH_{0,0} + \bsF)  = \sigma_{\rm ess}(\bsH_{0,0}) = [0,\infty).
\end{equation}
Finally, \eqref{6.28} implies $\bsH \geq \bsH_{0,0} + \bsF + a \bsI$, and hence
\begin{equation}
\sigma_{\rm ess}(\bsH_j)  = \sigma_{\rm ess}(\bsH) \subseteq 
[a,\infty), \quad j=1,2.
\end{equation}
\end{proof}

Theorem \ref{t8.xi} now easily yields the following Fredholm index result:

\begin{corollary}  \lb{c8.index}
Assume Hypothesis \ref{h2.1} and define $\xi(\,\cdot\,; A_+, A_-)$ and
$\xi(\,\cdot\,; \bsH_2, \bsH_1)$ as in \eqref{2.31} and \eqref{2.35}, 
\eqref{2.36}, respectively.  Moreover, suppose that $0 \in 
\rho(A_+)\cap\rho(A_-)$. Then $\bsD_\bsA^{}$
is a Fredholm operator in $L^2(\bbR;\cH)$ and
\beq\lb{indform}
\begin{split}
\ind(\bsD_\bsA^{}) &= \xi(0_+; \bsH_2, \bsH_1)   \\
&=\xi(0; A_+, A_-).
\end{split} 
\enq
\end{corollary}
\begin{proof}
Since $\sigma_{\rm ess} (\bsH_2) = \sigma_{\rm ess} (\bsH_1)$ by equation
\eqref{6.22}, $\bsH_1$ and $\bsH_2$ have an essential spectral gap 
near $0$ by Lemma \ref{l8.2}. In addition,
$\bsH_1=\bsD_\bsA^* \bsD_\bsA^{}$ and $\bsH_2=\bsD_\bsA^{} \bsD_\bsA^*$ 
have the same 
nonzero eigenvalues including multiplicities, and hence one concludes 
by the general properties of $\xi(\,\cdot\,; \bsH_2, \bsH_1)$ in 
essential spectral gaps of $\bsH_2$ and $\bsH_1$ (cf.\ \cite[p.\ 276, 
300]{Ya92}) that
\begin{equation}
\ind(\bsD_\bsA^{}) = \dim(\ker(\bsH_1)) - \dim(\ker(\bsH_2)) = \xi 
(\lambda; \bsH_2, \bsH_1), \quad
\lambda \in (0,\lambda_0),   \lb{2.46}
\end{equation}
for $\lambda_0 < \inf(\sigma_{\rm ess} (\bsH_2)) = \inf(\sigma_{\rm ess} (\bsH_1))$. 

On the other hand, since  $0 \in \rho(A_+)\cap\rho(A_-)$, there exists a constant 
$c\in\bbR$ such that $\xi(\,\cdot\,; A_+, A_-) = c$ a.e.\ on the interval 
 $(-\nu_0,\nu_0)$ for $0 < \nu_0$ sufficiently small. (This follows from the basic properties of the spectral shift function in joint essential spectral gaps of $A_-$ 
 and $A_+$, cf.\ \cite[p.\ 300]{Ya92}.) Hence, one may define 
\begin{equation}
\xi(\nu; A_+, A_-) = \xi(0; A_+, A_-), \quad \nu \in (-\nu_0,\nu_0).    \lb{2.47}
\end{equation}
Thus, taking $\lambda \to 0$ in \eqref{2.37}, utilizing \eqref{2.46}, \eqref{2.47}, and
\begin{equation}
\frac{1}{\pi}\int_{-\lambda^{1/2}}^{\lambda^{1/2}}
\frac{d\nu}{(\lambda- \nu^2)^{1/2}} = 1 \, \text{ for all $\lambda > 0$},
\end{equation}
finally yields \eqref{indform}.
\end{proof}

\subsection{Supersymmetry and the Atiyah--Patodi--Singer Spectral Asymmetry} 
We conclude this section with an application involving the Atiyah--Patodi--Singer 
(APS) spectral asymmetry (cf., e.g.,  \cite{APS73}--\cite{APS76}, 
\cite{BB04}, \cite{BL99a}, \cite{DW91}, \cite{Gi84}, \cite{GS83}, \cite{Gr01}, 
\cite{GS96}, \cite{KL04}, \cite{LW96}, \cite{Lo84}, \cite{Mu94}, \cite{Mu98}, 
\cite{NS84}, \cite{NT85}, \cite{Si87}, and the extensive list of references in 
\cite{BGGSS87}) applied to the case of supersymmetric Dirac-type 
operators $\bsQ_m$ (cf.\ \cite{BGGSS87}, \cite{Ge86}, 
\cite{GS88}, \cite[Ch.\ 5]{Th92}, and the references cited therein) 
defined as follows: In the Hilbert 
space $L^2(\bbR; \cH) \oplus L^2(\bbR; \cH)$ we 
consider the $2 \times 2$ block operator-valued matrix
\begin{equation}
\bsQ_m = \begin{pmatrix} m & \bsD_{\bsA} \\ \bsD_{\bsA}^* & - m \end{pmatrix}, 
\quad m \in \bbR\backslash\{0\},     \lb{8.39}
\end{equation}
such that 
\begin{equation}
\bsQ_m^2 = \begin{pmatrix} \bsH_2 + m^2 \bsI & 0 \\ 
0 & \bsH_1 + m^2 \bsI \end{pmatrix},      \lb{8.40}
\end{equation}
and hence
\begin{equation}
\bsQ_m e^{-t \bsQ_m^2} = \begin{pmatrix}
m \, e^{-t(\bsH_2 + m^2 \bsI)} & \bsD_{\bsA} e^{-t(\bsH_1 + m^2 \bsI)} \\
\bsD_{\bsA}^* e^{-t(\bsH_2 + m^2 \bsI)} & - m \, e^{-t(\bsH_1 + m^2 \bsI)}
\end{pmatrix}, \quad t\geq 0. 
\end{equation}

The zeta function regularized spectral asymmetry $\eta_m(t)$, $t>0$, associated 
with $\bsQ_m$, is defined by 
\begin{align} 
& \eta_m(s) = \f{m}{\Gamma ((s+1)/2)} \int_{[0,\infty)} t^{(s-1)/2} \, 
{\tr}_{L^2(\bbR; \cH)} 
\big(e^{-t (\bsH_1 + m^2 \bsI)} - e^{-t (\bsH_2 + m^2 \bsI)}\big) dt,&    \no \\  
& \hspace*{8cm} m \in \bbR\backslash\{0\}, \; s>0,       \lb{8.41}
\end{align}
and the APS spectral asymmetry (eta invariant) $\eta_m$ is then given by 
\begin{equation}
\eta_m = \lim_{s \downarrow 0} \eta_m (s), \quad m \in \bbR\backslash\{0\}, 
\lb{8.42}
\end{equation}
whenever the limit in \eqref{8.42} exists. Intuitively, $\eta_m$ measures the 
asymmetry of the positive and negative spectrum of $Q_m$, 
$m\in\bbR\backslash\{0\}$. The 
asymmetry vanishes if $m=0$ since then $Q_0$ is unitarily equivalent to 
$- Q_0$ (cf.\ \cite{GSS91}).  

Similarly, using the fact that 
\begin{align} 
& \bsQ_m |\bsQ_m|^{-1} e^{-t \bsQ_m^2}    \no \\ 
& \quad = \begin{pmatrix} 
m \, (\bsH_2 + m^2 \bsI)^{-1/2} e^{-t (\bsH_2 + m^2 \bsI)}  
& \bsD_{\bsA} (\bsH_1 + m^2 \bsI)^{-1/2} e^{-t (\bsH_1 + m^2 \bsI)}  \\
\bsD_{\bsA}^* (\bsH_2 + m^2 \bsI)^{-1/2} e^{-t (\bsH_2 + m^2 \bsI)}  
& - m \, (\bsH_1 + m^2 \bsI)^{-1/2} e^{-t (\bsH_1 + m^2 \bsI)} 
\end{pmatrix},    \no \\
& \hspace*{10cm} t \geq 0,    
\end{align}
the heat kernel regularized spectral asymmetry $\wti \eta_m(t)$, $t>0$, associated 
with $\bsQ_m$, is defined by 
\begin{align}
& \wti \eta_m(t) = m \, {\tr}_{L^2(\bbR; \cH)} 
\big((\bsH_2 + m^2 \bsI)^{-1/2} e^{-t (\bsH_2 + m^2 \bsI)}     \lb{8.43} \\
& \hspace*{3.2cm} - (\bsH_1 + m^2 \bsI)^{-1/2} e^{-t (\bsH_1 + m^2 \bsI)}\big), 
\quad m \in \bbR\backslash\{0\},  \; t>0,   \no
\end{align}
and the corresponding spectral asymmetry $\wti \eta_m$ is then given by 
\begin{equation}
\wti \eta_m = \lim_{t \downarrow 0} \wti \eta_m (t), \quad m \in \bbR\backslash\{0\}, 
\lb{8.44}
\end{equation}
whenever the limit in \eqref{8.44} exists.  

Denoting by $\Gamma(\cdot)$ the gamma function \cite[Sect.\ 6.1]{AS72}, 
by $K_0(\cdot)$ the modified (irregular) Bessel function of order zero 
\cite[Sect.\ 9.6]{AS72}, and by $W_{\kappa,\mu}(\cdot)$ the (irregular) 
Whittaker function \cite[Sect.\ 13.1]{AS72}, one obtains the following explicit 
result for $\eta_m$ and $\wti \eta_m$ and their regularizations:

\begin{lemma} \lb{l8.5}
Assume Hypothesis \ref{h2.1} and $m \in \bbR\backslash\{0\}$. Then 
\begin{align}
\begin{split}
\eta_m(s) &= - m \f{s+1}{2} \int_{[0,\infty)} 
\f{\xi(\lambda; \bsH_2, \bsH_1) \, d\lambda}{(\lambda + m^2)^{(s+3)/2}},    \\[1mm]
&= - m \f{s+1}{2 \pi^{1/2}} \f{\Gamma((s+2)/2)}{\Gamma((s+3)/2)} \int_{\bbR} 
\f{\xi(\nu; A_+, A_-) \, d\nu}{(\nu^2 + m^2)^{(s+2)/2}}, \quad s>0,    \lb{8.45}
\end{split}
\end{align}
and $\eta_m (\cdot)$ extends analytically to the open right half-plane $\Re(s) > - 1/2$. Moreover, 
\begin{align}
\wti \eta_m(t) &= m \int_{[0, \infty)} \xi(\lambda; \bsH_2, \bsH_1) \, d \lambda  
\bigg(\f{d}{d\lambda} 
\big[(\lambda + m^2)^{-1/2} e^{-t (\lambda + m^2)}\big]\bigg)    \no \\
& = - \f{m}{2 \pi^{1/2}} \int_{\bbR} \f{\xi(\nu; A_+, A_-) \, d\nu}{\nu^2 + m^2} 
W_{-1/2,-1/2} (t(\nu^2 + m^2)) e^{-t (\nu^2 + m^2)/2}      \lb{8.46}  \\
& \quad - \f{m}{\pi} t \int_{\bbR} \xi(\nu; A_+, A_-) \, d\nu \, K_0(t(\nu^2 + m^2)/2)
e^{-t (\nu^2 + m^2)/2},  \quad t>0.    \no 
\end{align}
In addition, 
\begin{align}
\begin{split}
\eta_m = \wti \eta_m 
& = - \f{m}{2} \int_{[0,\infty)} 
\f{\xi(\lambda; \bsH_2, \bsH_1) \, d\lambda}{(\lambda + m^2)^{3/2}}    \\
&= - \f{m}{\pi} \int_{\bbR} \f{\xi (\nu; A_+, A_-) \, d\nu}{\nu^2 + m^2}. 
\lb{8.47}
\end{split} 
\end{align}
\end{lemma}
\begin{proof}
Using \eqref{8.41}, 
one obtains from the standard trace formula applied to the pair $(\bsH_2, \bsH_1)$ 
(cf.\ \cite[Theorem\ 8.7.1]{Ya92}), Fubini's theorem, and the gamma function 
representation \cite[no.\ 6.1.1, p.\ 255]{AS72}, that 
\begin{align}
\eta_m (s) &= \f{m}{\Gamma((s+1)/2)} \int_0^{\infty} t^{(s-1)/2} \, 
{\tr}_{L^2(\bbR; \cH)} \big(e^{- t (\bsH_2 + m^2 \bsI)} - e^{- t (\bsH_1 + m^2 \bsI)}\big) 
\, dt     \no \\
&= - \f{m}{\Gamma((s+1)/2)} \int_0^{\infty} t^{(s+1)/2} \, e^{- t m^2}
\bigg(\int_{[0,\infty)} \xi(\lambda; \bsH_2, \bsH_1) e^{-t \lambda} \, d\lambda\bigg) 
dt   \no \\
&=  - \f{m}{\Gamma((s+1)/2)} \int_{[0,\infty)}  \xi(\lambda; \bsH_2, \bsH_1) 
\bigg(\int_0^{\infty} t^{(s+1)/2} \, e^{- t (\lambda + m^2)} \, dt \bigg) d\lambda  \no \\
&= - m \f{\Gamma((s+3)/2)}{\Gamma((s+1)/2)} \int_{[0,\infty)}  
\f{\xi(\lambda; \bsH_2, \bsH_1) \, d\lambda} 
{(\lambda + m^2)^{(s+3)/2}}    \no \\
&= - m \f{s+1}{2}  \int_{[0,\infty)}  \f{\xi(\lambda; \bsH_2, \bsH_1) \, d\lambda}
{(\lambda + m^2)^{(s+3)/2}},     \no \\
&= - m \f{s+1}{2 \pi}  \int_{[0,\infty)} 
\bigg(\int_{- \lambda^{1/2}}^{\lambda^{1/2}} 
\f{\xi(\nu; A_+, A_-) \, d\nu}{(\lambda - \nu^2)^{1/2}} 
\bigg) \f{d\lambda}{(\lambda + m^2)^{(s+3)/2}}, \quad s>0,     \lb{8.49}
\end{align}
using the functional equation $\Gamma(z+1) = z \Gamma(z)$ to arrive at the next 
to last step and inserting \eqref{2.37} in the last step.

Next, one transforms the double integral in \eqref{8.49}, where   
$(\lambda,\nu) \in [0,\infty) \times [0,\lambda^{1/2})$, to 
$(\nu,\lambda) \in [0,\infty) \times [\nu^2,\infty)$, and similarly, that where  
$(\lambda,\nu) \in [0,\infty) \times [- \lambda^{1/2},0]$, to 
$(\nu,\lambda) \in (-\infty,0] \times [\nu^2,\infty)$, and using Fubini's theorem 
again one obtains  
\begin{align}
\eta_m (s) &=  - m \f{s+1}{2 \pi} \int_{\bbR} \xi(\nu; A_+, A_-) 
\bigg( \int_{[\nu^2, \infty)} 
\f{d \lambda}{(\lambda - \nu^2)^{1/2} (\lambda + m^2)^{(s+3)/2}}\bigg) d\nu  \no \\
&= - m \f{s+1}{2 \pi^{1/2}} \f{\Gamma((s+2)/2)}{\Gamma((s+3)/2)} 
\int_{\bbR} \f{\xi(\nu; A_+, A_-) \, d\nu}{(\nu^2 + m^2)^{(s+2)/2}},  
\quad s > 0,    \lb{8.50}
\end{align} 
where we used 
\begin{equation}
\int_{\alpha}^{\infty} \f{d\lambda}{(\lambda - \alpha)^{1-a} (\lambda + \beta)^{b}} 
= (\alpha + \beta)^{- (b - a)} B(b-a,a),  \quad \alpha + \beta > 0, \; b>a>0,      \lb{8.51}
\end{equation}
according to \cite[no.\ 3.1962, p.\ 285]{GR80}, 
with $B(z,w) = \Gamma(z) \Gamma(w)/\Gamma(z+w)$, the beta function, and 
$\Gamma(1/2) = \pi^{1/2}$ (cf.\ \cite[Sects.\ 6.1, 6.2]{AS72}). This proves \eqref{8.45}. 
By \eqref{2.36aa}, the first equation in \eqref{8.45} proves the existence of an 
analytic continuation of $\eta_m(\cdot)$ to the open right half-plane $\Re(s) > - 1/2$.  
The facts \eqref{2.33xi} and \eqref{2.36aa} together with Lebesgue's dominated 
convergence theorem employed in both equalities in \eqref{8.45} then prove 
\eqref{8.47} in the case of $\eta_m$.

The corresponding proof of \eqref{8.46}, and the remaining proof of \eqref{8.47} 
in the case of $\wti \eta_m$ proceed along entirely analogous steps, but 
naturally, the second equality in \eqref{8.46} is based on more involved arguments. 
To shorten the remainder of this proof a bit we now focus just on the major steps 
in the computations: Employing \eqref{8.43}, 
one concludes from the standard trace formula in 
\cite[Theorem\ 8.7.1]{Ya92} that
\begin{align}
\wti \eta_m (t) &= m \int_{[0,\infty)} \xi(\lambda; \bsH_2, \bsH_1) \, d \lambda 
\bigg(\f{d}{d\lambda} \big[(\lambda + m^2)^{-1/2} e^{- t (\lambda + m^2)}\big]
\bigg)  \no \\
&= - m \int_{[0,\infty)} \xi(\lambda; \bsH_2, \bsH_1) \, d \lambda 
\bigg[ \f{1}{2 (\lambda + m^2)^{3/2}} + \f{t}{(\lambda + m^2)^{1/2}}\bigg] 
e^{- t (\lambda + m^2)}   \no \\
&= - \f{m}{\pi} \int_{[0,\infty)} \bigg(\int_{- \lambda^{1/2}}^{\lambda^{1/2}} 
\f{\xi(\nu; A_+, A_-) \, d\nu}{(\lambda - \nu^2)^{1/2}}\bigg)   \no \\
& \hspace*{2.4cm} \times 
\bigg[ \f{1}{2 (\lambda + m^2)^{3/2}} + \f{t}{(\lambda + m^2)^{1/2}}\bigg] 
e^{- t (\lambda + m^2)} \, d\lambda  \no \\
&= - \f{m}{\pi} \int_{\bbR} \xi(\nu; A_+, A_-) \bigg(\int_{[\nu^2,\infty)} 
\f{1}{(\lambda - \nu^2)^{1/2}}    \no \\
& \hspace{2.4cm} \times  
\bigg[ \f{1}{2 (\lambda + m^2)^{3/2}} + \f{t}{(\lambda + m^2)^{1/2}}\bigg] 
e^{- t (\lambda + m^2)} \, d\lambda\bigg) d\nu   \no \\
& = - \f{m}{2 \pi^{1/2}} \int_{\bbR} \f{\xi(\nu; A_+, A_-) \, d\nu}{\nu^2 + m^2} 
W_{-1/2,-1/2} (t(\nu^2 + m^2)) e^{-t (\nu^2 + m^2)/2}     \no \\
& \quad - \f{m}{\pi} t \int_{\bbR} \xi(\nu; A_+, A_-) \, d\nu \, K_0(t(\nu^2 + m^2)/2)
e^{-t (\nu^2 + m^2)/2},  \quad t>0.        \lb{8.53}
\end{align}
Here we used
\begin{align}
& \int_{\alpha}^{\infty} \f{e^{-c \lambda} \, d\lambda}
{(\lambda - \alpha)^{1-a} (\lambda + \beta)^{b}} 
= (\alpha + \beta)^{- (b - a + 1)/2} c^{(b-a-1)/2} e^{-c(\alpha - \beta)/2} \Gamma (a) 
\no \\
& \quad \times W_{(1-b-a)/2,(a-b)/2} (c(\alpha + \beta)),  
\quad \alpha > 0, \; \alpha + \beta > 0, \; c>0, \; a>0,      \lb{8.54}
\end{align}
according to \cite[no.\ 3.3843, p.\ 320]{GR80}, and 
\begin{equation}
W_{0,0} (z) = \pi^{-1/2} z^{1/2} K_0(z/2),     \lb{8.55}
\end{equation}
combining no.\ 13.1.33 on p.\ 505 and no.\ 13.6.21 on p.\ 510 in \cite{AS72}.
\end{proof}

Equation \eqref{8.45} and the existence of an analytic continuation of   
$\eta_m(\cdot)$ to the open right half-plane $\Re(s) > - 1/2$ suggests the 
possibility that under the assumptions of Hypothesis \ref{h2.1} (and in analogy 
to \eqref{2.36aa}), one actually has 
$\xi(\,\cdot\,; A_+, A_-) \in L^1(\bbR; (1+|\nu|)^{-1} d\nu)$, but this is left to a future investigation.

\section{Connections Between the Index and the Spectral Flow}  \lb{s9}

In this section we briefly discuss connections of our results to the topic of the spectral flow for 
the family of operators $\{A(t)\}_{t=-\infty}^\infty$ defined in \eqref{defAT}, \eqref{dfnAplus}. While 
there are several definitions of the spectral flow available in the literature, we will follow the scheme originated in \cite{Ph96} (see also \cite{BBLP05} and  \cite{Le05}), but also note, for instance, the definition in \cite[Theorem 4.23]{RS95} that uses the Kato Selection Theorem (cf.\ e.g., 
\cite[Theorems II.5.4 and II.6.8]{Ka80} and \cite[Theorem 4.23]{RS95}), and the definition in \cite{Pu01} and \cite{Pu09}.

The spectral flow is defined in \cite[Definition 1.1]{Le05} for a family of operators continuous 
with respect to the graph metric (which induces convergence in the norm resolvent sense, 
cf.\ \cite[Sect.\ VIII.7]{RS80}).  The {\em graph} metric,  $d_G$, on the space of (unbounded) self-adjoint operators on the Hilbert space $\cH$ is defined as follows: for any two self-adjoint operators, $S_1$ and $S_2$, we set
\begin{equation}
d_G(S_1,S_2)=\|(S_2-i)^{-1}-(S_1-i)^{-1}\|_{\cB(\cH)}.
\end{equation}
Another metric on the space of (unbounded) self-adjoint operators is the {\em Riesz} metric, $d_R$,  defined by the formula
\begin{equation}
d_R(S_2,S_1)=\|g(S_2)-g(S_1)\|_{\cB(\cH)}, \quad g(x)=x(1+x^2)^{-1}.
\end{equation}
Finally, given a self-adjoint operator $A_-$, let us consider the set of all (unbounded) self-adjoint operators having the same domain as $A_-$. 
On this set one can define a metric, $d_{|A_-|}$, by the formula
\beq d_{|A_-|}(S_2,S_1)=\|(S_2-S_1)(|A_-|+I)^{-1}\|_{\cB(\cH)}.      \lb{8.3}
\enq
The metric $d_{|A_-|}$ is strictly stronger than $d_R$, and the metric $d_R$ is strictly stronger than $d_G$, see \cite[Proposition 2.2]{Le05} (as well as comments 
following that proposition and further references therein) for the proof of this result. 

\begin{lemma}\label{lem_cont}
Assume Hypothesis \ref{h2.1}. Then the family $\{A(t)\}_{t=-\infty}^\infty$ of operators defined in \eqref{defAT}, \eqref{dfnAplus} is continuous at each $t\in\bbR$, and also $\lim_{t\to\pm\infty}A(t)=A_\pm$ holds with respect to each of the metrics $d_{|A_-|}$, $d_R$, and $d_G$.
\end{lemma}
\begin{proof}
By the observation following \eqref{8.3}, it suffices to consider $d_{|A_-|}$ only. However, to make the underlying issues more transparent, we will present 
independent proofs for  all three metrics.

Metric $d_{|A_-|}$: For any $-\infty\le a<b\le+\infty$
 the distance $d_{|A_-|}(A(b),A(a))$ (cf.\ \eqref{2.13jk}), is dominated by
 \beq
 \begin{split}
\|(A(b)-A(a))(|A_-|+I)^{-1}\|_{\cB_1(\cH)}
 \le\int_a^b
 \|B'(s)(|A_-|+I)^{-1}\|_{\cB_1(\cH)}\,ds,
 \end{split}        \label{intineq}
 \enq
and the required in the lemma assertions follow from \eqref{2.13i}.

Metric $d_R$: This follows, essentially, from Lemma \ref{ext_subtle}.  
Indeed, using \eqref{subtle2} with $S_+=A(b)$ and $S_-=A(a)$, the distance $d_R(A(b),A(a))$ is dominated by
\begin{align}
\|g(A(b))&-g(A(a))\|_{\cB_1(\cH)}\le
\big\|T_\phi\big[\overline{\varkappa(A(b))^{-1/2}\big(A(b)-A(a)\big)\varkappa(A(a))^{-1/2}}\big]\big\|_{\cB_1(\cH)} \no \\
&\le\|T_\phi\|_{\cB(\cB_1(\cH))}
\big\|\overline{\varkappa(A(b))^{-1/2}(|A_-|+I)^{1/2}}\big\|_{\cB(\cH)}\label{intin1}\\
&\quad \times 
\big\|\overline{(|A_-|+I)^{-1/2}(A(b)-A(a))(|A_-|+I)^{-1/2}}\big\|_{\cB_1(\cH)} 
\label{lastin}\\ 
&\quad \times
\big\|(|A_-|+I)^{1/2}\varkappa(A(a))^{-1/2}\big\|_{\cB(\cH)}.    \lb{intin2}
\end{align}
We claim that 
\beq\label{bddcl}
\sup_{t\in\bbR} \big\|(|A_-|+I)^{1/2}\varkappa(A(t))^{-1/2}\big\|_{\cB(\cH)}<\infty.
\enq

Assuming the claim, the proof is completed as follows. First,
for $T_\phi=T_\phi^{(A(b),A(a))}$ in \eqref{intin1} the norms $\|T_\phi\|_{\cB(\cB_1(\cH))}$ are bounded uniformly for 
$a,b\in\bbR$ due to \eqref{normTphi} and Lemma \ref{ext_subtle} (i). The $\cB(\cH)$-norms in \eqref{intin1} and \eqref{intin2} are also bounded uniformly for $a,b\in\bbR$ due to claim \eqref{bddcl} and the relation
\beq\big(\overline{\varkappa(A(b))^{-1/2}(|A_-|+I)^{1/2}}\big)^*= (|A_-|+I)^{1/2}\varkappa(A(b))^{-1/2}.\enq
It remains to estimate the norm in \eqref{lastin}. Using \eqref{normin1} for 
$S=A(b)-A(a)$ and $T= |A_-|+I$, we infer that the norm in \eqref{lastin} is dominated by the expression in the left-hand side of \eqref{intineq}.
Putting all this together, one concludes that there is a constant $c>0$ such that, for any interval $-\infty\le a<b\le+\infty$, 
\begin{equation}\label{Fac}
\|g(A(b))-g(A(a))\|_{\cB_1(\cH)}\le c\int_a^b\|B'(\tau)(|A_-|+I)^{-1}\|_{\cB_1(\cH)}\,d\tau,
\end{equation}
and the assertions in the lemma follow from \eqref{2.13i}.

To proof the claim \eqref{bddcl}, one notes that 
\begin{align} 
&  \|(|A_-|+I)^{1/2}\varkappa(A(t))^{-1/2}\|_{\cB(\cH)}
\le\|(|A_-|+I)^{1/2}(|A(t)|+I)^{-1/2}\|_{\cB(\cH)}  \no \\
& \qquad\times \|(|A(t)|+I)^{1/2}\varkappa(A(t))^{-1/2}\|_{\cB(\cH)}    \no \\
& \quad \le
\|(|A_-|+I)^{1/2}(|A(t)|+I)^{-1/2}\|_{\cB(\cH)} \, 
\sup_{y\in\bbR}\frac{(|y|+1)^{1/2}}{(|y^2|+1)^{1/4}}, 
\end{align} 
and thus it suffices to show that 
\beq\label{bddcl2}\sup_{t\in\bbR}\|(|A_-|+I)^{1/2}(|A(t)|+I)^{-1/2}\|_{\cB(\cH)}<\infty.\enq By  Lemma \ref{UNIFnorms}, there are constants $c_1, c>0$ such that, for all $f\in\cH$ and $t\in\bbR$, 
\begin{align}
\|(|A_-|&+I)^{1/2}(|A(t)|+I)^{-1/2}f\|_{\cH}\le\|(|A(t)|+I)^{-1/2}f\|_{\cH_{1/2}(|A_-|)}\nonumber\\&\le c_1 \|(|A(t)|+I)^{-1/2}f\|_{\cH_{1/2}(|A(t)|)}\\&
=c_1\big(
\|(|A(t)|+I)^{-1/2}f\|_{\cH}^2+\|\,|A(t)|^{1/2}(|A(t)|+I)^{-1/2}f\|_{\cH}^2\big)^{1/2}
\le c\|f\|_{\cH},\nonumber\end{align}
completing the proof.

Metric $d_G$: Using the resolvent identity,  $d_G(A(b),A(a))$ is dominated by
\begin{align}
\|(&A(a)-i)^{-1}\|_{\cB(\cH)}\|(A(b)-A(a))(|A_-|+I)^{-1}\|_{\cB_1(\cH)}\label{ff1}\\&\times
\|(|A_-|+I)(|A(b)|+I)^{-1}\|_{\cB(\cH)}\|(|A(b)|+I)(A(b)-i)^{-1}\|_{\cB(\cH)}\label{ff2}.\end{align}
The first factor in \eqref{ff1} and the second factor in \eqref{ff2} are uniformly bounded for $a,b\in\bbR$ employing the fact that $A(a),A(b)$ are self-adjoint and using 
Lemma \ref{UNIFnorms}. 
By \eqref{UNn1} in Lemma \ref{UNIFnorms}, the first factor in \eqref{ff2} is uniformly bounded for $b\in\bbR$. The second factor in \eqref{ff1} is estimated as in \eqref{intineq}, and again the assertions in the lemma follow from \eqref{2.13i}.
\end{proof}

For additional references in connection with metrics for closed operators we also 
refer to \cite{BBLP05}, \cite{CL63}, \cite{Ka84}, \cite{MMAD08}, \cite{Me99}, 
\cite{Ni07}, \cite{Sh09}, and \cite{Wa08}.

Assuming Hypothesis \ref{h2.1} and $0\in\rho(A_-)\cap\rho(A_+)$, we will now recall the definition of the spectral flow for the operator path $\{A(t)\}_{t=-\infty}^{\infty}$, 
following the line of arguments in \cite{Ph96} (see also \cite{BBLP05,Le05}), where the spectral flow has been defined for paths with $t\in[0,1]$. 

\begin{remark}\label{rem:pril} 
Since   $0\in\rho(A_-)\cap\rho(A_+)$, there exists $\varepsilon_0>0$ such that
$[-\varepsilon_0,\varepsilon_0]\cap\sigma(A_\pm)=\emptyset$.
Since the family $\{A(t)\}_{t=-\infty}^{\infty}$ is $d_G$-continuous by Lemma \ref{lem_cont}, the 
function $\bbR\ni t\mapsto\sigma(A(t))$ is upper semicontinuous by \cite[Theorem VIII.2.3(a)]{RS80}.
Since $d_R(A(t),A_\pm)\to0$ as $t\to\pm\infty$ by Lemma \ref{lem_cont}, there exists $T_0>0$ 
such that $[-\varepsilon_0,\varepsilon_0]\cap\sigma(A(t))=\emptyset$ for all $|t|\ge T_0$. Moreover, using \eqref{3.esssp}, $[-\varepsilon_0,\varepsilon_0]\cap\sigma_{\text{ess}}(A(t))=\emptyset$ for all 
$t\in\bbR$. Thus, the operators $A_\pm$ and $A(t)$, $t\in\bbR$, are Fredholm, and for each $t\in\bbR$,  the set $[-\varepsilon_0,\varepsilon_0]\cap\sigma(A(t))$ consists at most of finitely many isolated eigenvalues of finite multiplicity.\end{remark}

\begin{remark}\label{rem:pr} By Remark \ref{rem:pril}, for each $t\in\bbR$, there exist 
$\e\in(0,\e_0)$ and $\delta>0$ such that the following assertions hold: 
\begin{align} \label{spas2}
&\pm\e\notin\sigma(A(s))\text{ for all $s\in(t-\delta,t+\delta)$},   \\
\label{spas3} 
& E_{[-\e,\e]}(A(s)) \text{ has finite rank}\\
&\quad \text{and is norm continuous as a function of  $s\in(t-\delta,t+\delta)$}. 
\label{spas3a}
\end{align} 
Indeed, since $[-\varepsilon_0,\varepsilon_0]\cap\sigma_{\text{ess}}(A(t))=\emptyset$, the interval 
$[-\varepsilon_0,\varepsilon_0]$ contains at most finitely many points of $\sigma(A(t))$. Fix 
$\e\in(0,\e_0)$ such that
$\pm\e\notin\sigma(A(t))$. Since 
\begin{equation} 
\sigma(A(t))\subseteq\bbR\backslash\{-\e,+\e\},
\end{equation}
there is an open $d_G$-ball containing $A(t)$, such that 
\begin{equation}
\sigma(A(s))\subseteq\bbR\backslash\{-\e,+\e\}, 
\end{equation} 
provided $A(s)$ is in this ball 
(cf.\ \cite[Theorem VIII.2.3(a)]{RS80}). In addition, 
since $A(\cdot)$ is $d_G$-continuous, there is a $\delta > 0$ such that \eqref{spas2} holds. The inclusion  $[-\e,\e]\subset[-\e_0,\e_0]$ yields that 
$[-\varepsilon,\varepsilon]\cap\sigma_{\text{ess}}(A(s))=\emptyset$, and thus assertion \eqref{spas3}  for all $s\in(t-\delta,t+\delta)$. Finally, the norm continuity in 
\eqref{spas3a} follows by \cite[Theorem VIII.2.3(b)]{RS80}.
\end{remark}

\begin{remark}\label{subd}
By compactness of $[-T_0,T_0]$ (with $T_0$ as in Remark \ref{rem:pril}) and Remark \ref{rem:pr}, 
we may choose a subdivision $-T_0=t_0<t_1<\dots<t_{n-1}<t_n=T_0$, and numbers 
$\varepsilon_j\in(0,\e_0)$ (with $\varepsilon_0 > 0$ as in Remark \ref{rem:pril}), such that for 
each $j=1,\dots,n$, and for all $t\in[t_{j-1},t_j]$, the following assertions hold: \\
$(i)$ $\pm\varepsilon_j\notin\sigma(A(t))$.  \\
$(ii)$ $[-\varepsilon_j,\varepsilon_j]\cap\sigma_{\text{ess}}(A(t))=\emptyset$. \\
$(iii)$ $E_{[-\e_j,\e_j]}(A(t))$ is of finite rank and is norm continuous in  $t\in[t_{j-1},t_j]$. 
\end{remark}

\begin{definition} [\cite{BBLP05,Le05,Ph96}] \label{defSPF} 
Given the notation used in Remark \ref{subd}, we define the spectral flow of the $d_G$-continuous 
path $\{A(t)\}_{t=-\infty}^\infty$ of self-adjoint Fredholm operators by the formula
\begin{align}
\begin{split}
& \text{\rm SpFlow} (\{A(t)\}_{t=-\infty}^\infty)   \label{dfnSpFl} \\
& \quad =\sum_{j=1}^n\big(\dim(\ran (E_{A(t_{j-1})}([0,\varepsilon_j)))) 
- \dim(\ran (E_{A(t_{j})}([0,\varepsilon_j))))\big). 
\end{split} 
\end{align}
\end{definition}

\begin{remark}\label{indep}
As in \cite{Ph96}, one can see that the definition is independent of the choice of 
$T_0$, the subdivision, and the numbers $\e_j > 0$ with the properties described in 
Remarks \ref{rem:pril} and \ref{subd}. Indeed, since  $A(t)$ does not have the 
eigenvalue zero for all $|t|\ge T_0$, the right-hand side  of \eqref{dfnSpFl} does not depend on $T_0$. Adding a point $t_*$ to the subdivision yields adding and 
subtracting the term $\dim(\ran (E_{A(t_*)}([0,\e_*))))$ on the right-hand side of 
\eqref{dfnSpFl}. Finally, changing $\e_j$ by, say, a smaller $\e'_j > 0$, we remark 
that the dimension of the range of 
$E_{[0,\e_j)}(A(t))-E_{[0,\e'_j)}(A(t))=E_{[\e'_j,\e_j)}(A(t))$ is constant
for $t\in[t_{j-1},t_j]$ by the norm continuity of the spectral projections.
Therefore, this change does not affect the right-hand side of  \eqref{dfnSpFl} either.
\end{remark}

\begin{remark}\label{reparam} Equivalently, the definition of  
$\text{SpFlow} (\{A(t)\}_{t=-\infty}^\infty)$ can be reduced to the  definition in 
\cite{BBLP05,Le05,Ph96} for $t\in[0,1]$, by a re-parameterization: Indeed, for any continuous strictly increasing function $r:[0,1]\to\bbR$, we introduce the path 
$\{S(t)\}_{t=0}^1$ by letting $S_0=A_-$, $S(t)=A(r(t))$, $t\in(0,1)$, and $S_1=A_+$, and then define $\text{SpFlow} (\{A(t)\}_{t=-\infty}^\infty)=\text{SpFlow} (\{S(t)\}_{t=0}^1)$; the latter spectral flow is defined by formula \eqref{dfnSpFl}, 
with $A(t)$ replaced by $S(t)$ and the $t_j$'s representing a subdivision of $[0,1]$. An argument similar to Remark \ref{indep} shows that this new definition does not depend on the choice of the re-parameterization $r$ and is equivalent to Definition \ref{defSPF}. An advantage of the definition by re-parameterization is that the proof of \eqref{spf1} becomes shorter as one does not need to show \eqref{tails}. Nevertheless, we prefer to use Definition  \ref{defSPF} as it provides direct insight into the process where eigenvalues of $A(t)$ are passing through zero as $t$ changes from $-\infty$ to 
$+\infty$.
\end{remark} 

Next,  we recall some terminology and several results from \cite{Ab01}, 
\cite{ASS94}, \cite{Le05}:

\begin{definition}\label{defFRP}
A pair $(P,Q)$ of orthogonal projections on $\cH$ is called Fredholm $($see, e.g., 
\cite{ASS94}$)$, if $QP$ is a Fredholm operator from $\ran (P)$ to $\ran (Q)$; the index of the pair $(P,Q)$ is defined to be the Fredholm index of the operator $QP$, that is, by the formula
\beq \label{dfnIn} 
\ind(P,Q)=\dim\big(\ran (P) \cap (\ran (Q))^\bot\big) 
- \dim \big((\ran (P))^\bot \cap \ran (Q)\big).
\enq
\end{definition}

We note that a pair $(P,Q)$ is a Fredholm pair if and only if the essential spectrum of the difference $P-Q$ is a subset of the open interval $(-1,1)$.

\begin{remark}\label{rem:term} $(i)$ If $(P,Q)$ is a Fredholm pair, then  $(Q,P)$ is a Fredholm pair and $\ind(P,Q)=-\ind(Q,P)=-\ind(I-P,I-Q)$ (see \cite[Theorem 3.4(a)]{ASS94}). \\
$(ii)$ If $P-Q$ is compact, then $(P,Q)$ is Fredholm (see \cite[Proposition 3.1]{ASS94}). \\
$(iii)$ If $P-Q\in\cB_1(\cH)$, then $\ind(P,Q)={\tr}_{\cH}(P-Q)$ (see 
\cite[Theorem 4.1]{ASS94}).
\end{remark}

\begin{definition}\label{defFRS}
A pair $(M,N)$ of closed subspaces of $\cH$ is called Fredholm $($see, e.g., \cite[Section 2.2]{Ab01}, \cite[Sect.\ IV.4]{Ka80}$)$, if $M\cap N$ is finite-dimensional, $M+N$ is closed and has finite codimension; the index of the pair $(M,N)$ is defined as 
\beq\label{indssp} \ind(M,N)=\dim(M\cap N)-\dim(M^\bot\cap N^\bot).\enq
The number on the right-hand side  of \eqref{indssp} is also called the relative dimension of the subspaces $M$ and $N^\bot$.
\end{definition}

\begin{remark}\label{rem:term2} Clearly, the pair $(P,Q)$ of orthogonal projections is Fredholm if and only if
the pair of subspaces $M=\ran (P)$ and $N = (\ran (Q))^\bot$ is Fredholm; the indices of the pairs 
$(P,Q)$ and $(M,N)$ are equal. The subspaces $M, N$ are called commensurable 
if $P-Q$ is 
compact (see, e.g., \cite[Section 2.2]{Ab01}); in this case the pair $(M,N^\bot)$ is Fredholm by Remark \ref{rem:term}\,$(ii)$ (see also  \cite[Lemma 7.3]{LT05}). We refer to \cite{BP08} for a detailed discussion of  relations between Fredholm pairs of projections and Fredholm pairs of subspaces.
\end{remark}

For a variety of additional material on closed subspaces, including a number of classical references on the subject, as well as the study of pairs of projections that differ by a compact operator (and necessarily being far from complete), we refer, for instance, to \cite{AS94}, \cite{Ar70}, \cite{ASS94}, \cite{BS10}, \cite{Bo79}, \cite{BL01}, 
\cite{Da58}, \cite{Di48}, \cite{Di49}, \cite{Fr37}, \cite{GM00}, \cite{Go05}, 
\cite{Ha69}, \cite{Ka55}, \cite{Ka97}, \cite{KMM03}, 
\cite{KK47}, \cite{KKM48}, \cite{Pu09}, \cite{Sp94}, \cite{Wo85}, and the numerous references cited therein.

\begin{proposition} [Lesch \cite{Le05}]  \label{prop:lesch}  
Assume that $\{S_t\}_{t=0}^1$ is a $d_R$-continuous path of $($unbounded$\,)$ self-adjoint  Fredholm operators. Assume furthermore that the domain of $S_t$ does not depend on $t$, $\dom (S_t) = \dom (S_0)$, and that for $t\in[0,1]$, the difference
$S_t-S_0$ is an $S_0$-compact symmetric operator. Then the following assertions hold: \\
$(i)$ Suppose that $\lambda\notin\sigma(S_t)$, $t \in [0,1]$. Then the path of spectral projections $t\mapsto E_{S_t}((\lambda,\infty))$ is norm continuous $(cf.\ $\cite[Lemma 3.3]{Le05}$)$.  \\
$(ii)$ Assume that $\lambda\notin\sigma_{\text{ess}}(S_t)$, $t \in [0,1]$. Then the difference of the spectral projections $E_{S_t}([\lambda,\infty))-E_{S_0}([\lambda,\infty))$ is a compact operator $(cf.\ $\cite[Corollary 3.5]{Le05}$)$.  \\
$(iii)$ The pair of spectral projections $(E_{S_1}([0,\infty)),E_{S_0}([0,\infty)))$
is Fredholm and 
\begin{equation} 
\text{\em SpFlow} (\{S_t\}_{t=0}^1)=\ind(E_{S_1}([0,\infty)),E_{S_0}([0,\infty)))
\end{equation} 
$(cf.\ $\cite[Theorem 3.6]{Le05}$)$.
\end{proposition} 

Assuming Hypothesis \ref{h2.1} and $0\in\rho(A_-)\cap\rho(A_+)$, we are now ready to proceed with the main result of this section. Its proof uses $d_R$-continuity of the family $\{A(t)\}_{t=-\infty}^\infty$ since it requires the norm continuity in $t$
of the spectral projections $E_{A(t)}([0,\infty))$ when $0\notin\sigma(A(t))$. This is in contrast to the definition of the spectral flow which requires $d_G$-continuity yielding the norm continuity of 
$E_{A(t)}([0,\e))$, $\varepsilon > 0$, for just a {\em finite} $\e\notin\sigma(A(t))$.

The spectral projections $E_{A_+}((-\infty,0))$ and $E_{A_-}((-\infty,0))$ are called Morse projections. We recall that by \eqref{2.43a} in Corollary \ref{c2.11}  the difference  $E_{A_-}((-\infty,0))-E_{A_+}((-\infty,0))$ of the Morse projections is of trace class.
 We introduce the notation $\cS_\pm=\ran (E_{A_\pm}((-\infty,0)))$ for the ranges of the Morse projections.
 
\begin{theorem}\label{thSPFI}
Assume Hypothesis \ref{h2.1} and suppose that $0 \in \rho(A_+)\cap\rho(A_-)$. 
Then the pair $\big(E_{A_+}((-\infty,0)),E_{A_-}((-\infty,0))\big)$ of the Morse projections is Fredholm, the pair of subspaces $(\cS_+,\cS_-)$ is commensurable,
the pair of subspaces $(\cS_+,\cS_-^\bot)$ is Fredholm, 
and  the following equalities hold:
\begin{align}
\text{\rm SpFlow} (\{A(t)\}_{t=-\infty}^\infty) & =\ind(E_{A_-}((-\infty,0)),E_{A_+}((-\infty,0)))\label{spf1}\\
&=\ind(\cS_-,\cS_+^\bot)=
\dim (\cS_-\cap \cS_+^\bot)-\dim(\cS_-^\bot\cap \cS_+)\label{spf1a}\\&
={\tr}_{\cH}(E_{A_-}((-\infty,0))-E_{A_+}((-\infty,0)))\label{spf2} \\
&=\xi(0;A_+,A_-) \label{spf3}\\
& = \xi(0_+; \bsH_2, \bsH_1)    \lb{spf3A} \\ 
&=\ind (\bsD_\bsA^{}).    \label{spf4}
\end{align}
\end{theorem}
\begin{proof}
All assertions about the Fredholm properties of the pairs of projections and subspaces follow from Remark \ref{rem:term}\,$(ii)$ and Remark \ref{rem:term2},  
using compactness of the difference of the Morse projections in 
\eqref{2.43a}. Equality \eqref{spf4} of the Fredholm index of $\bsD_\bsA^{}$ and the $\xi$-function $\xi(0;A_+,A_-)$ is one of the main results of this paper; it is contained in Corollary \ref{c8.index}. Similarly, equality of \eqref{spf4} and \eqref{spf3A} is proved 
in Corollary \ref{c8.index}.  
Equality of \eqref{spf2} and \eqref{spf3} of the trace and the $\xi$-function is proved in Lemma \ref{l7.trxi}. Equality of the trace \eqref{spf2} and the index of the pair of the Morse projections holds due to \eqref{2.43a} by 
Remark \ref{rem:term}\,$(iii)$. Equality \eqref{spf1a} holds by Remark \ref{rem:term2}.

It remains to prove equality \eqref{spf1}. In fact, the main step in its proof is an application of 
\cite[Theorem 3.6]{Le05} as recorded in Proposition \ref{prop:lesch}\,$(iii)$ above.  First, we recall that $T_0$ is chosen as in Remark \ref{rem:pril}. By Lemma \ref{lem_cont},  the path
$\{A(t)\}_{t=-T_0}^{T_0}$ of self-adjoint Fredholm operators is $d_R$-continuous. Moreover, 
it follows from
Hypothesis \ref{h2.1} that the domain of $A(t)$ does not depend on $t$, and the difference $A(T_0)-A(-T_0)$ is $A(-T_0)$-compact. Indeed, the operator
\begin{align}
(A(T_0)-A(-T_0))(A(-T_0))^{-1}=\int_{-T_0}^{T_0}B'(s)(|A_-|+I)^{-1}\,ds\,
(|A_-|+I)(A(-T_0))^{-1}    \lb{8.28}
\end{align}
is compact since the integral on the right-hand side of \eqref{8.28} is a trace class operator 
by \eqref{2.12a} and
$(|A_-|+I)(A(-T_0))^{-1}\in\cB(\cH)$ (see \eqref{UNn1}). Thus, the assumptions of 
Proposition \ref{prop:lesch} are satisfied for $S_t=A(t)$. By
Proposition \ref{prop:lesch}\,$(iii)$, the pair of projections
$(E_{A(T_0)}([0,\infty)), E_{A(-T_0)}([0,\infty)))$ is Fredholm, and
\beq\label{spfT0} \text{SpFlow} (\{A(t)\}_{t=-T_0}^{T_0})=\ind(E_{A(T_0)}([0,\infty)), E_{A(-T_0)}([0,\infty))).
\enq
According to Definition \ref{defSPF}, one has $\text{SpFlow} (\{A(t)\}_{t=-T_0}^{T_0})=\text{SpFlow} (\{A(t)\}_{t=-\infty}^{\infty})$, and thus it remains to show that 
\beq\label{tails}
\ind(E_{A(T_0)}([0,\infty)), E_{A(-T_0)}([0,\infty)))=\ind(E_{A_-}((-\infty,0)),E_{A_+}((-\infty,0))).
\enq
For each $t\ge T_0$, the difference of the projections
$E_{A(t)}([0,\infty)) - E_{A(-t)}([0,\infty))$ is compact by Proposition \ref{prop:lesch}\,$(ii)$ and thus the pair $(E_{A(t)}([0,\infty)), E_{A(-t)}([0,\infty)))$ is Fredholm by Remark \ref{rem:term}\,$(ii)$. Since 
$0\notin\sigma(A(t))$ for $|t|\ge T_0$, one infers that $E_{A(t)}([0,\infty))=E_{A(t)}((0,\infty))$. Since $A(t)$ is $d_R$-continuous
and $d_R(A(t),A_\pm)\to0$ as $t\to\pm\infty$ by Lemma \ref{lem_cont},  
the function $\bbR\ni t\mapsto E_{A(t)}((0,\infty))$ is norm continuous and $\|E_{A(t)}((0,\infty))-E_{A_\pm}((0,\infty))\|_{\cB(\cH)}\to0$ as $t\to\pm\infty$ by Proposition  \ref{prop:lesch}\,$(i)$. The index of a norm-continuous family of Fredholm pairs of projections is constant (see, e.g., \cite[Lemma 3.2]{Le05}), and thus, if $t\ge T_0$, then
\beq \ind(E_{A(t)}([0,\infty)), E_{A(-t)}([0,\infty)))=\ind(E_{A_+}((0,\infty)), E_{A_-}((0,\infty))), \enq
yielding \eqref{tails} by Remark \ref{rem:term}\,$(i)$.
\end{proof}
Finally, we note that if both subspaces $\cS_+$ and $\cS_-$ are finite-dimensional then formulas 
\eqref{spf1}, \eqref{spf1a}, \eqref{spf2} become the well-known formula in finite-dimensional Morse theory (see, e.g., \cite{Ab01,AM03,RS95} and the much earlier literature cited therein):
\begin{equation}\label{indppr2}\begin{split}
\ind (\bsD_\bsA^{}) &=\ind(E_{A_-}((-\infty,0)),E_{A_+}((-\infty,0)))\\
&=\dim (\cS_+) - \dim (\cS_-), \quad \dim(\cS_\pm) <\infty.\end{split} 
\end{equation}

\appendix
\section{Some Facts on Direct Integrals of Closed Operators}
\lb{sA}
\renewcommand{\theequation}{A.\arabic{equation}}
\renewcommand{\thetheorem}{A.\arabic{theorem}}
\setcounter{theorem}{0} \setcounter{equation}{0}

We briefly recall some basic facts on closed operators and their graphs discussed 
in detail in Stone's fundamental paper \cite{St51} and then review some of its 
consequences for direct integrals of (unbounded) closed operators as developed in 
Nussbaum \cite{Nu64} (see also Pallu de la Barri\`ere \cite{Pa51}). For a detailed 
treatment of some of the material in this appendix we refer to \cite{GGST10}.

For simplicity, we make the following assumption:

\begin{hypothesis} \lb{hA.-1}
Let $\cH$ be a complex separable Hilbert space and $T$ a densely defined, closed,  
linear operator in $\cH$. 
\end{hypothesis}

We note that Stone \cite{St51} considers a more general situation, but Hypothesis 
\ref{hA.-1} perfectly fits the purpose of our paper.

By $\Gamma (T)$ we denote the graph of 
$T$, that is, the following subspace of the direct sum $\cH \oplus \cH$,
\begin{equation}
\Gamma (T) = \{\langle f, Tf\rangle \,|\, f \in \dom(T)\} \subseteq \cH \oplus \cH.  \lb{A.1}
\end{equation}  
Since $T$ is assumed to be closed, $\Gamma (T)$ is a closed subspace of 
$\cH \oplus \cH$. Here $\langle f, g \rangle$ denotes the ordered pair of $f, g \in \cH$,  and we use the standard norm
\begin{equation}
\|\langle f, g\rangle\|_{\cH\oplus \cH} = \big[\|f\|^2_{\cH} + \|g\|^2_{\cH}\big]^{1/2},   
\quad f, g \in \cH,    \lb{A.2}
\end{equation}
and scalar product 
\begin{equation}
(\langle f_1, g_1\rangle, \langle f_2, g_2\rangle)_{\cH\oplus \cH} 
= (f_1,f_2)_{\cH} + (g_1,g_2)_{\cH},   \quad f_j, g_j \in \cH, \; j=1,2,     \lb{A.3}
\end{equation}
in $\cH\oplus\cH$. 

If $B \in \cB(\cH\oplus\cH)$, one can uniquely represent $B$ as the $2\times 2$ block operator matrix
\begin{equation}
B = \begin{pmatrix} B_{1,1} & B_{1,2} \\ B_{2,1} & B_{2,2} \end{pmatrix},  \lb{A.4}
\end{equation} 
where $B_{j,k} \in \cB(\cH)$, $j, k \in \{1,2\}$. 

Denoting by 
\begin{equation}
P(\Gamma(T)) =  \begin{pmatrix} P(\Gamma(T))_{1,1} & P(\Gamma(T))_{1,2} \\ 
P(\Gamma(T))_{2,1} & P(\Gamma(T))_{2,2} \end{pmatrix}    \lb{A.5}
\end{equation}
the orthogonal projection onto $\Gamma (T)$, the corresponding matrix 
$(P(\Gamma(T))_{j,k})_{1\leq j,k \leq 2}$ will be called the {\it characteristic matrix} 
of $T$. Since by hypothesis $T$ is closed and densely defined, one actually obtains 
(cf.\ \cite{St51})
\begin{align}
\begin{split} 
& P(\Gamma(T))_{1,1} = (T^* T + I)^{-1},    \\ 
& P(\Gamma(T))_{1,2} = T^* (T T^* + I)^{-1},   \\ 
& P(\Gamma(T))_{2,1} = T (T^* T + I)^{-1} = (P(\Gamma(T))_{1,2})^*,   \lb{A.14}  \\
& P(\Gamma(T))_{2,2} = T T^* (T T^* + I)^{-1} = I - (T T^* + I)^{-1}. 
\end{split}  
\end{align} 

Next, we turn to families of densely defined, closed operators $\{T(t)\}_{t\in\bbR}$ 
in $\cH$ and use the following assumption for the remainder of this appendix:

\begin{hypothesis} \lb{hA.0}
Let $T(t)$, $t\in\bbR$, be densely defined, closed, linear operators in $\cH$. 
\end{hypothesis}

We need the following notions of measurable vector and operator families:

\begin{definition} \lb{dA.1}
$(i)$ Let $\bbR \ni t \mapsto g(t)\in \cH$. Then the family $\{g(t)\}_{t\in\bbR}$ 
is called {\it weakly measurable} in $\cH$ if $\bbR \ni t \mapsto (h, g(t))_{\cH}$ 
is $($Lebesgue$)$ measurable for each $h \in \cH$. \\ 
Next, assume Hypothesis \ref{hA.0}: \\
$(ii)$ The family $\{T(t)\}_{t\in\bbR}$ is called {\it weakly measurable} if for any 
weakly measurable family of elements $\{f(t)\}_{t\in\bbR}$ in $\cH$ such that 
$f(t) \in \dom(T(t))$ for all $t\in\bbR$, the family of elements $\{T(t) f(t)\}_{t\in\bbR}$ is weakly measurable in $\cH$. \\
$(iii)$ The family $\{T(t)\}_{t\in\bbR}$ is called {\it $N$-measurable} if the entries of the characteristic matrix of $T(t)$ are weakly measurable, that is, if 
$\big\{P(\Gamma(T(t)))_{j,k}\big\}_{t\in\bbR}$, $j,k \in \{1, 2\}$, are weakly measurable.  
\end{definition} 

We note that measurability of the characteristic matrix 
$(P(\Gamma(T(\cdot)))_{j,k})_{1\leq j,k \leq 2}$ of $T(\cdot)$ was introduced by 
Nussbaum \cite{Nu64}. In fact, he considered the more general situation of a general measure $d\mu$ and a $\mu$-measurable family of Hilbert spaces $\{\cH(t)\}_{t\in\bbR}$.   

We refer to \cite{Nu64} for more details in connection with items $(ii)$--$(iv)$ in 
Remark \ref{rA.2} below: 

\begin{remark}  \lb{rA.2} 
$(i)$ Since $\cH$ is assumed to be separable,
weak measurability of the family  $\{g(t)\}_{t\in\bbR}$ in $\cH$
is equivalent to measurability, that is, there exists a sequence of
countably-valued elements $\{g_n(t)\}_{t\in\bbR} \subset \cH$,
$n\in\bbN$, and a set $\cE \subset \bbR$ of Lebesgue measure zero
such that $\lim_{n\to\infty} \|g_n(t) - g(t)\|_{\cH} =0$ for each
$t\in \bbR\backslash \cE$. Thus, the family $\{g(t)\}_{t\in\bbR}$ is 
(weakly) measurable in $\cH$ if there exists a dense set $\cY \subset \cH$
such that the function $(y,g(\cdot))_{\cH}$ is measurable for every $y \in \cY$, see, 
for instance, \cite[Corollary\ 1.1.3]{ABHN01}, \cite[p.\ 42--43]{DU77}. Moreover,
\begin{equation}
f,g: \bbR \mapsto \cH \, \text{ measurable } \,
\Longrightarrow \, (f(\cdot), g(\cdot))_{\cH} \, \text{ is measurable}.   \lb{A.14aa}
\end{equation}
$(ii)$ If $\cH_1, \cH_2, \cH_3$ are complex, separable Hilbert spaces and 
$F: \bbR \mapsto \cB(\cH_1, \cH_2)$ and $G: \bbR \mapsto \cB(\cH_2, H_3)$
are strongly measurable, then $G \, F : \bbR \mapsto \cB(\cH_1, \cH_3)$ is strongly measurable, see, for instance, \cite[Lemma A4]{Ka73a}. (Here strong (operator) measurability of $F: \bbR \mapsto \cB(\cH_1, \cH_2)$ is defined pointwise, i.e., 
for all $f\in\cH_1$, $\{F(t) f\}_{t\in\bbR}$ is (weakly) measurable in $\cH_2$.)  \\ 
$(iii)$ One can show that 
\begin{equation}
\text{$N$-measurability of $\{T(t)\}_{t\in\bbR}$ $\Longrightarrow$ weak measurability of 
$\{T(t)\}_{t\in\bbR}$},   \lb{A.14a}
\end{equation} 
but the converse is false. For an example of a weakly measurable family of symmetric operators which is not $N$-measurable, we refer to Example \ref{eA.5} below. \\
$(iv)$ Since $P(\Gamma(T(t)))_{2,1} = (P(\Gamma(T(t)))_{1,2})^*$, or 
equivalently, since 
\begin{align}
\begin{split}
\big[T(t) \big(T(t)^* T(t) + I\big)^{-1}\big]^* &= T(t)^* (T(t) T(t)^* + I)^{-1} \\ 
& \supseteq  (T(t)^* T(t) + I)^{-1} T(t)^*,  
\end{split} 
\end{align}
as $T(t)$ is closed in $\cH$, weak measurability of 
$\{P(\Gamma(T(t)))_{1,2}\}_{t\in\bbR}$ is equivalent to that of 
$\{P(\Gamma(T(t)))_{2,1}\}_{t\in\bbR}$. 
Thus, by \eqref{A.14}, 
\begin{align}
& \text{$N$-measurability of $\{T(t)\}_{t\in\bbR}$ is equivalent to weak measurability of}  \no  \\
& \quad \big\{\big(|T(t)|^2 + I\big)^{-1}\big\}_{t\in\bbR}, 
\quad \big\{T(t) \big(|T(t)|^2 + I\big)^{-1}\big\}_{t\in\bbR},    \lb{A.14A} \\
& \quad \text{and } \, \big\{\big(|T(t)^*|^2 + I\big)^{-1}\big\}_{t\in\bbR}.    \no 
\end{align} 
\end{remark}

\begin{example}  [\cite{GGST10}] \lb{eA.5}
Let $T_0$ and $T_1$ be densely defined, closed, unbounded, symmetric operators in $\cH$ satisfying
\begin{equation}
T_0 \subsetneq T_1.   \lb{4.2}
\end{equation}
Let $\gE\subset \bbR$ be a nonmeasurable subset of $\bbR$ (in the sense of Lebesgue measure) and introduce the linear operators 
\begin{equation}
\wti T(t) = \begin{cases} T_0, & t \in\gE,  \\
T_1, & t \in \bbR\backslash\gE,       
\end{cases}      \lb{4.3}
\end{equation} 
in $\cH$. Then the family $\big\{\wti T(t)\big\}_{t\in\bbR}$ is weakly measurable, but not 
$N$-measurable.
\end{example}

\smallskip 
The Hilbert space $L^2(\bbR; dt; \cH)$, in short, $L^2(\bbR; \cH)$, consists of equivalence classes $f$ of (weakly) Lebesgue measurable 
$\cH$-valued elements $f(\cdot)\in\cH$ (whose elements are 
equal a.e.\ on $\bbR$), such that $\|f(\cdot)\|_{\cH} \in L^2(\bbR; dt)$. The 
norm and scalar product on $L^2(\bbR; \cH)$ are then given by
\beq
  \|f\|_{L^2(\bbR; \cH)}^2 = \int_{\bbR}  \|f(t)\|_{\cH}^2 \, dt, \;\;
  (f,g)_{L^2(\bbR; \cH)} = \int_{\bbR}  (f(t), g(t))_{\cH} \, dt, \;\;
  f, g \in L^2(\bbR; \cH).   \lb{A.15} 
\enq

Of course, $L^2(\bbR; \cH)$ can be identified with the constant fiber 
direct integral $\int_{\bbR}^{\oplus} \cH \, dt$, that is,
\beq
L^2(\bbR; \cH) = \int_{\bbR}^{\oplus} \cH \, dt.    \lb{A.16} 
\enq

Throughout the rest of this appendix, operators denoted by a calligraphic boldface letter such as $\boldsymbol{\cS}$ in the 
Hilbert space $L^2(\bbR;\cH)$ represent operators associated with a 
family of operators $\{S(t)\}_{t\in\bbR}$ in $\cH$, defined by
\begin{align}
&(\boldsymbol{\cS} f)(t) = S(t) f(t) \, \text{ for a.e.\ $t\in\bbR$,}    \no \\
& f \in \dom(\boldsymbol{\cS}) = \bigg\{g \in L^2(\bbR;\cH) \,\bigg|\,
g(t)\in \dom(S(t)) \text{ for a.e.\ } t\in\bbR,    \lb{A.17}  \\
& \quad t \mapsto S(t)g(t) \text{ is (weakly) measurable,} \, 
\int_{\bbR} \|S(t) g(t)\|_{\cH}^2 \, dt < 
\infty\bigg\}.   \no
\end{align}

Assuming Hypothesis \ref{hA.0}, we note that $\boldsymbol{\cT}$, defined according to \eqref{A.17}, with $T(t)$ satisfying Hypothesis \ref{hA.0}, is closed in $L^2(\bbR; \cH)$ since $T(t)$, $t\in\bbR$, are closed in $\cH$ (but $\boldsymbol{\cT}$ might not be densely defined). If in addition, the family $\{T(t)\}_{t\in\bbR}$ is $N$-measurable, then $\boldsymbol{\cT}$ is called {\it decomposable} in 
$L^2(\bbR; \cH) = \int_{\bbR}^{\oplus} \cH \, dt$ and also denoted by the direct integral of the family $\{T(t)\}_{t\in\bbR}$ over $\bbR$ with respect to Lebesgue measure, 
\begin{equation}
\bsT = \int_{\bbR}^{\oplus} T(t) \, dt.   \lb{A.18}
\end{equation}
In this case, one also has 
\begin{equation}
\bsP(\Gamma(\bsT))_{j,k} = \int_{\bbR}^{\oplus} P(\Gamma(T(t)))_{j,k} \, dt, \quad 
j,k \in \{1,2\}.   \lb{A.19}
\end{equation} 

If $T(t) \in \cB(\cH)$, $t\in\bbR$, then 
\begin{equation}
\boldsymbol{\cT} \in \cB(L^2(\bbR; \cH)) \Longleftrightarrow {\rm esssup}_{t\in\bbR} \|T(t)\|_{\cB(\cH)} 
< \infty, 
\end{equation}
in particular, if $\boldsymbol{\cT} \in \cB(L^2(\bbR; \cH))$, then 
\begin{equation}
\|\boldsymbol{\cT} \|_{\cB(L^2(\bbR; \cH))} = {\rm esssup}_{t\in\bbR} \|T(t)\|_{\cB(\cH)}.  
\end{equation}

We recall the following results of Nussbaum \cite{Nu64} (in fact, he deals with the more general situation where the constant fiber space $\cH$ is replaced by a measurable family of Hilbert spaces $\{\cH(t)\}_{t\in\bbR}$):

\begin{lemma} [Nussbaum \cite{Nu64}]  \lb{lA.5} 
Assume Hypothesis \ref{hA.0} and suppose in addition that the family 
$\{T(t)\}_{t\in\bbR}$ is weakly measurable. Define $\boldsymbol{\cT}$ according 
to \eqref{A.17},
\begin{align}
&(\boldsymbol{\cT} f)(t) = T(t) f(t) \, \text{ for a.e.\ $t\in\bbR$,}    \no \\
& f \in \dom(\boldsymbol{\cT}) = \bigg\{g \in L^2(\bbR;\cH) \,\bigg|\,
g(t)\in \dom(T(t)) \text{ for a.e.\ } t\in\bbR,    \lb{A.19a}  \\
& \quad t \mapsto T(t)g(t) \text{ is $($weakly\,$)$ measurable,} \, 
\int_{\bbR} \|T(t) g(t)\|_{\cH}^2 \, dt < 
\infty\bigg\}.   \no
\end{align}
Then $\boldsymbol{\cT}$ is a closed, decomposable operator in 
$L^2(\bbR; \cH) = \int_{\bbR}^{\oplus} \cH \, dt$. Thus, there exists an $N$-measurable 
family of closed operators $\big\{\hatt T(t)\big\}_{t\in\bbR}$ in $\cH$ such that
\begin{equation}
\boldsymbol{\cT} = \int_{\bbR}^{\oplus} \hatt T(t) \, dt
\end{equation}
and
\begin{equation}
\hatt T(t) \subseteq T(t) \, \text{ for a.e.\ $t\in\bbR$.}
\end{equation}
\end{lemma}

We note that in general $\boldsymbol{\cT}$ is not densely defined in $L^2(\bbR; \cH)$ (cf.\ \cite{GGST10}). 

\begin{theorem} [Nussbaum \cite{Nu64}]  \lb{tA.6} 
Assume Hypothesis \ref{hA.0} and suppose in addition that the family 
$\{T(t)\}_{t\in\bbR}$ is $N$-measurable. Then the following assertions hold: \\
$(i)$ $\bsT = \int_{\bbR}^{\oplus} T(t) \, dt$ is densely defined and closed in 
$L^2(\bbR; \cH) = \int_{\bbR}^{\oplus} \cH \, dt$ and  
\begin{equation}
\bsT^* = \int_{\bbR}^{\oplus} T(t)^* \, dt, \quad 
|\bsT| = \int_{\bbR}^{\oplus} |T(t)| \, dt.      \lb{A.20}
\end{equation}
$(ii)$ $\bsT$ is symmetric $($resp., self-adjoint, or normal\,$)$ if and only if $T(t)$ is 
symmetric $($resp., self-adjoint, or normal\,$)$ for a.e.\ $t\in\bbR$. \\
$(iii)$ $\ker(\bsT) = \{0\}$ if and only if $\ker(T(t)) = \{0\}$ for a.e.\ $t\in\bbR$. In addition, 
if $\ker(\bsT) = \{0\}$ then $\big\{T(t)^{-1}\big\}_{t\in\bbR}$ is $N$-measurable and 
\begin{equation}
\bsT^{-1} = \int_{\bbR}^{\oplus} T(t)^{-1} \, dt.    \lb{A.21}
\end{equation}
$(iv)$ If $\bsT$ is self-adjoint in $L^2(\bbR; \cH)$, then $\bsT \geq 0$ if and only 
if $T(t) \geq 0$ for a.e.\ $t\in\bbR$. \\
$(v)$ If $\bsT$ is normal in $L^2(\bbR; \cH)$, then 
\begin{equation}
p(\bsT) = \int_{\bbR}^{\oplus} p(T(t)) \, dt    \lb{A.22} 
\end{equation}
for any polynomial $p$. \\
$(vi)$ Let $S(t)$, $t\in\bbR$, be densely defined, closed operators in $\cH$ and assume 
that the family $\{S(t)\}_{t\in\bbR}$ is $N$-measurable and 
$\bsS = \int_{\bbR}^{\oplus} S(t)\, dt$. Then $\bsT \subseteq \bsS$ if and only if 
$T(t) \subseteq S(t)$ for a.e.\ $t\in\bbR$.
\end{theorem} 

Since $N$-measurability is a crucial hypothesis in Theorem \ref{tA.6}, we emphasize 
Remark \ref{rA.2}\,$(iv)$ which represents necessary and sufficient conditions which 
seem verifiable in practical situations. In addition, we note the following result: 

\begin{lemma} \lb{lA.6} 
Assume Hypothesis \ref{hA.0} and suppose that 
\begin{equation} 
\{T(t)\}_{t\in\bbR}, \quad \big\{\big(|T(t)|^2 + I)^{-1}\big\}_{t\in\bbR}, \, \text{ and } \, 
\big\{T(t)\big(|T(t)|^2+I\big)^{-1}\big\}_{t\in\bbR}
\end{equation} 
are weakly measurable. Then $\{T(t)\}_{t\in\bbR}$ is $N$-measurable. 
\end{lemma}
\begin{proof}
Since $T(t)\big(|T(t)|^2+I\big)^{-1} \in \cB(\cH)$, $t\in\bbR$, and 
\begin{equation}
\big(T(t) \big(|T(t)|^2+I\big)^{-1}\big)^* = T(t)^* \big(|T(t)^*|^2+I\big)^{-1}, 
\quad t\in\bbR,
\end{equation}
one concludes that $\big\{T(t)^* \big(|T(t)^*|^2+I\big)^{-1}\big\}_{t\in\bbR}$ is weakly measurable too. Thus, for each $g \in \cH$, 
$\big\{T(t)^* \big(|T(t)^*|^2+I\big)^{-1} g\big\}_{t\in\bbR}$ is (weakly)  
measurable in $\cH$, in addition, $T(t)^* \big(|T(t)^*|^2+I\big)^{-1} g \in \dom (T(t))$ for all $t\in\bbR$. Since 
$\{T(t)\}_{t\in\bbR}$ is weakly measurable, one thus concludes that 
\begin{equation}
\big\{T(t) T(t)^* \big(|T(t)^*|^2+I\big)^{-1}\big\}_{t\in\bbR} 
= \big\{I - \big(|T(t)^*|^2+I\big)^{-1}\big\}_{t\in\bbR}, 
\end{equation}
and hence $\big\{\big(|T(t)^*|^2+I\big)^{-1}\big\}_{t\in\bbR}$, is weakly measurable as well. 
 \end{proof}

Next, we recall a result due to Lennon \cite{Le74} on sums and products of decomposable operators (actually, Lennon considers a slightly more general situation). We use the usual conventions that if $A$ and $B$ are linear operators in $\cH$ then 
\begin{equation}
\dom(A+B) = \dom(A) \cap \dom(B)
\end{equation}
and 
\begin{equation}
\dom(AB) = \{f \in \dom(B) \,|\, Bf \in \dom(A)\}. 
\end{equation}

\begin{theorem} [Lennon \cite{Le74}]  \lb{lA.7} 
Let $\bsA = \int_{\bbR}^{\oplus} A(t)\, dt$ and $\bsB = \int_{\bbR}^{\oplus} B(t)\, dt$ 
be closed decomposable operators in 
$L^2(\bbR; \cH) = \int_{\bbR}^{\oplus} \cH \, dt$ with the 
$N$-measurable families $\{A(t)\}_{t\in\bbR}$ and $\{B(t)\}_{t\in\bbR}$ in $\cH$ 
satisfying Hypothesis \ref{hA.0}. Then the following holds: \\
$(i)$ $\dom(\bsA + \bsB)$ is dense in $L^2(\bbR; \cH)$ if and only if 
$\dom(A(t)\cap B(t))$ is dense in $\cH$ for a.e.\ $t\in\bbR$. In addition, 
$\bsA + \bsB$ is closable  in $L^2(\bbR; \cH)$ if and only if $A(t) + B(t))$ is 
closable in $\cH$ for a.e.\ $t\in\bbR$. In this case the family 
$\big\{\ol{[A(t) + B(t)]}\big\}_{t\in\bbR}$ is $N$-measurable and 
\begin{equation}
\ol{\bsA + \bsB} = \int_{\bbR}^{\oplus} \ol{[A(t) + B(t)]} \, dt. 
\end{equation}
$(ii)$ $\dom(\bsA \bsB)$ is dense in $L^2(\bbR; \cH)$ if and only if $\dom(A(t) B(t))$ 
is dense in $\cH$ for a.e.\ $t\in\bbR$. In addition, $\bsA \bsB$ is closable  in 
$L^2(\bbR; \cH)$ if and only if $A(t) B(t))$ is closable in $\cH$ for a.e.\ $t\in\bbR$. In this case the family 
$\big\{\ol{[A(t) B(t)]}\big\}_{t\in\bbR}$ is $N$-measurable and 
\begin{equation}
\ol{\bsA \bsB} = \int_{\bbR}^{\oplus} \ol{[A(t) B(t)]} \, dt.    \lb{A.43}
\end{equation}
\end{theorem}
\begin{lemma} \lb{l2.9}
Assume Hypotheses \ref{h2.1}. Then  
\begin{equation}
\{B(t)\}_{t\in\bbR}, \quad \{B(t)^*\}_{t\in\bbR}, \quad \{B'(t)\}_{t\in\bbR},  
\quad \{(B'(t))^*\}_{t\in\bbR},    \lb{2.37B}
\end{equation}
as well as 
\begin{align}
\begin{split}
& \big\{B(t) \big(|B(t)|^2 + I\big)^{-1}\big\}_{t\in\bbR},  \quad 
\big\{B'(t) \big(|B'(t)|^2 + I\big)^{-1}\big\}_{t\in\bbR},  \lb{2.37A} \\ 
& \big\{\big(|B(t)^*|^2 + I\big)^{-1}\big\}_{t\in\bbR}, \quad 
\big\{\big(|(B'(t))^*|^2 + I\big)^{-1}\big\}_{t\in\bbR},
\end{split}
\end{align} 
are weakly measurable. In particular, \eqref{2.13B} and \eqref{2.37A} 
together imply that $\{B(t)\}_{t\in\bbR}$ and $\{B'(t)\}_{t\in\bbR}$ are 
$N$-measurable. Consequently, $\bsB$ and $\bsB'$, defined according to \eqref{2.28},  
are densely defined in $L^2(\bbR; \cH)$, and the analogs of \eqref{2.29} hold in either case. 
\end{lemma}
\begin{proof} 
Fix a (weakly) measurable family of elements $\{f(t)\}_{t\in\bbR}$ in
$\cH$ such that $f(t) \in \dom(B(t))$ for a.e.\ $t\in\bbR$. By
Hypothesis \ref{h2.1}\,$(ii)$, for every $g \in \dom(|A_{-}|)$, 
\begin{equation}
(g,B(\cdot)f(\cdot))_{\cH} = (B(\cdot)g, f(\cdot))_{\cH},
\end{equation}
where $\{B(t)g\}_{t\in\bbR}$ (as well as $\{f(t)\}_{t\in\bbR}$) is weakly measurable 
and hence measurable in $\cH$. By \eqref{A.14aa}, the function 
$(f(\cdot), B(\cdot)g)_{\cH}$ is measurable.
Since $\dom (|A_-|)$ is dense, $\{B(t)f(t)\}_{t\in\bbR}$ is
measurable in $\cH$ by Remark \ref{rA.2}\,$(i)$. Thus $\{B(t)\}_{t \in\bbR}$
is weakly measurable. Using \eqref{2.10a}, one similarly infers that 
$\{B'(t)\}_{t\in\bbR}$ is weakly measurable. Utilizing Remark \ref{r2.6}, one then 
also concludes that $\{B(t)^*\}_{t\in\bbR}$ and $\{(B'(t))^*\}_{t\in\bbR}$ are weakly measurable, proving \eqref{2.37B}.

Next, we invoke the fact that $\big\{\big(|B(t)|^2+I\big)^{-1}\big\}_{t\in\bbR}$
is  assumed to be weakly measurable by Hypothesis \ref{h2.1}\,$(v)$: As above, for a (weakly) measurable family
of elements $\{f(t)\}_{t\in\bbR}$ in $\cH$ such that $f(t) \in \dom(B(t))$ 
for a.e.\ $t\in\bbR$, and for every $g \in \dom(|A_{-}|)$, the function
\begin{equation}
\big(B(\cdot) \big(|B(\cdot)|^2+I\big)^{-1}f(\cdot), g\big)_{\cH}  
= \big(\big(|B(\cdot)|^2+I\big)^{-1}f(\cdot), B(\cdot) g\big)_{\cH}
\end{equation}
is  measurable since $\big\{\big(|B(t)|^2+I\big)^{-1}f(t)\big\}_{t\in\bbR}$ and
$\{B(t) g\}_{t\in\bbR}$ are measurable in $\cH$. Since $\dom (|A_{-} |)$ is
dense, Remark \ref{rA.2}\,$(ii)$ implies that 
$\big\{B(t)\big(|B(t)|^2+I\big)^{-1}\big\}_{t\in\bbR}$ is weakly measurable. Similarly one proves 
the weak measurability of the family $\big\{B'(t)\big(|B'(t)|^2+I\big)^{-1}\big\}_{t\in\bbR}$. 

Weak measurability of $\big\{\big(|B(t)^*|^2+I\big)^{-1}\big\}_{t\in\bbR}$ then follows from Lemma \ref{lA.6}; the weak measurability of the family 
$\big\{\big(|(B'(t))^*|^2 + I\big)^{-1}\big\}_{t\in\bbR}$ is proved analogously, completing 
the proof of \eqref{2.37A}. 

$N$-measurability of $\{B(t)\}_{t\in\bbR}$ and $\{B'(t)\}_{t\in\bbR}$ then follows from 
\eqref{A.14A}.

Finally, that $\bsB$ and $\bsB'$ are densely defined in $L^2(\bbR; \cH)$ and the 
analogs of \eqref{2.29} hold follows from Theorem \ref{tA.6}\,$(i)$. 
\end{proof}

Next, we show that Hypothesis \ref{h2.1}\,$(v)$ is essential, in particular, we will show 
that weak measurability of the family 
$\big\{\big(|B'(t)|^2 + I\big)^{-1}\big\}_{t\in\bbR}$ does not follow from weak 
measurability of $\{B'(t)\}_{t\in\bbR}$ and weak measurability of 
$\big\{B'(t)(|A_-| + I)^{-1}\big\}_{t\in\bbR}$.
For this purpose it suffices to consider the following example (a slight refinement of Example \ref{eA.5}):

\begin{example} \lb{e2.11}
Let $B_0$ and $B_1$ be densely defined, closed, unbounded, symmetric operators in $\cH$ satisfying
\begin{equation}
B_0 \subsetneq B_1   \lb{4.2A}
\end{equation}
and 
\begin{equation}
\dom(A_-) \subseteq \dom(B_0).    \lb{4.2B}
\end{equation}
Let $\gE\subset \bbR$ be a nonmeasurable subset of $\bbR$ (in the sense of Lebesgue measure) and introduce the linear operators 
\begin{equation}
\wti B(t) = \begin{cases} B_0, & t \in\gE,  \\
B_1, & t \in \bbR\backslash\gE,       
\end{cases}      \lb{4.3A}
\end{equation} 
in $\cH$. Then the family $\big\{\wti B(t)\big\}_{t\in\bbR}$ is weakly measurable, 
but not $N$-measurable, in particular, 
\begin{equation}
\big\{\big(\big|\wti B(t)\big|^2 + I\big)^{-1}\big\}_{t\in\bbR} \, 
\text{ is not weakly  measurable.}    \lb{4.3B}
\end{equation}
On the other hand, obviously, 
\begin{equation}
\wti B(t) (|A_-| + I)^{-1} = B_0 (|A_-| + I)^{-1}
\end{equation}
is $N$-measurable, in fact, even constant with respect to $t\in\bbR$. 
\end{example} 
\begin{proof}
Let $\{f(t)\}_{t\in\bbR}$ be a (weakly) measurable family of elements  in $\cH$ such that 
$f(t) \in \dom\big(\wti B(t)\big)$ for all $t\in\bbR$. Then, using the fact that 
\begin{equation}
B_0 \subset B_1 \subseteq B_1^* \subset B_0^*,     \lb{4.4}
\end{equation}
one concludes that 
\begin{equation}
\big(\wti B(t) f(t), g\big)_{\cH} = (f(t), B_0 g)_{\cH}, \quad t\in\bbR, \; g \in \dom(B_0),  
\lb{4.5}
\end{equation}
is measurable, and since $\dom(B_0)$ is dense in $\cH$, the family 
$\big\{\wti B(t)\big\}_{t\in\bbR}$ is weakly measurable by Remark \ref{rA.2}\,$(i)$. 

Since by hypothesis, $B_0 \subsetneq B_1$, $B_0^* B_0 \neq B_1^* B_1$, and hence there exists $0 \neq h \in \cH$ such that 
\begin{equation}
(h, (B_0^* B_0 + I)^{-1} h)_{\cH} 
\neq (h, (B_1^* B_1 + I)^{-1} h)_{\cH}.   \lb{4.6}
\end{equation}
Since nonmeasurability of $\gE$ is equivalent to nonmeasurability of its characteristic function $\chi_{\gE}$, one similarly infers that 
\begin{equation}
\big(h, \big(\big(\wti B(t)\big)^* \wti B(t) + I\big)^{-1} h\big)_{\cH} 
= \begin{cases}
(h, (B_0^* B_0 + I)^{-1} h)_{\cH}, & t \in \gE, \\
(h, (B_1^* B_1 + I)^{-1} h)_{\cH}, & t \in \bbR\backslash\gE,  
\end{cases}      \lb{4.7}
\end{equation}
is nonmeasurable, implying that the family $\big\{\wti B(t)\big\}_{t\in\bbR}$ is not 
$N$-measurable by \eqref{A.14A} and hence \eqref{4.3B} follows. 
\end{proof}

As another application of the notion of $N$-measurability we now conclude this 
appendix with an alternative proof of Lemma \ref{l2.3}\,$(iii)$, that is we reprove 
the fact that the operator $\bsD_{\bsA_-}^{}$ is normal in $L^2(\bbR; \cH)$: 

\begin{lemma} \lb{lA.10}
Suppose $A_-$ is self-adjoint  in $\cH$ on $\dom(A_-)\subseteq\cH$, 
and define the operator $\bsD_{\bsA_-}^{}$ as in \eqref{3.DA-1}. Then 
$\bsD_{\bsA_-}^{}$ is a normal $($and hence closed$)$ operator in 
$L^2(\bbR; \cH)$. 
\end{lemma}
\begin{proof} We start by considering the direct integral decomposition 
\begin{align}
& \, \widetilde{\bsD}_{\bsA_-}  
=\int_{\bbR}^{\oplus} D(t) \, dt,  \no \\
& 
\dom \big(\widetilde{\bsD}_{\bsA_-}\big) = \bigg\{g\in L^2(\bbR;\cH) \,\bigg|\,  
g(t)\in \dom(D(t)) \text{ for a.e.\ } t\in\bbR,   \lb{3.7til}
\\&
 \quad t \mapsto D(t) g(t) \text{ is (weakly) measurable,} \, 
\int_{\bbR} \|D(t)g(t)\|_{\cH}^2 \, dt < \infty \bigg\}   \no
\end{align}
in $L^2(\bbR; \cH)$. 
Here $\{D(t)\}_{t\in \bbR}$ is the family of normal operators in $\cH$
given by
\begin{equation}
D(t)f=it f +A_- f,\quad f  \in\dom(D(t))=\dom(A_-),  \; t\in \bbR.
\end{equation}
Next we show, that the family  $\{D(t)\}_{t\in \bbR}$ is $N$-measurable. Indeed, the 
orthogonal projection $P(D (t))$, $t\in \bbR$, in $\cH\oplus \cH$  onto the graph of 
the operator $D(t)$ is given by the $2 \times 2$ operator-valued matrix in 
$\cB(\cH) \oplus \cB(\cH)$, 
\begin{align}
& P(D(t))   \\
& \quad =\begin{pmatrix}
(A_-^2+(t^2+1)I_{\cH})^{-1}&(A_- - itI_{\cH})(A_-^2+(t^2+1)I_{\cH})^{-1}\\
(A_- +t^2I_{\cH})(A_- ^2+(t^2+1)I_{\cH})^{-1}&I_{\cH}-(A_-^2+(t^2+1)I_{\cH})^{-1}
\end{pmatrix}.    \no 
\end{align} 
The family $\{P(D (t))\}_{t\in \bbR}$ is a norm-continuous family of bounded operators 
and hence $\{P(D (t))\}_{t\in \bbR}$ is weakly measurable, which in turn proves
 that the family $\{D(t)\}_{t\in \bbR}$ is $N$-measurable. One observes that 
 $N$-measurabily of $\{D(t)\}_{t\in \bbR}$ implies
its weak measurability (cf. \eqref{A.14a}), and 
therefore, the requirement in \eqref{3.7til} that the map $t \mapsto D(t) g(t)$ 
is (weakly) measurable holds automatically and hence is redundant in this case. 
Combining Lemma \ref{lA.5} and Theorem \ref{tA.6}\,(ii), one concludes that 
the direct integral
\begin{equation}\label{dii}
\widetilde{\bsD}_{A_-}=\int_{\bbR}^{\oplus} D(t) \, dt, 
\end{equation} 
on the domain provided in \eqref{3.7til}, is a normal operator.  
 
Since $A_-$ is a self-adjoint operator, the following estimate holds, 
\begin{equation}
t^2\|f\|_{\cH}^2\le \|(A_-+itI_{\cH})f\|_{\cH}^2, \quad f \in \dom(A_-), \; t\in \bbR,
\end{equation}
and one concludes that  the requirement 
$\int_{\bbR} \|D(t) g(t)\|_{\cH}^2 \, dt < \infty$ in \eqref{3.7til} for
 $g\in L^2(\bbR, \cH)$
is equivalent to the conditions
\begin{equation}
\int_{\bbR} \|(1+t^2)g(t)\|_{\cH}^2 \, dt < \infty \, \text{ and } \, 
\int_{\bbR} \|A_-g(t)\|_{\cH}^2 \, dt < \infty, 
\end{equation}
and thus to 
\begin{equation}\label{dominter}
\dom \big(\wti{\bsD}_{\bsA_-}\big)=\dom (it\bsI)\cap \dom (\bsA_-).
\end{equation} 
Thus,  $\widetilde{\bsD}_{\bsA_-}$ on \eqref{dominter} is a normal operator.
Here, in obvious notation, $it \, \bsI$ denotes the maximally 
defined operator of multiplication by $it$ in $L^2(\bbR; \cH)$ with 
domain
\beq
\dom(it \, \bsI) = \bigg\{g\in L^2(\bbR; \cH) \,\bigg|\,
\int_{\bbR} (1+ t^2) \|g(t)\|_{\cH}^2 \, dt < \infty \bigg\}. 
\enq
Applying the unitary vector-valued Fourier transform $\gF_{\cH}$ (cf.\ the comments in connection with \eqref{ft}) 
one notes that
\beq
\gF_{\cH} \bsA_- \gF_{\cH}^{-1} = \bsA_-,    \lb{3.10F}
\enq
since $\bsA_-$ has constant fiber operators $A_-(t) = A_-$, 
$t\in\bbR$, in $\cH$, and
$\gF_{\cK}$ is unitary on any Hilbert space $L^2(\bbR; \cK)$, and 
hence particularly in the case $\cK = \cH_1(A_-)$ (cf.\ \eqref{grA-}). 
In this context one also notes that
\beq\label{dit}
\gF_{\cH} \bigg(\f{d}{dt}\bigg) \gF_{\cH}^{-1} = it \, \bsI.
\enq
In particular, 
\begin{equation}\label{iti}
\widetilde{\bsD}_{\bsA_-}=it \bsI+\bsA_- \, \text{ on } \, 
\dom \big(\widetilde{\bsD}_{\bsA_-}\big).
\end{equation}
Combining \eqref{3.10F}, \eqref{dit}, and \eqref{iti}, 
 one concludes  that 
\begin{equation}\label{fufu}
\gF_{\cH}^{-1} \widetilde{\bsD}_{\bsA_-}
 \gF_{\cH}={\bsD}_{\bsA_-}.
\end{equation}
Since $\widetilde{\bsD}_{\bsA_-}$
is a normal operator, from \eqref{fufu} one 
concludes  that 
$\bsD_{\bsA_-}^{}$ is a normal operator 
 on $\dom(d/dt) \cap \dom(\bsA_-)$ in $L^2(\bbR;\cH)$. 
\end{proof}

\section{Trace Norm Analyticity of $[g_z(A_+) - g_z(A_-)]$}
\lb{sB}
\renewcommand{\theequation}{B.\arabic{equation}}
\renewcommand{\thetheorem}{B.\arabic{theorem}}
\setcounter{theorem}{0} \setcounter{equation}{0}

The purpose of this appendix is to provide a straightforward proof of 
Lemma \ref{l7.4}, given the fact \eqref{7.trclgz}: 

\begin{lemma} \lb{lB.1}
Assume Hypothesis \ref{h2.1} and let $z\in\bbC\backslash [0,\infty)$.\  
Then $[g_z(A_+) - g_z(A_-)]$ is differentiable with respect to the 
$\cB_1(\cH)$-norm and 
\begin{align}
& \f{d}{dz} \tr_{\cH} \big(g_z(A_+) - g_z(A_-)\big) 
= \tr_{\cH}\bigg(\frac{d}{dz} g_z(A_+) - \frac{d}{dz} g_z(A_-)\bigg)     \lb{B.38a} \\
& \quad = \f{1}{2} {\tr}_{\cH} \big(A_+ (A_+^2 - z I)^{-3/2} - A_- (A_-^2 - z I)^{-3/2}\big), 
\quad z\in\bbC\backslash [0,\infty).    \no 
\end{align} 
\end{lemma}
\begin{proof}
Throughout this proof we choose $z\in\bbC\backslash [0,\infty)$ 
and $h \in\bbC$ satisfying 
$|h| < \varepsilon$ with $0 < \varepsilon$ sufficiently small such that also 
$z, (z + h) \in\bbC\backslash [0,\infty)$. Due to the self-adjointness 
of $A_\pm$ in $\cH$, 
\begin{equation}
\sigma\big(A_+^2\big) \cup \sigma\big(A_-^2\big) 
\subseteq \big[\sigma_0, \infty\big) \subseteq [0,\infty).
\end{equation}  
where we abbreviated
\begin{equation}
\sigma_0 = \min\big\{\inf \big(\sigma\big(A_+^2\big)\big), 
\inf \big(\sigma\big(A_-^2\big)\big)\big\} \geq 0. 
\end{equation}

We recall the integral representations 
\begin{equation}
A_{\pm} (A_{\pm}^2 -z I)^{-1/2}f = \f{1}{\pi} \int_0^{\infty} t^{-1/2} 
(A_{\pm}^2 + (- z + t) I)^{-1} A_{\pm} f \, dt, \quad f \in \dom(A_{\pm}),    \lb{B.39}
\end{equation}
valid in the strong sense in $\cB(\cH)$ (cf., e.g., \cite[Sect.\ V.3.11]{Ka80}). As a consequence 
of \eqref{B.39} one computes
\begin{align}
& \f{1}{h} [g_{z + h}(A_+) - g_{z }(A_+)] 
- \f{d}{dz} g_{z } (A_+)  
- \f{1}{h} [g_{z + h}(A_-) - g_{z }(A_-)] 
 + \f{d}{dz} g_{z } (A_-)    \no \\
& \quad = \f{h}{\pi} \int_0^{\infty} t^{-1/2} 
 \bigg[A_+ (A_+^2 + (-z +t) I)^{-2} 
 (A_+^2 + (-z - h + t) I)^{-1}   \no \\
& \hspace*{2.8cm}    - A_- (A_-^2 + (-z +t) I)^{-2} 
 (A_-^2 + (-z - h + t) I)^{-1}\bigg] dt  \no \\
 & \quad = \f{h}{\pi} \int_0^{\infty} t^{-1/2} 
 \bigg[(A_+ - A_-) (A_+^2 + (-z +t) I)^{-2} 
 (A_+^2 + (-z - h + t) I)^{-1}   \no \\
& \hspace*{2.8cm} + A_- (A_+^2 + (-z +t) I)^{-2} 
 (A_+^2 + (-z - h + t) I)^{-1}   \no \\
& \hspace*{2.8cm}    - A_- (A_-^2 + (-z +t) I)^{-2} 
 (A_-^2 + (-z - h + t) I)^{-1}\bigg] dt.    \lb{B.40} 
\end{align} 
One notes that in contrast to \eqref{B.39}, \eqref{B.40} now holds in the norm 
sense in $\cB(\cH)$. 

Next, we recall \eqref{7.trclgz}, that is, 
\begin{equation}
[g_{z}(A_+) - g_{z}(A_-)] \in\cB_1(\cH), \quad z\in \bbC\backslash [0, \infty), 
\lb{B.40a}
\end{equation}
and note that 
\begin{equation}
\bigg[\f{d}{dz} g_{z } (A_+) -  \f{d}{dz} g_{z } (A_-)\bigg] 
\in\cB_1(\cH), \quad z\in \bbC\backslash [0, \infty).     \lb{B.40b}
\end{equation}
Indeed, \eqref{B.40b} follows from \cite[Theorem\ 8.7.1]{Ya92}, as 
$(d/dz) g_z(\cdot)$ satisfies the conditions \eqref{yaf1} (with $\varepsilon = 1$) 
and \eqref{yaf2} (both limits vanishing). 

Hence, 
\begin{align}
& \bigg\|\f{1}{h} [g_{z + h}(A_+) - g_{z }(A_+)] 
- \f{d}{dz} g_{z } (A_+)   
- \f{1}{h} [g_{z + h}(A_-) - g_{z }(A_-)] 
 + \f{d}{dz} g_{z } (A_-)\bigg\|_{\cB_1(\cH)}      \no \\
 & \; \leq \f{|h|}{\pi} \int_0^{\infty} t^{-1/2} 
 \big\|(A_+ - A_-) (A_+^2 + (-z +t) I)^{-2} 
 (A_+^2 + (-z - h + t) I)^{-1}\big\|_{\cB_1(\cH)} dt   \no \\
& \;\quad + \f{|h|}{\pi} \int_0^{\infty} t^{-1/2} 
\big\|A_- (A_+^2 + (-z +t) I)^{-2} 
 (A_+^2 + (-z - h + t) I)^{-1}     \lb{B.41}  \\
& \hspace*{3.2cm}    - A_- (A_-^2 + (-z +t) I)^{-2} 
 (A_-^2 + (-z - h + t) I)^{-1}\big\|_{\cB_1(\cH)} \, dt.  \no
\end{align} 
Investigating the terms in \eqref{B.41} individually, and recalling,
\begin{equation}
(A_+ - A_-) (A_-^2 - z I)^{-1/2}, \, (A_+ - A_-) (A_+^2 - z I)^{-1/2} 
\in \cB_1(\cH), \quad z \in \rho(A_-^2),  \lb{B.42}
\end{equation}
by \eqref{3.Apmrtrcl}, one estimates for the first term on the right-hand 
side of \eqref{B.41}
\begin{align}
& \f{|h|}{\pi} \int_0^{\infty} t^{-1/2} 
 \big\|(A_+ - A_-) (A_+^2 + (-z +t) I)^{-2} 
 (A_+^2 + (-z - h + t) I)^{-1}\big\|_{\cB_1(\cH)} dt   \no \\
& \quad \leq C (\varepsilon, z) \f{|h|}{\pi} 
\big\|(A_+ - A_-) (|A_+| + I)^{-1}\big\|_{\cB_1(\cH)}  
\int_0^{\infty} t^{-1/2} (\eta_0 (\varepsilon, z) + t)^{-1}  \, dt < \infty,  \lb{B.43}
\end{align}
where
\begin{align}
& \|(|A_+| + I)(A_+^2 + (-z + t)I)^{-1}\|_{\cB(\cH)} 
= \sup_{\mu \geq \sigma_0} \bigg|\f{\mu^{1/2} + 1}{\mu - z + t}\bigg| 
\leq C (\varepsilon, z),    \lb{B.43a} \\
& \|(A_+^2 + (-z - h + t)^{-1}\|_{\cB(\cH)} = 
\sup_{\mu \geq \sigma_0} \f{1}{|\mu - z - h + t|} \leq \f{1}{\eta_0(\varepsilon, z) + t} 
\lb{B.43b} 
\end{align}
for $C (\varepsilon, z) > 0$ independent of $t>0$,  
and for some $\eta_0 (\varepsilon, z) >0$, with $\eta_0 (\varepsilon, z)$ independent of $h\in\bbC$ since we assumed $z, (z + h) \in \rho\big(A_+^2\big) \cap \rho\big(A_-^2\big)$ for all $h \in\bbC$, $|h| < \varepsilon$, with $0 < \varepsilon$ sufficiently small. 

Next, we turn to the second term on the right-hand side of \eqref{B.41} and write
\begin{align}
& \f{|h|}{\pi} \int_0^{\infty} t^{-1/2} 
\big\|A_- \big[(A_+^2 + (-z +t) I)^{-2} 
 (A_+^2 + (-z - h + t) I)^{-1}     \no  \\
& \hspace*{2.7cm}    - (A_-^2 + (-z +t) I)^{-2} 
 (A_-^2 + (-z - h + t) I)^{-1}\big]\big\|_{\cB_1(\cH)} \, dt   \no \\
& \quad = \f{|h|}{\pi} \int_0^{\infty} t^{-1/2} 
\big\|A_- \big[(A_+^2 + (-z +t) I)^{-2} 
 (A_+^2 + (-z - h + t) I)^{-1}     \no  \\
 & \hspace*{3.5cm}    - (A_+^2 + (-z +t) I)^{-2} 
 (A_-^2 + (-z - h + t) I)^{-1}     \no \\
 & \hspace*{3.5cm}    + (A_+^2 + (-z +t) I)^{-2} 
 (A_-^2 + (-z - h + t) I)^{-1}     \no \\
& \hspace*{3.5cm}    - (A_-^2 + (-z +t) I)^{-2} 
 (A_-^2 + (-z - h + t) I)^{-1}\big]\big\|_{\cB_1(\cH)} \, dt   \no \\
& \quad \leq \f{|h|}{\pi} \int_0^{\infty} t^{-1/2} 
 \big\|A_- (A_+^2 + (-z + t) I)^{-2}   \no \\
& \hspace*{2.78cm}  \times \big[(A_+^2 + (-z - h +t) I)^{-1} 
 - (A_-^2 + (-z - h + t) I)^{-1}\big]\big\|_{\cB_1(\cH)} \, dt   \no \\
& \qquad + \f{|h|}{\pi} \int_0^{\infty} t^{-1/2}  
 \big\|A_- \big[(A_+^2 + (-z +t) I)^{-2} 
-  (A_-^2 + (-z + t) I)^{-2}\big]    \no \\
& \hspace*{3.1cm} \times 
(A_-^2 + (-z - h + t) I)^{-1} \big\|_{\cB_1(\cH)} \, dt  \no \\
& \quad \leq \f{|h|}{\pi} \int_0^{\infty} t^{-1/2} (\eta_0 (\varepsilon, z) + t)^{-1}  
 \big\|A_- (A_+^2 + (-z + t)I)^{-1}\big\|_{\cB(\cH)}  \no \\
& \hspace*{1.8cm} \times \big\|(A_+^2 + (-z - h +t) I)^{-1} 
 - (A_-^2 + (-z - h + t) I)^{-1}\big\|_{\cB_1(\cH)} \, dt    \no \\
& \qquad + \f{|h|}{\pi} \int_0^{\infty} t^{-1/2} (\eta_0 (\varepsilon, z) + t)^{-1} 
\lb{B.44} \\
& \hspace*{2.1cm}  \times 
\big\|A_- \big[(A_+^2 + (-z +t) I)^{-2} 
-  (A_-^2 + (-z + t) I)^{-2}\big]\big\|_{\cB_1(\cH)} \, dt.  \no  
\end{align}

To complete the proof one estimates the following norms:
\begin{align}
& \big\|A_- (A_+^2 + (-z + t)I)^{-1}\big\|_{\cB(\cH)}    \no \\
& \quad \leq \big\|A_- (|A_+| + I)^{-1}\big\|_{\cB(\cH)} 
\big\|(|A_+| + I) (A_+^2 + (-z + t)I)^{-1}\big\|_{\cB(\cH)}   \no \\
& \quad \leq C_1 (\varepsilon, z)
\sup_{\mu \geq \sigma_0} \bigg|\f{\mu^{1/2} + 1}{\mu - z + t}\bigg| 
\leq \wti C_1 (\varepsilon, z)       \lb{B.45} 
\end{align}
and 
\begin{align} 
& \big\|(A_+^2 + (-z - h +t) I)^{-1} 
 - (A_-^2 + (-z - h + t) I)^{-1}\big\|_{\cB_1(\cH)}   \no \\
& \quad = \big\|A_+(A_+^2 + (-z - h +t) I)^{-1} 
\big[(A_- - A_+)\big] (A_-^2 + (-z - h + t) I)^{-1}  \no \\ 
& \qquad + \big[(A_- - A_+) (A_+^2 + (-z - h + t) I)^{-1}\big]^* 
A_- (A_-^2 + (-z - h + t) I)^{-1}\big\|_{\cB_1(\cH)}   \no \\
& \quad \leq \big\|A_+ (A_+^2 + (-z - h +t) I)^{-1} \big\|_{\cB(\cH)}   \no \\
& \qquad \quad \times \big\|(A_- - A_+) (A_-^2 + (-z - h + t) I)^{-1}\big\|_{\cB_1(\cH)}
   \no \\
& \qquad + \big\|(A_- - A_+) (A_+^2 + (-z - h +t) I)^{-1} \big\|_{\cB_1(\cH)}  \no \\
& \qquad \quad \times \big\|A_- (A_-^2 + (-z - h + t) I)^{-1}\big\|_{\cB(\cH)}   \no \\
& \quad = C_1 (\varepsilon, z)\big\|(A_- - A_+) (|A_-| + I)^{-1}\big\|_{\cB_1(\cH)}   
\no \\
& \qquad + C_2 (\varepsilon, z) \big\|(A_- - A_+) (|A_+| + I)^{-1}\big\|_{\cB_1(\cH)}, 
\lb{B.46}
\end{align}
for appropriate constants $C_j (\varepsilon, z)>0$, $j=1,2$, independent of 
$t>0$ and $h\in\bbC$, $|h|<\varepsilon$, and similarly, 
\begin{align}
& \big\|A_- \big[(A_+^2 + (-z +t) I)^{-2} 
-  (A_-^2 + (-z + t) I)^{-2}\big]\big\|_{\cB_1(\cH)}  \no \\
& \quad = \big\|A_-\big[(A_+^2 + (-z +t) I)^{-2} 
- (A_-^2 + (-z +t) I)^{-1} (A_+^2 + (-z +t) I)^{-1}   \no \\ 
& \qquad \;\,\, + (A_-^2 + (-z +t) I)^{-1} (A_+^2 + (-z +t) I)^{-1} 
-  (A_-^2 + (-z + t) I)^{-2}\big]\big\|_{\cB_1(\cH)}  \no \\
& \quad = \big\|A_- (A_-^2 + (-z +t) I)^{-1} 
(A_+^2 + (-z +t) I)^{-1}(A_-^2 - A_+^2) (A_-^2 + (-z +t) I)^{-1} 
\no \\
& \qquad \;\,\, + A_- (A_+^2 + (-z +t) I)^{-1} (A_-^2 - A_+^2) 
(A_-^2 + (-z +t) I)^{-1}     \no \\
& \qquad \quad \;\,\, \times (A_+^2 + (-z +t) I)^{-1}\big\|_{\cB_1(\cH)} 
\no \\ 
& \quad = \big\|A_- (A_-^2 + (-z +t) I)^{-1} 
A_+ (A_+^2 + (-z +t) I)^{-1}     \no \\ 
& \qquad \;\,\, \times \big[(A_- - A_+) (A_-^2 + (-z +t) I)^{-1}\big]     \no \\ 
& \qquad \;\,\, + A_- (A_-^2 + (-z +t) I)^{-1} 
\big[(A_- - A_+) (A_+^2 + (-z +t) I)^{-1}\big]^*     \no \\ 
& \qquad \quad \;\,\, \times A_- (A_-^2 + (-z +t) I)^{-1}  \no \\ 
& \qquad \;\,\, + A_- A_+ (A_+^2 + (-z +t) I)^{-1} 
\big[(A_- - A_+) (A_-^2 + (-z +t) I)^{-1}\big]     \no \\ 
& \qquad \quad \;\,\, \times (A_+^2 + (-z +t) I)^{-1}  \no \\
& \qquad \;\,\, + A_- (A_+^2 + (-z +t) I)^{-1} 
\big[(A_- - A_+) (A_-^2 + (-z +t) I)^{-1}\big]     \no \\ 
& \qquad \quad \;\,\, \times A_- (A_+^2 + (-z +t) I)^{-1}\big\|_{\cB_1(\cH)}    
\no \\
& \quad \leq \big\|A_-(A_-^2 + (-z +t) I)^{-1}\big\|_{\cB(\cH)} 
\big\|A_+ (A_+^2 + (-z +t) I)^{-1}\big\|_{\cB(\cH)}    \no \\
& \qquad \quad \; \times \big\|(A_- - A_+)  (A_-^2 + (-z +t) I)^{-1} \big\|_{\cB_1(\cH)}    
\no \\
& \qquad \; + \big\|A_-(A_-^2 + (-z +t) I)^{-1}\big\|_{\cB(\cH)}^2 
\big\|(A_- - A_+) (A_+^2 + (-z +t) I)^{-1}\big\|_{\cB_1(\cH)}      \no \\
& \qquad \; + \big\|A_- A_+ (A_+^2 + (-z +t) I)^{-1} \big\|_{\cB(\cH)} 
\big\|(A_+^2 + (-z +t) I)^{-1}\big\|_{\cB(\cH)}    \no \\
& \qquad \quad \; \times \big\|(A_- - A_+) (A_-^2 + (-z +t) I)^{-1}\big\|_{\cB_1(\cH)}    
\no \\
& \qquad \; + \big\|A_- (A_+^2 + (-z +t) I)^{-1}\big\|_{\cB(\cH)}^2 
\big\|(A_- - A_+) (A_-^2 + (-z +t) I)^{-1}\big\|_{\cB_1(\cH)}     \no \\
& \quad = C_3 (\varepsilon, z) \big\|(A_- - A_+) (|A_-| + I)^{-1}\big\|_{\cB_1(\cH)} 
\no \\
& \qquad + C_4 (\varepsilon, z) \big\|(A_- - A_+) (|A_+| + I)^{-1}\big\|_{\cB_1(\cH)},   
\lb{B.47}
\end{align}
for appropriate constants $C_k  (\varepsilon, z)>0$, $k=3,4$, independent of 
$t>0$ and $h\in\bbC$, $|h|<\varepsilon$, repeatedly applying estimates of the 
type \eqref{B.43a}, \eqref{B.43b}, and \eqref{B.45}.

Finally, combining \eqref{B.41}--\eqref{B.47} yields 
\begin{align}
& \bigg\|\f{1}{h} \Big[[g_{z + h}(A_+) - g_{z + h}(A_-)] - [g_{z }(A_+) - g_{z }(A_-)]\Big] 
\no \\
& \;\;  - \bigg(\f{d}{dz} g_{z } (A_+) - \f{d}{dz} g_{z } (A_-)\bigg)\bigg\|_{\cB_1(\cH)} 
 \underset{h\to 0}{=} \Oh(h)    \lb{B.49}
\end{align}
and proves the required differentiability in trace norm. Since 
$z\in \bbC\backslash [0,\infty)$ was arbitrary, 
one concludes that \eqref{B.38a} holds. 
\end{proof}

We note that Lemma \ref{lB.1} extends to $z\in \rho(A_+^2) \cap \rho(A_-^2)$.

The function $g_z(x)$, $x\in\bbR$, in Lemma \ref{lB.1} should be viewed as a smooth version of a step function approaching $\pm 1$ as $x \to \pm \infty$. In this context 
we also note that compactness for operators of the type 
\begin{equation}
[\arg(A_+ - z I) - \arg(A_- - z I)], \quad z \in \bbC_+= \{z\in\bbC\,|\, \Im(z) > 0\}, 
\end{equation}
was proved in \cite[Theorem\ 7.3]{Pu01}. 

\medskip

\noindent {\bf Acknowledgments.}
We are indebted to Alan Carey, Alexander Gomilko, Galina Levitina, Alexander Pushnitski, 
Arnd Scheel, Barry Simon, and Alexander Stroh\-maier for helpful discussions. 
We are particularly grateful to Alan Carey for his steadfast support of this 
project.

 

\begin{thebibliography}{99}
%
\bi{Ab01} A.\ Abbondandolo, {\it Morse theory for Hamiltonian systems}, 
Res. Notes  Math., Vol.\ 425, Chapman/Hall/CRC, Boca Raton, FL, 2001.
%
\bi{AM03} A.\ Abbondandolo and P.\ Majer, {\it Ordinary differential operators 
in Hilbert spaces and Fredholm pairs}, Math. Z. {\bf 243}, 525--562 (2003). 
%
\bi{AS72} M.\ Abramovitz, I.\ A.\ Stegun, {\it Handbook of Mathematical
Functions}, Dover, New York, 1972.
%
\bi{AS94} W.\ O.\ Amrein and K.\ B.\ Sinha, {\it On pairs of projections
in a Hilbert space}, Linear Algebra Appl. {\bf 208/209}, 425 -- 435 (1994).
%
\bi{An89} N.\ Anghel, {\it Remark on Callias' index theorem}, Rep. Math. Phys. 
{\bf 28}, 1--6 (1989).
%
\bi{An90} N.\ Anghel, {\it $L^2$-index formulae for perturbed Dirac operators},  
Comm. Math. Phys. {\bf 128}, 77--97 (1990).
%
\bi{An90a} N.\ Anghel, {\it The two-dimensional magnetic field problem 
revisited}, J. Math. Phys. {\bf 31}, 2091--2093 (1990).
%
\bi{An93} N.\ Anghel, {\it On the index of Callias-type operators}, Geom. 
Funct. Anal. {\bf 3}, 431--438 (1993).
%
\bi{An94} N.\ Anghel, {\it Index theory for short-ranged fields in higher 
dimensions}, J. Funct. Anal. {\bf 119}, 19--36 (1994).
%
\bi{Ar70} H.\ Araki, {\it On quasifree states of CAR and Bogoliubov automorphisms}, 
Publ. Res. Inst. Math. Sci.  {\bf 6}, 385--442 (1970/71).
%
\bi{ABHN01} W.\ Arendt, C.\ K.\ Batty, M.\ Hieber, F.\ Neubrander, 
{\it Vector-Valued Laplace Transforms and Cauchy Transforms}, 
Monographs in Mathematics, Vol.\ 96, Birkh\"auser, Basel, 2001.
%
\bibitem{AD56} N.\ Aronszajn and W.\ F.\ Donoghue, {\it On
exponential representations of analytic functions in the upper
half-plane with positive imaginary part}, J. Analyse Math. {\bf 5},
321--388 (1956--57).
%
\bi{APS73} M.\ F.\ Atiyah, V.\ K.\ Patodi, I.\ M.\ Singer, {\it Spectral asymmetry and Riemannian geometry},  Bull. London Math. Soc. {\bf 5}, 229--234 (1973).
%
\bi{APS75} M.\ F.\ Atiyah, V.\ K.\ Patodi, I.\ M.\ Singer, {\it Spectral asymmetry and Riemannian geometry. I.}, Math. Proc. Cambridge Philos. Soc. {\bf 77}, 43--69 (1975).
%
\bi{APS75a} M.\ F.\ Atiyah, V.\ K.\ Patodi, I.\ M.\ Singer, {\it Spectral asymmetry and Riemannian geometry. II.}, Math. Proc. Cambridge Philos. Soc. {\bf 78}, 
405--432 (1975).
%
\bi{APS76} M.\ F.\ Atiyah, V.\ K.\ Patodi, I.\ M.\ Singer, {\it Spectral asymmetry and Riemannian geometry. III.},  Math. Proc. Cambridge Philos. Soc. {\bf 79},  
71--99 (1976).
%
\bi{ASS94} J.\ Avron, R.\ Seiler, and B.\ Simon, {\it The index of a pair 
of projections}, J. Funct. Anal. {\bf 120}, 220--237 (1994).
%
\bi{ACS07} N.\ Azamov, A.\ Carey, and F.\  Sukochev, {\it The spectral 
shift function and spectral flow}, Commun. Math. Phys. {\bf 276}, 51--91 
(2007).
%
\bi{ACDS09} N.\ A.\ Azamov, A.\ L.\ Carey, P.\ G.\ Dodds, and F.\ A.\ Sukochev, 
{\it Operator integrals, spectral shift, and spectral flow}, Canad. J. Math.,
{\bf 61}, 241--263 (2009).
%
\bi{AW09} S.\ Azzali and C.\ Wahl, {\it Spectral flow, index and the signature operator}, 
 J. Topol. Anal. {\bf 3}, no.\ 1, 37--67 (2011). 
%
\bi{BCPRSW06} M.-T.\ Benameur, A.\ L.\ Carey, J.\ Phillips, A.\ Rennie, 
F.\ A.\ Sukochev, and K.\ P.\ Wojciechowski, {\it An analytic approach 
to spectral flow in von Neumann algebras}, in {\it Analysis, Geometry 
and Topology of Elliptic Operators}, B.\ Booss-Bavnbek, S.\ Klimek, 
M.\ Lesch, and W.\ Zhang (eds.), World Scientific, Singapore, 2006, 
pp.\ 297--352.
%
\bi{BP08} P.\ Benevieri and P.\ Piccione, {\it On a formula for the spectral 
flow and its applications}, Math. Nachr. {\bf 283}, 659--685 (2010). 
%
\bi{BGV92} N.\ Berline, E.\ Getzler, and M.\ Vergne, {\it Heat Kernels and Dirac 
Operators}, Springer, Berlin, 1992. 
%
\bi{BK62} M.\ Sh. Birman and M.\ G.\ Krein, {\it On the theory of wave
operators and scattering operators}, Sov. Math. Dokl. {\bf 3}, 740--744
(1962).
%
\bi{BS65} M.\ \u{S}.\ Birman and M.\ Z.\ Solomjak, {\it  Stieltjes 
double operator integrals}, Sov. Math. Dokl. {\bf 6}, 1567--1571 (1965). 
%
\bi{BS67} M.\ Sh.\ Birman and M.\ Z.\ Solomyak, {\it  Stieltjes 
double-integral operators}, in {\it Topics in Mathematical Physics, Vol.\ 1, 
Spectral Theory and Wave Processes}, M.\ Sh.\ Birman (ed.), 
Consultants Bureau Plenum Publishing Corporation, New York, 1967, 
pp.\ 25--54.
%
\bi{BS68} M.\ Sh.\ Birman and M.\ Z.\ Solomyak, {\it  Stieltjes 
double-integral operators. II}, in {\it Topics in Mathematical Physics, 
Vol.\ 2, Spectral Theory and Problems in Diffraction}, M.\ Sh.\ 
Birman (ed.), Consultants Bureau, New York, 1968, pp.\ 19Ð46.
%
\bi{BS73} M.\ Sh.\ Birman and M.\ Z.\ Solomyak, {\it  Double Stieltjes 
operator integrals. III},  in {\it Problems of  Mathematical Physics, 
Vol.\ 6, Theory of Functions, Spectral Theory, Wave Propagation}, M.\ 
Sh.\ Birman (ed.), Izdat. Leningrad. Univ., Leningrad, 1973, pp.\ 
27--53. (Russian).
%
\bi{BS75} M.\ Sh.\ Birman and M.\ Z.\ Solomyak, {\it Remarks on the 
spectral shift function}, J. Sov. Math. {\bf 3}, 408--419 (1975).  
%
\bi{BS87} M.\ S.\ Birman and M.\ Z.\ Solomyak, {\it Spectral Theory of 
Self-Adjoint Operators in Hilbert Space}, Reidel, Dordrecht, 1987. 
%
\bi{BS93} M.\ Sh.\ Birman and M.\ Z.\ Solomyak, {\it Operator integration, 
perturbations, and commutators}, J. Sov. Math. {\bf 63}, 129--148 (1993).  
%
\bi{BS03} M.\ Sh.\ Birman and M.\ Solomyak, {\it  Double operator 
integrals in a Hilbert space},  Integr.\ Eqns. Oper. Theory {\bf 47}, 
131--168 (2003).
%
\bi{BY93} M.\ Sh.\ Birman and D.\ R.\ Yafaev, {\it The spectral shift function. 
The work of M.\ G.\ Krein and its further development}, 
St. Petersburg Math. J. {\bf 4}, 833--870 (1993).
%
\bi{BB04} D.\ Bleecker and B.\ Booss-Bavnbek, {\it Spectral invariants 
of operators of Dirac type on partitioned manifolds} in {\it Aspects of 
Boundary Problems in Analysis and Geometry}, J.\ Gil, T.\ Krainer, 
I.\ Witt (eds.),  Operator Theory:  Advances and Applications, Vol.\ 151, 
Birkh\"auser, Basel, 2004, pp.\ 1--130.
%
\bi{BS10} A.\ B\"ottcher and I.\ M.\ Spitkovsky, {\it A gentle guide to the basics of 
two projections theory}, Lin. Algebra Appl. {\bf 432}, 1412--1459 (2010).
%
\bi{Bo79} B.\ Bojarski, {\it Abstract linear conjugation problems and Fredholm pairs of subspaces}, in {\it Differential and Integral equations, Boundary value problems. Collection of papers dedicated to the memory of Academician I, Vekua}, Tbilisi University Press, Tbilisi, 1979, pp.\ 45Ð60. (Russian.)
%
\bi{BGGSS87} D.\ Boll{\' e}, F.\ Gesztesy, H.\ Grosse, W.\ Schweiger, and 
B.\ Simon, {\it Witten index, axial anomaly, and Krein's spectral 
shift function in supersymmetric quantum mechanics}, J. Math. Phys.
{\bf 28}, 1512--1525  (1987).
%
\bi{BB85} B.\ Booss and D.\ D.\ Bleecker, {\it Topology and Analysis. The 
Atiyah--Singer Index Formula and Gauge-Theoretic Physics}, Springer, 
New York, 1985. 
%
\bi{BBLP05} B.\ Booss-Bavnbek, M.\ Lesch, and J.\ Phillips, {\it Unbounded 
Fredholm operators and spectral flow}, Canad. J. Math. {\bf 57}, 225--250 
(2005).
%
\bi{BBW93} B.\ Boo{\ss}-Bavnbek and K.\ P.\ Wojciechowski, {\it Elliptic 
Boundary Problems for Dirac Operators}, Birkh\"auser, Boston, 1993. 
%
\bi{BMS88} N.\ V.\ Borisov, W.\ M\"uller, and R.\ Schrader, {\it Relative 
index theorems and supersymmetric scattering theory}, Commun. Math. 
Phys. {\bf 114}, 475--513 (1988).  
%
\bi{BS78} R.\ Bott and R.\ Seeley, {\it Some remarks on the paper of Callias}, 
Comm. Math. Phys. {\bf 62}, 235--245 (1978).
%
\bi{BL99a} J.\ Br\"uning and M.\ Lesch, {\it On the $\eta$-invariant of certain 
nonlocal boundary value problems}, Duke Math. J. {\bf 96}, 425--468 (1999). 
%
\bi{BL01} J.\ Br\"uning and M.\ Lesch, {\it On boundary value problems for 
Dirac type operators. I. Regularity and self-adjointness}, J. Funct. Anal. 
{\bf 185}, 1--62 (2001).
%
\bi{Bu92} U.\ Bunke, {\it Relative index theory}, J. Funct. Anal. {\bf 105}, 
63--76 (1992).
%
\bi{Ca78} C.\ Callias, {\it Axial anomalies and index theorems on 
open spaces},  Commun. Math. Phys. {\bf 62}, 213--234 (1978).
%
\bi{CPS09} A.\ Carey, D.\ Potapov, and F.\  Sukochev, {\it Spectral 
flow is the integral of one forms on the Banach manifold of self adjoint 
Fredholm operators}, Adv. Math. {\bf 222}, 1809--1849 (2009).
%
\bi{CL99} C.\ Chicone and Y.\ Latushkin, {\it Evolution Semigroups 
in Dynamical Systems and Differential Equations}, Math. Surv. Monogr., 
Vol.\ 70, Amer. Math. Soc., Providence, RI, 1999.
%
\bi{CL63} H.\ O.\ Cordes and J.\ P.\ Labrousse, {\it The invariance of the 
index in the metric space of closed operators}, J. Math. Mech. {\bf 12}, 
693--719 (1963). 
%
\bi{Da58} C.\ Davis, {\it Separation of two linear
subspaces}, Acta Scient. Math. (Szeged) {\bf 19}, 172--187 (1958).
%
\bi{dPS04} B.\ de Pagter and F.\ A.\ Sukochev, {\it  Differentiation 
of operator functions in non-commutative $L_p$-spaces},  J. Funct. 
Anal. {\bf 212},  28--75 (2004). 
%
\bi{dPSW02}  B.\ de Pagter, F.\ A.\ Sukochev, and H.\ Witvliet, {\it 
Double operator integrals}, J. Funct. Anal. {\bf 192}, 52--111 (2002).
%
\bi{DU77} J.\ Diestel and J.\ J.\ Uhl, {\it Vector Measures}, 
Mathematical Surveys, Vol.\ 15, Amer. Math. Soc., Providence, RI, 
1977.
%
\bi{Di48} J.\ Dixmier, {\it Position relative de deux vari\'{e}t\'{e}s ferm\'{e}es dans un espace de Hilbert}, Revue Scientifique {\bf 86}, 387--399 (1948).
%
\bi{Di49} J.\ Dixmier, {\it \'{E}tude sur les vari\'{e}t\'{e}s et les op\'{e}rateurs Julia, avec quelques applications}, Bull. Soc. Math. France {\bf 77}, 11--101 (1949).
%
\bi{DW91} R.\ G.\ Douglas and K.\ P.\ Wojciechowski, {\it Adiabatic limits of the 
$\eta$-invariants. The odd-dimensional Atiyah--Patodi--Singer problem}, 
Commun. Math. Phys. {\bf 142}, 139--168 (1991).
%
\bi{DS88} N.\ Dunford and J.\ Schwartz, {\it Linear operators. Part II. Spectral theory. Selfadjoint operators in Hilbert space},  Wiley \& Sons, New York, 1988.
%
\bi{EE89} D.\ E.\ Edmunds and W.\ D.\ Evans, {\it Spectral Theory and
Differential Operators}, Clarendon Press, Oxford, 1989.
%
\bi{Es98} G.\ Esposito, {\it Dirac Operators and Spectral Geometry}, Cambridge 
University Press, Cambridge, 1998. 
%
\bi{Fr37} K.\ Friedrichs, {\it On certain inequalities and characteristic value problems for analytic functions and for functions of two variables}, Trans. Amer. Math. Soc. 
{\bf 41}, 321--364 (1937).
%
\bi{Fu02} T.\ Furuta, {\it Invitation to Linear Operators. From Matrices to Bounded
Linear Operators in a Hilbert Space}, Taylor \& Francis, London, 2002. 
%
\bi{Ge86} F.\ Gesztesy, {\it Scattering theory for one-dimensional 
systems with nontrivial spatial asymptotics},  in {\it Schrûdinger operators, 
Aarhus 1985}, Lecture Notes in Math., Vol.\ 1218, Springer, Berlin, 1986, pp.\ 93--122.
%
\bi{GGST10} F.\ Gesztesy, A.\ Gomilko, F.\ Sukochev, and Y.\ Tomilov, 
{\it On a question of A.\ E.\ Nussbaum on measurability of families of closed linear 
operators in a Hilbert space}, Israel J. Math. {\bf 188}, 195--219 (2012). 
%
\bibitem{GLMZ05} F.\ Gesztesy, Y.\,Latushkin, M.\,Mitrea, and M. Zinchenko,
{\it Nonselfadjoint operators, infinite determinants, and some
applications}, Russ. J. Math. Phys. {\bf 12}, 443--471 (2005).
%
\bi{GLST15} F.\ Gesztesy, Y.\ Latushkin, F.\ Sukochev, 
and Y.\ Tomilov, {\it Some operator bounds employing complex interpolation revisited}, 
in {\it Operator Semigroups Meet Complex Analysis, Harmonic Analysis and Mathematical 
Physics}, W.\ Arendt, R.\ Chill and Yu.\ Tomilov (eds.), Operator Theory: Advances and 
Applications, Birkh\"auser--Springer, Basel, to appear. 
%
\bi{GM00} F.\ Gesztesy and K.\ A.\ Makarov, {\it The $\Xi$
operator and its relation to Krein's spectral shift function},  J. d'Anal. Math. 
{\bf 81}, 139--183 (2000).
%
\bi{GMMN09} F.\ Gesztesy,  M.\ Malamud, M.\ Mitrea, and S.\ Naboko,
{\it Generalized polar decompositions for closed operators in Hilbert
spaces  and some applications}, Integral Eq. Operator Th. {\bf 64}, 83--113  
(2009). 
%
\bi{GSS91} F.\ Gesztesy, W.\ Schweiger, and B.\ Simon, {\it Commutation methods applied to the mKdV-equation}, Trans. Amer. Math. Soc. {\bf 324}, 465--525 (1991). 
%
\bi{GS88} F.\ Gesztesy and B.\ Simon, {\it Topological invariance of the Witten 
index}, J. Funct. Anal. {\bf 79}, 91--102 (1988). 
%
\bi{Gi84} P.\ B.\ Gilkey, {\it Invariance Theory, the Heat Equation, and the 
Atiyah--Singer Index Theorem}, Publish or Perish, Wilmington, DE, 1984. 
%
\bi{GS83} P.\ B.\ Gilkey and L.\ Smith, {The eta invariant for a class of elliptic 
boundary value problems}, Commun. Pure Appl. Math. {\bf 36}, 85--131 (1983).
%
\bi{GK69} I.\ Gohberg and M.\ G.\ Krein, {\it Introduction to the Theory of
Linear Nonselfadjoint Operators}, Translations of Mathematical Monographs,
Vol.\ 18, Amer. Math. Soc., Providence, RI, 1969.
%
\bi{Go05} M.\ Gonz\'alez, {\it Fredholm theory for pairs of closed subspaces 
of a Banach space}, J. Math. Anal. Appl. {\bf 305}, 53--62 (2005).
%
\bi{GR80} I.\ S.\ Gradshteyn and I.\ M.\ Ryzhik, {\it Table of Integrals, Series, 
and Products}, corrected and enlarged edition, prepared by 
A.\ Jeffrey, Academic Press, San Diego, 1980. 
%
\bi{Gr04} L.\ Grafakos, {\it Classical and Modern Fourier Analysis}, Pearson,
Prentice Hall, Upper Saddle River, NJ, 2004.
%
\bi{GN83} G.\ Greiner and R.\ Nagel, {\it On the stability of strongly continuous semigroups of positive operators on $L^2 (\mu)$},  Ann. Scuola Norm. Sup. 
Pisa, Cl. Sci. (4), {\bf 10}, 257--262 (1983). 
%
\bi{Gr01} G.\ Grubb, {\it Poles of zeta and eta functions for perturbations of 
the Atiyah--Patodi--Singer problem}, Commun. Math. Phys. {\bf 215}, 583--589 (2001).
%
\bi{GS96} G.\ Grubb and R.\ T.\ Seeley, {\it Zeta and eta functions for 
Atiyah--Patodi--Singer operators}, J. Geom. Anal. {\bf 6}, 31--77 (1996).
%
\bi{Gr73} H.\ R.\ Gr\"umm, {Two theorems about $\cC_p$}, Rep. Math. Phys.
{\bf 4}, 211--215 (1973).
%
\bi{Ha69} P.\ R.\ Halmos, {\it Two subspaces}, Trans. Amer. Math. Soc. 
{\bf 144}, 381--389 (1969).
%
\bi{He51} E.\ Heinz, {\it Beitr\"age zur St\"orungstheorie der
Spektralzerlegung}, Math. Ann. {\bf 123}, 415--438 (1951).
%
\bi{HK70} P.\ Hess and T.\ Kato, {\it Perturbation of closed operators and 
their adjoints}, Comment. Math. Helv.  {\bf 45},  524--529 (1970).
%
\bi{HP85} E.\ Hille and R.\ S.\ Phillips, {\it Functional Analysis and Semi-Groups},
Colloquium Publications, Vol.\ 31, rev. ed., Amer. Math. Soc., 
Providence, RI, 1985.
%
\bi{Ka97} N.\ Kalton, {\it A note on pairs of projections}, Bull. Soc. Mat. Mexicana 
{\bf 3}, 309--311 (1997).
%
\bi{Ka52} T.\ Kato, {\it Notes on some inequaliies for linear operators}, 
Math. Ann. {\bf 125}, 208--212 (1952). 
%
\bi{Ka55} T.\ Kato, {\it Notes on projections and perturbation theory}, Technical 
Report No.\ 9, University of California at Berkeley, 1955.
%
\bi{Ka66} T.\ Kato, {\it Wave operators and similarity for some
non-selfadjoint operators}, Math. Ann. {\bf 162}, 258--279 (1966).
%
\bi{Ka73a} T. \ Kato, {\it Linear evolution equations of ``hyperbolic'' type, II}, 
J. Math. Soc. Japan  {\bf 25}, 648--666 (1973).
%
\bi{Ka80} T.\ Kato, {\it Perturbation Theory for Linear Operators}, corr.\ 
printing of the 2nd ed., Springer, Berlin, 1980.
%
\bi{Ka84} W.\ E.\ Kaufman, {\it A stronger metric for closed operators in Hilbert space}, 
Proc. Amer. Math. Soc. {\bf 90}, 83--87 (1984). 
%
\bi{KL04} P.\ Kirk and M.\ Lesch, {\it The $\eta$-invariant, Maslov index, and 
spectral flow for Dirac-type operators on manifolds with boundary}, Forum 
Math. {\bf 16}, 553--629 (2004).  
%
\bi{KMM03} V.\ Kostrykin, K.\ A. \ Makarov, and A.\ K.\ Motovilov,
{\it Existence and uniqueness of solutions to the operator Riccati equation. 
A geometric approach}, in {\it  Advances in Differential Equations and Mathematical Physics}, Yulia Karpeshina, G\"unter Stolz, Rudi Weikard, and Yanni Zeng (eds.), Contemp. Math. {\bf 327}, 181--198 (2003).  
%
\bi{Ko09} C.\ Kottke, {\it An index theorem of Callias type for pseudodifferential operators}, 
J. K-Theory {\bf 8}, no.\ 3, 387--417 (2011).
%
\bi{KK47} M.\ G.\ Krein and M.\ A.\ Krasnoselsky, {\it Fundamental theorems about extensions of Hermite operators and some applications to the
theory of orthogonal polynomials and to the moment problem}, Uspekhi Mat. Nauk 
{\bf 2}, 60--106 (1947). (Russian.)
%
\bi{KKM48} M.\ G.\ Krein, M.\ A.\ Krasnoselsky, and D.\ P.\ Milman, {\it On defect numbers of linear operators in Banach space and some geometric problems}, Sbornik Trudov Instituta Matematiki Akademii Nauk Ukrainskoy SSR {\bf 11}, 
97--112 (1948). (Russian.)
%
\bi{KY81} M.\ G.\ Krein and V.\ A.\ Yavryan, {\it Spectral shift functions that arise in perturbations of a positive operator}, J. Operator Th. {\bf 6}, 155--191 (1981). 
%
\bi{KPS82} S.\ G.\ Krein, Ju.\ I.\ Petunin, and E.\ M.\ Semenov,
{\it Interpolation of Linear Operators}, Transl. Math. Monographs,
Vol.~54, Amer. Math. Soc., Providence, RI, 1982.
%
\bi{LP08} Y.\ Latushkin and A.\ Pogan, {\it The Dichotomy Theorem for evolution bi-families}, 
 J. Diff. Eq. {\bf 245}, 2267-2306 (2008).
%
\bi{LT05} Y.\ Latushkin and Y.\ Tomilov, {\it Fredholm differential operators with unbounded coefficients}, J. Diff. Eq. {\bf 208}, 388--429 (2005).
%
\bi{LM89} H.\ B.\ Lawson and M.-L.\ Michelson, {\it Spin Geometry}, Princeton 
University Press, Princeton, 1989. 
%
\bi{Le74} M.\ J.\ J.\ Lennon, {\it On Sums and Products of Unbounded Operators in Hilbert Space}, Trans. Amer. Math. Soc. {\bf 198}, 273--285 (1974).
%
\bi{Le05} M.\ Lesch, {\it The uniqueness of the spectral flow
on spaces of unbounded self-adjoint Fredholm operators}, in {\it Spectral 
Geometry of Manifolds with Boundary and Decomposition of Manifolds}, 
B.\ Boss-Bavnbek, G.\ Grubb, and K.\ P.\ Wojciechowski (eds.), 
Contemp. Math., {\bf 366}, 193--224 (2005).
%
\bi{LW96} M.\ Lesch and K.\ P.\ Wojciechowski, {\it On the $\eta$-invariant of 
generalized Atiyah--Patodi--Singer boundary value problems}, Illinois J. 
Math. {\bf 40}, 30--46 (1996).
%
\bi{LL01} E.\ H.\ Lieb and M.\ Loss, {\it Analysis}, 2nd ed., Amer. Math. Soc., Providence, 
RI, 2001. 
%
\bi{LM72} J.\ L.\ Lions and E.\ Magenes, {\it Non-Homogeneous Boundary Value Problems and Applications}, Springer, New York, 1972. 
%
\bi{Lo34} K.\ L\"owner, {\it \"Uber monotone Matrixfunktionen}, Math. Z. 
{\bf 38}, 177--216 (1934). 
%
\bi{Lo84} J.\ Lott, {\it The eta function and some new anomalies}, Phys. 
Lett. B {\bf 145}, 179--180 (1984).
%
\bi{Mc80} A.\ McIntosh, {\it Heinz inequalities and perturbation of spectral families}, 
Macquarie Mathematics Reports, Report 79-006, revised, 1980. 
%
\bi{Me93} R.\ Melrose, {\it The Atiyah-Patodi-Singer index theorem}, 
Research Notes in Mathematics, Vol.\ 4, A.\ K.\ Peters, Ltd., Wellesley, MA, 1993.
%
\bi{MMAD08} B.\ Messirdi, M.\ H.\ Mortad, A.\ Azzouz, and G.\ Djellouli, {\it A 
topological characterization of the product of two closed operators}, Coll. Math. 
{\bf 112}, 269--278 (2008). 
%
\bi{Me99} Y.\ Mezroui, {\it Le compl\'et\'e des op\'erateurs ferm\'es \`a domaine 
dense pour la m\'etrique du gap}, J. Operator Th. {\bf 41}, 69--92 (1999).
%
\bi{Mu87} W.\ M\"uller, {\it Manifolds with Cusps of Rank One. Spectral Theory 
and $L^2$-Index Theorem}, Lecture Notes in Math., Vol.\ 1244, Springer, 
Berlin, 1987. 
%
\bi{Mu88} W.\ M\"uller, {\it $L^2$-index and resonances}, in {\it Geometry 
and Analysis on Manifolds}, T.\ Sunada (ed.), Lecture Notes in Math., Vol.\ 1339, Springer, Berlin, 1988, pp.\  203--21. 
%
\bi{Mu94} W.\ M\"uller, {\it Eta invariants and manifolds with boundary}, J. Diff. 
Geom. {\bf 40}, 311--377 (1994). 
%
\bi{Mu98} W.\ M\"uller, {\it Relative zeta functions, relative determinants and 
scattering theory}, Commun. Math. Phys. {\bf 192}, 309--347 (1998). 
%
\bi{Ni07} L.\ I.\ Nicolaescu, {\it On the space of Fredholm operators}, 
 An. \c{S}tiin\c{t}. Univ. Al. I. Cuza Ia\c{s}i. Mat. (N.S.)  {\bf 53}, 209--227 (2007).
%
\bi{NS84} A.\ J.\ Niemi and G.\ W.\ Semenoff, {\it Spectral asymmetry on an 
open space}, Phys. Rev. D (3) {\bf 30}, 809--818 (1984).
%
\bi{NS86} A.\ J.\ Niemi and G.\ W.\ Semenoff, {\it Index theorems on open infinite manifolds}, 
Nuclear Phys. B  {\bf 269}, 131--169 (1986).
%
\bi{NT85} M.\ Ninomiya and C.\ I.\ Tan, {\it  Axial anomaly and index theorem for manifolds 
with boundary}, Nuclear Phys. B {\bf 257}, 199--225 (1985).
%
\bi{Nu64} A.\ E.\ Nussbaum, {\it Reduction theory for unbounded closed 
operators in Hilbert space}, Duke Math. J. {\bf 31}, 33--44 (1964). 
%
\bi{Pa65} R.\ S.\ Palais, {\it Seminar on the Atiyah--Singer Index Theorem}, 
Annals of Math. Studies, Vol.\ 57, Princeton University Press, Princeton, 1965. 
%
\bi{Pa51} R.\ Pallu de la Barri\`ere, {\it D\'ecomposition des op\'erateurs non 
born\'es dans les sommes continues d'espaces de Hilbert}, Comptes Rendus 
Acad. Sci. Paris  {\bf 232}, 2071--2073 (1951). 
%
\bi{Pe09} V.\ V.\ Peller, {\it The behavior of functions of operators under 
perturbations}, in {\it A glimpse at Hilbert space operators}, Operator Theory: Advances 
and Applications, S.\ Axler, P.\ Rosenthal, D.\ Sarason (eds.), Vol.\ 207, Birkh\"auser Verlag, 
Basel, 2010, pp.\  287--324.
%
\bi{Ph96} J.\ Phillips, {\it Self-adjoint Fredholm operators and spectral flow}, 
Canad. Math. Bull. {\bf 39}, 460--467 (1996). 
%
\bi{PS08} D.\ Potapov and F.\ Sukochev, {\sl Lipschitz and
commutator estimates in symmetric operator spaces},  J. Operator
Theory  {\bf 59}  (2008),  no. 1, 211--234.
%
\bi{PS09} D.\ Potapov and F.\ Sukochev,
{\it Unbounded Fredholm modules and double operator integrals},
J. reine angew. Math. {\bf 626}, 159--185 (2009).
%
\bi{PS10} D.\ Potapov and F.\ Sukochev, {\it Double operator integrals and submajorization}, 
Math. Model. Nat. Phenom. {\bf 5}, No.\ 4, 317--339 (2010). 
%
\bi{Pu01} A.\ Pushnitski, 
{\it The spectral shift function and the invariance principle}, J. Funct. Anal. {\bf 183} (2001), 
269--320.
%
\bi{Pu08} A.\ Pushnitski,
{\it The spectral flow, the Fredholm index, and the spectral shift 
function}, in {\it Spectral Theory of Differential Operators: M.\ Sh.\ 
Birman 80th Anniversary Collection}, T.\ Suslina and D.\ Yafaev (eds.),
AMS Translations, Ser.\  2, Advances in the Mathematical Sciences, Vol.\  225,
Amer. Math. Soc., Providence, RI, 2008, pp.\ 141--155.
%
\bi{Pu09} A.\ Pushnitski, {\it Operator theoretic methods for the eigenvalue counting function 
in spectral gaps}, Ann. H.\ Poincar\'e {\bf 10}, 793--822 (2009).  
%
\bi{Ra04} P.\ Rabier, {\it The Robbin-Salamon index theorem in Banach 
spaces with UMD}, Dyn. Partial Diff. Eqs. {\bf 1}, 303--337 (2004).
%
\bi{RS80} M.\ Reed and B.\ Simon, {\it Methods of Modern Mathematical
Physics. I: Functional Analysis},  revised and enlarged edition, Academic 
Press, New York, 1980.
%
\bi{RS75} M.\ Reed and B.\ Simon, {\it Methods of Modern Mathematical
Physics. II: Fourier Analysis, Self-Adjointness},  Academic Press, New York, 1975.
%
\bi{RS78} M.\ Reed and B.\ Simon, {\it Methods of Modern Mathematical
Physics. IV: Analysis of Operators},  Academic Press, New York, 1978.
%
\bi{RS95} J.\ Robbin and D.\ Salamon, {\it  The spectral flow and the Maslov
index},  Bull. London Math. Soc. {\bf 27}, 1--33 (1995).
%
\bi{Ro88} J.\ Roe, {\it Elliptic Operators, Topology and Asymptotic Methods}, 
Pitman Research Notes in Math. Series, Vol.\ 179, Longman Scientific \& 
Technical, Harlow, Essex, UK, 1988. 
%
\bi{Sc60} R.\ Schatten, {\it Norm Ideals of Completely Continuous Operators}, 
Springer, Berlin, 1960. 
%
\bi{Sc81} M.\ Schechter, {\it Operator Methods in Quantum Mechanics}, North 
Holland, New York, 1981.
%
\bi{Sh09} K.\ Sharifi, {\it The gap between unbounded regular operators}, 
 J. Operator Theory {\bf 65}, 241--253 (2011).
%
\bi{Si71} B.\ Simon, {\it Quantum Mechanics for Hamiltonians Defined as
Quadratic Forms}, Princeton University Press, Princeton, NJ, 1971.
%
\bi{Si05} B.\ Simon, {\it Trace Ideals and Their Applications}, 2nd ed.,
Mathematical Surveys and Monographs, Vol.\ 120, Amer. Math. Soc.,
Providence, RI, 2005.
%
\bi{Si87} I.\ M.\ Singer, {\it The $\eta$-invariant and the index}, in 
{\it Mathematical Aspects of String Theory}, S.\ T.\ Yau (ed.), Adv. Ser. Math. Phys., 
Vol.\ 1, World Scientific, Singapore, 1987, pp.\ 239--258.
%
\bi{Sp94} I.\ Spitkovsky, {\it Once more on algebras generated by two projections}, 
Linear Algebra Appl. {\bf 208/209}, 377--395 (1994).
%
\bi{St51} M.\ H.\ Stone, {\it On unbounded operators in Hilbert space}, J. Indian 
Math. Soc. {\bf 15}, 155--192 (1951). 
%
\bi{Th92} B.\ Thaller, {\it The Dirac Equation}, Texts and Monographs in 
Physics, Springer, Berlin, 1992. 
%
\bi{Wa08} C.\ Wahl, {\it A new topology on the space of unbounded 
selfadjoint operators, $K$-theory and spectral flow}, in {\it $C\sp \ast$-algebras 
and elliptic theory II}, Trends in Mathematics, D.\ Burghelea, R.\ Melrose, A.\ S.\ Mishchenko, and E.\ V.\ Troitsky (eds.), Birkh\"auser, Basel, 2008, pp.\ 297--309. 
%
\bi{We80} J.\ Weidmann, {\it Linear Operators in Hilbert Spaces},
Graduate Texts in Mathematics, Vol.\ 68, Springer, New York, 1980. 
%
\bi{Wo85} K.\ Wojciechowski, {\it Spectral flow and the general linear conjugation problem}, Simon Stevin {\bf 59}, 59--91 (1985).  
%
\bi{Ya92} D.\ R.\ Yafaev, {\it Mathematical Scattering Theory. 
General Theory}, Amer. Math. Soc., Providence, RI, 1992.
%
\bi{Ya64} V.\ A.\ Yavryan, {\it On certain perturbations of selfadjoint operators}, 
Akad. Nauk Armyan. SSR Dokl. {\bf 38}, no.\ 1, 3--7 (1964). (Russian.) 
%
\end{thebibliography}
\end{document}